\documentclass{article}

\usepackage{nicefrac}
\usepackage{wrapfig}
\usepackage[table]{xcolor}
\usepackage{amsmath}
\usepackage{amssymb}
\usepackage{bbold}
\usepackage{graphicx}
\usepackage{mathtools}
\usepackage{amsthm,bm}
\usepackage{color}
\usepackage{algorithm}
\usepackage{authblk}
\usepackage{bbm}
\usepackage{subcaption}
\usepackage{setspace}
\newcommand{\N}{\mathbb{N}}
\newcommand{\R}{\mathbb{R}}
\newcommand{\cL}{\mathcal{L}}

\newtheorem{theorem}{Theorem}[section]
\newtheorem{proposition}{Proposition}[section]
\newtheorem{lemma}{Lemma}[section]
\newtheorem{corollary}{Corollary}[section]
\newtheorem{assumption}{Assumption}[section]

\newtheorem{definition}{Definition}[section]
\newtheorem{remark}{Remark}[section]

\graphicspath{{new_figures_rho_reduced/}}
\usepackage[noend]{algpseudocode}
\usepackage{hyperref}
\usepackage[colorinlistoftodos,bordercolor=orange,backgroundcolor=orange!20,linecolor=orange,textsize=scriptsize]{todonotes}

 \usepackage{pifont}
\newcommand{\Diag}{\mathop{\mathrm{Diag}}}
\newtheorem*{theorem*}{Theorem}
\newtheorem*{corollary*}{Corollary}
\newtheorem*{lemma*}{Lemma}
\newtheorem*{proposition*}{Proposition}
\DeclareMathOperator{\argmin}{arg\,min}

\usepackage{pgfplots}
\newlength\figureheight
\newlength\figurewidth

\newcommand{\ie}{{\em i.e.,~}}

 \usepackage[preprint]{neurips_2019}

\usepackage{hyperref}       
\usepackage{url}            
\usepackage{booktabs}       
\usepackage{amsfonts}       
\usepackage{nicefrac}       
\usepackage{microtype}      
\usepackage{multirow}       
\usepackage{makecell}       
\usepackage[justification=centering]{caption} 

\newcommand{\norm}[1]{\left\lVert#1\right\rVert}
\newcommand{\dotprod}[1]{\left< #1\right>}
\newcommand{\E}[1]{\mathbb{E}\left[#1\right] }
\newcommand{\ED}[1]{\mathbb{E}_{\cD}\left[#1\right] }
\newcommand{\EE}[2]{\mathbb{E}_{#1}\left[#2\right] }
\newcommand{\Prb}[1]{\mathbb{P}\left[#1\right] }
\newcommand{\Tr}[1]{\mbox{Tr}\left( #1\right)}
\newcommand{\eqdef}{\overset{\text{def}}{=}}
\newcommand{\floor}[1]{\left\lfloor #1 \right\rfloor}

\newcommand{\Var}[1]{\mathrm{Var}\left[#1\right]}
\newcommand{\cD}{\mathcal{D}}
\newcommand{\cP}{\mathcal{\psi}}
\newcommand{\cB}{\mathcal{B}}

\newcolumntype{C}[1]{>{\centering\arraybackslash}m{#1}} 
\title{Towards closing the gap between the theory and practice of SVRG}

\author{%
   Othmane Sebbouh \\ %
  T\'{e}l\'{e}com Paris, IPP, Paris\\
  \texttt{othmane.sebbouh@gmail.com}
  \and
  Nidham Gazagnadou\\
  T\'{e}l\'{e}com Paris, IPP, Paris\\
  \texttt{nidham.gazagnadou@gmail.com}\\
  \and
  Samy Jelassi\\
  ORFE Department\\
  Princeton University\\
  \texttt{sjelassi@princeton.edu}
  \and
  Francis Bach\\
  INRIA - Ecole Normale Supérieure\\
  PSL Research University\\
  \texttt{francis.bach@inria.fr}\\
  \and
  Robert Gower\\
  T\'{e}l\'{e}com Paris, IPP, Paris\\
  \texttt{robert.gower@telecom-paristech.fr}
}

\begin{document}
\maketitle

\begin{abstract}
  Among the very first variance reduced stochastic methods for solving the empirical risk minimization problem was the SVRG method~\cite{johnson2013accelerating}. . SVRG is an inner-outer loop based method, where in the outer loop a reference full gradient is evaluated, after which $m \in \N$ steps of an inner loop are executed where the reference gradient is used to build a variance reduced estimate of the current gradient.
  The simplicity of the SVRG method and its analysis have led to multiple extensions and variants for even non-convex optimization. We provide a more general analysis of SVRG than had been previously done by using arbitrary sampling, which allows us to analyse virtually all forms of mini-batching through a single theorem. Furthermore, our analysis is focused on more practical variants of SVRG including a new variant of the loopless SVRG~\cite{HofmanNSAGA,SVRGloopless,mairal19} and a variant of k-SVRG~\cite{kSVRGstitch} where $m=n$ and where $n$ is the number of data points. Since our setup and analysis  reflect what is done in practice, we are able to set the parameters such as the mini-batch size and step size using our theory in such a way that produces a more efficient algorithm in practice, as we show in extensive numerical experiments.
\end{abstract}
\section{Introduction} \label{sec:introduction}

Consider the problem of minimizing a $\mu$--strongly convex and $L$--smooth  function $f$ where
\begin{equation}\label{eq:prob}
x^* = \arg\min_{x \in \R^d} \frac{1}{n}\sum_{i=1}^n f_i(x) =: f(x),
\end{equation}
and each $f_i$ is convex and $L_i$--smooth.  Several training problems in machine learning  fit this format,
e.g. least-squares, logistic regressions and conditional random fields. Typically each $f_i$ represents a regularized loss of an $i$th data point. When $n$ is  large, algorithms that rely on full passes over the data, such as gradient descent, are no longer competitive. Instead, the stochastic version of gradient descent SGD~\cite{robbins1985convergence} is often used since it requires only a mini-batch of data to make progress towards the solution.  However, SGD suffers from high variance, which keeps the algorithm from converging unless a carefully often hand-tuned decreasing sequence of step sizes is chosen. This often results in a cumbersome parameter tuning and a slow convergence.

To address this issue, many variance reduced methods have been designed in recent years including SAG~\cite{SAG}, SAGA~\cite{SAGA} and SDCA~\cite{shalev2013stochastic} that require only a constant step size to achieve linear convergence. In this paper, we are interested in variance reduced methods with an inner-outer loop structure, such as S2GD~\cite{konevcny2013semi}, SARAH~\cite{SARAH}, L-SVRG~\cite{SVRGloopless} and the orignal SVRG~\cite{johnson2013accelerating} algorithm.
Here we present not only a more general analysis that allows for any mini-batching strategy, but also a more \emph{practical} analysis, by analysing methods that are based on what works in practice, and thus providing an analysis that can inform practice.

\section{Background and Contributions} \label{sec:background_contributions}

\begin{wrapfigure}{r}{0.53\textwidth}
    \centering
   \begin{minipage}{0.265\textwidth}
  \centering
 \includegraphics[width =\textwidth,height = 0.8\textwidth]{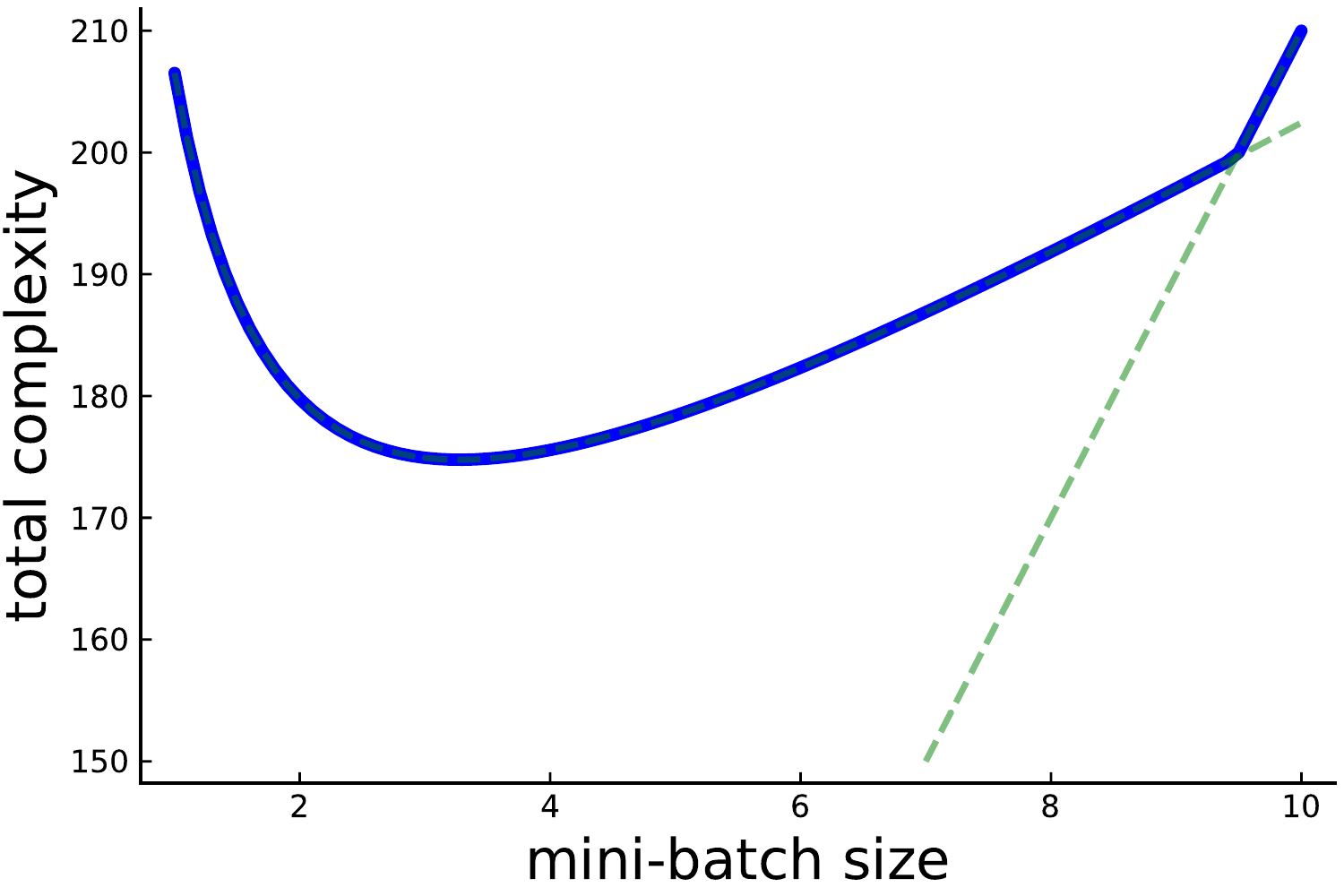}
 \end{minipage}%
\begin{minipage}{0.265\textwidth}
  \centering
  \includegraphics[width =\textwidth,, height = 0.8\textwidth]{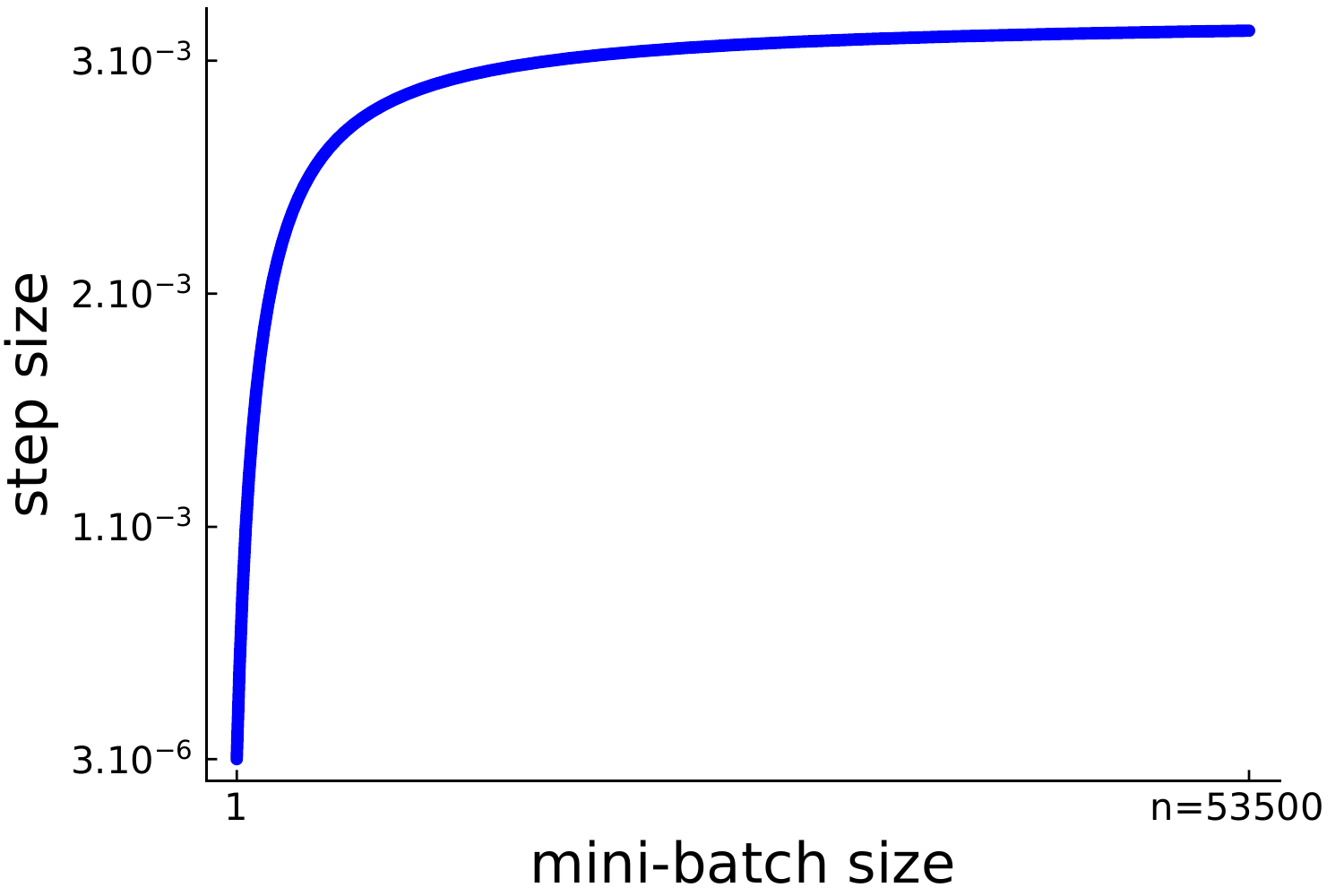}
 \end{minipage}%
       \caption{\small Left:  the total complexity~\eqref{eq:totcomplexintro} for random Gaussian data, right: the step size~\eqref{eq:stepsizeintro} as $b$ increases.}
\label{fig:intro}
\end{wrapfigure}

\paragraph{Convergence under arbitrary samplings.} We give the first arbitrary sampling convergence results for SVRG type methods in the convex setting\footnote{SVRG has very recently been analysed under arbitrary samplings in the non-convex setting~\cite{SVRG-AS-nonconvex}.}. That is our analysis includes all  forms of sampling including mini-batching and importance sampling as a special case. To better understand the significance of this result, we use mini-batching $b$ elements \emph{without} replacement as a running example throughout the paper. With this sampling the update step of SVRG, starting from $x^0=w_0 \in \R^d$, takes the form of
\begin{equation}\label{eq:SVRG_update}
x^{t+1} = x^t - \alpha \left(\frac{1}{b}\sum_{i\in B}\nabla f_i(x^t) - \frac{1}{b}\sum_{i\in B} \nabla f_i(w_{s-1}) + \nabla f(w_{s-1})\right),
\end{equation}
where $\alpha >0$ is the step size, $B \subseteq [n] \eqdef \{1,\dots,n\}$ and $b = |B|$. Here $w_{s-1}$ is the \emph{reference point} which is updated after $m \in \N$ steps, the $x^t$'s are the \emph{inner iterates} and $m$ is the loop length. As a special case of our forthcoming analysis in Corollary~\ref{cor:complex_total_f-SVRG}, we show that the \emph{total complexity} of the  SVRG method based on~\eqref{eq:SVRG_update} to reach  an $\epsilon >0$ accurate solution
has a simple expression which depends on $n$, $m$, $b$, $\mu$, $L$ and $L_{\max} \eqdef \max_{i\in[n]}L_i$:
\begin{equation}
C_m(b) \; \eqdef \; 2\left( \frac{n}{m}+2b \right)\max\left\{ \frac{3}{b}\frac{n-b}{n-1}\frac{L_{\max}}{\mu} + \frac{n}{b}\frac{b-1}{n-1}\frac{L}{\mu}   , m\right\}\log\left(\frac{1}{\epsilon}\right), \label{eq:totcomplexintro}
\end{equation}
so long as the step size is
\begin{equation} \label{eq:stepsizeintro}
\alpha \; = \;\frac{1}{2}\;\frac{b(n-1)}{3(n-b)L_{\max}+n(b-1)L}.
\end{equation}
By total complexity we mean the total number of individual  $\nabla f_i$ gradients evaluated. This shows that the total complexity is a simple function of the loop length $m$ and the mini-batch size $b$. See Figure~\ref{fig:intro} for an example for how total complexity evolves as we increase the mini-batch size.


\paragraph{Optimal mini-batch and step sizes for SVRG.}
The size of the mini-batch $b$ is often left as a parameter for the user to choose or set using a rule of thumb.
The current analysis in the literature for mini-batching shows that when increasing the mini-batch size $b$, while the iteration complexity can decrease\footnote{Note that the total complexity is equal to the iteration complexity times the mini-batch size $b$.}, the total complexity increases or is invariant. See for instance results in the
 non-convex case~\cite{Nitanda2014,Reddi2016}, and for the  convex case~\cite{Konecny:wastinggrad:2015},~\cite{Konecny2015},~\cite{Allenzhu2017-katyusha} and finally~\cite{Murata:2017} where one can find the iteration complexity of several variants of SVRG with mini-batching.
 However, in practice, mini-batching can often lead to faster algorithms.
 In contrast our total complexity~\eqref{eq:totcomplexintro} clearly highlights that when increasing the mini batch size,
the total complexity can decrease and the step size increases, as can be seen in our plot of~\eqref{eq:totcomplexintro} and~\eqref{eq:stepsizeintro} in Figure~\ref{fig:intro}. Furthermore $C_m(b)$ is a convex function in $b$ which  allows us to determine the optimal mini-batch a priori. For $m=n$ -- a widely used loop length in practice -- the optimal mini-batch size, depending on the problem setting, is given in Table \ref{tab:optimal_mini-batch}.
\begin{table}[h]
\begin{center}
\setlength{\extrarowheight}{5pt}
\begin{tabular}{ | C{1.7cm} | C{1.7cm} | C{2.5cm} | C{2.5cm} | C{1.75cm}|}
\hline
\multicolumn{2}{|c|}{$n \leq \frac{L}{\mu}$} & \multicolumn{2}{|c|}{$ \frac{L}{\mu}< n < \frac{3L_{\max}}{\mu}$} & \multirow{2}{8em}{\vspace{3mm} \\ $n \geq \frac{3L_{\max}}{\mu}$} \\[10pt]
\cline{1-4}
$L_{\max} \geq \frac{nL}{3}$ & $L_{\max} < \frac{nL}{3}$ & $L_{\max} \geq \frac{nL}{3}$  & $L_{\max} < \frac{nL}{3}$ &  \\[10pt]
\hline
$n$ & $\floor{\hat{b}}$ & $\floor{\tilde{b}}$ & $\floor{\min\{\hat{b}, \tilde{b}\}}$ & $1$  \\[10pt]
\hline
\end{tabular}
\caption{Optimal mini-batch sizes for Algorithm~\ref{alg:f-SVRG} with $m = n$. The last line presents the optimal mini-batch sizes depending on all the possible problem settings, which are presented in the first two lines. Notations: ${\hat{b} = \sqrt{\frac{n}{2}\frac{3L_{\max} - L}{nL - 3L_{\max}}}}$, $\tilde{b} =  \frac{(3L_{\max} - L)n}{n(n-1)\mu- nL + 3 L_{\max}}.$}
\label{tab:optimal_mini-batch}
\end{center}
\end{table}
Moreover, we can also determine the optimal loop length.
The reason we were able to achieve these new tighter mini-batch complexity bounds was by using the recently introduced concept of expected smoothness~\cite{JakSketch} alongside a new constant we introduce in this paper called the expected residual constant. The expected smoothness and residual constants, which we present later in Lemmas~\ref{lem:exp_smooth} and~\ref{lem:exp_residual}, show how mini-batching (and arbitrary sampling in general) combined with the smoothness of our data can determine how smooth in expectation our resulting mini-batched functions are. The expected smoothness constant has been instrumental in providing a tight mini-batch analysis for SGD~\cite{SGD-AS}, SAGA~\cite{SAGAminib} and now SVRG.

\paragraph{Practical variants.}
We took special care so that our analysis allows for practical parameter settings. In particular, often the loop length is set to $m=n$ or $m=n/b$ in the case of mini-batching\footnote{See for example the lightning package from  \texttt{scikit-learn} \cite{scikit-learn}:  \href{http://contrib.scikit-learn.org/lightning/}{http://contrib.scikit-learn.org/lightning/} and \cite{SARAH} for examples where $m=n$. See \cite{allen-zhua16} for an example where $m=5n/b$.}. And yet, the classical SVRG analysis given in~\cite{johnson2013accelerating}
requires $m \geq 20 L_{\max}/\mu$  in order to ensure a resulting iteration complexity of $O((n + L_{\max}/\mu)\log(1/\epsilon))$. Furthermore, the standard SVRG analysis relies on averaging the $x^t$ inner iterates after every $m$ iterations of~\eqref{eq:SVRG_update}, yet this too is not what works well in practice\footnote{Perhaps an exception to the above issues in the literature is the Katyusha method and its analysis~\cite{Allenzhu2017-katyusha}, which is an accelerated variant of SVRG. In~\cite{Allenzhu2017-katyusha} the author shows that using a loop length $m=2n$ and by not averaging the inner iterates, the  Katyusha method achieves the accelerated complexity of $O((n + \sqrt{(n L_{\max})/\mu})\log(1/\epsilon))$. Though a remarkable advance in the theory of accelerated methods,  the analysis in~\cite{Allenzhu2017-katyusha} does not hold for the unaccelerated case. This is important since, contrary to the name, the accelerated variants of stochastic methods are not always faster than their non-accelerated counterparts. Indeed, acceleration only helps in the stochastic setting when  $L_{\max}/\mu \geq n,$ in other words for problems that are sufficiently ill-conditioned.}. To remedy this, we propose~\textit{Free-SVRG}, a variant of SVRG where the inner iterates are not averaged at any point. Furthermore, by developing a new Lyapunov style convergence for~\textit{Free-SVRG}, our analysis holds for any choice of $m$, and in particular, for $m=n$ we show that the resulting complexity is also given by  $O((n + L_{\max}/\mu)\log(1/\epsilon))$. 

We would like to clarify that, after submitting this work in 2019, it come to our attention that \textit{Free-SVRG} is a special case of \textit{k-SVRG} (2018) ~\cite{kSVRGstitch} when $k=1$.

The only downside of \textit{Free-SVRG} is that the reference point is set using a weighted averaging based on the strong convexity parameter. To fix this issue, \cite{HofmanNSAGA}, and later \cite{SVRGloopless, mairal19}, proposed a loopless version of SVRG. This loopless variant has no explicit inner-loop structure, it instead updates the reference point based on a coin toss and lastly requires no knowledge of the strong convexity parameter and no averaging whatsoever.
  We introduce \textit{L-SVRG-D}, an improved variant of \textit{Loopless-SVRG} that takes much larger step sizes
after the reference point is reset, and gradually smaller step sizes thereafter. We provide an complexity analysis of \textit{L-SVRG-D} that allows for arbitrary sampling and
  achieves the same complexity as \textit{Free-SVRG}, albeit at the cost of introducing more variance into the procedure due to the coin toss.

\section{Assumptions and Sampling}\label{sec:assump_reform}

We collect all of the assumptions we use in the following.
\begin{assumption}\label{ass:smoothness}
There exist $L\geq 0$ and $\mu \geq 0$  such that for all $x,y \in \mathbb{R}^d$,
\begin{eqnarray}
f(x) &\leq &  f(y) + \langle \nabla f(y), x - y \rangle + \frac{L}{2}\norm{x - y}_2^2, \label{eq:smoothness_f} \\
f(x)  &\leq & f(y) + \dotprod{\nabla f(x), x- y} -\frac{\mu}{2} \norm{x - y}_2^2. \label{eq:strconv2}
\end{eqnarray}
We say that $f$ is $L$--smooth \eqref{eq:smoothness_f} and $\mu$--strongly convex \eqref{eq:strconv2}.
Moreover, for all $i \in [n]$, $f_i$ is convex and there exists $L_i \geq 0$ such that $f_i$ is $L_i$--smooth.
\end{assumption}





So that we can analyse all forms of mini-batching simultaneously through arbitrary sampling we make use of a \emph{sampling vector}.

\begin{definition}[The sampling vector]\label{ass:unbsn} We say that the random vector $v = \left[v_1,\dots,v_n\right] \in \R^n$ with distribution $\cD$ is a sampling vector if  $\ED{v_i}  =1$ for all $i \in [n].$
\end{definition}
With a sampling vector we can compute an unbiased estimate of $f(x)$ and $\nabla f(x)$ via
\begin{equation} \label{eq:fvt_exp}
f_{v}(x) \eqdef \frac{1}{n} \sum_{i=1}^n v_if_i(x) \quad \mbox{and}\quad \nabla f_v(w) \; \eqdef \;  \frac{1}{n} \sum_{i=1}^n v_i\nabla f_i(x).
\end{equation}
Indeed these are unbiased estimators since
\begin{equation} \label{eq:unbiased}
\ED{f_{v}(x)} \; = \;  \frac{1}{n} \sum_{i=1}^n \ED{v_i}f_i(x) \; = \;\frac{1}{n} \sum_{i=1}^n f_i(x) \; = \; f(x).
\end{equation}
Likewise we can show that $\ED{\nabla f_{v}(x)} = \nabla f(x).$
Computing $\nabla f_{v}$ is cheaper than computing the full gradient $\nabla f$ whenever $v$ is a  sparse vector. In particular, this is the case when the support of $v$ is based on a mini-batch sampling.



\begin{definition}[Sampling]\label{def:sampling}
A sampling  $S\subseteq [n]$ is any random set-valued map which is uniquely defined by the probabilities $\sum_{B \subseteq [n]} p_B =1$ where
$p_{B} \;\eqdef \; \mathbb{P}(S=B)$ for all $B \subseteq [n]$.
A sampling $S$ is called proper if for every $i \in [n]$, we have that $p_i \eqdef \Prb{i \in S} =  \underset{C:i\in C}{\sum}p_C > 0$.
\end{definition}
We can build a sampling vector using sampling as follows.
\begin{lemma}[Sampling vector] \label{lem:sampling_vector}
Let $S$ be a proper sampling. Let $p_i \eqdef \Prb{i \in S}$ and $\mathbf{P} \eqdef \Diag\left(p_1,\dots,p_n\right)$. Let $v = v(S)$ be a random vector defined by
\begin{eqnarray} \label{eq:vSdef}
v(S) \;= \; \mathbf{P}^{-1}\sum_{i \in S}e_i  \;\eqdef \; \mathbf{P}^{-1}e_S.
\end{eqnarray}
It follows that $v$ is a sampling vector.
\end{lemma}
\emph{Proof.}
The $i$-th coordinate of $v(S)$ is $v_i(S) =  \mathbb{1}(i \in S) / p_i$ and thus
\[\E{v_i(S)}\; =\; \frac{\E{\mathbb{1}(i \in S)}}{ p_i} \;=\; \frac{\Prb{i \in S}}{p_i} \;= \;1.  \qquad \qquad \qed\]

Our forthcoming analysis holds for all samplings. However, we will pay particular attention to $b$-nice sampling, otherwise known as mini-batching without replacement, since it is often used in practice.

\begin{definition}[$b$-nice sampling] \label{def:bnice_sampling}
$S$ is a $b$-nice sampling if it is sampling such that
  \[\Prb{S = B} = \frac{1}{\binom{n}{b}}, \quad \forall	B \subseteq [n] : |B| = b.\]
\end{definition}

To construct such a sampling vector based on the $b$--nice sampling, note that $p_i = \tfrac{b}{n}$ for all $i \in [n]$ and thus we have that $v(S) = \tfrac{n}{b}\sum_{i\in S}e_i$ according to Lemma \ref{lem:sampling_vector}. The resulting subsampled function is then $f_v(x) = \tfrac{1}{|S|}\sum_{i\in S}f_i(x)$, which is simply the mini-batch average over $S$.

Using arbitrary sampling also allows us to consider non-uniform samplings, and for completeness, we present this sampling and several others in Appendix \ref{sec:samplings_appendix}.

\section{\textit{Free-SVRG}: freeing up the inner loop size}
Similarly to SVRG, \textit{Free-SVRG} is an inner-outer loop variance reduced algorithm. It differs from the original SVRG \cite{johnson2013accelerating} on two major points: how the reference point is reset and how the first iterate of the inner loop is defined, see Algorithm~\ref{alg:f-SVRG}.

First, in SVRG, the reference point is the average of the iterates of the inner loop. Thus, old iterates and recent iterates have equal weights in the average. This is counterintuitive as one would expect that to reduce the variance of the gradient estimate used in \eqref{eq:SVRG_update}, one needs a reference point which is closer to the more recent iterates. This is why, inspired by \cite{Nesterov-average}, we use the weighted averaging in \textit{Free-SVRG} given in~\eqref{eq:Smpts}, which gives more importance to recent iterates compared to old ones.

Second, in SVRG, the first iterate of the inner loop is reset to the reference point. Thus, the inner iterates of the algorithm are not updated using a one step recurrence. In contrast, \textit{Free-SVRG} defines the first iterate of the inner loop as the last iterate of the previous inner loop, as is also done in practice. These changes and a new Lyapunov function analysis are what allows us to freely choose the size of the inner loop\footnote{Hence the name of our method \emph{Free}-SVRG.}.
To declutter the notation, we define for a given step size $\alpha > 0$:
\begin{equation}
    S_m \eqdef \sum_{i=0}^{m-1}(1-\alpha\mu)^{m-1-i} \quad \mbox{and} \quad p_t \eqdef  \frac{(1 - \alpha\mu)^{m-1-t}}{S_m}, \quad \mbox{for }t=0,\ldots, m-1. \label{eq:Smpts}
\end{equation}

\begin{algorithm}
    \begin{algorithmic}
        \State \textbf{Parameters} inner-loop length $m$, step size $\alpha$, a sampling vector $v \sim \cD$, and $p_t$ defined in \eqref{eq:Smpts}
        \State \textbf{Initialization} $w_0 = x_0^m \in \mathbb{R}^d$
        \For {$s=1, 2,\dots, S$}\vskip 1ex
            \State $x_s^0 = x_{s-1}^m$ 
            \For {$t=0, 1,\dots, m-1$}\vskip 1ex
            \State Sample $v_t \sim \cD$
            \State $g_s^t = \nabla f_{v_t}(x_s^t)- \nabla f_{v_t}(w_{s-1}) + \nabla f(w_{s-1})$ 
            \State $x_s^{t+1} = x_s^t - \alpha g_s^t$
        \EndFor
        \State $w_s = \sum_{t=0}^{m-1} p_t x_s^t$
        \EndFor
        \State
        \Return $x_S^m$
    \end{algorithmic}
    \caption{\textit{Free-SVRG}}
    \label{alg:f-SVRG}
\end{algorithm}

\subsection{Convergence analysis}
Our analysis relies on two important constants called the \textit{expected smoothness} constant and the \textit{expected residual} constant. Their existence is a result of the smoothness of the function $f$ and that of the individual functions $f_i, i \in [n]$.

\begin{lemma}[Expected smoothness, Theorem~3.6 in~\cite{SGD-AS}]\label{lem:exp_smooth} Let $v \sim \cD$ be a sampling vector and assume that Assumption~\ref{ass:smoothness} holds. There exists $\mathcal{L} \geq 0$ such that for all $x \in \R^d$,
\begin{equation}\ED{\norm{\nabla f_{v}(x) - \nabla f_{v}(x^*)}_2^2} \leq 2\mathcal{L}\left(f(x)-f(x^*) \right).\label{eq:Expsmooth}\end{equation}
\end{lemma}

\begin{lemma}[Expected residual]
\label{lem:exp_residual}Let $v \sim \cD$ be a sampling vector and assume that Assumption~\ref{ass:smoothness} holds. There exists $\mathcal{\rho} \geq 0$ such that for all $x \in \R^d$,
\begin{equation}\ED{\norm{\nabla f_{v}(x) - \nabla f_{v}(x^*)  - \nabla f(x)}_2^2} \leq 2\rho \left(f(x)-f(x^*) \right).\label{eq:Expresidual}\end{equation}
\end{lemma}

For completeness, the proof of Lemma~\ref{lem:exp_smooth} is given in Lemma~\ref{lemma:master_lemma} in the supplementary material. The proof of Lemma \ref{lem:exp_residual} is also given in the supplementary material, in Lemma~\ref{lem:exp_residual_general}. Indeed, all proofs are deferred to the supplementary material.

Though Lemma \ref{lem:exp_smooth} establishes the existence of the expected smoothness $\cL$, it was only very recently that a tight estimate of $\cL$ was conjectured in~\cite{SAGAminib} and proven in~\cite{SGD-AS}. In particular, for our working example of $b$--nice sampling, we have that the constants $\cL$ and $\rho$ have simple closed formulae that depend on $b$.
 \begin{lemma}[$\cL$ and $\rho$ for $b$-nice sampling] \label{lem:exp_smooth_residual_bnice} Let $v$ be a sampling vector based on the $b$--nice sampling. It follows that.
\begin{eqnarray}
\cL =\cL(b) &\eqdef & \frac{1}{b}\frac{n-b}{n-1}L_{\max} + \frac{n}{b}\frac{b-1}{n-1}L, \label{reminder_L_1} \\
\rho = \rho(b) &\eqdef& \frac{1}{b}\frac{n-b}{n-1}L_{\max}. \label{reminder_rho_1}
\end{eqnarray}
\end{lemma}
The reason that the expected smoothness and expected residual constants are so useful in obtaining a tight mini-batch analysis is because, as the mini-batch size $b$ goes from $n$ to $1$,  $\cL(b)$ (resp. $\rho(b)$) gracefully interpolates between the smoothness of the full function $\cL(n) = L$ (resp. $\rho(n) = 0$), and the smoothness of the individual $f_i$ functions $\cL(1) =L_{\max}$ (resp $\rho(1) = L_{\max}$).
%
Also, we can bound the second moment of a variance reduced gradient estimate using $\cL$ and $\rho$ as follows.
\begin{lemma}\label{lem:gt_sq_bndfval_opt3}
Let Assumption~\ref{ass:smoothness} hold. Let $x, w \in \mathbb{R}^d$ and $v \sim \cD$ be sampling vector. Consider $g(x, w) \eqdef \nabla f_{v}(x) - \nabla f_{v}(w) + \nabla f(w)$. As a consequence of~\eqref{eq:Expsmooth} and~\eqref{eq:Expresidual} we have that
\begin{equation}\ED{\|g(x, w)\|_2^2} \leq 4 \cL (f(x) - f(x^*))+ 4\rho(f(w) -  f(x^*)).
\label{eq:gt_sq_bndfval_opt3}\end{equation}
\end{lemma}

 Next we present a new Lyapunov style convergence analysis through which we will establish the convergence of the iterates and the function values simultaneously.
\begin{theorem} \label{convergence_f-SVRG}
Consider the setting of Algorithm~\ref{alg:f-SVRG} and the following Lyapunov function
\begin{eqnarray}
\phi_s \; \eqdef \; \norm{x_s^m - x^*}_2^2 + \cP_s \quad \mbox{where}\quad \cP_s \; \eqdef \; 8\alpha^2\rho S_m (f(w_s) - f(x^*)) .\label{eq:Dsdef}
\end{eqnarray}
If Assumption~\ref{ass:smoothness} holds and  if $ \alpha \leq \frac{1}{2(\cL+2\rho)}$, then
\begin{eqnarray}
\E{\phi_s} \; \leq  \; \beta^s \phi_0, \quad \mbox{where} \quad \beta = \max\left\{(1 - \alpha\mu)^m, \tfrac{1}{2}\right\}.
\end{eqnarray}
\end{theorem}



\subsection{Total complexity for $b$--nice sampling}
To gain better insight into the convergence rate stated in Theorem \ref{convergence_f-SVRG}, we present the total complexity of Algorithm \ref{alg:f-SVRG} when $v$ is defined via the $b$--nice sampling introduced in Definition \ref{def:bnice_sampling}.

\begin{corollary}\label{cor:complex_total_f-SVRG}
Consider the setting of Algorithm \ref{alg:f-SVRG} and suppose that we use $b$--nice sampling. Let $\alpha = \frac{1}{2(\cL(b) + 2\rho(b))}$, where $\cL(b)$ and $\rho(b)$ are given in~\eqref{reminder_L_1} and~\eqref{reminder_rho_1}. We have that the total complexity of finding an $\epsilon > 0$ approximate solution that satisfies $\E{\norm{x_s^m - x^*}_2^2} \leq \epsilon \, \phi_0$ is
\begin{eqnarray}
C_m(b) \quad \eqdef \quad 2 \left(\frac{n}{m}+2b\right) \max \left\{\frac{\cL(b) + 2\rho(b)}{\mu}, m\right\}\log\left(\frac{1}{\epsilon}\right). \label{eq:total_cplx_def}
\end{eqnarray}
\end{corollary}
Now~\eqref{eq:totcomplexintro} results from plugging~\eqref{reminder_L_1} and~\eqref{reminder_rho_1}  into~\eqref{eq:total_cplx_def}.
As an immediate sanity check, we check the two extremes $b=n$ and $b=1$. When $b=n$, we would expect to recover the iteration complexity of gradient descent, as we do in the next corollary\footnote{Though the resulting complexity is $6$ times the tightest gradient descent complexity, it is of the same order. }.
\begin{corollary}
Consider the setting of Corollary~\ref{cor:complex_total_f-SVRG} with $b=n$ and $m=1$, thus $\alpha = \frac{1}{2(\cL(n)+2\rho(n))} = \frac{1}{2L}$. Hence, the resulting total complexity~\eqref{eq:total_cplx_def} is given by
$C_1(n) \; = \; 6 n \tfrac{L}{\mu} \log\left(\tfrac{1}{\epsilon}\right). $
\end{corollary}

In practice, the most common setting is choosing $b=1$ and the size of the inner loop $m = n$. Here we recover a complexity that is common to other non-accelerated algorithms \cite{SAG}, \cite{SAGA}, \cite{konevcny2013semi}, and for a range of values of $m$ including $m=n.$


\begin{corollary}\label{cor:total_comp_m_1nice}
Consider the setting of Corollary~\ref{cor:complex_total_f-SVRG} with $b=1$ and thus $\alpha = \frac{1}{2(\cL(1)+2\rho(1))} = \frac{1}{6L_{\max}}$. Hence the resulting total complexity~\eqref{eq:total_cplx_def} is given by $C_m(1) \; = \;  18\left(n + \tfrac{L_{\max}}{\mu}\right)\log\left(\tfrac{1}{\epsilon}\right),$
so long as $m \in \left[\min(n, \tfrac{L_{\max}}{\mu}), \max(n, \tfrac{L_{\max}}{\mu})\right]$.
\end{corollary}
Thus total complexity is essentially invariant for $m =n$, $m =L_{\max}/\mu$ and everything in between.
\section{\textit{L-SVRG-D}: a decreasing step size approach}
Although \textit{Free-SVRG} solves multiple issues regarding the construction and analysis of SVRG, it still suffers from an important issue: it requires the knowledge of the strong convexity constant, as is the case for the original SVRG algorithm \cite{johnson2013accelerating}. One can of course use an explicit small regularization parameter as a proxy, but this can result in a slower algorithm.

A loopless variant of SVRG was proposed and analysed in~\cite{HofmanNSAGA,SVRGloopless,mairal19}. At each iteration, their method makes a coin toss. With (a low) probability $p$, typically $1/n$, the reference point is reset to the previous iterate, and with probability $1-p$, the reference point remains the same. This method does not require knowledge of the strong convexity constant.

Our method, \textit{L-SVRG-D}, uses the same loopless structure as in~\cite{HofmanNSAGA, SVRGloopless, mairal19} but introduces different step sizes at each iteration, see Algorithm~\ref{alg:L-SVRG-D}. We initialize the step size to a fixed value $\alpha >0$. At each iteration we toss a coin, and if it lands heads  (with probability $1 - p$) the step size decreases by a factor $\sqrt{1-p}$. If it lands tails (with probability $p$) the reference point is reset to the most recent iterate and the step size is reset to its initial value $\alpha$.
%

This allows us to take larger steps than \textit{L-SVRG} when we update the reference point, \ie when the variance of the unbiased estimate of the gradient is low, and smaller steps when this variance increases.

\begin{algorithm}
  \begin{algorithmic}
    \State \textbf{Parameters} step size $\alpha$, $p\in (0,1]$, and a sampling vector $v \sim \cD$
    \State \textbf{Initialization}   $w^0 = x^0 \in \mathbb{R}^d,\; \alpha_0 = \alpha$
    \For {$k=1, 2,\dots, K-1$}\vskip 1ex
      \State Sample $v_k \sim \cD$
      \State $g^k = \nabla f_{v_k}(x^k)- \nabla f_{v_k}(w^k) + \nabla f(w^k)$ 
      \State $x^{k+1} = x^k - \alpha_k g^k$
      \State $ (w^{k+1}, \alpha_{k+1}) = \left\{
      \begin{array}{ll}
          (x^k, \alpha) & \mbox{with probability }p \\
          (w^{k}, \sqrt{1 - p} \; \alpha_k) & \mbox{with probability } 1-p
      \end{array} \right.$
    \EndFor
    \State
    \Return $x^K$
  \end{algorithmic}
  \caption{L-SVRG-D}
  \label{alg:L-SVRG-D}
\end{algorithm}

\begin{theorem}\label{thm_dec_step}
Consider the iterates  of Algorithm \ref{alg:L-SVRG-D} and the following Lyapunov function
\begin{eqnarray}
\phi^k \eqdef \norm{x^k - x^*}_2^2 + \cP^k \quad \mbox{where} \quad  \cP^k \eqdef  \frac{8 \alpha_k^2 \cL}{p(3 - 2p)} \left(f(w^k) - f(x^*)\right),  \quad\forall k \in \N . \label{gradlearn_dec_def}
\end{eqnarray}
If Assumption~\ref{ass:smoothness} holds and
\begin{eqnarray}
\alpha \leq \frac{1}{2\zeta_p\cL}, \quad \text{where} \quad \zeta_p \; \eqdef \; \frac{(7-4p)(1 - (1 - p)^{\frac{3}{2}})}{p(2-p)(3 - 2p)}, \label{eq:alphadecreasingcyclic}
\end{eqnarray}
then
\begin{eqnarray}
\E{\phi^k} \leq \beta^k \phi^0, \quad \text{where} \quad \beta \; = \; \max\left\{1 - \frac{2}{3}\alpha \mu, 1 - \frac{p}{2}\right\}.
\end{eqnarray}
\end{theorem}
\begin{remark}
To get a sense of the formula of the step size given in~\eqref{eq:alphadecreasingcyclic}, it is easy to show that $\zeta_p$ is an increasing function of $p$ such that
$7/4\; \leq \; \zeta_p \; \leq \; 3.$
Since typically  $p \approx 0$, we often take a step which is approximately $\alpha \leq 2/(7\cL)$.
\end{remark}

\begin{corollary}\label{cor:total_cplx_LDSVRG}
Consider the setting of Algorithm \ref{alg:L-SVRG-D} and suppose that we use $b$--nice sampling. Let $\alpha = \frac{1}{2\zeta_p\cL(b)}$. We have that the total complexity of finding an $\epsilon > 0$ approximate solution that satisfies $\E{\norm{x^k - x^*}_2^2} \leq \epsilon \, \phi^0$ is
\begin{equation}\label{eq:total_cplx_dsvrg}
C_p(b)\, \eqdef \, 2(2b+pn)\max\left\{\frac{3\zeta_p}{2}\frac{\cL(b)}{\mu}, \frac{1}{p}\right\}\log\left(\frac{1}{\epsilon}\right).
\end{equation}
\end{corollary}

\section{Optimal parameter settings: loop, mini-batch and step sizes}\label{sec:optim_param}

In this section, we restrict our analysis to $b$--nice sampling. First, we determine the optimal loop size for Algorithm \ref{alg:f-SVRG}. Then, we examine the optimal mini-batch and step sizes for particular choices of the inner loop size $m$ for Algorithm \ref{alg:f-SVRG} and of the probability $p$ of updating the reference point in Algorithm \ref{alg:L-SVRG-D}, that play analogous roles. Note that the steps used in our algorithms depend on $b$ through the expected smoothness constant $\cL(b)$ and the expected residual constant $\rho(b)$. Hence, optimizing the total complexity in the mini-batch size also determines the optimal step size.

Examining the total complexities of Algorithms \ref{alg:f-SVRG} and \ref{alg:L-SVRG-D}, given in \eqref{eq:total_cplx_def} and \eqref{eq:total_cplx_dsvrg}, we can see that, when setting $p = 1/m$ in Algorithm \ref{alg:L-SVRG-D}, these complexities only differ by constants. Thus, to avoid redundancy, we present the optimal mini-batch sizes for Algorithm \ref{alg:L-SVRG-D} in Appendix \ref{sec:optim_mini-batch_d-SVRG} and we only consider here the complexity of Algorithm \ref{alg:f-SVRG} given in~\eqref{eq:total_cplx_def}.

\subsection{Optimal loop size for Algorithm \ref{alg:f-SVRG}}
Here we determine the optimal value of $m$ for a fixed batch size $b$, denoted by $m^*(b)$, which minimizes the total complexity \eqref{eq:total_cplx_def}.

\begin{proposition}\label{prop:optimalloopsize} The loop size that minimizes~\eqref{eq:total_cplx_def} and the resulting total complexity is given by
\begin{eqnarray}
m^*(b) = \frac{\cL(b)+2\rho(b)}{\mu} \quad \text{ and }\quad  C_{m^*}(b) = 2\left(n + 2b\frac{\cL(b)+2\rho(b)}{\mu}\right)\log\left(\frac{1}{\epsilon}\right). \label{eq:optimal_m}
\end{eqnarray}
\end{proposition}

For example when $b\!=\!1$, we have that $m^*(1) = 3L_{\max}/\mu$ and $C_{m^*}(1) = O((n + L_{\max}/\mu)\log(1/\epsilon)),$ which is the same complexity as achieved by the range of $m$ values given in Corollary~\ref{cor:total_comp_m_1nice}. Thus, as we also observed in Corollary~\ref{cor:total_comp_m_1nice}, the total complexity is not very sensitive to the choice of $m$, and $m=n$ is a perfectly safe choice as it achieves the same complexity as $m^*$. We also confirm this  numerically with a series of experiments in Section~\ref{sec:app_exp_inner_loop}.

\subsection{Optimal mini-batch and step sizes}
%
%

In the following proposition, we determine the optimal mini-batch and step sizes for two practical choices of the size of the loop $m$.
\begin{proposition}\label{prop:optim_batch}
Let $b^* \eqdef \underset{b \in [n]}{\argmin}\,C_m(b)$, where $C_m(b)$ is defined in \eqref{eq:total_cplx_def}.
For the widely used choice $m = n$, we have that $b^*$ is given by Table~\ref{tab:optimal_mini-batch}.
For another widely used choice $m = n/b$, which allows to make a full pass over the data set during each inner loop, we have
\begin{eqnarray}
b^* = \left\{
            \begin{array}{ll}
            		\floor{\bar{b}} & \mbox{if } n > \frac{3L_{\max}}{\mu}\\
            		1 & \mbox{if } \frac{3L_{\max}}{L}<n\leq \frac{3L_{\max}}{\mu}\\
                n & \mbox{otherwise, if } n \leq \frac{3L_{\max}}{L}
            \end{array} \right.,
            \quad \mbox{where}\ \ \bar{b} \, \eqdef \, \frac{n(n-1)\mu - 3n(L_{\max} - L)}{3(nL - L_{\max})} \enspace.
\end{eqnarray}

\end{proposition}

Previously, theory showed that the total complexity would increase as the mini-batch size increases, and thus established that single-element sampling was optimal.  However, notice that for $m = n$ and $m = n/b$, the usual choices for $m$ in practice, the optimal mini-batch size is different than $1$ for a range of problem settings. Since our algorithms are closer to the SVRG variants used in practice, we argue that our results explain why practitioners experiment that mini-batching works, as we verify in the next section.

\section{Experiments}

We performed a series of experiments on data sets from LIBSVM~\cite{chang2011libsvm} and the UCI repository~\cite{asuncion2007uci}, to validate our theoretical findings. We tested $l_2$--regularized logistic regression on \textit{ijcnn1} and \textit{real-sim}, and ridge regression on \textit{slice} and \textit{YearPredictionMSD}. We used two choices for the regularizer: $\lambda = 10^{-1}$ and $\lambda = 10^{-3}$. All of our code is implemented in \texttt{Julia 1.0}. Due to lack of space, most figures have been relegated to Section~\ref{sec:additiona_exps} in the supplementary material.

\begin{figure}[h!]
    \vskip 0.2in
    \begin{center}
        \begin{subfigure}[b]{0.8\textwidth}
          \includegraphics[width=\textwidth]{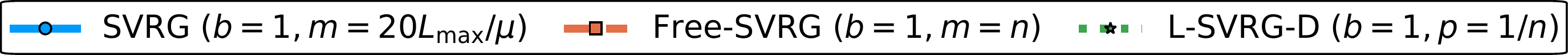}
        \end{subfigure}\\
        \begin{subfigure}[b]{0.4\textwidth}
          \includegraphics[width=\textwidth]{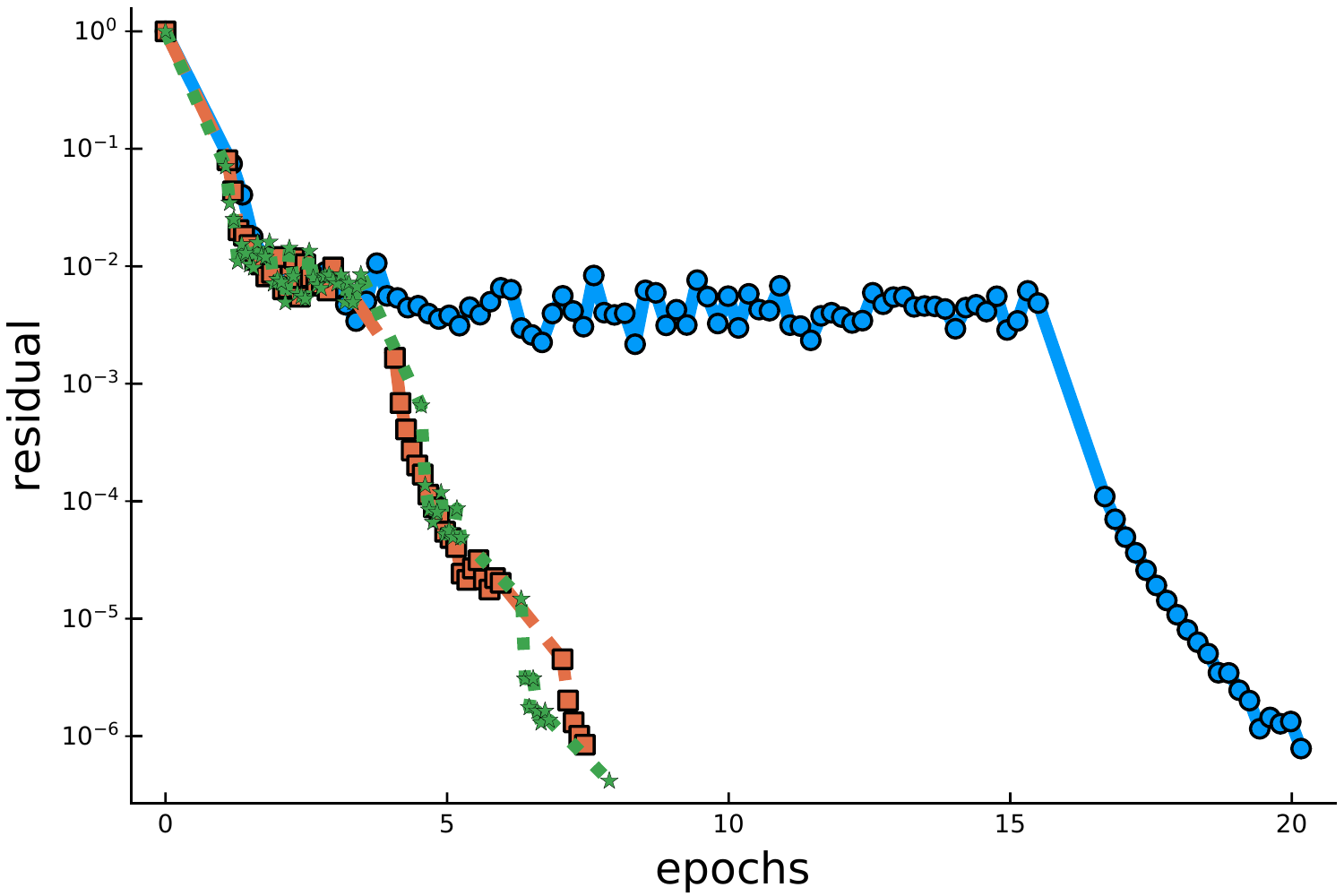}
        \end{subfigure}%
        \begin{subfigure}[b]{0.4\textwidth}
          \includegraphics[width=\textwidth]{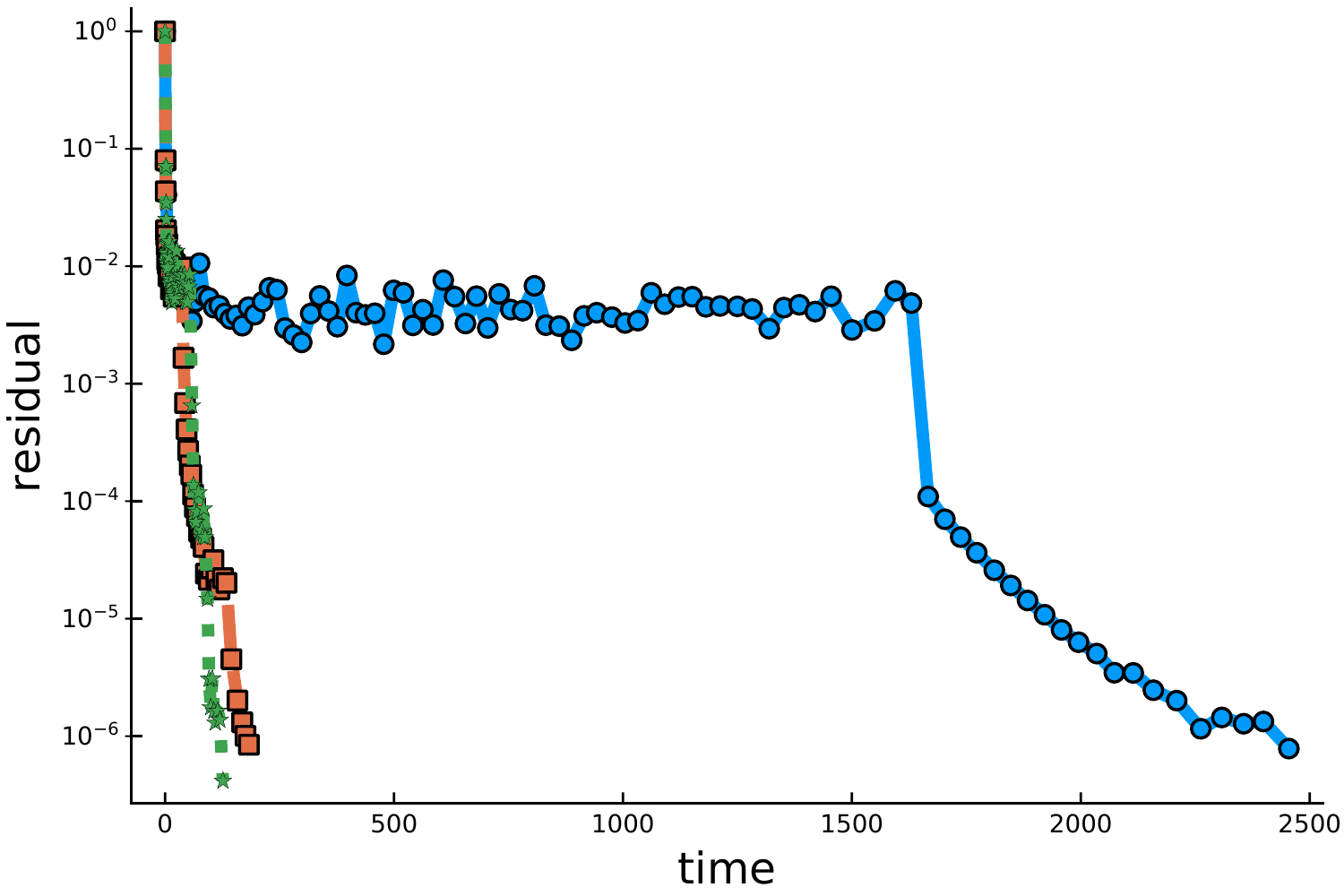}
        \end{subfigure}
        \caption{Comparison of theoretical variants of SVRG without mini-batching ($b=1$) on the \textit{ijcnn1} data set.} 
\label{fig:exp2a_ijcnn1_1e-03}
    \end{center}
    \vskip -0.2in
\end{figure}

\begin{figure}[ht]
    \vskip 0.2in
    \begin{center}
        \begin{subfigure}[b]{0.4\textwidth}
          \includegraphics[width=\textwidth]{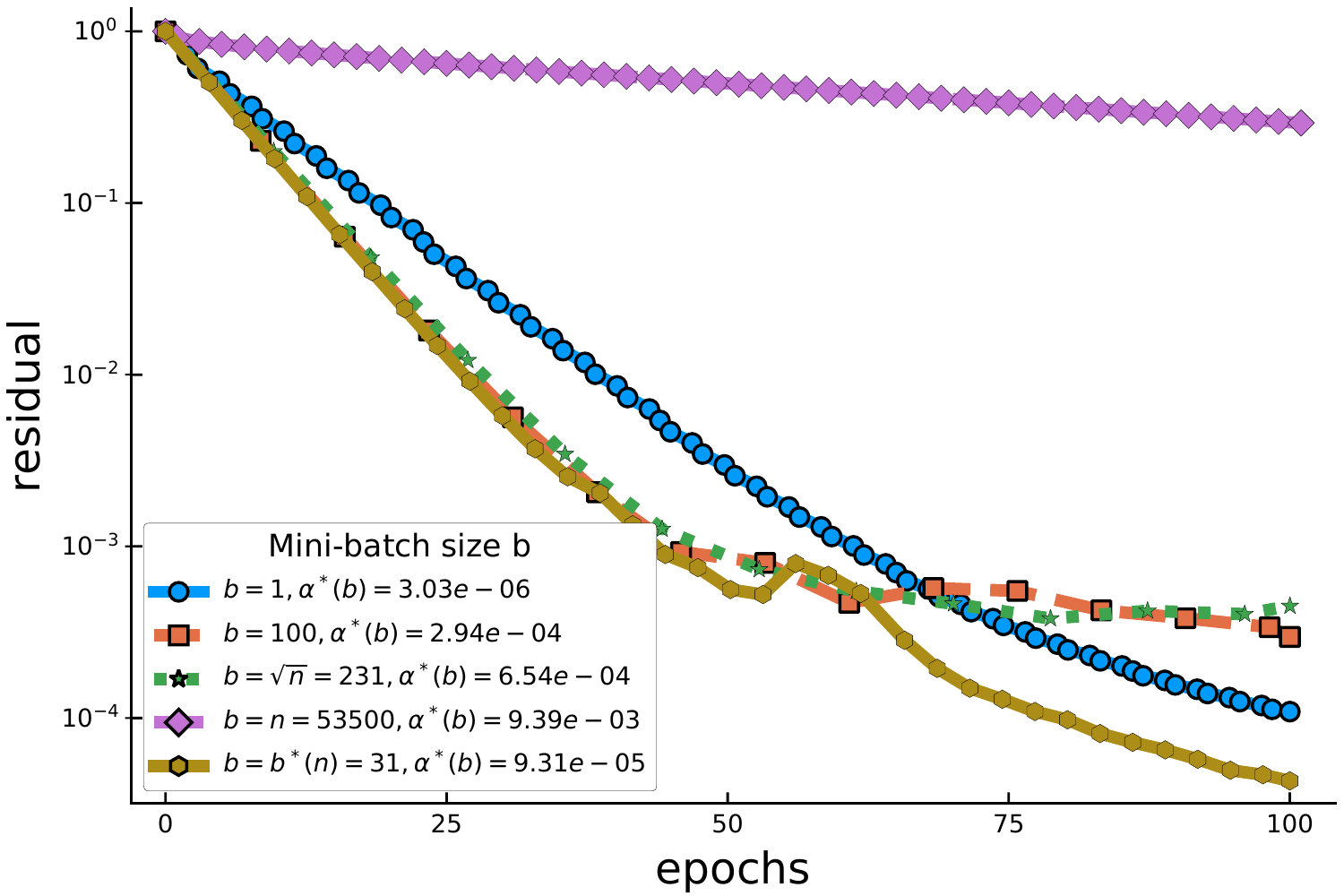}
        \end{subfigure}%
        \begin{subfigure}[b]{0.4\textwidth}
          \includegraphics[width=\textwidth]{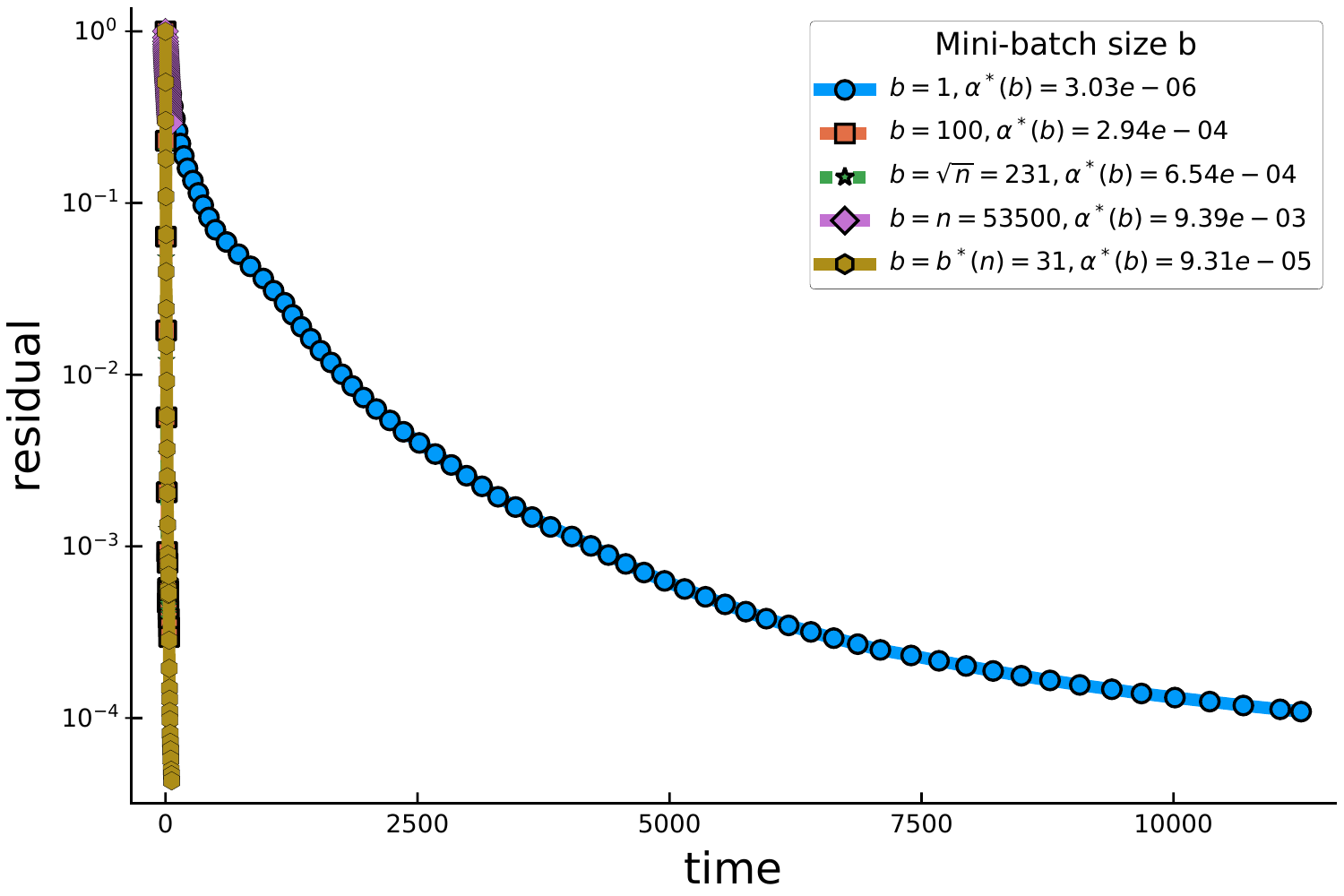}
        \end{subfigure}
        \caption{Optimality of our mini-batch size $b^*$ given in Table~\ref{tab:optimal_mini-batch} for \textit{Free-SVRG} on the \textit{slice} data set.}
    \label{fig:exp2A_slice_1e-01}
    \end{center}
    \vskip -0.2in
\end{figure}



\paragraph{Practical theory.} Our first round of experiments aimed at verifying if our theory does result in efficient algorithms. Indeed, we  found that \textit{Free-SVRG} and \textit{L-SVRG-D} with the parameter setting given by our theory are often faster than SVRG with  settings suggested by the theory in~\cite{johnson2013accelerating}, that is $m=20L_{\max}/\mu$ and $\alpha = 1/10L_{\max}$. See Figure~\ref{fig:exp2a_ijcnn1_1e-03}, and Section~\ref{sec:add_theoretical_exps} for more experiments comparing different theoretical parameter settings.

\paragraph{Optimal mini-batch size.} We
also confirmed numerically that when using
\textit{Free-SVRG} with $m=n$, the optimal mini-batch size $b^*$ derived in Table \ref{tab:optimal_mini-batch} was highly competitive as compared to the range of mini-batch sizes $b\in \{ 1, 100, \sqrt{n}, n\}$.
See Figure~\ref{fig:exp2A_slice_1e-01}
 and several more such experiments in Section~\ref{sec:app_optimal_minibatch}. We also explore the optimality of our $m^*$ in more experiments in Section~\ref{sec:app_exp_inner_loop}.


\clearpage

\subsubsection*{Acknowledgments}
RMG acknowledges the support by grants from
DIM Math Innov R\'egion Ile-de-France (ED574 - FMJH),
 reference ANR-11-LABX-0056-LMH, LabEx LMH.

\bibliographystyle{plain}
\bibliography{../biblio}{}

\clearpage
\appendix

This is the supplementary material for the paper: ``Towards closing the gap between the theory and practice of SVRG'' authored by O. Sebbouh, N. Gazagnadou, S. Jelassi, F. Bach and R. M. Gower (NeurIPS 2019).

In Section \ref{sec:gen_prop} we present general properties that are used in our proofs. In Section \ref{sec:convergence_proofs}, we present the proofs for the convergence and the complexities of our algorithms. In Section \ref{sec:samplings_appendix}, we define several samplings. In Section \ref{sec:Exp_smoothness_appendix}, we present the expected smoothness constant for the samplings we consider. In Section \ref{sec:Exp_residual_appendix}, we present the expected residual constant for the same samplings.
\section{General properties} \label{sec:gen_prop}

\begin{lemma} \label{lem:sq_upp_bnd}
For all $a,b \in \R^d, \; \|a+b\|_2^2 \leq 2\|a\|_2^2 + 2\|b\|_2^2.$
\end{lemma}

\begin{lemma} \label{lem:eq_var_bias}
  For any random vector $X \in \R^d$,
  \begin{equation*}
    \E{\|X - \E{X} \|_2^2} = \E{\|X\|_2^2} - \|\E{X} \|_2^2
    \leq \E{\|X\|_2^2}\enspace.
  \end{equation*}
\end{lemma}

\begin{lemma}  \label{lem:cvx_ineq}
For any convex function $f$, we have
\[f(y) \geq f(x) + \nabla f(x)^\top (y-x), \quad \forall x,y \in \R^d.\]
\end{lemma}

 \begin{lemma}[Logarithm inequality]
  For all $x>0$,
  \begin{equation} \label{eq:log_ineq2}
    \log(x)\leq x-1\enspace.
  \end{equation}
\end{lemma}

 \begin{lemma}[Complexity bounds]\label{lem:complex_bnd}
 Consider the sequence $(\alpha_k)_k \in \mathbb{R}_+$ of positive scalars that converges to $0$ according to \[\alpha_k \leq \rho^k \alpha_0,\] where $\rho \in [0,1)$. For a given $\epsilon \in (0, 1)$, we have that
 \begin{equation}
 k \geq \frac{1}{1 - \rho}\log\left(\frac{1}{\epsilon}\right) \implies \alpha_k \leq \epsilon \alpha_0.
 \end{equation}
 \end{lemma}

\begin{lemma}\label{lem:nLgeqLmax}
Consider convex and $L_i$--smooth functions $f_i$, where $L_i \geq 0$ for all $i \in [n]$, and define $L_{\max} = \max_{i\in[n]}L_i$. Let
\begin{equation*}
f(x) = \frac{1}{n}\sum_{i=1}^n f_i(x)
\end{equation*}
for any $x \in \mathbb{R}^d$. Suppose that $f$ is $L$--smooth, where $L \geq 0$. Then,
\begin{equation}
nL \geq L_{\max}.
\end{equation}
\begin{proof}
Let $x, y \in \mathbb{R}^d$. Since $f$ is $L$--smooth, we have
\begin{equation*}
f(x) \leq f(y) + \nabla f(y)^{\top}(x - y) + \frac{L}{2}\norm{x-y}_2^2.
\end{equation*}
Hence, multiplying by $n$ on both sides,
\begin{equation*}
\sum_{i=1}^n f_i(x) \leq \sum_{i=1}^n f_i(y) + \sum_{i=1}^n \nabla f_i(y)^{\top}(x-y) + \frac{nL}{2}\norm{x-y}_2^2.
\end{equation*}
Rearranging this inequality,
\begin{equation}\label{eq:sum_smaller_nL}
\sum_{i=1}^n \left(f_i(x) - f_i(y) - \nabla f_i(y)^{\top}(x-y)\right) \leq \frac{nL}{2}\norm{x-y}_2^2.
\end{equation}
Since the functions $f_i$ are convex, we have for all $i \in [n]$,
\begin{equation*}
f_i(x) - f_i(y) - \nabla f_i(y)^{\top}(x-y) \geq 0.
\end{equation*}
Then, as a consequence of \eqref{eq:sum_smaller_nL}, we have that for all $i \in [n]$,
\begin{equation*}
f_i(x) - f_i(y) - \nabla f_i(y)^{\top}(x-y) \leq \frac{nL}{2}\norm{x-y}_2^2.
\end{equation*}
Rearranging this inequality,
\begin{equation*}
f_i(x) \leq f_i(y) + \nabla f_i(y)^{\top}(x-y) + \frac{nL}{2}\norm{x-y}_2^2.
\end{equation*}
But since for all $i \in [n]$, $L_i$ is the smallest positive constant that verifies
\begin{equation*}
f_i(x) \leq f_i(y) + \nabla f_i(y)^{\top}(x-y) + \frac{L_i}{2}\norm{x-y}_2^2,
\end{equation*}
we have for all $i \in [n], L_i \leq nL$. Hence $L_{\max} \leq nL$.
\end{proof}
\end{lemma}

\section{Proofs of the results of the main paper}\label{sec:convergence_proofs}
In this section, we will use the abbreviations $\EE{t}{X}\eqdef \E{X | x^t,\dots, x^0}$ for any random variable $X \in \R^d$ and iterates $(x^t)_{t\geq0}$.

\subsection{Proof of Lemma \ref{lem:gt_sq_bndfval_opt3}}
\begin{proof}
\begin{eqnarray*}
\ED{\|g(x, w)\|_2^2} &= &
\ED{\norm{\nabla f_v(x) -\nabla f_v(x^*)+\nabla f_v(x^*) -\nabla f_v(w)+ \nabla f(w)}_2^2}\\
&\overset{\text{Lem.}~\ref{lem:sq_upp_bnd}}{\leq} & 2 \ED{\norm{\nabla f_v(x) - \nabla f_v(x^*)}_2^2} \\
&& +  2\ED{\norm{\nabla f_v(w) -  \nabla f_v(x^*)
-  \nabla f(w)}_2^2} \\
&\overset{\eqref{eq:Expsmooth}+\eqref{eq:Expresidual}}{\leq}& 4\mathcal{L}(f(x) - f(x^*))+ 4\rho(f(w) -  f(x^*)).\end{eqnarray*}
\end{proof}
\subsection{Proof of Theorem \ref{convergence_f-SVRG}}
\begin{proof}
To clarify the notations, we recall that $g_s^t \eqdef g(x_s^t, w_{s-1})$. Then, we get
\begin{eqnarray}
\EE{t}{\|x_{s}^{t+1}-x^*\|_2^2} &= &\EE{t}{ \|x_s^t -x^* - \alpha g_s^{t}\|_2^2} \nonumber \\
& = &\|x_s^t -x^*\|_2^2 - {2\alpha}\EE{t}{g_s^t}^{\top}(x_s^t -x^*) + \alpha^2\EE{t}{\|g_s^t\|_2^2} \nonumber \\
&\overset{\eqref{eq:unbiased}+\eqref{eq:gt_sq_bndfval_opt3}}{\leq} &\|x_s^t -x^*\|_2^2 - {2\alpha}\nabla f(x_s^t)^{\top}(x_s^t-x^*) \nonumber\\%
& & +2\alpha^2\left[2 \mathcal{L}(f(x_s^t) - f(x^*))+  2\rho(f(w_{s-1}) -  f(x^*)) \right]\nonumber\\
&\overset{\eqref{eq:strconv2}}{\leq}&\left(1- \alpha\mu \right) \|x_s^t -x^*\|_2^2 -  2\alpha(1-2\alpha \mathcal{L})\left(f(x_s^t)-f(x^*)\right)\nonumber\\
&&+  4  \alpha^2\rho( f(w_{s-1}) -  f(x^*)).\label{eq:tempsmoerjiose}
\end{eqnarray}
Note that since $\alpha \leq \frac{1}{2(\cL + 2\rho)}$ and $\rho \geq 0$, we have that
\[\alpha \quad \overset{\mbox{\scriptsize Lemma~\ref{lem:gammasmallcL}}}{\leq}  \quad \frac{1}{2\mu},\]
and consequently $(1-\alpha \mu) >0.$
Thus by iterating~\eqref{eq:tempsmoerjiose} over $t=0, \dots,m-1$ and taking the expectation,  since $x_s^0 = x_{s-1}^m$, we obtain
\begin{eqnarray}
\E{\|x_s^{m}-x^*\|_2^2}&\leq& \left(1- \alpha\mu \right)^{m} \E{\|x_{s-1}^m -x^*\|_2^2} \nonumber \\
&& - 2\alpha(1-2\alpha \mathcal{L})\sum\limits_{t=0}^{m-1}\left(1- \alpha\mu \right)^{m-1-t} \E{f(x_s^{t})-f(x^*)}  \nonumber\\
&& + 4  \alpha^2\rho \E{f(w_{s-1}) -  f(x^*)}\sum\limits_{t=0}^{m-1}\left(1- \alpha\mu \right)^{m-1-t}  \nonumber\\
&\overset{\eqref{eq:Smpts}}{=} & \left(1- \alpha\mu \right)^{m} \E{\|x_{s-1}^m -x^*\|_2^2} - 2\alpha(1-2\alpha \mathcal{L})S_m \sum_{t=0}^{m-1}p_t\E{f(x_s^{t})-f(x^*)}  \nonumber\\
&& + 4  \alpha^2\rho S_m \E{f(w_{s-1}) -  f(x^*)} \nonumber \\
&\overset{\eqref{eq:Dsdef}}{=} &  \left(1- \alpha\mu \right)^{m} \E{\|x_{s-1}^m -x^*\|_2^2} - 2\alpha(1-2\alpha \mathcal{L})S_m \sum_{t=0}^{m-1}p_t\E{f(x_s^{t})-f(x^*)}  \nonumber\\
&& + \frac{1}{2} \E{\cP_{s-1}}. \label{eq:temp8j438sj}
\end{eqnarray}

Weights $p_t$ are defined in~\eqref{eq:Smpts}. We note that $(1-\alpha \mu) >0$ implies that $p_t > 0$ for all $t=0,\ldots,m-1$, and by construction we get $\sum_{t=0}^{m-1} p_t = 1$.
Since $f$ is convex, we have by Jensen's inequality that
\begin{eqnarray}
f(w_{s}) - f(x^*) & = & f \left(\sum_{t=0}^{m-1} p_t x_s^t\right) - f(x^*) \nonumber \\
& \leq &  \sum_{t=0}^{m-1} p_t (f(x_s^t) - f(x^*) ). \label{eq:tempje89j8s}
\end{eqnarray}
Consequently,
\begin{eqnarray}
\E{\cP_s} &\overset{\eqref{eq:Dsdef}+\eqref{eq:tempje89j8s}}{ \leq }& 8\alpha^2\rho S_m \sum_{t=0}^{m-1} p_t \E{(f(x_{s}^{t}) - f(x^*))}. \label{eq:Dsbnd2}
\end{eqnarray}

As a result,
\begin{eqnarray*}
\E{\phi_s} &=& \E{\norm{x_s^m - x^*}_2^2} + \E{\cP_s}\\
&\overset{\eqref{eq:temp8j438sj}+\eqref{eq:Dsbnd2}}{\leq}& \left(1- \alpha\mu \right)^{m}\E{\|x_{s-1}^m -x^*\|_2^2} + \frac{1}{2}\E{\cP_{s-1}}\\
&&- 2\alpha(1 - 2\alpha(\cL+2\rho)) S_m \sum_{t=0}^{m-1} p_t \E{(f(x_{s}^{t})- f(x^*))}.
\end{eqnarray*}
Since $\alpha \leq \frac{1}{2(\cL+2\rho)}$, the above implies
\begin{eqnarray*}
\E{\phi_s} &\leq& \left(1- \alpha\mu \right)^{m}\E{\|x_{s-1}^m -x^*\|_2^2} + \frac{1}{2}\E{\cP_{s-1}} \\
&\leq& \beta \E{\phi_{s-1}},
\end{eqnarray*}
where $\beta = \max\{(1-\alpha\mu)^m, \frac{1}{2}\}.$

Moreover, if we set $w_s = x_s^t $ with probability  $p_t$, \, for $t =0,\ldots, m-1$, the result would still hold. Indeed~\eqref{eq:tempje89j8s} would hold with equality and the rest of the proof would follow verbatim.
\end{proof}

\subsection{Proof of Corollary \ref{cor:complex_total_f-SVRG}}
\begin{proof}
Noting $\beta = \max\left\{\left(1 - \frac{\mu}{2(\cL(b)+2\rho(b))}\right)^m, \frac{1}{2}\right\}$, we need to chose $s$ so that $\beta^s \leq \epsilon$, that is ${s \geq \frac{\log(1/\epsilon)}{\log(1/\beta)}}$. Since in each inner iteration we evaluate $2b$ gradients of the $f_i$ functions, and in each outer iteration we evaluate all $n$ gradients, this means that the total complexity will be given by
\begin{eqnarray*}
C &\eqdef& (n + 2bm) \frac{\log(1/\epsilon)}{\log(1/\beta)} \\
&=& (n + 2bm) \max\left\{-\frac{1}{m\log(1-\frac{\mu}{2(\cL(b)+2\rho(b))})}, \frac{1}{\log2}\right\}\log\left(\frac{1}{\epsilon}\right) \\
&\overset{\eqref{eq:log_ineq2}}{\leq} & (n + 2bm) \max\left\{\frac{1}{m} \frac{2(\cL(b)+2\rho(b))}{\mu}, 2\right\}\log\left(\frac{1}{\epsilon}\right).
\end{eqnarray*}
\end{proof}

\subsection{Proof of Corollary \ref{cor:total_comp_m_1nice}}
\begin{proof}
Recall that from \eqref{eq:total_cplx_def}, using the fact that $\cL(1) = \rho(1) = L_{\max}$, we have
\begin{equation*}
C_m(1) = 2\left(\frac{n}{m}+2\right)\max\left\{\frac{3L_{\max}}{\mu}, m\right\}\log\left(\frac{1}{\epsilon}\right).
\end{equation*}
When $n \geq \frac{L_{\max}}{\mu}$, then, $m \in \left[\frac{L_{\max}}{\mu}, n\right]$. We can rewrite $C_m(1)$ as
\begin{equation*}
C_m(1) = 2(n + 2m)\max\left\{\frac{1}{m}\frac{3L_{\max}}{\mu}, 1\right\}\log\left(\frac{1}{\epsilon}\right).
\end{equation*}
We have $\frac{1}{m}\frac{3L_{\max}}{\mu} \leq 3$ and $n + 2m \leq 3n$. Hence,
\begin{equation*}
C_m(1) \leq 18n\log\left(\frac{1}{\epsilon}\right) = O\left(\left(n + \frac{L_{\max}}{\mu}\right)\log\left(\frac{1}{\epsilon}\right)\right).
\end{equation*}

When $n \leq \frac{L_{\max}}{\mu}$, then, $m \in \left[n, \frac{L_{\max}}{\mu}\right]$. We have $\frac{n}{m} \leq 1$ and $m \leq \frac{3L_{\max}}{\mu}$. Hence,
\begin{equation*}
C_m(1) \leq \frac{18L_{\max}}{\mu} \log\left(\frac{1}{\epsilon}\right) = O\left(\left(n + \frac{L_{\max}}{\mu}\right)\log\left(\frac{1}{\epsilon}\right)\right).
\end{equation*}
\end{proof}

\subsection{Proof of Theorem \ref{thm_dec_step}}

Before analysing Algorithm~\ref{alg:L-SVRG-D}, we present a lemma that allows to compute the expectations $\E{\alpha_k}$ and $\E{\alpha_k^2}$, that will be used in the analysis.
\begin{lemma}\label{exp_alpha_2}
    Consider the step sizes defined by Algorithm \ref{alg:L-SVRG-D}. We have
    \begin{eqnarray}
    \E{\alpha_k} &=& \frac{(1 - p)^{\frac{3k + 2}{2}}(1 - \sqrt{1 - p})+p}{1 - (1 - p)^{\frac{3}{2}}}  \alpha. \label{eq:Ealphaktemp} \\
    \E{\alpha_k^2} &=& \frac{1 + (1 - p)^{2k + 1}}{2 - p}\alpha^2.
    \end{eqnarray}
\end{lemma}

\begin{proof}
Taking expectation with respect to the filtration induced by the sequence of step sizes $\{\alpha_1,\dots,\alpha_k\}$
\begin{eqnarray}
\EE{p}{\alpha_{k+1}} = (1 - p)\sqrt{1 - p} \; \alpha_k + p \alpha.
\end{eqnarray}
Then taking total expectation
\begin{eqnarray}
\E{\alpha_{k+1}} = (1 - p)\sqrt{1 - p} \E{\alpha_k} + p \alpha. \label{eq:tempalpharecur}
\end{eqnarray}
Hence the sequence $(\E{\alpha_k})_{k\geq1}$ is uniquely defined by
\begin{eqnarray}
    \E{\alpha_k} &=& \frac{(1 - p)^{\frac{3k + 2}{2}}(1 - \sqrt{1 - p})+p}{1 - (1 - p)^{\frac{3}{2}}}  \alpha. \label{eq:Ealphak}
\end{eqnarray}

Indeed, applying~\eqref{eq:tempalpharecur} recursively gives
$$
\E{\alpha_{k}} = (1 - p)^{\frac{3k}{2}} \alpha +
p \alpha  \sum_{i=0}^{k-1} (1 - p)^{\frac{3i}{2}} .
$$
Adding  up the geometric series gives
\begin{eqnarray*}
\E{\alpha_{k}} &=& \alpha (1 - p)^{\frac{3k}{2}} +
p \alpha \frac{1 -(1 - p)^{\frac{3k}{2}} }{1-(1 - p)^{\frac{3}{2}} }\\
& = &
 \frac{ (1 - p)^{\frac{3k}{2}}(1-(1 - p)^{\frac{3}{2}}) -(1 - p)^{\frac{3k}{2}}  p + p }{1-(1 - p)^{\frac{3}{2}} }\alpha \enspace.
\end{eqnarray*}
Which leads to \eqref{eq:Ealphak} by factorizing. The same arguments are used to compute $\E{\alpha_k^2}$.
\end{proof}

We now present a proof of Theorem \ref{thm_dec_step}.
\begin{proof}
We recall that $g^k \eqdef \nabla f (x^k)$. First, we get
\begin{eqnarray}
\EE{k}{\|x^{k+1} - x^*\|_2^2} &= &\EE{k}{ \|x^k - x^* - \alpha_k g^k\|_2^2}\nonumber\\
& = &\|x^k - x^*\|_2^2 - {2\alpha_k}\EE{k}{g^k}^{\top}(x^k - x^*) + \alpha_k^2\EE{k}{\|g^k\|_2^2}\nonumber\\
&\overset{\eqref{eq:unbiased}+\eqref{eq:gt_sq_bndfval_opt3}+\text{Rem.}~\ref{rem:rho_equal_cL}}{\leq} &\|x^k - x^*\|_2^2 - {2\alpha_k}\nabla f(x^k)^{\top}(x^k - x^*)\nonumber \\
& & +2\alpha_k^2\left[2 \cL(f(x^k) - f(x^*))+  2\cL(f(w^k) -  f(x^*)) \right] \nonumber\\
&\overset{\eqref{eq:strconv2}}{\leq}&\left(1- \alpha_k\mu \right) \|x^k - x^*\|_2^2 -  2\alpha_k(1-2\alpha_k \cL)\left(f(x^k)-f(x^*)\right) \nonumber\\
&& + 4\alpha_k^2\cL(f(w^k) -  f(x^*)) \nonumber \\
&\overset{\eqref{gradlearn_dec_def}}{=}& \left(1- \alpha_k\mu \right) \|x^k - x^*\|_2^2 -  2\alpha_k(1-2\alpha_k \cL)\left(f(x^k)-f(x^*)\right) \nonumber\\
&& + p\left(\frac{3}{2} - p\right) \cP^k.  \nonumber
\end{eqnarray}

Hence we have, taking total expectation and noticing that the variables $\alpha_k$ and $x^k$ are independent,
\begin{eqnarray}
\E{ \norm{x^{k+1} - x^*}_2^2} &\leq& \left(1 - \E{\alpha_k}\mu \right)\E{\norm{x^k - x^*}_2^2} - 2\E{\alpha_k(1-2\alpha_k \cL)}\E{f(x^k)-f(x^*)}\nonumber\\
&& + p\left(\frac{3}{2} - p\right)\E{\cP^k}. \label{phi_first_dec}
\end{eqnarray}
We have also have
\begin{eqnarray*}
\EE{k}{\cP^{k+1}} &=& (1 - p)  \frac{8(1-p)\alpha_k^2\cL}{p(3 - 2p)}\left(f(w^k) - f(x^*)\right) + p \frac{8\alpha^2\cL}{p(3 - 2p)}\left(f(x^k) - f(x^*)\right) \\
&=& (1 - p)^2 \cP^k + \frac{8\alpha^2\cL}{3 - 2p}\left(f(x^k) - f(x^*)\right).
\end{eqnarray*}
Hence, taking total expectation gives
\begin{eqnarray}
\E{\cP^{k+1}} = (1 - p)^2 \E{\cP^k} + \frac{8\alpha^2\cL}{3 - 2p} \E{f(x^k) - f(x^*)} \label{phi_second_dec}
\end{eqnarray}
Consequently,
\begin{eqnarray}
\E{\phi^{k+1}} &\overset{\eqref{gradlearn_dec_def} + \eqref{phi_first_dec} + \eqref{phi_second_dec}}{\leq}& \left(1 - \E{\alpha_k}\mu \right)\E{\norm{x^k - x^*}_2^2} \nonumber \\
&&- 2\left(\E{\alpha_k(1-2\alpha_k \cL)} - 4\frac{\alpha^2\cL}{3 - 2p}\right)\E{f(x^k)-f(x^*)}\nonumber\\
&& + \left(1 - \frac{p}{2}\right)\E{\cP^k} \nonumber \\
&=&  \left(1 - \E{\alpha_k}\mu \right)\E{\norm{x^k - x^*}_2^2} \nonumber \\
&& - 2 \left(\E{\alpha_k} - 2 \left(\E{\alpha_k^2} + \frac{2}{3 - 2p}\alpha^2\right)\cL\right)\E{f(x^k)-f(x^*)}\nonumber\\
&& + \left(1 - \frac{p}{2}\right)\E{\cP^k}.
\end{eqnarray}
From Lemma \ref{exp_alpha_2}, we have ${\E{\alpha_k} = \frac{(1 - p)^{\frac{3k + 2}{2}}(1 - \sqrt{1 - p})+p}{1 - (1 - p)^{\frac{3}{2}}}  \alpha}$, and we can show that for all $k$
\begin{eqnarray} \label{l_bnd_exp_alpha}
\E{\alpha_k} \geq \frac{2}{3}\alpha,
\end{eqnarray}
Letting $q =1-p$ we have that
\begin{eqnarray*}
 \frac{(1 - p)^{\frac{3k + 2}{2}}(1 - \sqrt{1 - p})+p}{1 - (1 - p)^{\frac{3}{2}}}  &=&  \frac{q^{\frac{3k + 2}{2}}(1 - \sqrt{q})+1-q}{1 - q^{\frac{3}{2}}} \\
 &=& q^{\frac{3k + 2}{2}}\frac{1-\sqrt{q}}{1-q^{3/2}} + \frac{1-q}{1-q^{3/2}}\\
 &\geq&  \frac{1-q}{1-q^{3/2}}\\
 &\geq& \frac{2}{3}, \quad \forall q \in [0,1] \enspace.
\end{eqnarray*}
Consequently,
\begin{eqnarray}
\E{\phi^{k+1}} &\overset{\eqref{phi_first_dec} + \eqref{phi_second_dec} + \eqref{l_bnd_exp_alpha}}{\leq}& \left(1 - \frac{2}{3}\alpha\mu \right)\E{\norm{x^k - x^*}_2^2} \nonumber \\
&& - 2 \left(\E{\alpha_k} - 2 \left(\E{\alpha_k^2} + \frac{2}{3 - 2p}\alpha^2\right)\cL\right)\E{f(x^k)-f(x^*)}\nonumber\\
&& + \left(1 - \frac{p}{2}\right)\E{\cP^k}. \label{major_dec}
\end{eqnarray}
To declutter the notations, Let us define
\begin{eqnarray}
a_k &\eqdef& \frac{(1 - p)^{\frac{3k + 2}{2}}(1 - \sqrt{1 - p})+p}{1 - (1 - p)^{\frac{3}{2}}}\\
b_k &\eqdef& \frac{1 + (1 - p)^{2k + 1}}{2 - p}
\end{eqnarray}
so that $\E{\alpha_k} = a_k \alpha$ and $\E{\alpha_k^2} = b_k \alpha^2$. Then \eqref{major_dec} becomes
\begin{eqnarray}
\E{\phi^{k+1}} &\leq& \left(1 - \frac{2}{3}\alpha\mu \right)\E{\norm{x^k - x^*}_2^2} \nonumber \\
&& - 2\alpha \left(a_k - 2 \alpha\left(b_k + \frac{2}{3 - 2p}\right)\cL\right)\E{f(x^k)-f(x^*)} \nonumber\\
&& + \left(1 - \frac{p}{2}\right)\E{\cP^k}. \label{eq:8jaj8ja83a}
\end{eqnarray}

Next we would like to drop the second term in~\eqref{eq:8jaj8ja83a}. For this we need to guarantee that
\begin{equation} \label{eq:necessary_condition_ak_bk}
  a_k - 2 \alpha\cL\left(b_k + \frac{2}{3 - 2p}\right) \geq 0
\end{equation}
Let $q \eqdef 1-p$ so that the above becomes
\[ \frac{q^{\frac{3k + 2}{2}}(1 - \sqrt{q})+1-q}{1 - q^{\frac{3}{2}}} - 2 \alpha\cL\left(\frac{1 + q^{2k + 1}}{1 +q} + \frac{2}{1 + 2q}\right) \geq 0.\]
In other words, after dividing through by $\left(\frac{1 + q^{2k + 1}}{1 +q} + \frac{2}{1 + 2q}\right)$ and re-arranging, we require that
\begin{eqnarray}
  2 \alpha\cL & \leq & \frac{\frac{q^{\frac{3k + 2}{2}}(1 - \sqrt{q})+1-q}{1 - q^{\frac{3}{2}}}}{\frac{1 + q^{2k + 1}}{1 +q} + \frac{2}{1 + 2q}} \nonumber \\
  &=& \frac{\frac{q^{\frac{3k + 2}{2}}(1 - \sqrt{q})+1-q}{1 - q^{\frac{3}{2}}}}{\frac{(1 + q^{2k + 1})(1+2q)+2(1+q)}{(1 +q)(1+2q)}} \nonumber \\
  &=& \frac{q^{\frac{3k + 2}{2}}(1 - \sqrt{q})+1-q}{q^{2k+1}(1+2q) + 3 + 4q} \; \frac{(1 +q)(1+2q)}{  1 - q^{\frac{3}{2}}} \enspace.\label{eq:alpabnddepednk}
\end{eqnarray}

%
We are now going to show that
\begin{eqnarray}\label{eq:alphatoshow}
\frac{q^{\frac{3k + 2}{2}}(1 - \sqrt{q}) + 1 - q}{q^{2k + 1}(1+2q)+3+4q} \geq \frac{1-q}{3+4q}.
\end{eqnarray}
Indeed, multiplying out the denominators of the above gives
\begin{eqnarray*}
F(q) &\eqdef & (3+4q)\left(q^{\frac{3k + 2}{2}}(1 - \sqrt{q})+1-q\right) - (1 - q)\left(q^{2k+1}(1+2q) + 3 + 4q\right)\\
&=& q^{\frac{3k + 2}{2}}(1 - \sqrt{q})(3+4q) - q^{2k+1}(1+2q)(1-q)\\
&=& q^{\frac{3k + 2}{2}}(1 - \sqrt{q}) \left(3 + 4q - q^{\frac{k}{2}}(1+2q)(1+\sqrt{q}) \right).
\end{eqnarray*}
And since $q^{\frac{k}{2}} \leq 1$, we have
\begin{eqnarray*}
F(q) &\geq& q^{\frac{3k + 2}{2}}(1 - \sqrt{q}) \left(3 + 4q - (1+2q)(1+\sqrt{q}) \right) \\
&=& 2q^{\frac{3k + 2}{2}}(1 - \sqrt{q})(1 - q\sqrt{q})\\
&\geq & 0.
\end{eqnarray*}

As a result~\eqref{eq:alphatoshow} holds. And thus if
\begin{eqnarray*}
  2 \alpha\cL &\leq& \frac{1 - q}{3 + 4q} \frac{(1+q)(1+2q)}{1 - q^{\frac{3}{2}}}
\end{eqnarray*}
holds, then \eqref{eq:alpabnddepednk} is verified for all $k$. This is why we impose the upper bound on the step size given in~\eqref{eq:alphadecreasingcyclic},
which ensures that~\eqref{eq:necessary_condition_ak_bk} is satisfied. Finally, this condition being verified, we get that
\begin{eqnarray}
\E{\phi^{k+1}} &\overset{\eqref{eq:8jaj8ja83a}+\eqref{eq:necessary_condition_ak_bk}}{\leq}& \left(1 - \frac{2}{3}\alpha\mu \right)\E{\norm{x^k - x^*}_2^2} + \left(1 - \frac{p}{2}\right)\E{\cP^k} \nonumber \\
&\leq& \beta \E{\phi^{k}},
\end{eqnarray}
where $\beta = \max\left\{1 - \frac{2}{3}\alpha\mu, 1 - \frac{p}{2}\right\}$.
\end{proof}

\subsection{Proof of Corollary \ref{cor:total_cplx_LDSVRG}}
\begin{proof}
We have that \[\E{\phi^k} \leq \beta^k\phi^0,\] where $\beta = \max\left\{1 - \frac{1}{3\zeta_p}\frac{\mu}{\cL(b)}, 1 - \frac{p}{2}\right\}$. Hence using Lemma \ref{lem:complex_bnd}, we have that the iteration complexity for an $\epsilon > 0$ approximate solution that verifies $\E{\phi^k} \leq \epsilon \phi^0$ is
\[2\max\left\{\frac{3\zeta_p}{2}\frac{\cL(b)}{\mu}, \frac{1}{p}\right\}\log\left(\frac{1}{\epsilon}\right).\]
For the total complexity, one can notice that in expectation, we compute $2b + pn$ stochastic gradients at each iteration.
\end{proof}

\subsection{Proof of Proposition~\ref{prop:optimalloopsize} }
\begin{proof} Dropping the $\log(1/\epsilon)$ for brevity, we distinguish two cases,
 $m \geq \frac{2(\cL(b)+2\rho(b))}{\mu}$ and ${m \leq \frac{2(\cL(b)+2\rho(b))}{\mu}}$.
 \begin{enumerate}
 \item {\bf $m \geq \frac{2(\cL(b)+2\rho(b))}{\mu}$}:  Then $C_m(b) = 2(n+2bm)$,
 and hence we should use the smallest $m$ possible, that is, $m=\frac{2(\cL(b)+2\rho(b))}{\mu}$.
 \item {\bf $m \leq \frac{2(\cL(b)+2\rho(b))}{\mu}$}: Then $C_m(b) = \frac{2(n+2bm)}{m}\frac{2(\cL(b)+2\rho(b))}{\mu}  = 2\left(\frac{n}{m} + 2b\right)\frac{2(\cL(b)+2\rho(b))}{\mu} $.
 Hence $C_m(b)$ is decreasing in $m$ and we should then use the highest possible value for $m$, that is $m =\frac{2(\cL(b)+2\rho(b))}{\mu}$.
\end{enumerate}
The result now follows by substituting $m =\frac{2(\cL(b)+2\rho(b))}{\mu}$ into~\eqref{eq:total_cplx_def}.
\end{proof}

\subsection{Proof of Proposition \ref{prop:optim_batch}} \label{sec:proof_optim_batch_f-SVRG}
\begin{proof}
Recall that have from Lemma \ref{lem:exp_smooth_residual_bnice}:
\begin{eqnarray}
\cL(b) &=& \frac{1}{b}\frac{n-b}{n-1}L_{\max} + \frac{n}{b}\frac{b-1}{n-1}L, \label{reminder_L}\\
\rho(b) &=& \frac{1}{b}\frac{n-b}{n-1}L_{\max} .\label{reminder_rho}
\end{eqnarray}
For brevity, we temporarily drop the term $\log\left(\frac{1}{\epsilon}\right)$ in $C_m(b)$ defined in Equation~\eqref{eq:total_cplx_def}. Hence, we want to find, for different values of $m$:
\begin{equation}
b^* = \underset{b \in [n]}{\argmin}\, C_m(b):= 2\left(\frac{n}{m}+2b\right)\max\{\kappa(b), m\},
\end{equation}
where $\kappa(b) \eqdef \frac{\cL(b) + 2\rho(b)}{\mu}$.

\paragraph{When $m = n$.} In this case we have
\begin{eqnarray}
C_n(b) \overset{\eqref{eq:total_cplx_def}}{=} 2(2b + 1) \max\{\kappa(b),n\},
\end{eqnarray}

Writing $\kappa(b)$ explicitly: \[\kappa(b) = \frac{1}{\mu (n-1)}\left((3L_{\max} - L)\frac{n}{b} + nL - 3L_{\max} \right).\]
Since $3L_{\max} > L$, $\kappa(b)$ is a decreasing function of $b$. In the light of this observation, we will determine the optimal mini-batch size. The upcoming analysis is summarized in Table \ref{tab:optimal_mini-batch}.

We distinguish three cases:
\begin{itemize}
\item If $n \leq \frac{L}{\mu}$: then $\kappa(n) = \frac{L}{\mu} \geq n$. Since $\kappa(b)$ is decreasing, this means that for all $b \in [n], \kappa(b) \geq n$. Consequently, $C_n(b)= 2(2b + 1) \kappa(b)$. Differentiating twice: \[C_n^{''}(b) = \frac{4}{\mu(n-1)}\frac{(3L_{\max} - L)n}{b^3} > 0.\] Hence $C_n(b)$ is a convex function. Now examining its first derivative:\[C_n^{'}(b) = \frac{2}{\mu(n-1)}\left(-\frac{(3L_{\max} - L)n}{b^2} + 2(nL - 3L_{\max})\right),\] we can see that:
\begin{itemize}
\item If $n \leq \frac{3L_{\max}}{L}$, $C_n(b)$ is a decreasing function, hence \[b^* = n.\]
\item If $n > \frac{3L_{\max}}{L}$, $C_n(b)$ admits a minimizer, which we can find by setting its first derivative to zero. The solution is \[\hat{b} \eqdef \sqrt{\frac{n}{2}\frac{3L_{\max} - L}{nL - 3L_{\max}}}.\] Hence, \[b^* = \floor{\hat{b}}\]
\end{itemize}

\item If $n \geq \frac{3L_{\max}}{\mu}$, then $\kappa(1) = 3 \frac{L_{\max}}{\mu}$. Since $\kappa(b)$ is decreasing, this means that for all $b \in [n], \kappa(b) \leq n$. Hence, $C_n(b) = 2(2b + 1)n$. $C_n(b)$ is an increasing function of $b$. Therefore, \[b^* = 1.\]

\item If $\frac{L}{\mu} < n < \frac{3L_{\max}}{\mu}$, we have ${\kappa(1) > n}$ and ${\kappa(n) < n}$. Hence there exists $\tilde{b} \in [1,\, n]$ such that $\kappa(b) = n$, and it is given by
\begin{eqnarray}
\tilde{b} \eqdef \frac{(3L_{\max} - L)n}{n(n-1)\mu - nL + 3 L_{\max}}. \label{eq:tilde_b}
\end{eqnarray}
Define $G(b):\eqdef(2b+1)\kappa(b)$. Then,
\begin{equation}
\underset{b\in[1,\, n]}{\argmin}\, G(b) = \begin{cases}
        n & \mbox{if } n \leq \frac{3L_{\max}}{L}, \\
        \hat{b} & \mbox{if } n > \frac{3L_{\max}}{L}.
\end{cases}
\end{equation}
As a result, we have that 
\begin{itemize}
\item if $n \leq \frac{3L_{\max}}{L}$, $G(b)$ is  decreasing on $[1, n]$, hence $C_n(b)$ is decreasing on $[1, \tilde{b}]$ and increasing on $[\tilde{b}, n]$. Then, \[b^* = \floor{\tilde{b}}.\]
\item if $n > \frac{3L_{\max}}{L}$, $G(b)$ is decreasing on $[1, \hat{b}]$ and increasing on $[\hat{b}, n]$. Hence $C_n(b)$ is decreasing on $[1, \min\{\hat{b}, \tilde{b}\}]$ and increasing on $[\min\{\hat{b}, \tilde{b}\}, 1]$. Then, \[b^* = \floor{\min\{\hat{b}, \tilde{b}\}}.\]
\end{itemize}
\end{itemize}
To summarize, we have for $m = n$,
\begin{eqnarray}
b^* &=& \left\{
            \begin{array}{ll}
                1 & \mbox{if } n \geq \frac{3L_{\max}}{\mu} \\
                \floor{\min(\tilde{b}, \hat{b})} & \mbox{if }\max\{\frac{L}{\mu}, \frac{3L_{\max}}{L}\} < n < \frac{3L_{\max}}{\mu}\\
				\floor{\hat{b}} & \mbox{if }  \frac{3L_{\max}}{L} < n < \frac{L}{\mu}\\
				\floor{\tilde{b}} & \mbox{if } \frac{L}{\mu} < n \leq \frac{3L_{\max}}{L}\\
                n & \mbox{otherwise, if } n \leq \min\{\frac{L}{\mu}, \frac{3L_{\max}}{L}\}
            \end{array} \right.
\end{eqnarray}

\paragraph{When $m=n/b$.} In this case we have \[C_m(b) \overset{\eqref{eq:total_cplx_def}}{=} 6 \max\{b\kappa(b), n\},\]
with \[b\kappa(b) = \frac{1}{\mu(n-1)}\left((3L_{\max} - L)n + (nL - 3 L_{\max})b\right),\]
and thus $\kappa(1) = \frac{3L_{\max}}{\mu}$ and $n\kappa(n) = \frac{nL}{\mu} \geq n$.
We distinguish two cases:
\begin{itemize}
\item if $n \leq \frac{3L_{\max}}{L}$, then $b\kappa(b)$ is decreasing in $b$. Since $n\kappa(n) \geq n$, $C_m(b) = 6b\kappa(b)$, thus $C_m(b)$ is decreasing in $b$,
hence\[b^* = n\]
\item if $n > \frac{3L_{\max}}{L}$, $b\kappa(b)$ is increasing in $b$. Thus,
\begin{itemize}
\item if $n \leq \frac{3L_{\max}}{\mu} = \kappa(1)$, then $C_m(b) = 6b\kappa(b)$. Hence $b^* = 1$.
\item if $n > \frac{3L_{\max}}{\mu}$, using the definition of $\tilde{b}$ in Equation \eqref{eq:tilde_b}, we have that
\begin{eqnarray*}
C_m(b) = \left\{
    \begin{array}{ll}
        6n & \mbox{for } b \in [1, \bar{b}] \\
        6b\kappa(b) & \mbox{for } b \in [\bar{b}, n]
    \end{array} \right.,
\end{eqnarray*}
where $$\bar{b} = \frac{n(n-1)\mu - (3L_{\max} - L)n}{nL - 3L_{\max}}$$
is the batch size $b \in [n]$ which verifies $b\kappa(b) = n$.
Hence $b^*$ can be any point in $\{1,\dots,\floor{\bar{b}}\}.$ In light of shared memory parallelism, $b^* = \floor{\bar{b}}$ would be the most practical choice.
\end{itemize}
\end{itemize}
\end{proof}

\section{Optimal mini-batch size for Algorithm \ref{alg:L-SVRG-D}\label{sec:optim_mini-batch_d-SVRG}}
By using a similar proof as in Section \ref{sec:proof_optim_batch_f-SVRG}, we derive the following result.

\begin{proposition}\label{prop:optim_batch2}
Note $b^* \eqdef \underset{b \in [n]}{\argmin}\,C_p(b)$, where $C_p(b)$ is defined in \eqref{eq:total_cplx_dsvrg}.
For the widely used choice $p = \frac{1}{n}$, we have that
\begin{eqnarray} \label{eq:optim_minibatch_dsvrg}
b^* &=& \left\{
            \begin{array}{ll}
                1 & \mbox{if } n \geq \frac{3\zeta_{1/n}}{2}\frac{L_{\max}}{\mu} \\
                \floor{\min(\tilde{b}, \hat{b})} & \mbox{if }\frac{3\zeta_{1/n}}{2}\frac{L}{\mu}< n < \frac{3\zeta_{1/n}}{2}\frac{L_{\max}}{\mu}\\
                \floor{\hat{b}} & \mbox{otherwise, if } n \leq \frac{3\zeta_{1/n}}{2}\frac{L}{\mu}
            \end{array} \right. ,
\end{eqnarray}
where $\zeta_p$ is defined in \eqref{eq:alphadecreasingcyclic} for $p \in (0,1]$ and:$${\hat{b} = \sqrt{\frac{n}{2}\frac{L_{\max} - L}{nL - L_{\max}}}}, \ \tilde{b} = \frac{\frac{3\zeta_p}{2}n(L_{\max} - L)}{\mu n(n-1) - \frac{3\zeta_p}{2}(nL - L_{\max})}.$$
\end{proposition}

Because $\zeta_p$ depends on $p$, optimizing the total complexity with respect to $b$ for the case $p = \frac{b}{n}$ is extremely cumbersome. Thus, we restrain our study for the optimal mini-batch sizes for Algorithm \ref{alg:L-SVRG-D} to the case where $p = \frac{1}{n}$.

\begin{proof}
For brevity, we temporarily drop the term $\log\left(\frac{1}{\epsilon}\right)$ in $C_p(b)$ defined in Equation~\eqref{eq:total_cplx_dsvrg}. Hence, we want to find, for different values of $m$:
\begin{equation}
b^* = \underset{b \in [n]}{\argmin}\, C_{1/n}(b):= 2\left(2b+1\right)\max\{\pi(b), m\},
\end{equation}
where $\pi(b) \eqdef \frac{3\zeta_p}{2}\frac{\cL(b)}{\mu}$. We have
\begin{eqnarray}
\pi(b) &=& \frac{3\zeta_p}{2}\frac{1}{\mu(n-1)}\left(\frac{n(L_{\max} - L)}{b} + nL - L_{\max}\right).
\end{eqnarray}
Since $L_{\max} \geq L$, $\pi(b)$ is a decreasing function on $[1, n]$.
We distinguish three cases:
\begin{itemize}
\item if $n > \pi(1) = \frac{3\zeta_p}{2}\frac{L_{\max}}{\mu}$, then for all $b \in [1, n], n > \pi(b)$. Hence, \[C_n(b)=2(2b+1)n.\]
$C_{1/n}(b)$ is an increasing function of $b$. Hence \[b^* = 1.\]
\item if $n < \pi(n) = \frac{3\zeta_p}{2}\frac{L}{\mu}$, then for all $b \in [1, n], n < \pi(b)$. Hence, \[C_{1/n}(b)=2(2b+1)\pi(b).\]

Now, consider the function
\begin{eqnarray*}
G(b) &\eqdef & (2b+1)\pi(b)\\
&=& \frac{3\zeta_p}{2}\frac{1}{\mu(n-1)}\left(2(nL - L_{\max})b + \frac{n(L_{\max}-L)}{b}\right) + \Omega,
\end{eqnarray*}
where $\Omega$ replaces constants which don't depend on $b$.
The first derivative of $G(b)$ is
\begin{eqnarray*}
G'(b) = \frac{3\zeta_p}{2}\frac{1}{\mu(n-1)}\left(-\frac{n(L_{\max}- L)}{b^2}+2(nL - L_{\max})\right),
\end{eqnarray*}
and its second derivative is
\begin{eqnarray*}
G''(b) = \frac{3\zeta_pn(L_{\max}- L)}{\mu(n-1)b^3} \geq 0.
\end{eqnarray*}
$G(b)$ is a convex function, and we can find its minimizer by setting its first derivative to zero. This minimizer is
\begin{eqnarray*}
\hat{b} \eqdef \sqrt{\frac{n}{2}\frac{L_{\max} - L}{nL - L_{\max}}}.
\end{eqnarray*}
Indeed, recall that from Lemma \ref{lem:nLgeqLmax}, we have $nL \geq L_{\max}$.

Thus, in this case, $C_{1/n}(b)$ is a convex function and its minimizer os \[b^* = \floor{\hat{b}}.\]
\item if $\frac{3\zeta_p}{2}\frac{L}{\mu}=\pi(n) \leq n \leq \pi(1) = \frac{3\zeta_p}{2}\frac{L_{\max}}{\mu}$. Then there exists $b \in [1, n]$ such that $\pi(b) = n$ and its expression is given by
\begin{eqnarray*}
\tilde{b} = \frac{\frac{3\zeta_p}{2}n(L_{\max} - L)}{\mu n(n-1) - \frac{3\zeta_p}{2}(nL - L_{\max})}.
\end{eqnarray*}

Consequently, the function $C_n(b)$ is decreasing on $\left[1, \min\left\{\tilde{b}, \hat{b}\right\}\right]$ and increasing on $\left[\min\left\{\tilde{b}, \hat{b}\right\}, n\right]$. Hence, \[b^* = \floor{\min\left\{\tilde{b}, \hat{b}\right\}}.\]
\end{itemize}
\end{proof}

\section{Samplings}\label{sec:samplings_appendix}
In Definition \ref{def:bnice_sampling}, we defined $b$--nice sampling. For completeness, we present here some other interesting possible samplings.

\begin{definition}[single-element sampling] \label{def:single-element_sampling}
Given a set of probabilities $(p_i)_{i\in[n]}$, $S$ is a single-element sampling if $\mathbb{P}(|S| = 1) = 1$ and
\[\mathbb{P}(S = \{i\}) = p_i \quad \forall i \in [n].\]
\end{definition}

\begin{definition}[partition sampling]\label{def:partition_sampling}
Given a partition $\mathcal{B}$ of $[n]$, $S$ is a partition sampling if \[ p_B \eqdef \mathbb{P}(S = B) > 0 \quad \forall B\in \mathcal{B}, \text{ and } \sum_{B \in \mathcal{B}}p_B = 1. \]
\end{definition}

\begin{definition}[independent sampling]\label{def:independent_sampling}
$S$ is an independent sampling if it includes every $i$ independently with probability $p_i > 0$.
\end{definition}

In Section \ref{sec:Exp_smoothness_appendix}, we will determine for each of these samplings their corresponding expected smoothness constant.

\section{Expected Smoothness} \label{sec:Exp_smoothness_appendix}
First, we present two general properties about the expected smoothness constant presented in Lemma \ref{lem:exp_smooth}: we establish its existence, and we prove that it is always greater than the strong convexity constant. Then, we determine the expected smoothness constant for particular samplings.
\subsection{General properties of the expected smoothness constant}
The following lemma is an adaptation of Theorem~3.6 in~\cite{SGD-AS}. It establishes the existence of the expected smoothness constant as a result of the smoothness and convexity of the functions $f_i, i \in [n]$.
\begin{lemma}[Theorem~3.6 in~\cite{SGD-AS}] \label{lemma:master_lemma}
Let $v$ be a sampling vector as defined in Definition \ref{ass:unbsn} with $v_i \geq 0$ with probability one . Suppose that $f_v (w) = \frac{1}{n} \sum_{i=1}^n f_i(w)v_i$  is $L_v$--smooth and convex. It follows that the expected smoothness constant \eqref{lem:exp_smooth} is given by
    \[\cL = \max_{i \in [n]}\E{ L_v v_i}.\]
\end{lemma}
\begin{proof}
Since the $f_i$'s are convex, each realization of $f_v$ is convex, and it follows from equation~2.1.7 in~\cite{Nesterov-convex}  that
\begin{equation}\label{eq:fvas2}
	\|\nabla f_v(x) - \nabla f_v(y)\|^2_2 \quad \leq \quad 2L_v\left( f_v(x) - f_v(y) - \langle \nabla f_v(y), x-y \rangle \right).
\end{equation}
Taking expectation over the sampling gives
\begin{eqnarray*}
		\E{\|\nabla f_v(x) - \nabla f_v(x^*)\|^2_2 } &\overset{\eqref{eq:fvas2}}{\leq}& 2\E{L_v \left( f_v(x) - f_v(x^*) - \langle \nabla f_v(x^*), x-x^* \rangle \right)} \\
    & \overset{\eqref{eq:fvt_exp}}{=}&
 \frac{2}{n}\E{\sum_{i =1}^n L_v v_i\left( f_i(x) - f_i(x^*) - \langle \nabla f_i(x^*), x-x^* \rangle \right)}  \\
 & =& 	 \frac{2}{n}\sum_{i =1}^n \E{ L_v v_i} \left( f_i(x) - f_i(y) - \langle \nabla f_i(x^*), x-x^* \rangle \right) \\
 &\overset{\eqref{eq:prob}}{\leq} & 2\max_{i=1,\ldots, n} \E{ L_v v_i} \left( f(x) - f(x^*) - \langle \nabla f(x^*), x-x^*\rangle \right) \\
 & =& 2\max_{i=1,\ldots, n} \E{ L_v v_i} \left( f(x) - f(x^*)\right).
\end{eqnarray*}
By comparing the above with~\eqref{eq:Expsmooth} we have that $\cL = \underset{i=1,\ldots, n}{\max} \E{ L_v v_i}.$
\end{proof}

\begin{lemma}[PL inequality]
If $f$ is $\mu$--strongly convex, then for all $x, y \in \R^d$
\begin{equation}
    \frac{1}{2 \mu} \norm{\nabla f(x)}_2^2 \geq f(x) -f(x^*), \quad \forall x \in \R^d.\label{eq:PL}
\end{equation}
\end{lemma}
\begin{proof}
Since $f$ is $\mu$--strongly convex, we have from, rearranging \eqref{eq:strconv2}, that for all $x, y \in \R^d$ \[f(y)  \geq  f(x) + \dotprod{\nabla f(x), y- x}+ \frac{\mu}{2} \norm{x - y}_2^2.\]
Minimizing both sides of this inequality in $y$ proves \eqref{eq:PL}.
\end{proof}

The following lemma shows that the expected smoothness constant is always greater than the strong convexity constant.
\begin{lemma} \label{lem:gammasmallcL}
If the expected smoothness inequality \eqref{eq:Expsmooth} holds with constant $\cL$ and $f$ is $\mu$--strongly convex, then  $\cL \geq \mu$.
\end{lemma}

\begin{proof}
 We have, since $\E{\nabla f_v(x) - \nabla f_v(x^*)} = \nabla f(x)$
 \begin{eqnarray}
 \E{\norm{\nabla f_v(x) - \nabla f_v(x^*) - \nabla f(x)}_2^2} &\overset{\text{Lem.}~\ref{lem:eq_var_bias}}{=}& \E{\norm{\nabla f_v(x) - \nabla f_v(x^*)}_2^2} - \norm{\nabla f(x)}_2^2 \nonumber \\
 &\overset{\eqref{eq:Expsmooth}+\eqref{eq:PL}}{\leq} &  2(\cL - \mu) (f(x) - f(x^*)). \label{eq:rho_proxy}
 \end{eqnarray}
 Hence $2(\cL - \mu)(f(x) - f(x^*)) \geq 0$, which means $\cL \geq \mu$.
\end{proof}

\begin{remark}\label{rem:rho_equal_cL}
Consider the expected residual constant $\rho$ defined in \ref{lem:exp_residual}. This constant verifies for all $x \in \R^d$,
\begin{eqnarray*}
\E{\norm{\nabla f_v(x) - \nabla f_v(x^*) - \nabla f(x)}} \leq 2\rho(f(x) - f(x^*)).
\end{eqnarray*}
From Equation \eqref{eq:rho_proxy}, we can see that we can use $\rho = \cL$ as the expected residual constant.
\end{remark}

\subsection{Expected smoothness constant for particular samplings}
The results on the expected smoothness constants related to the samplings we present here are all derived in \cite{SGD-AS} and thus are given without proof. The expected smoothness constant for $b$-nice sampling is given in Lemma \ref{lem:exp_smooth_residual_bnice}. Here, we present this constant for single-element sampling, partition sampling and independent sampling.

\begin{lemma}[$\cL$ for single-element sampling. Proposition 3.7 in \cite{SGD-AS}] \label{lem:exp_smoothness_single-element}
    Consider $S$ a single-element sampling from Definition \ref{def:single-element_sampling}. If for all $i \in [n]$, $f_i$ is $L_i$--smooth, then
    \begin{eqnarray*}
    \cL &=& \frac{1}{n}\max_{i\in[n]}\frac{L_i}{p_i}
    \end{eqnarray*}
    where $p_i = \mathbb{P}(S = \{i\})$.
\end{lemma}

\begin{remark}
Consider $S$ a single-element sampling from Definition \ref{def:single-element_sampling}. Then, the probabilities that maximize $\cL$ are \[p_i = \frac{L_i}{\sum_{j\in[n]}L_j}.\]
Consequently,
\[\cL = \bar{L} \eqdef \frac{1}{n}\sum_{i=1}^nL_i.\]
\end{remark}
In contrast, for uniform single-element sampling, \ie when $p_i = \frac{1}{n}$ for all $i$, we have $\cL = L_{\max}$, which can be significantly larger than $\bar{L}$. Since the step sizes of all our algorithms are a decreasing function of $\cL$, importance sampling can lead to much faster algorithms.\\

\begin{lemma}[$\cL$ for partition sampling. Proposition 3.7 in \cite{SGD-AS}] \label{lem:exp_smoothness_partition}
    Given a partition $\mathcal{B}$ of [n], consider $S$ a partition sampling from Definition \ref{def:independent_sampling}. For all $B \in \mathcal{B}$, suppose that $f_B(x) \eqdef \frac{1}{b} \sum_{i \in \mathcal{B}}f_i(x)$ is $L_B$--smooth. Then, with $p_B = \mathbb{P}(S = B)$
    \begin{eqnarray*}
    \cL &=& \frac{1}{n} \max_{B \in \mathcal{B}} \frac{L_B}{p_B}
    \end{eqnarray*}
\end{lemma}

\begin{lemma}[$\cL$ for independent sampling. Proposition 3.8 in \cite{SGD-AS}] \label{lem:exp_smoothness_independent}
    Consider $S$ a single-element sampling from Definition \ref{def:independent_sampling}. Note $p_i = \mathbb{P}(i \in S)$. If for all $i \in [n]$, $f_i$ is $L_i$--smooth and $f$ is $L$--smooth, then
    \begin{eqnarray*}
    \cL &=& L + \max_{i \in [n]} \frac{1 - p_i}{p_i}\frac{L_i}{n}
    \end{eqnarray*}
    where $p_i = \mathbb{P}(S = \{i\})$.
\end{lemma}

\section{Expected residual}\label{sec:Exp_residual_appendix}

In this section, we compute bounds on the expected residual $\rho$ from Lemma~\ref{lem:exp_residual}.

\begin{lemma} \label{lem:exp_residual_general}
Let $v = [v_1,\dots,v_n] \in \mathbb{R}^n$ be an unbiased sampling vector with $v_i \geq 0$ with probability one. It follows that the expected residual constant exists with
\begin{equation} 
    \rho = \frac{\lambda_{\max}(\Var{v})}{n}L_{\max},
\end{equation}
where $\Var{v} = \E{(v-\mathbb{1})(v - \mathbb{1})^\top}.$
\end{lemma}

Before the proof, let us introduce the following lemma (inspired from \url{https://www.cs.ubc.ca/~nickhar/W12/NotesMatrices.pdf}).
\begin{lemma}[Trace inequality]
    Let $A$ and $B$ be symmetric $n\times n$ such that $A \succcurlyeq 0$. Then,\\
    \begin{equation*}
        \Tr{AB} \leq  \lambda_{\max} (B) \Tr{A}
    \end{equation*}
\end{lemma}
\begin{proof}
Let $A = \sum_{i=1}^n \lambda_i (A) U_i U_i^\top$, where $\lambda_1 (A) \geq \ldots \geq \lambda_n (A) \geq 0$ denote the ordered eigenvalues of matrix $A$. Setting $V_i \eqdef \sqrt{\lambda_i (A)} U_i$ for all $i \in [n]$, we can write $A = \sum_{i=1}^n V_i V_i^\top$. Then,\\
\begin{align*}
    \Tr{AB} &= \Tr{\sum_{i=1}^n V_i V_i^\top B} = \sum_{i=1}^n \Tr{V_i V_i^\top B}            = \sum_{i=1}^n \Tr{V_i^\top B V_i} = \sum_{i=1}^n V_i^\top B V_i\\
            &\leq  \lambda_{\max} (B) \sum_{i=1}^n V_i^\top V_i = \lambda_{\max} (B) \Tr{A},
\end{align*}
where we use in the inequality that $B \preccurlyeq \lambda_{\max} (B) I_n$.
\end{proof}

We now turn to the proof of the theorem.
\begin{proof}
Let $v = [v_1,\dots,v_n] \in \mathbb{R}^n$ be an unbiased sampling vector with $v_i \geq 0$ with probability one. We will show that there exists $\rho \in \R_+$ such that:

\begin{equation}
    \E{\norm{\nabla f_v (w) -\nabla f_v (x^*) - ( \nabla f(w)-\nabla f(x^*)) }_2^2} \leq 2\rho \left(f(w)-f(x^*) \right).
\end{equation}

Let us expand the squared norm first. Define $DF(w)$ as the Jacobian of $F(w) \overset{def}{=} [f_1(w),\dots,f_n(w)]$ We denote $R \eqdef \left(DF (w) - DF (x^*)\right)$
\begin{align*}
    C &\eqdef \norm{\nabla f_v (w) -\nabla f_v (x^*) - (\nabla f(w)-\nabla f(x^*)) }_2^2\\
    &= \frac{1}{n^2}\norm{\left(DF (w) - DF (x^*)\right) ( v - \mathbb{1})}_2^2\\
    &= \frac{1}{n^2} \langle R ( v - \mathbb{1}), R ( v - \mathbb{1}) \rangle_{\mathbb{R}^d}\\
    &=\frac{1}{n^2} \Tr{(v - \mathbb{1})^\top R^\top R (v - \mathbb{1})}\\
    &= \frac{1}{n^2} \Tr{R^\top R (v - \mathbb{1}) (v - \mathbb{1})^\top}.\\
\end{align*}
Taking expectation,
\begin{eqnarray}
\E{C} &=& \frac{1}{n^2} \Tr{R^\top R \Var{v}} \nonumber \\
& \leq & \frac{1}{n^2} \Tr{R^\top R} \lambda_{\max}(\Var{v}). \label{eq:trace-eig}
\end{eqnarray}

Moreover, since the $f_i$'s are convex and $L_i$-smooth, it follows from equation~2.1.7 in~\cite{Nesterov-convex}  that
\begin{eqnarray}
\Tr{R^\top R} &=& \sum_{i=1}^n \norm{\nabla f_i(w) - \nabla f_i(x^*)}_2^2 \nonumber \\
&\leq& 2 \sum_{i=1}^n L_i (f_i(w) - f_i(x^*) - \langle \nabla f_i(x^*), w - x^* \rangle) \nonumber \\
&\leq& 2nL_{\max}(f(w) - f(x^*)). \label{eq:traceRtT}
\end{eqnarray}
Therefore,
\begin{eqnarray}
\E{C}\overset{\eqref{eq:trace-eig} + \eqref{eq:traceRtT}}{\leq}  2\frac{\lambda_{\max}(\Var{v})}{n}L_{\max} (f(w) - f(x^*)).
\end{eqnarray}
Which means
\begin{equation} \label{eq:residual_result}
    \rho = \frac{\lambda_{\max}(\Var{v})}{n}L_{\max}
\end{equation}
\end{proof}

Hence depending on the sampling $S$, we need to study the eigenvalues of the matrix $\Var{v}$, whose general term is given by
\begin{eqnarray}
    \label{eq:variance_v_gnl_term}
\left(\Var{v}\right)_{ij} &=& \left\{
    \begin{array}{ll}
        \frac{1}{p_i} - 1 & \mbox{if } i=j \\
        \frac{P_{ij}}{p_i p_j} - 1 & \mbox{otherwise, }
    \end{array} \right.
\end{eqnarray}

with
\begin{eqnarray}\label{eq:p_ip_ij}
p_i \eqdef \mathbb{P}(i\in S) \text{ and } P_{ij} \eqdef \mathbb{P}(i\in S, j \in S) \text{ for } i, j \in [n]
\end{eqnarray}

To specialize our results to particular samplings, we introduce some notations:
\begin{itemize}
    \item $\mathcal{B}$ designates all the possible sets for the sampling $S$,
    \item $b = |B|$, where $B \in \mathcal{B}$, when the sizes of all the elements of $\mathcal{B}$ are equal.
\end{itemize}



\subsection{Expected residual for uniform $b$-nice sampling}
\begin{lemma}[$\rho$ for $b$-nice sampling] \label{lem:exp_residual_bnice}
    Consider $b$-nice sampling from Definition \ref{def:bnice_sampling}. If each $f_i$ is $L_{\max}$-smooth, then
    \begin{eqnarray}
\rho = \frac{n-b}{(n-1)b}L_{\max}. \label{rho_b-nice}
	\end{eqnarray}
\end{lemma}

\begin{proof}
For uniform $b$-nice sampling, we have using notations from \eqref{eq:p_ip_ij}
\begin{eqnarray}
\forall i \in [n], p_i &=& \frac{c_1}{|\cB|}, \nonumber \\
\forall i, j \in [n], P_{ij} &=& \frac{c_2}{|\cB|}, \nonumber
\end{eqnarray}
with $c_1 = \binom{n-1}{b-1}$, $c_2 = \binom{n-2}{b-2}$ and $|\cB| = \binom{n}{b}$. Hence, \[
  \Var{v} \overset{\eqref{eq:variance_v_gnl_term}}{=}
  \left[ {\begin{array}{ccccc}
    \frac{|\cB|}{c_1} - 1& \frac{|\cB|c_2}{c_1^2} - 1  & \dots & \frac{|\cB|c_2}{c_1^2} - 1 & \frac{|\cB|c_2}{c_1^2} - 1\\
   \frac{|\cB|c_2}{c_1^2} - 1 & \frac{|\cB|}{c_1} - 1 & \dots & \frac{|\cB|c_2}{c_1^2} - 1 & \frac{|\cB|c_2}{c_1^2} - 1\\
   \vdots &  & \ddots &  & \vdots\\
   \frac{|\cB|c_2}{c_1^2} - 1 & \dots & \dots & \frac{|\cB|}{c_1} - 1& \frac{|\cB|c_2}{c_1^2} - 1\\
   \frac{|\cB|c_2}{c_1^2} - 1 & \dots & \dots & \frac{|\cB|c_2}{c_1^2} - 1 & \frac{|\cB|}{c_1} - 1\\
  \end{array} } \right].
\]

As noted in Appendix C of \cite{JakSketch}, $\Var{v}$ is then a circulant matrix with associated vector
\begin{eqnarray}
\left(\frac{|\cB|}{c_1} - 1, \frac{|\cB|c_2}{c_1^2} - 1, \dots, \frac{|\cB|c_2}{c_1^2} - 1\right), \nonumber
\end{eqnarray}
and, as such, it has two eigenvalues
\begin{eqnarray}
\lambda_1 &\eqdef& \frac{|\cB|}{c_1}\left(1 + (n-1)\frac{c_2}{c_1} \right) -n = 0, \nonumber \\
\lambda_2 &\eqdef& \frac{|\cB|}{c_1}\left(1 - \frac{c_2}{c_1} \right) = \frac{n(n-b)}{b(n-1)}. \label{eq:lambda_max}
\end{eqnarray}

Hence, the expected residual can be computed explicitly as
\begin{eqnarray}
\rho \overset{\eqref{eq:residual_result}}{=} \frac{n-b}{(n-1)b}L_{\max}.
\end{eqnarray}

\end{proof}

We can see that the residual constant is a decreasing function of $b$ and in particular: $\rho(1) = L_{\max}$ and $\rho(n) = 0$.

\subsection{Expected residual for uniform partition sampling}

\begin{lemma}[$\rho$ for uniform partition sampling]
Suppose that $b$ divises $n$ and consider partition sampling from Definition \ref{def:partition_sampling}. Given a partition $\mathcal{B}$ of $[n]$ of size $\frac{b}{n}$, if each $f_i$ is $L_{\max}$-smooth, then,
\begin{eqnarray}
\rho = \left(1 - \frac{b}{n} \right) L_{\max}.
\end{eqnarray}
\end{lemma}
\begin{proof}
Recall that for partition sampling, we choose \textit{a priori} a partition $\cB = B_1\sqcup\dots\sqcup B_{\frac{n}{b}}$ of $[n]$. Then, for $k \in [\frac{n}{b}]$,
\begin{eqnarray}
\forall i \in [n], p_i &=& \left\{
    \begin{array}{ll}
        p_{B_k} = \frac{b}{n} & \mbox{if } i \in B_k\\
        0 & \mbox{otherwise, }
    \end{array} \right. \\
\forall i, j \in [n], P_{ij} &=& \left\{
    \begin{array}{ll}
        p_{B_k} = \frac{b}{n} & \mbox{if } i, j \in B_k\\
        0 & \mbox{otherwise. }
    \end{array} \right.
\end{eqnarray}

Let $k \in [\frac{n}{b}]$. If $i,j \in B_k$, then $\frac{1}{p_i} - 1 = \frac{P_{ij}}{p_i p_j} - 1 = \frac{n}{b}-1$.

As a result, up to a reordering of the observations, $\Var{v}$ is a block diagonal matrix, whose diagonal matrices, which are all equal, are given by, for $k \in [\frac{n}{b}]$,
\[
  V_k = (\frac{n}{b} - 1)\mathbb{1}_{b}\mathbb{1}_{b}^\top =
  \left[ {\begin{array}{ccccc}
    \frac{n}{b} -1 & \frac{n}{b} - 1 & \dots & \frac{n}{b} - 1 & \frac{n}{b} - 1\\
   \frac{n}{b} - 1 & \frac{n}{b} - 1 & \dots & \frac{n}{b} - 1 & \frac{n}{b} - 1\\
   \vdots &  & \ddots &  & \vdots\\
   \frac{n}{b} - 1 & \dots & \dots & \frac{n}{b} - 1 & \frac{n}{b} - 1\\
   \frac{n}{b} - 1 & \dots & \dots & \frac{n}{b} - 1 & \frac{n}{b} - 1\\
  \end{array} } \right] \in \mathbb{R}^{b \times b}.
\]

Since all the matrices on the diagonal are equal, the eigenvalues of $\Var{v}$ are simply those of one of these matrices. Any matrix $ V_k = (\frac{n}{b} - 1)\mathbb{1}_{b}\mathbb{1}_{b}^\top$ we consider has two eigenvalues: $0$ and $n - b$. Then,
\begin{eqnarray}
\rho \overset{\eqref{eq:residual_result}}{=} \left(1 - \frac{b}{n} \right) L_{\max}.
\end{eqnarray}
\end{proof}

If $b=n$, SVRG with uniform partition sampling boils down to gradient descent as we recover $\rho = 0$. For $b=1$, we have $\rho = \left(1 - \frac{1}{n} \right) L_{\max}$.



\subsection{Expected residual for independent sampling}
\begin{lemma}[$\rho$ for independent sampling]
Consider independent sampling from Definition \ref{def:partition_sampling}. Let $p_i = \mathbb{P}(i \in S)$. If each $f_i$ is $L_{\max}$-smooth, then
\begin{equation}
\rho = \left( \frac{1}{\underset{i\in[n]}{\min}p_i} - 1 \right) \frac{L_{\max}}{n}.
\end{equation}
\end{lemma}

\begin{proof}
Using the notations from \eqref{eq:p_ip_ij}, we have
\begin{eqnarray}
\forall i \in [n], p_i &=& p_i, \nonumber \\
\forall i, j \in [n], P_{ij} &=& p_i p_j \quad \mbox{when } i \neq j.\nonumber
\end{eqnarray}
Thus, according to \eqref{eq:variance_v_gnl_term}:
\begin{equation*}
    \Var{v} = \Diag \left( \frac{1}{p_1} - 1, \frac{1}{p_2} - 1, \ldots, \frac{1}{p_n} - 1 \right).
\end{equation*}
whose largest eigenvalue is
\begin{equation*}
    \lambda_{\max}(\Var{v}) = \max_{i\in[n]} \frac{1}{p_i} - 1 = \frac{1}{\underset{i\in[n]}{\min}p_i} - 1.
\end{equation*}

Consequently,
\begin{equation}
    \rho \overset{\eqref{eq:residual_result}}{=} \left( \frac{1}{\underset{i\in[n]}{\min}p_i} - 1 \right) \frac{L_{\max}}{n}.
\end{equation}
\end{proof}
If $p_i = \frac{1}{n}$ for all $i \in [n]$, which corresponds in expectation to uniform single-element sampling SVRG since $\E{|S|} = 1$, we have $\rho = \frac{n-1}{n} L_{\max}$. While if $p_i = 1$ for all $i \in [n]$, this leads to gradient descent and we recover $\rho = 0$.

The following remark gives a condition to construct an independent sampling with $E{|S|} = b$.
\begin{remark}
One can add the following condition on the probabilities: $\sum_{i=1}^n p_i = b$, such that $\E{|S|} = b$. Such a sampling is called $b$-independent sampling. This condition is obviously met if $p_i = \frac{b}{n}$ for all $i \in [n]$.
\end{remark}
\begin{lemma}
    \label{lem:avg_b_independent_sampling}
    Let $S$ be a independent sampling from $[n]$ and let $p_i = \Prb{i \in S}$ for all $i \in [n]$. If $\sum_{i=1}^n p_i = b$, then $\E{|S|} = b$.
\end{lemma}
\begin{proof}
    Let us model our sampling by a tossing of $n$ independent rigged coins. Let $X_1, \ldots, X_n$ be $n$ Bernoulli random variables representing these tossed coin, \ie $X_i \sim \cB(p_i)$, with $p_i \in [0,1]$ for $i \in [n]$. If $X_i = 1$, then the point $i$ is selected in the sampling $S$. Thus the number of selected points in the mini-batch $|S|$ can be denoted as the following random variable $\sum_{i=1}^n X_i$, and its expectation equals
    \begin{equation*}
        \E{|S|} = \E{\sum_{i=1}^n X_i} = \sum_{i=1}^n \E{X_i} = \sum_{i=1}^n p_i = b \enspace.
    \end{equation*}
\end{proof}
\begin{remark}
    Note that one does not need the independence of the $(X_i)_{i=1,\ldots,n}$.
\end{remark}

\subsection{Expected residual for single-element sampling}
From Remark \ref{rem:rho_equal_cL}, we can take $\cL$ as the expected residual constant. Thus, we simply use the expected smoothness constant from Lemma \ref{lem:exp_smoothness_single-element}.
\begin{lemma}[$\rho$ for single-element sampling] \label{lem:exp_residual_single-element}
Consider single-element sampling from Definition \ref{def:single-element_sampling}. If for all $i \in [n]$, $f_i$ is $L_i$-smooth, then \[\rho = \frac{1}{n}\max_{i\in[n]}\frac{L_i}{p_i}.\]
\end{lemma}

\section{Additional experiments}
\label{sec:additiona_exps}

\subsection{Comparison of theoretical variants of SVRG}
\label{sec:add_theoretical_exps}

In this series of experiments, we compare the performance of the SVRG algorithm with the settings of \cite{johnson2013accelerating} against \textit{Free-SVRG} and \textit{L-SVRG-D} with the settings given by our theory.

\subsubsection{Experiment 1.a: comparison without mini-batching ($b=1$)}
\label{sec:exp_1a_no_minibatching}
A widely used choice for the size of the inner loop is $m = n$. Since our algorithms allow for a free choice of the size of the inner loop, we set $m=n$ for \textit{Free-SVRG} and $p=1/n$ for \textit{L-SVRG-D}, and use a mini-batch size $b=1$. For vanilla \textit{SVRG}, we set $m$ to its theoretical value $20L_{\max}/\mu$ as in~\cite{bubeck2015convex}. See Figures~\ref{fig:exp1A_YearPredictionMSD},~\ref{fig:exp1A_slice},~\ref{fig:exp1A_ijcnn1} and \ref{fig:exp1A_real-sim}. We can see that \textit{Free-SVRG} and \textit{L-SVRG-D} often outperform the SVRG algorithm \cite{johnson2013accelerating}. It is worth noting that, in Figure~\ref{fig:exp1A_YearPredictionMSD_1e-1},~\ref{fig:exp1A_ijcnn1_1e-1} and~\ref{fig:exp1A_real-sim} the classic version of SVRG can lead to increase of the suboptimality when entering the outer loop. This is due to the fact that the reference point is set to a weighted average of the iterates of the inner loop, instead of the last iterate.

\begin{figure}[!htb]
  \vskip 0.2in
  \begin{center}
    \begin{subfigure}[b]{0.8\textwidth}
        \includegraphics[width=\textwidth]{exp1a/legend_exp1a_horizontal}
      \end{subfigure}\\
      \begin{subfigure}[b]{\textwidth}
        \centering
        \includegraphics[width=0.45\textwidth]{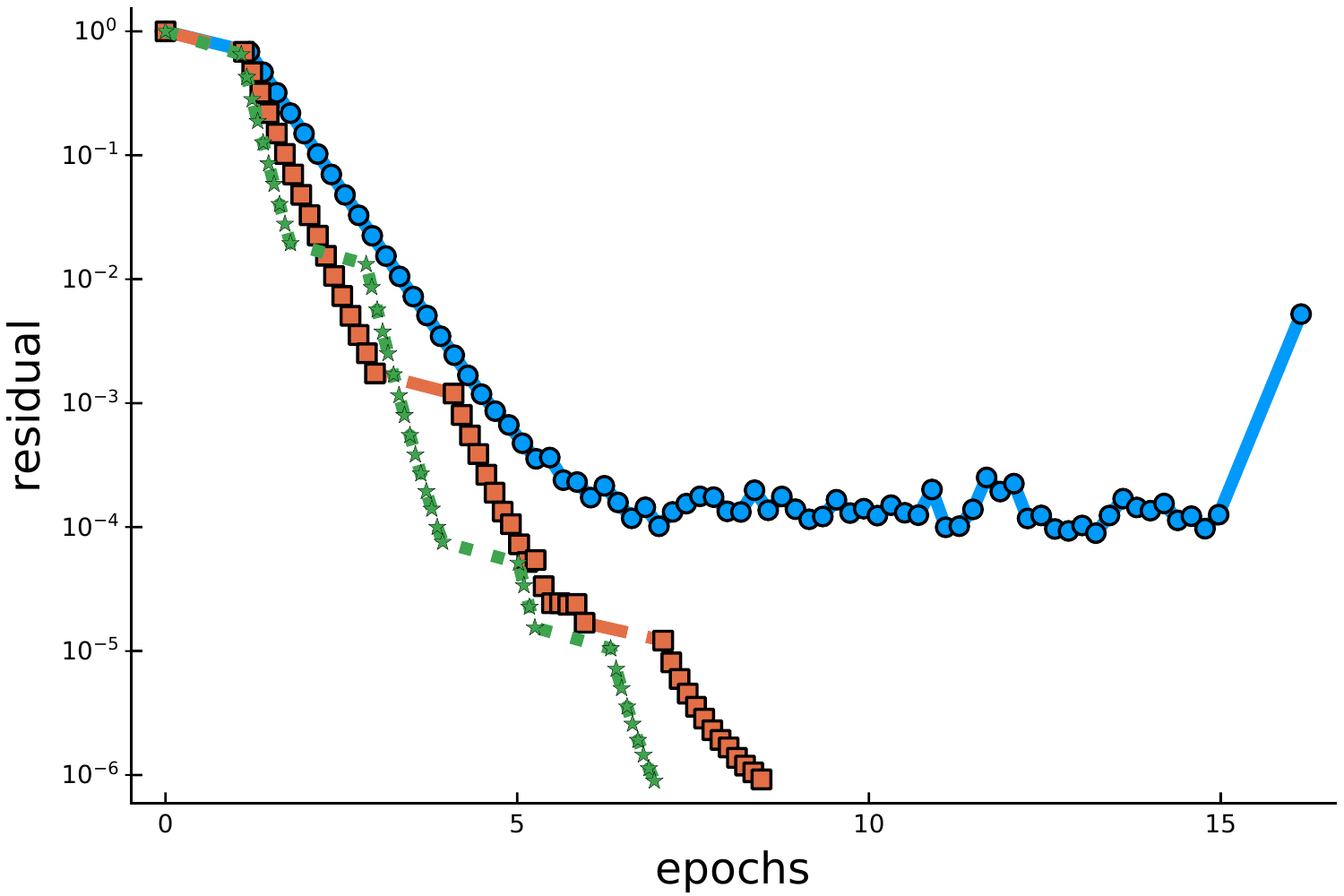}
        \includegraphics[width=0.45\textwidth]{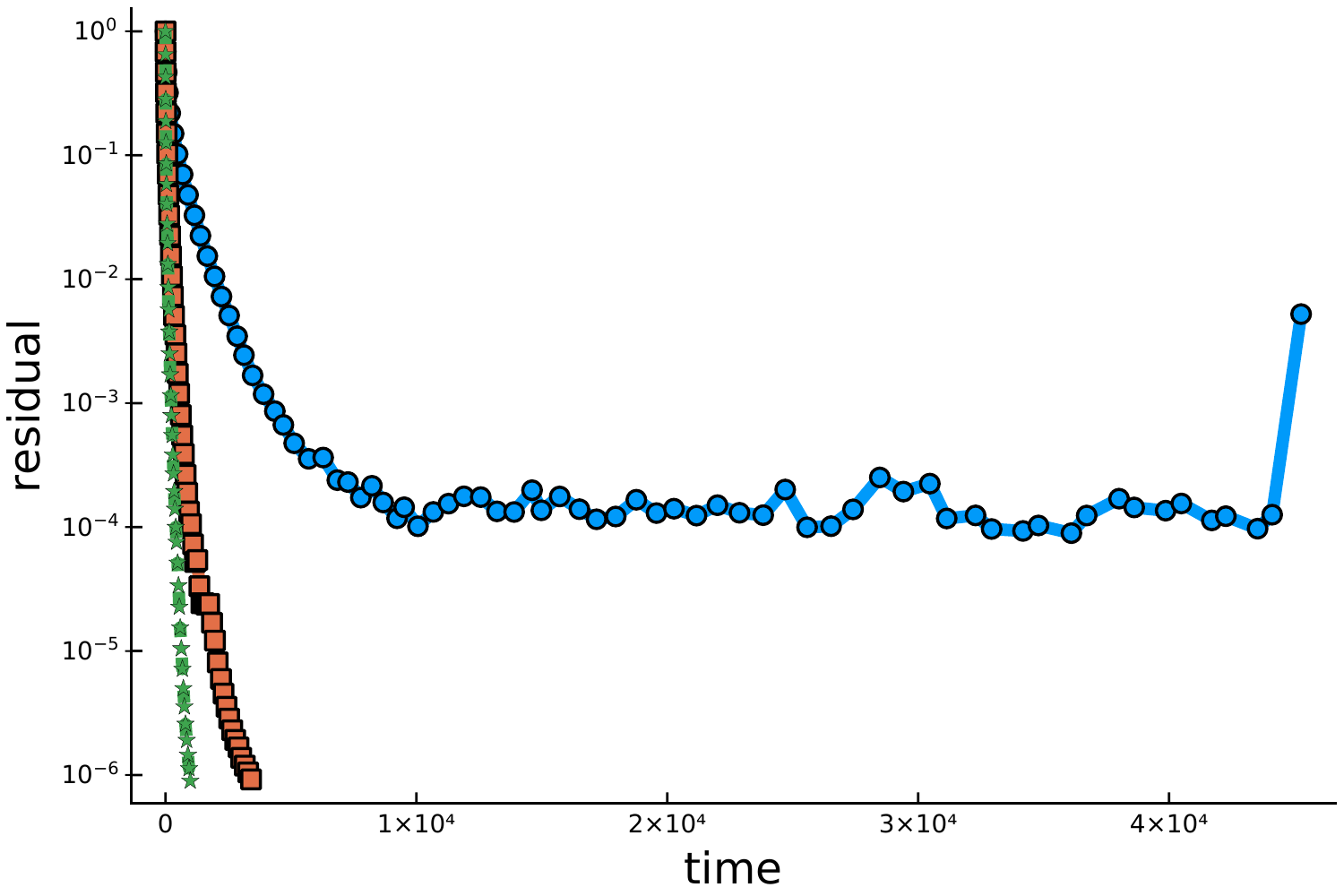}
        \caption{$\lambda = 10^{-1}$}
        \label{fig:exp1A_YearPredictionMSD_1e-1}
      \end{subfigure}\\
      \begin{subfigure}[b]{\textwidth}
        \centering
        \includegraphics[width=0.45\textwidth]{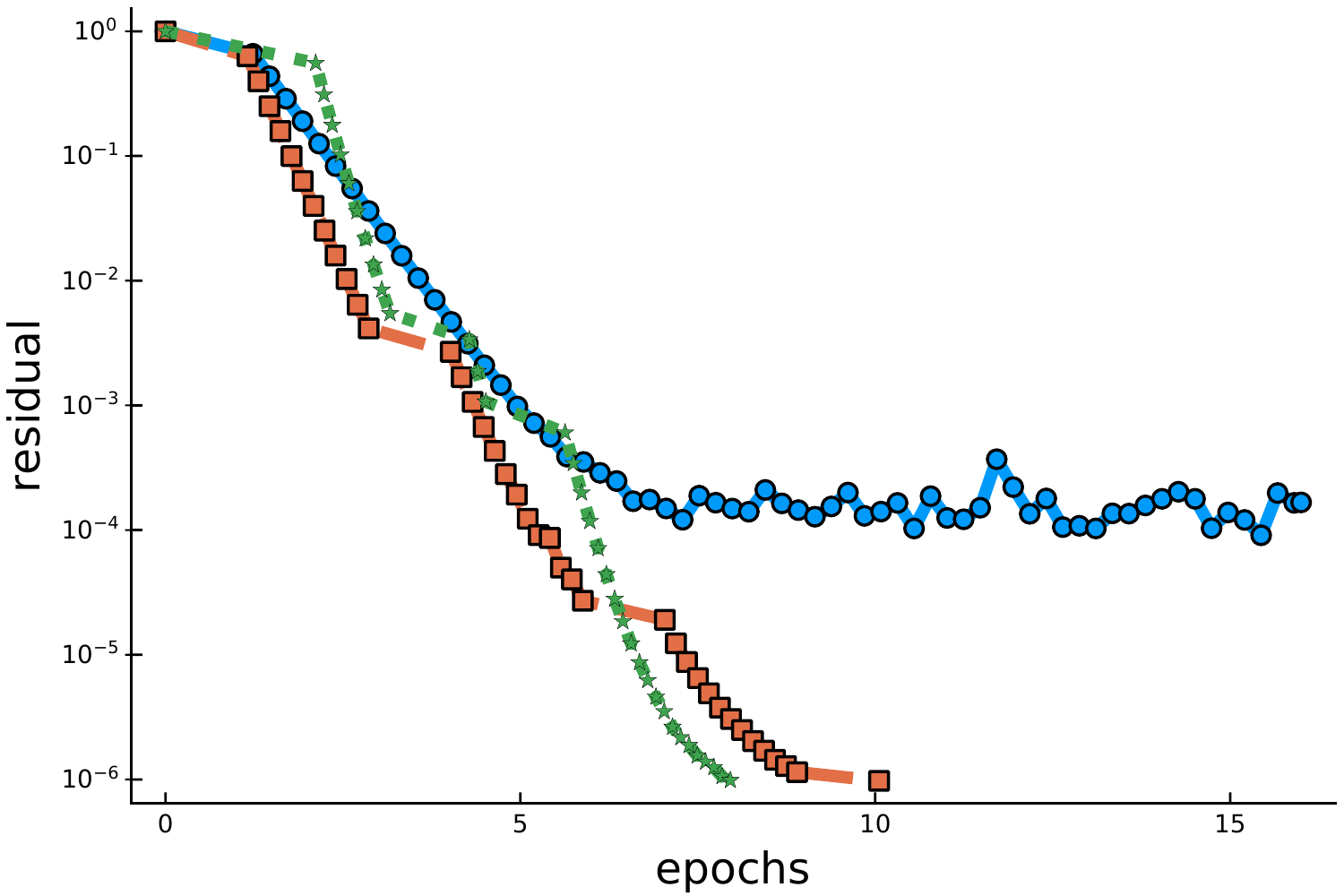}
        \includegraphics[width=0.45\textwidth]{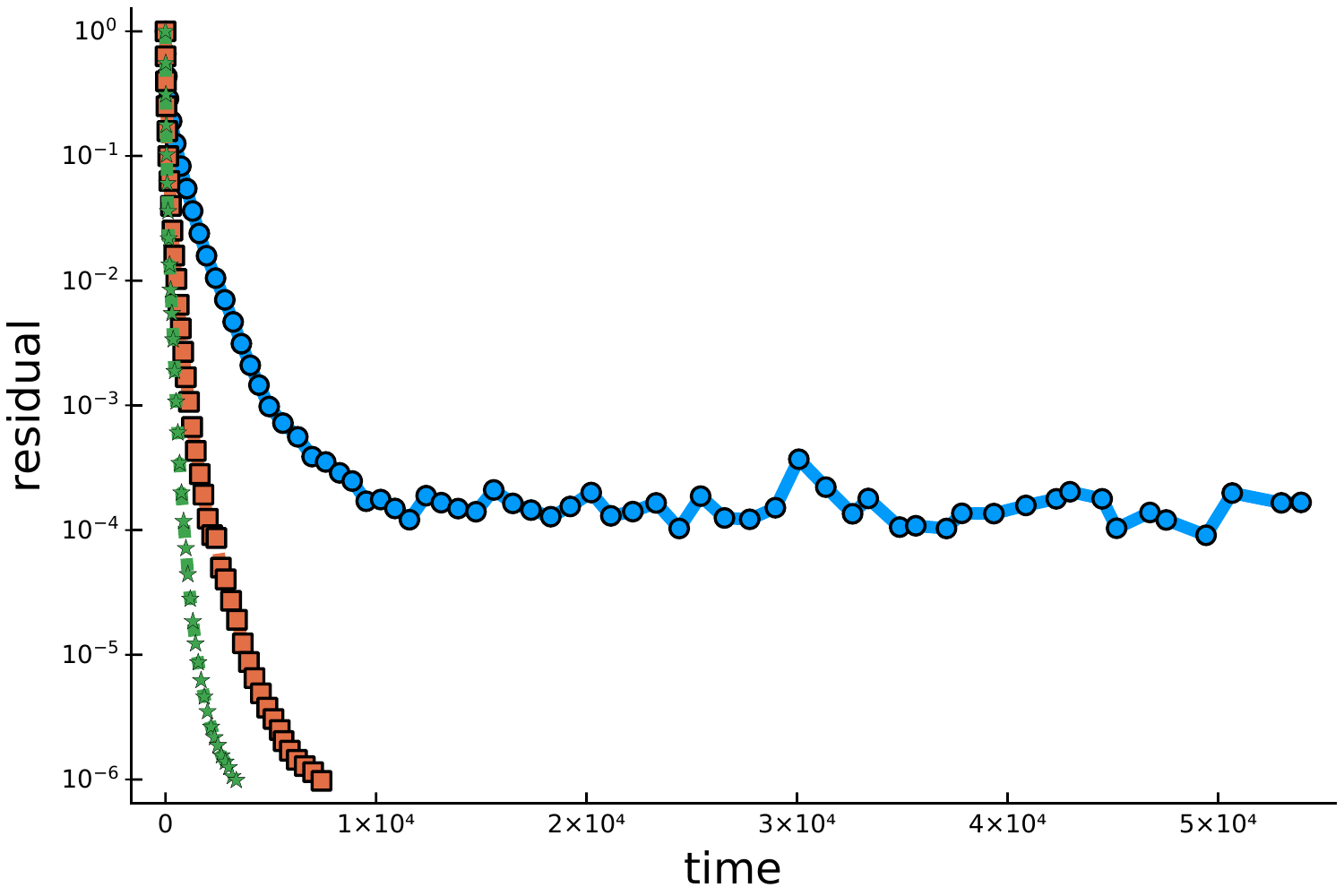}
        \caption{$\lambda = 10^{-3}$}
      \end{subfigure}
  \caption{Comparison of theoretical variants of SVRG without mini-batching ($b=1$) on the \textit{YearPredictionMSD} data set.}
  \label{fig:exp1A_YearPredictionMSD}
  \end{center}
  \vskip -0.2in
\end{figure}

\begin{figure}[!htb]
  \vskip 0.2in
  \begin{center}
    \begin{subfigure}[b]{0.8\textwidth}
        \includegraphics[width=\textwidth]{exp1a/legend_exp1a_horizontal}
      \end{subfigure}\\
      \begin{subfigure}[b]{\textwidth}
        \centering
        \includegraphics[width=0.45\textwidth]{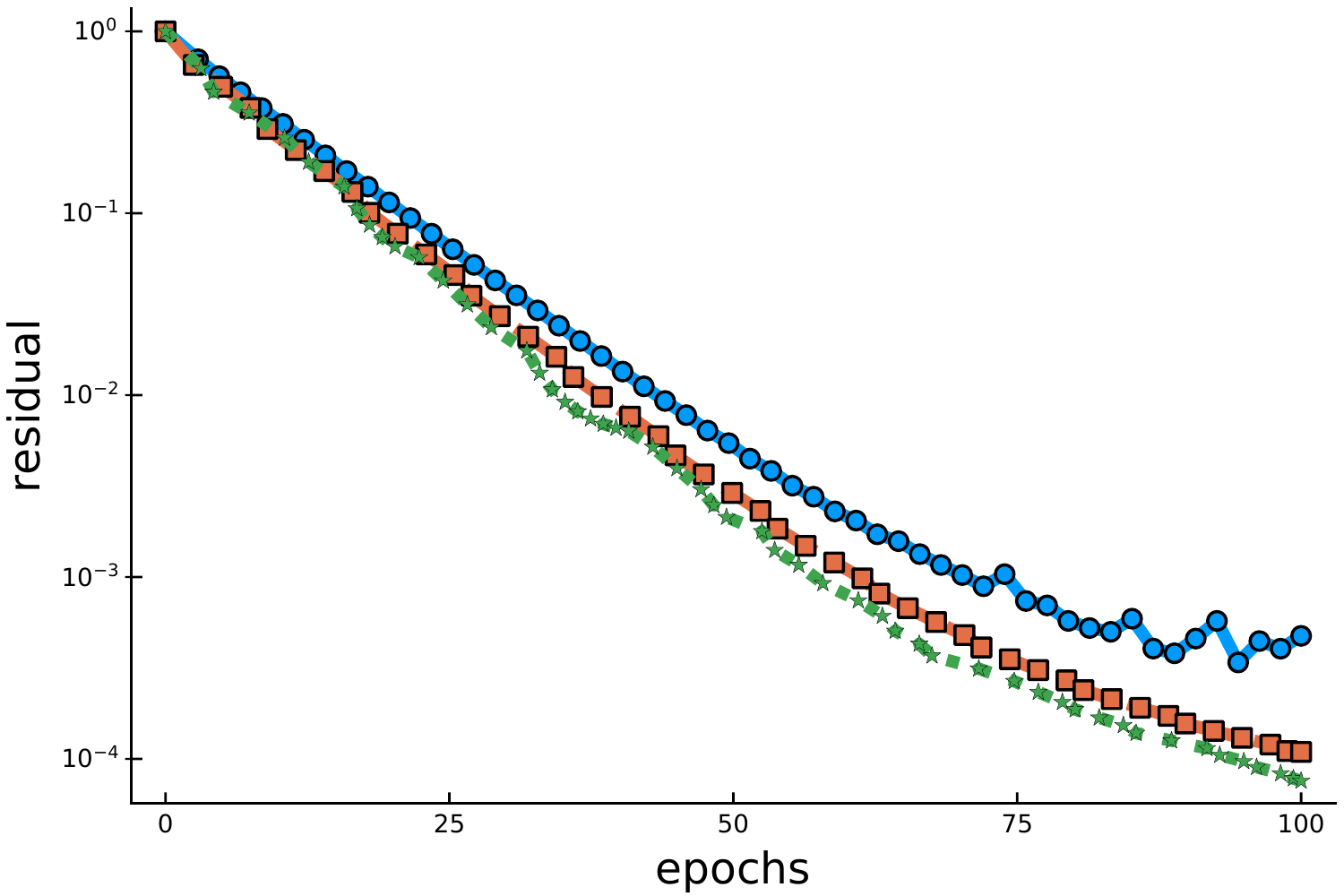}
        \includegraphics[width=0.45\textwidth]{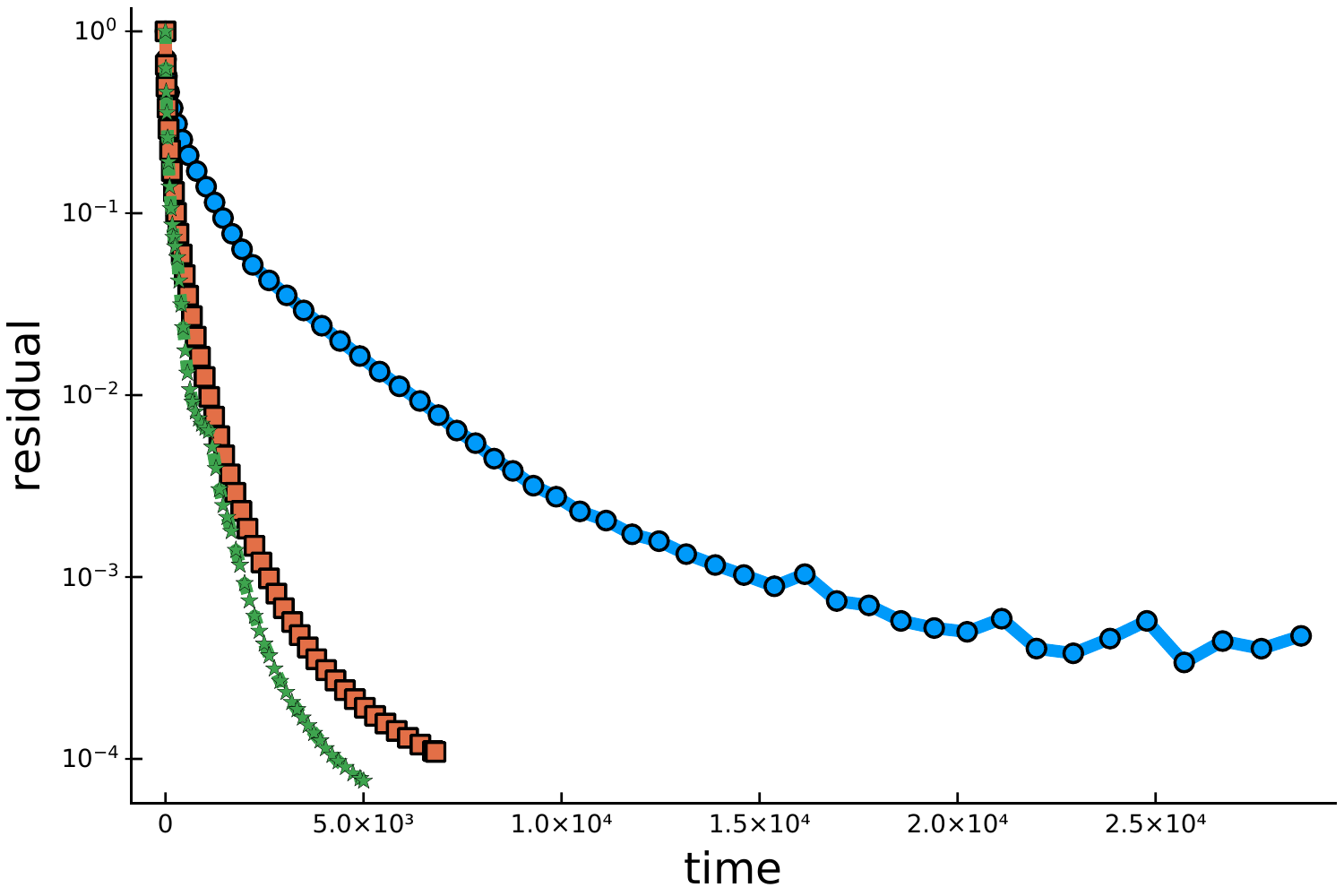}
        \caption{$\lambda = 10^{-1}$}
      \end{subfigure}\\
      \begin{subfigure}[b]{\textwidth}
        \centering
        \includegraphics[width=0.45\textwidth]{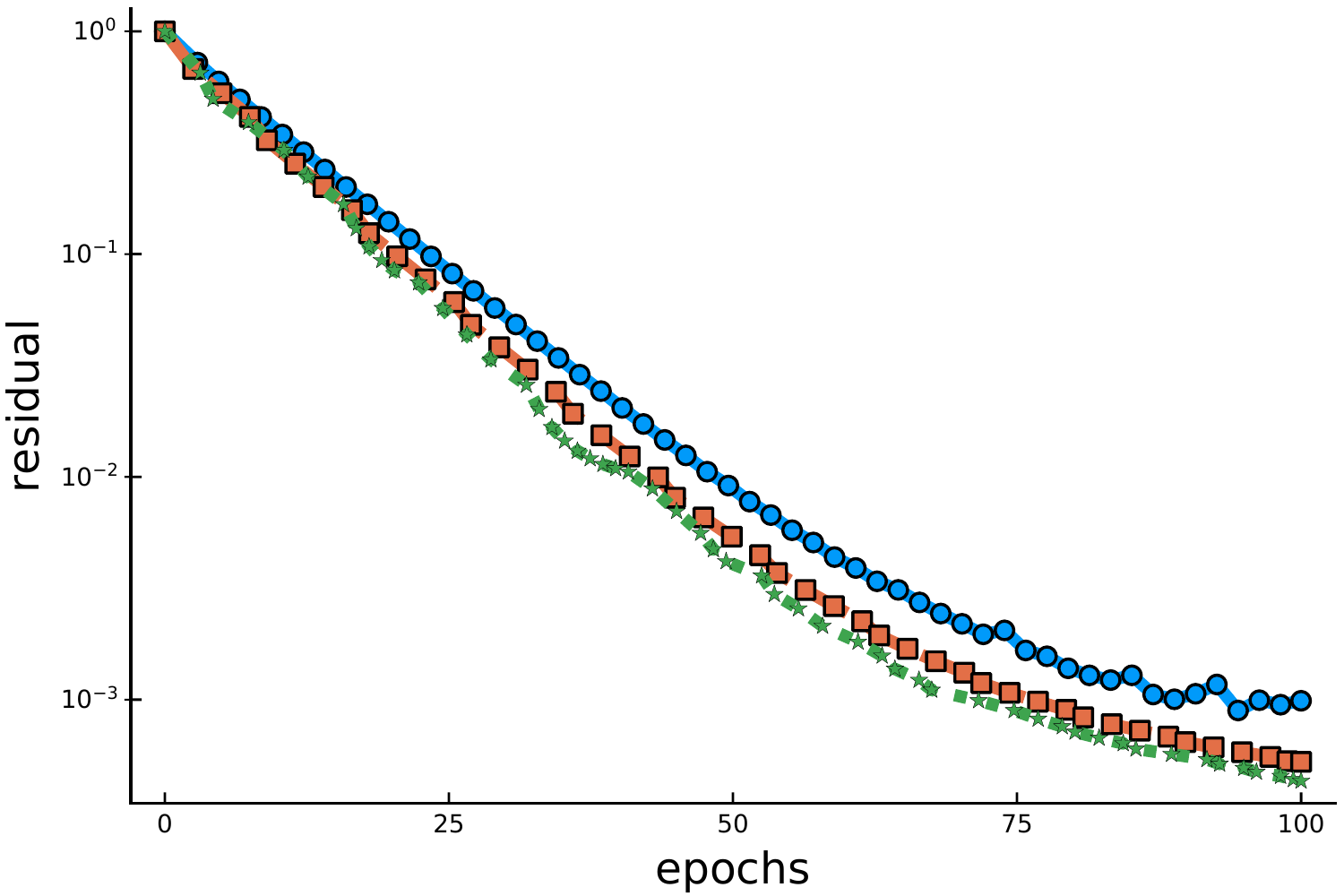}
        \includegraphics[width=0.45\textwidth]{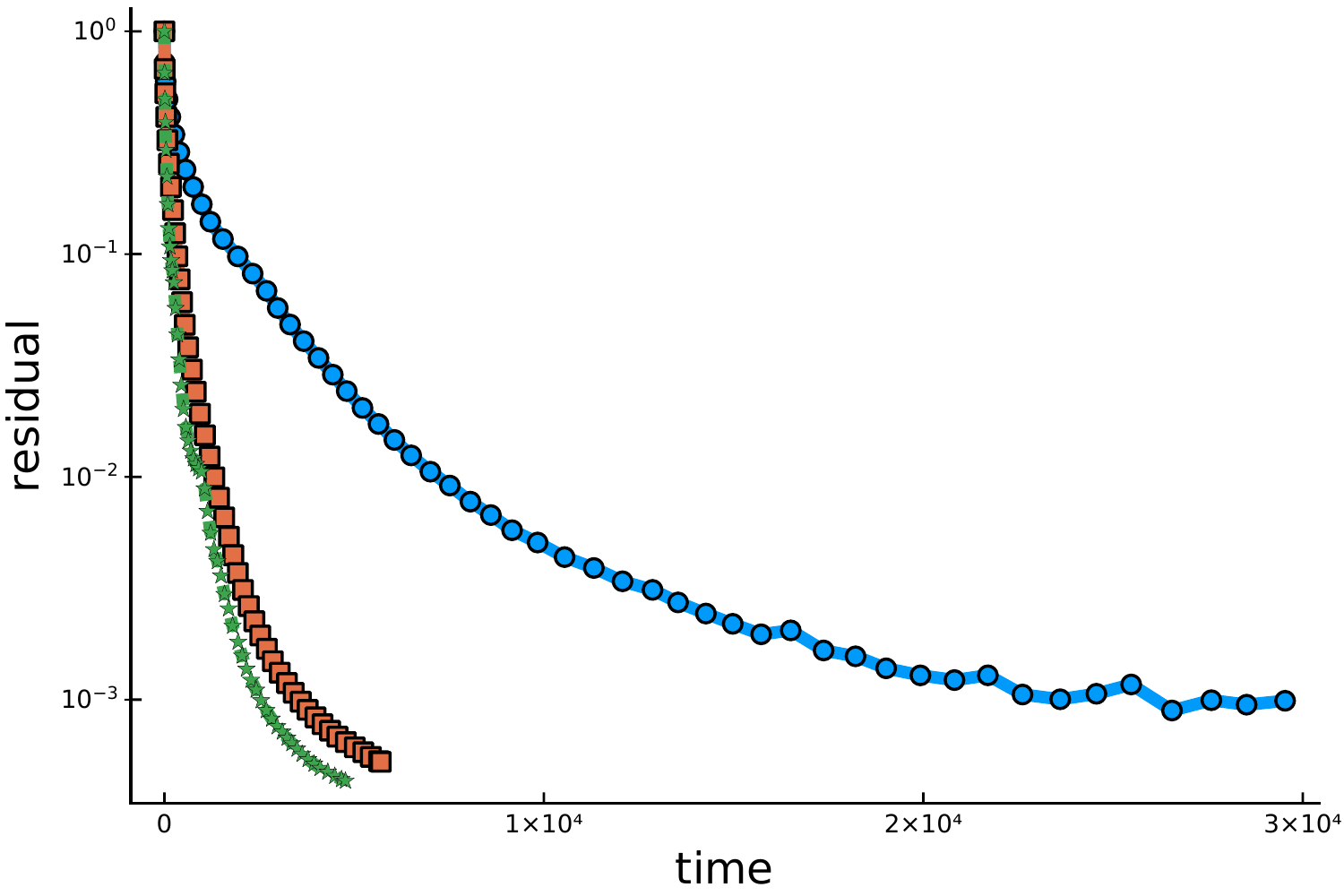}
        \caption{$\lambda = 10^{-3}$}
      \end{subfigure}
  \caption{Comparison of theoretical variants of SVRG without mini-batching ($b=1$) on the \textit{slice} data set.}
  \label{fig:exp1A_slice}
  \end{center}
  \vskip -0.2in
\end{figure}

\begin{figure}[!htb]
  \vskip 0.2in
  \begin{center}
    \begin{subfigure}[b]{0.8\textwidth}
        \includegraphics[width=\textwidth]{exp1a/legend_exp1a_horizontal}
      \end{subfigure}\\
      \begin{subfigure}[b]{\textwidth}
        \centering
        \includegraphics[width=0.45\textwidth]{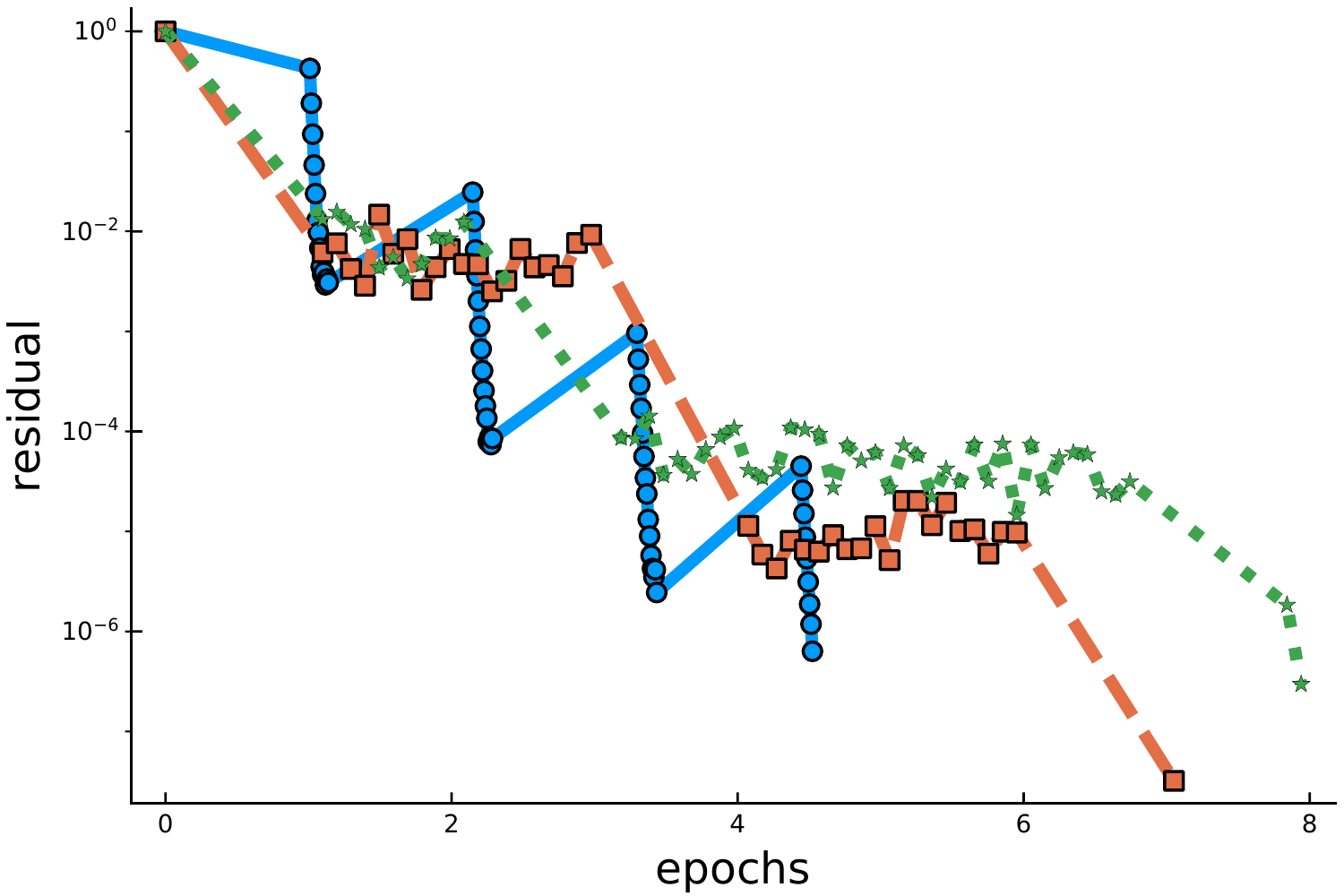}
        \includegraphics[width=0.45\textwidth]{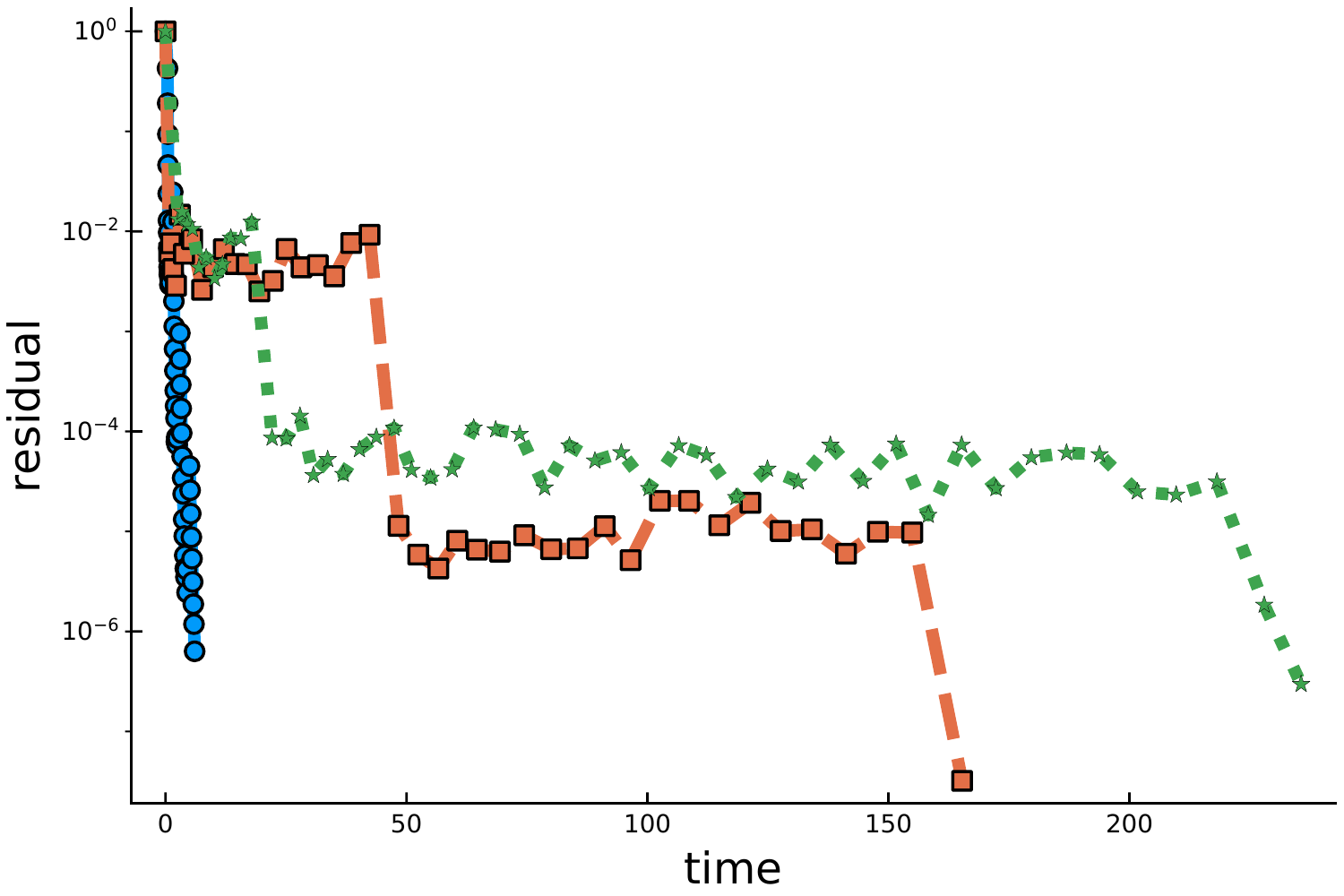}
        \caption{$\lambda = 10^{-1}$}
        \label{fig:exp1A_ijcnn1_1e-1}
      \end{subfigure}\\
      \begin{subfigure}[b]{\textwidth}
        \centering
        \includegraphics[width=0.45\textwidth]{exp1a/lgstc_ijcnn1_full-column-scaling-regularizor-1e-03-exp1a-lame23-final-epoc}
        \includegraphics[width=0.45\textwidth]{exp1a/lgstc_ijcnn1_full-column-scaling-regularizor-1e-03-exp1a-lame23-final-time}
        \caption{$\lambda = 10^{-3}$}
      \end{subfigure}
  \caption{Comparison of theoretical variants of SVRG without mini-batching ($b=1$) on the \textit{ijcnn1} data set.}
  \label{fig:exp1A_ijcnn1}
  \end{center}
  \vskip -0.2in
\end{figure}

\begin{figure}[!htb]
  \vskip 0.2in
  \begin{center}
    \begin{subfigure}[b]{0.8\textwidth}
        \includegraphics[width=\textwidth]{exp1a/legend_exp1a_horizontal}
      \end{subfigure}\\
      \begin{subfigure}[b]{\textwidth}
        \centering
        \includegraphics[width=0.45\textwidth]{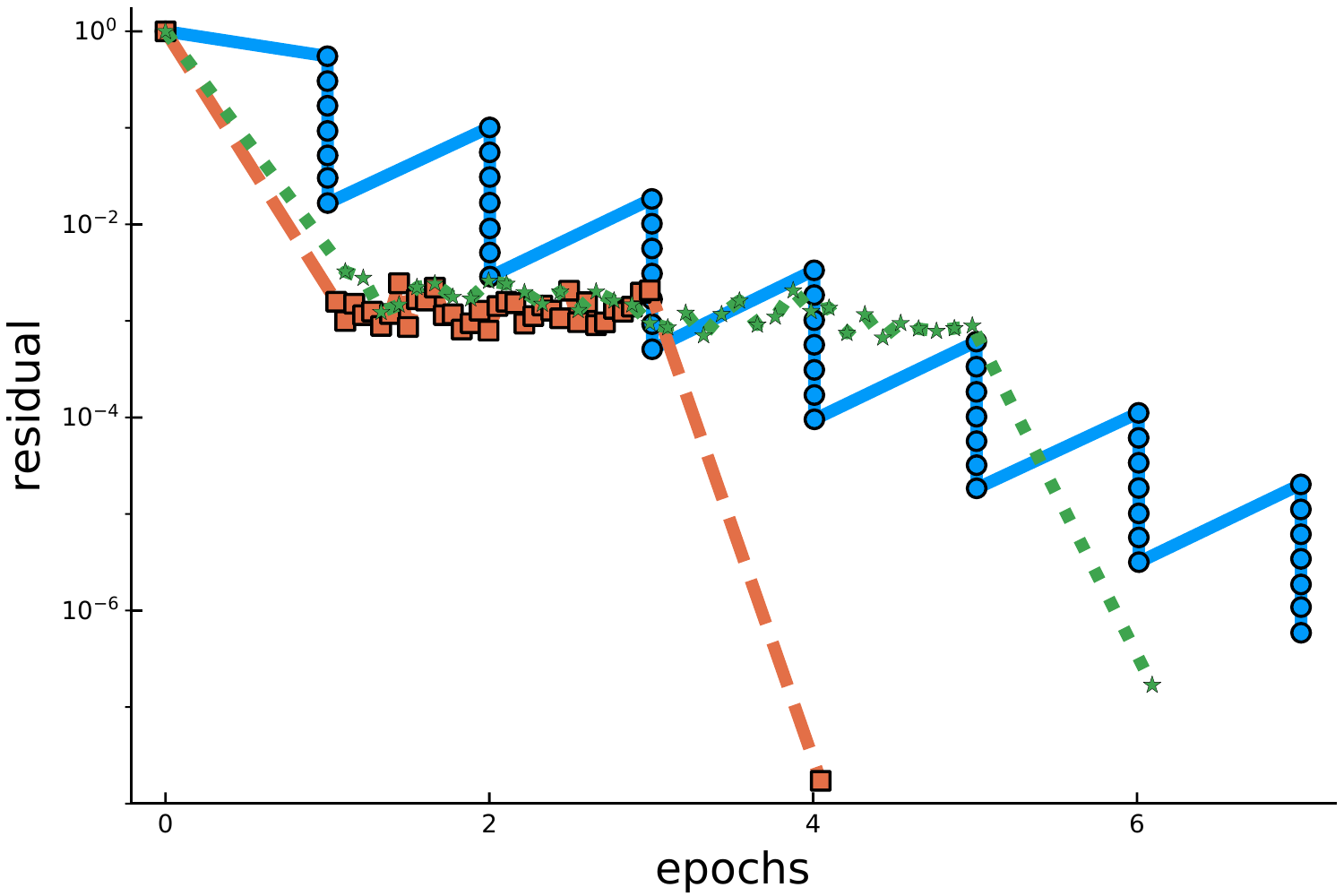}
        \includegraphics[width=0.45\textwidth]{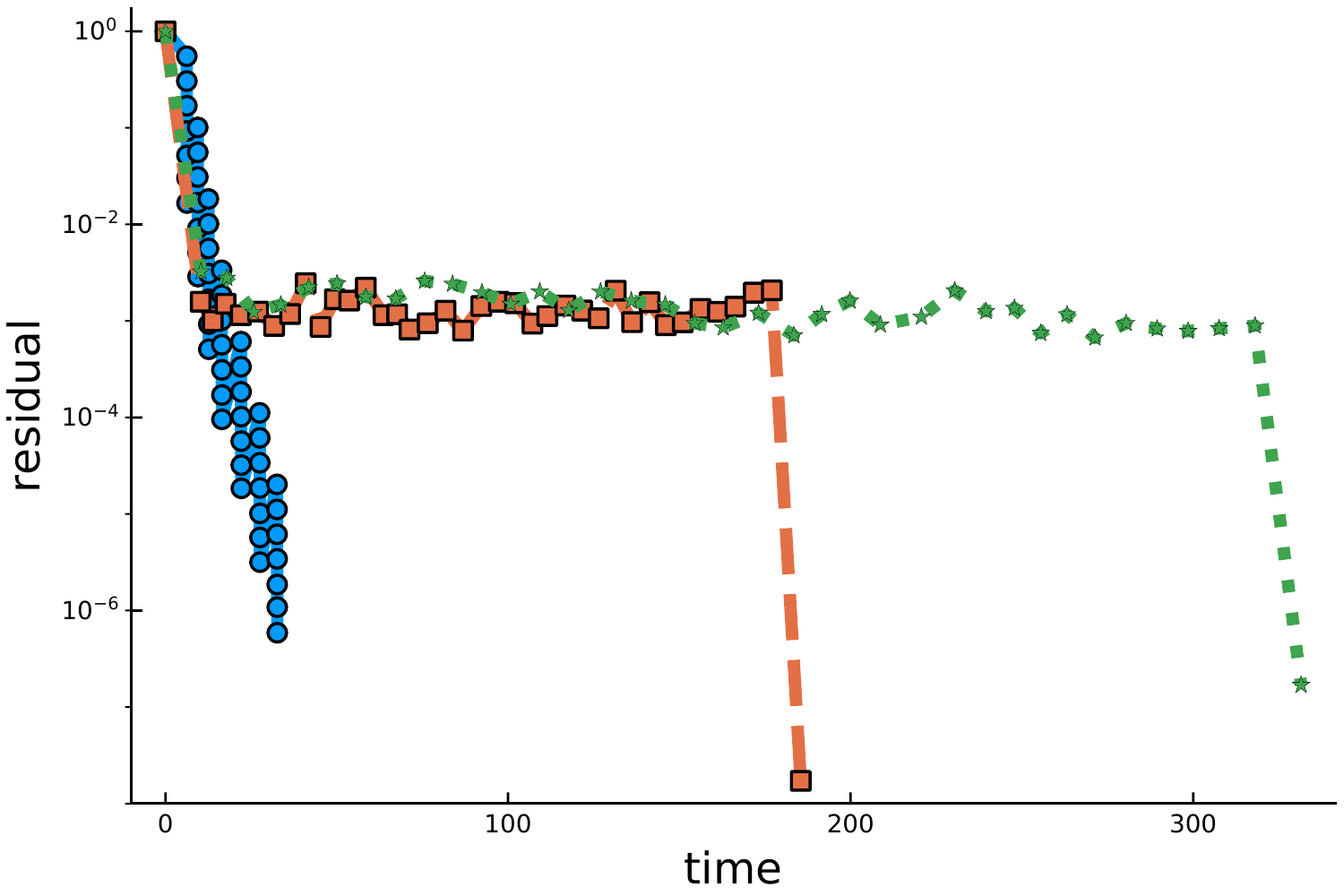}
        \caption{$\lambda = 10^{-1}$}
      \end{subfigure}\\
      \begin{subfigure}[b]{\textwidth}
        \centering
        \includegraphics[width=0.45\textwidth]{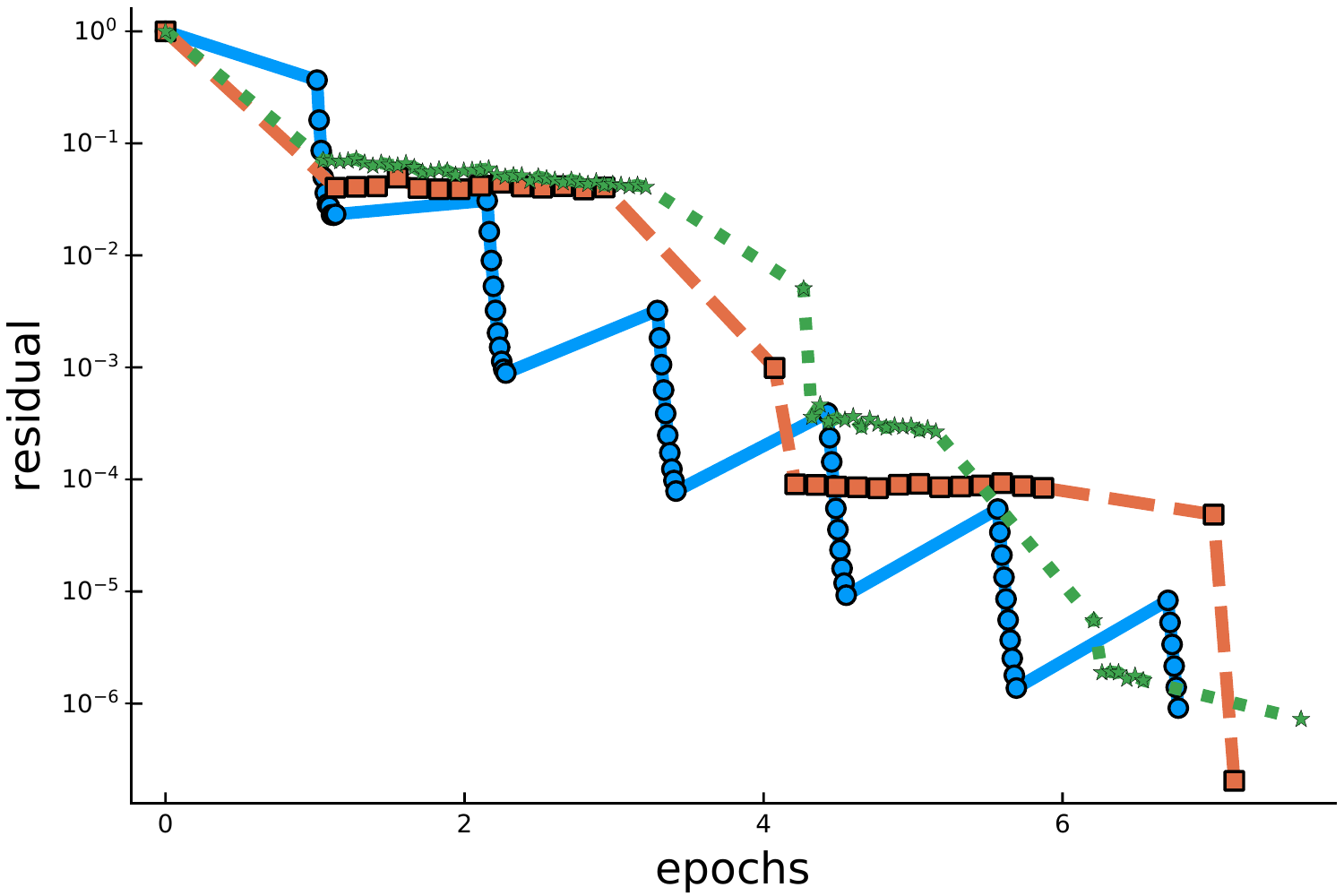}
        \includegraphics[width=0.45\textwidth]{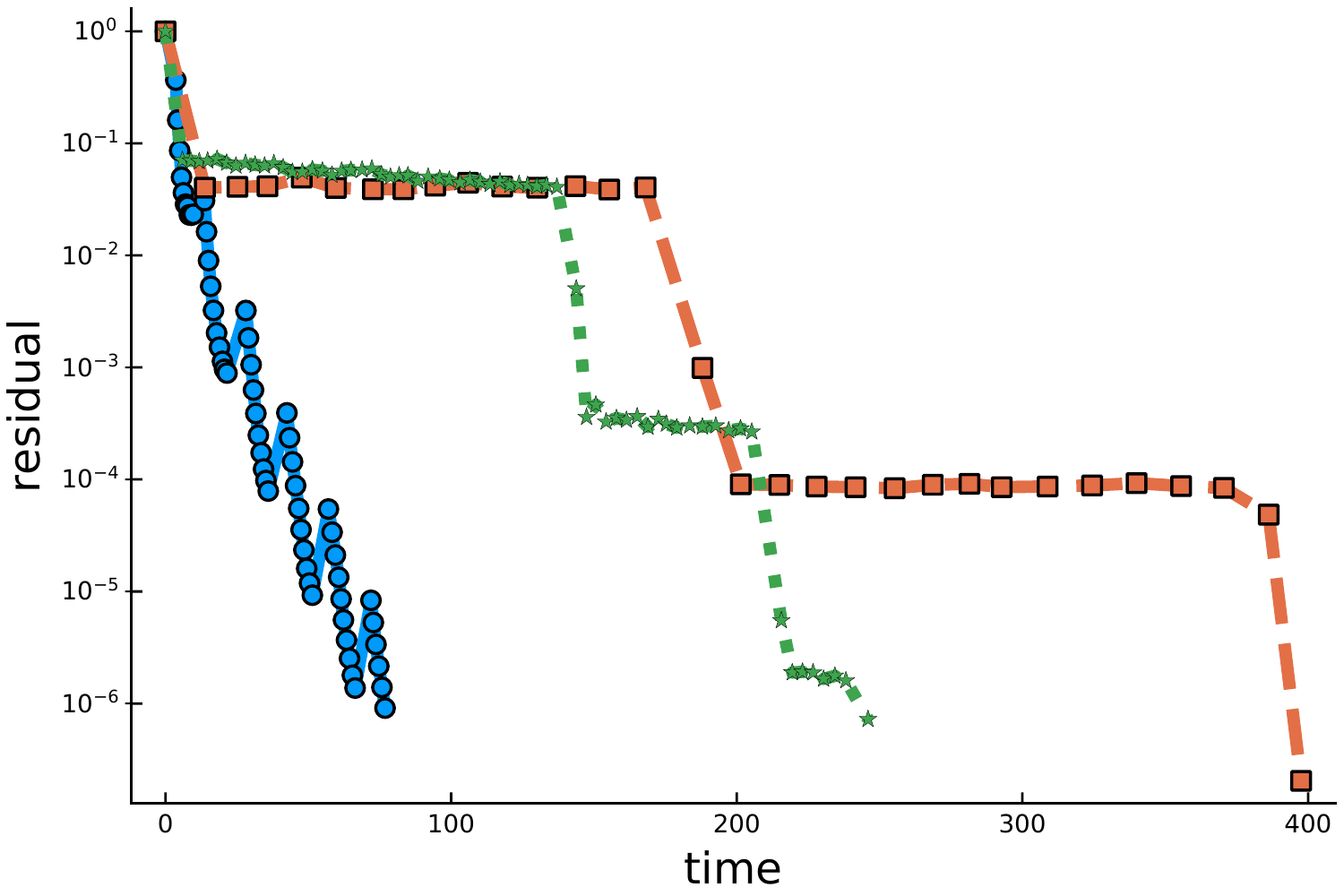}
        \caption{$\lambda = 10^{-3}$}
      \end{subfigure}
  \caption{Comparison of theoretical variants of SVRG without mini-batching ($b=1$) on the \textit{real-sim} data set.}
  \label{fig:exp1A_real-sim}
  \end{center}
  \vskip -0.2in
\end{figure}

\subsubsection{Experiment 1.b: optimal mini-batching}
Here we use the optimal mini-batch sizes we derived for \textit{Free-SVRG} in Table \ref{tab:optimal_mini-batch} and \textit{L-SVRG-D} in~\eqref{eq:optim_minibatch_dsvrg}. Since the original SVRG theory has no analysis for mini-batching, and the current existing theory shows that its total complexity increases with $b$, we use $b=1$ for SVRG. Like in Section~\ref{sec:exp_1a_no_minibatching}, the inner loop length is set to $m=n$. We confirm in these experiments that setting the mini-batch size to our predicted optimal value $b^*$ doesn't hurt our algorithms' performance. See Figures~\ref{fig:exp1B_YearPredictionMSD},~\ref{fig:exp1B_slice},~\ref{fig:exp1B_ijcnn1} and~\ref{fig:exp1B_real-sim}.  Note that in Section~\ref{sec:app_exp_inner_loop}, we further confirm that $b^*$ outperforms multiple other choices of the mini-batch size.
In most cases, \textit{Free-SVRG} and \textit{L-SVRG-D} outperform the vanilla SVRG algorithm both on the epoch and time plots, except for the regularized logistic regression on the \textit{real-sim} data set (see Figure~\ref{fig:exp1B_real-sim}), which is a very easy problem since it is well conditioned. Comparing Figures~\ref{fig:exp1A_slice} and~\ref{fig:exp1B_slice} clearly underlines the speed improvement due to optimal mini-batching, both in epoch and time plots.

\begin{figure}[!htb]
  \vskip 0.2in
  \begin{center}
      \begin{subfigure}[b]{\textwidth}
        \centering
        \includegraphics[width=0.45\textwidth]{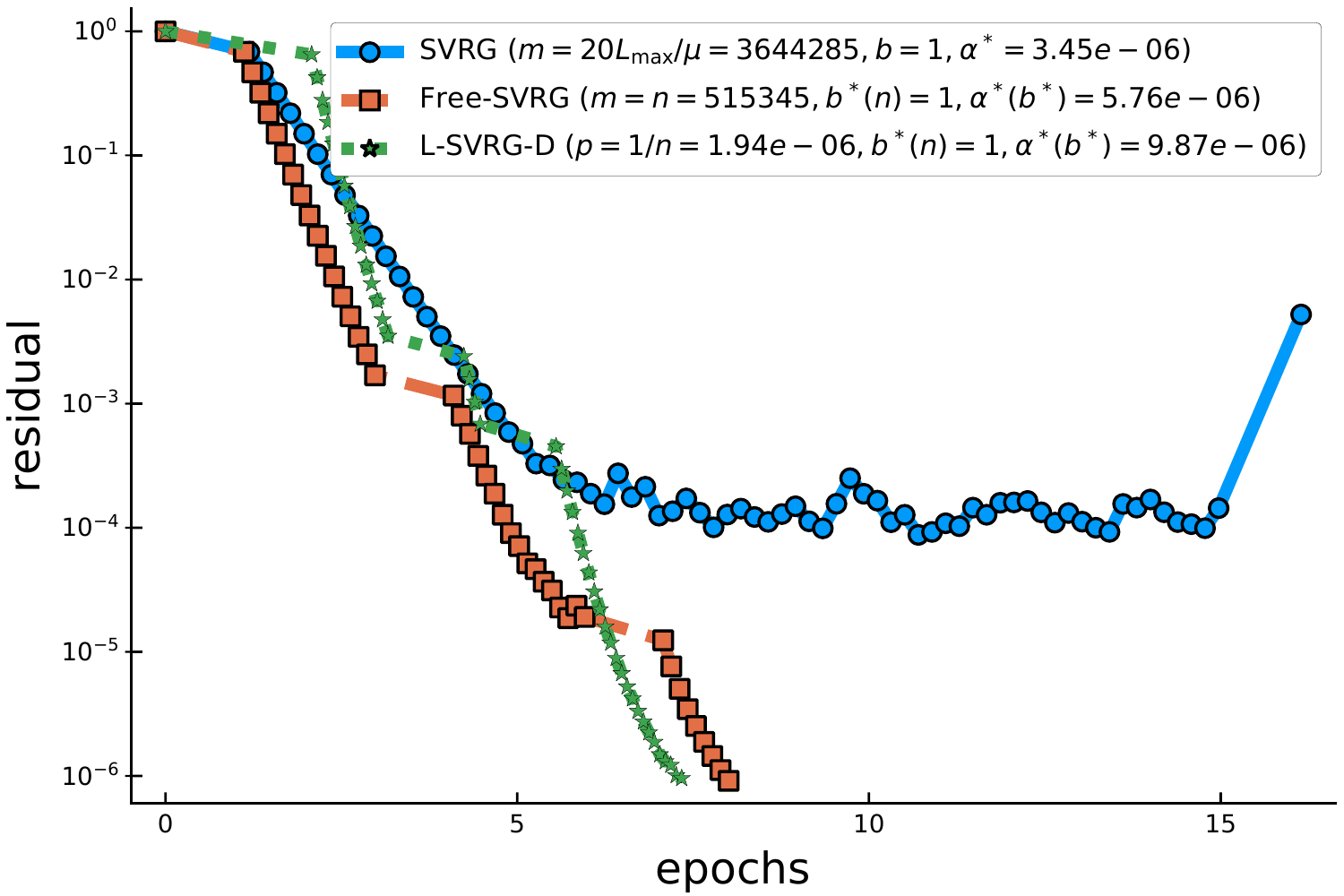}
        \includegraphics[width=0.45\textwidth]{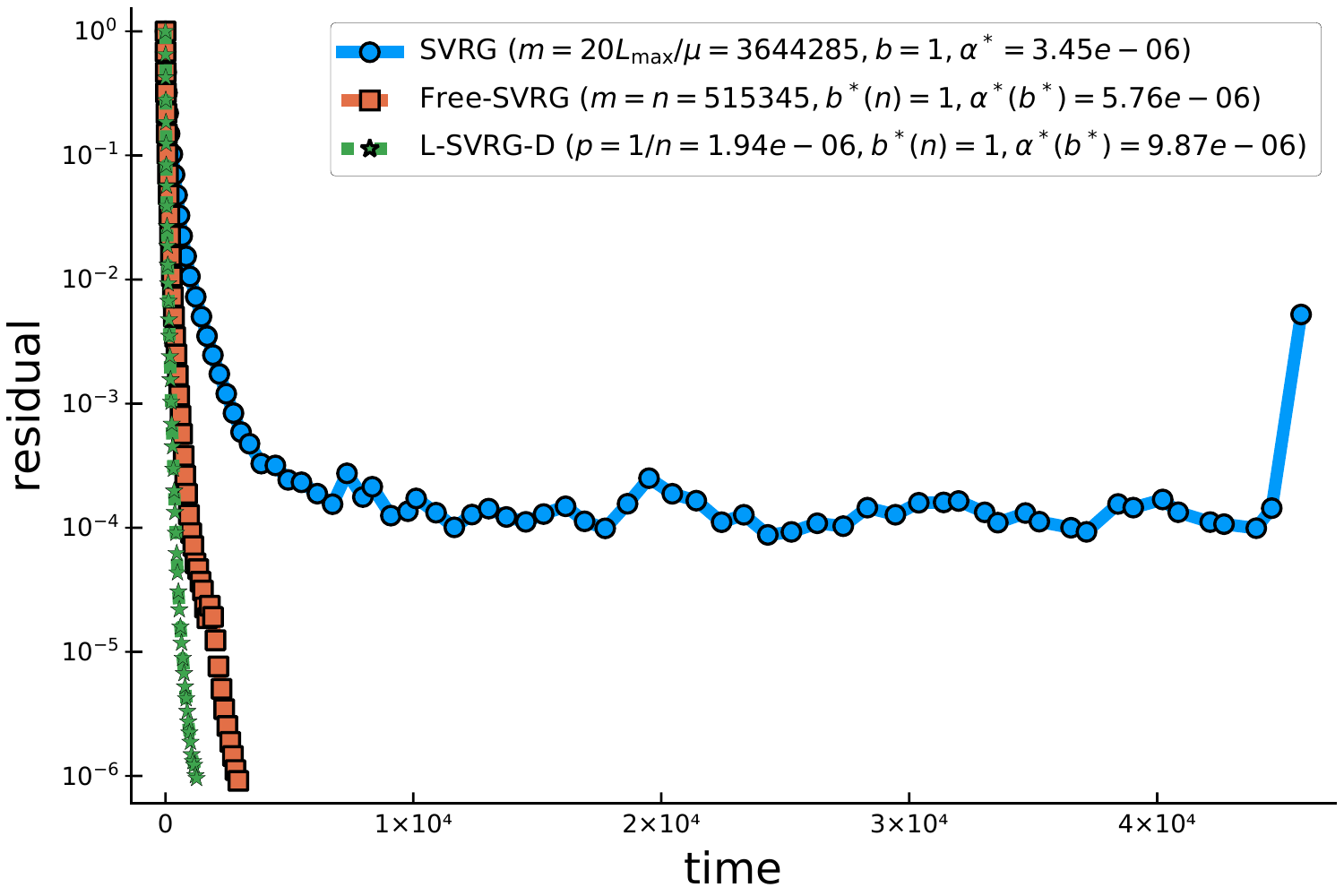}
        \caption{$\lambda = 10^{-1}$}
      \end{subfigure}\\
      \begin{subfigure}[b]{\textwidth}
        \centering
        \includegraphics[width=0.45\textwidth]{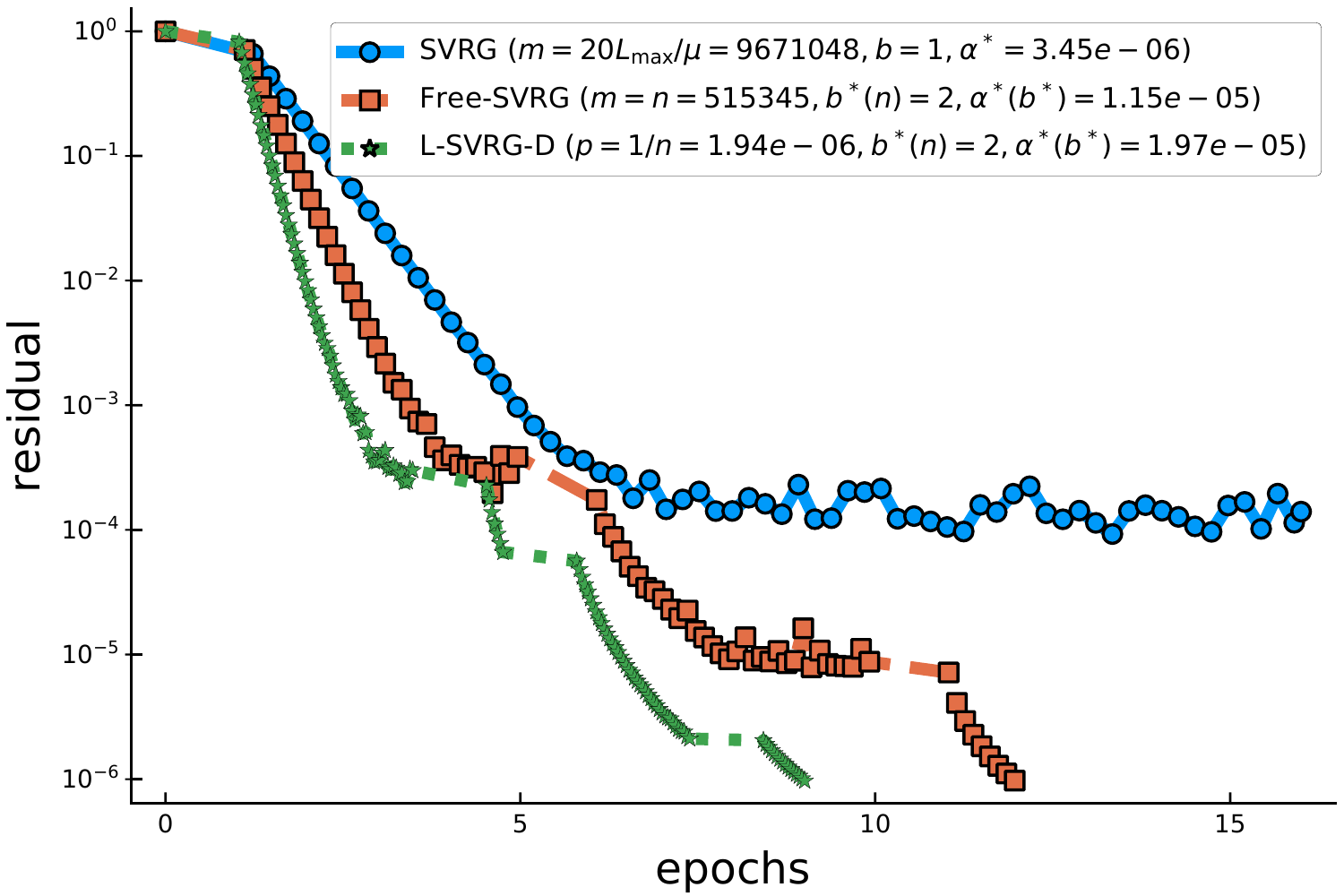}
        \includegraphics[width=0.45\textwidth]{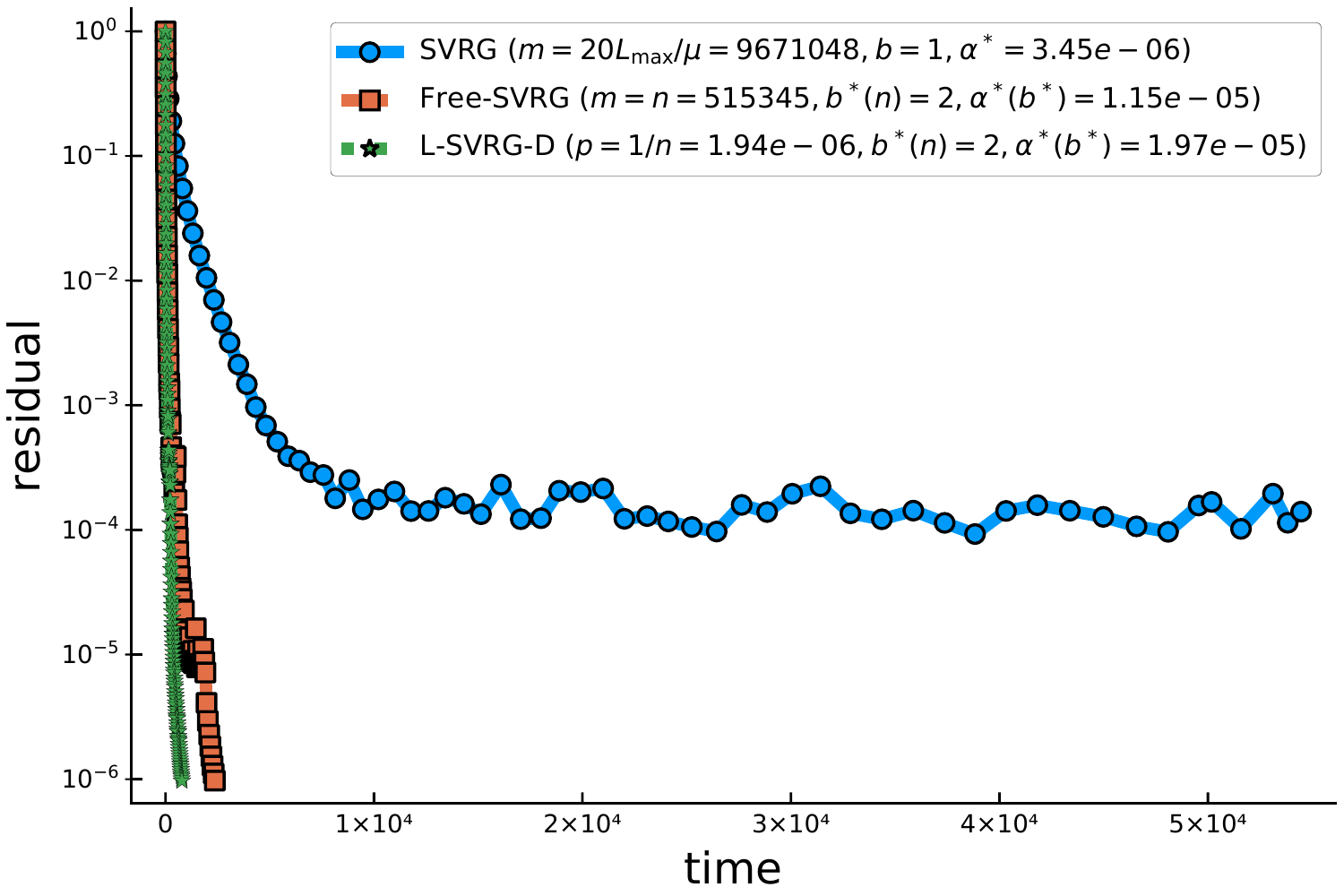}
        \caption{$\lambda = 10^{-3}$}
      \end{subfigure}
    \caption{Comparison of theoretical variants of SVRG with optimal mini-batch size $b^*$ when theoretically available on the \textit{YearPredictionMSD} data set.}
    \label{fig:exp1B_YearPredictionMSD}
  \end{center}
  \vskip -0.2in
\end{figure}


\begin{figure}[!htb]
 \vskip 0.2in
 \begin{center}
    \begin{subfigure}[b]{\textwidth}
      \centering
      \includegraphics[width=0.45\textwidth]{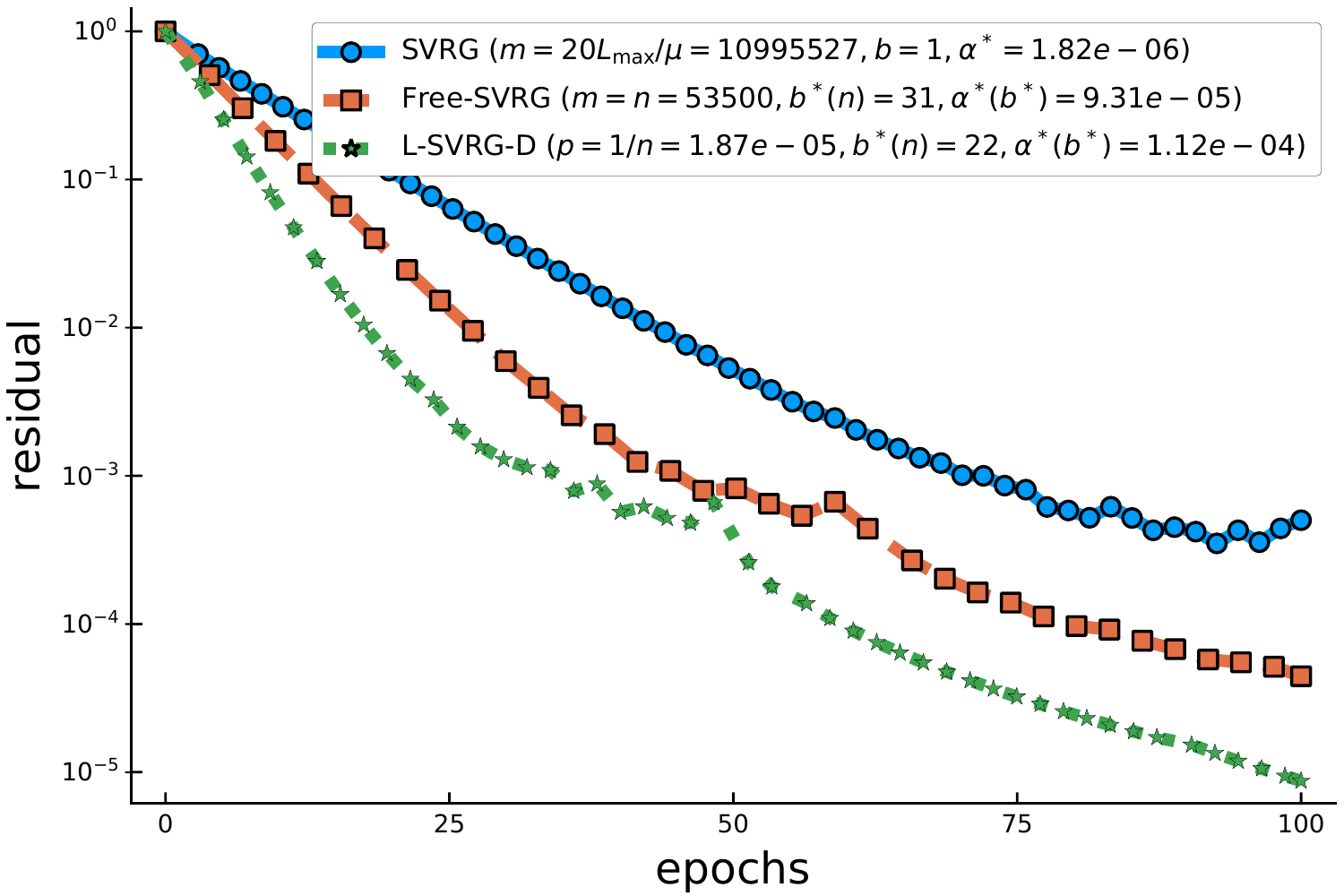}
      \includegraphics[width=0.45\textwidth]{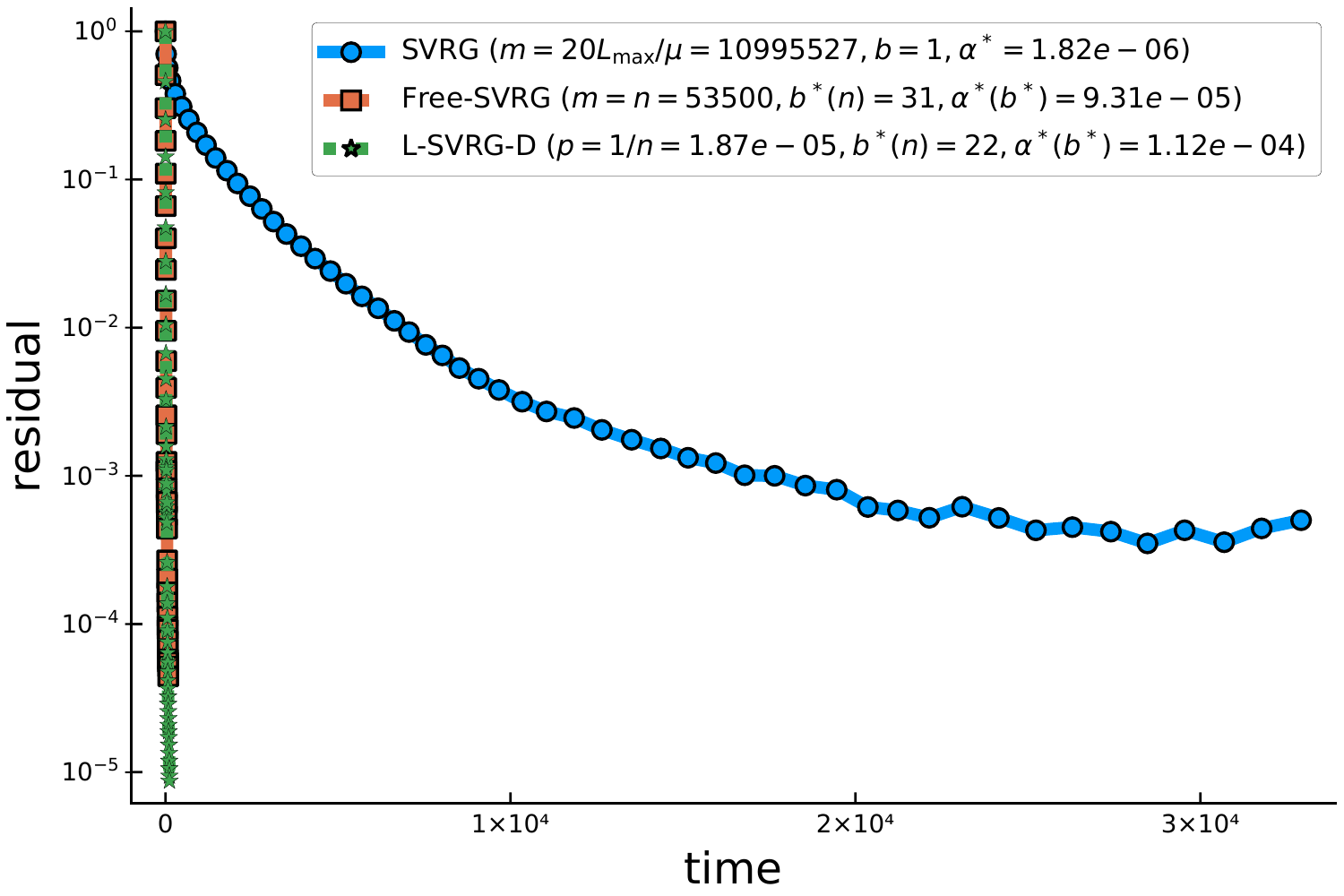}
      \caption{$\lambda = 10^{-1}$}
    \end{subfigure}\\
    \begin{subfigure}[b]{\textwidth}
      \centering
      \includegraphics[width=0.45\textwidth]{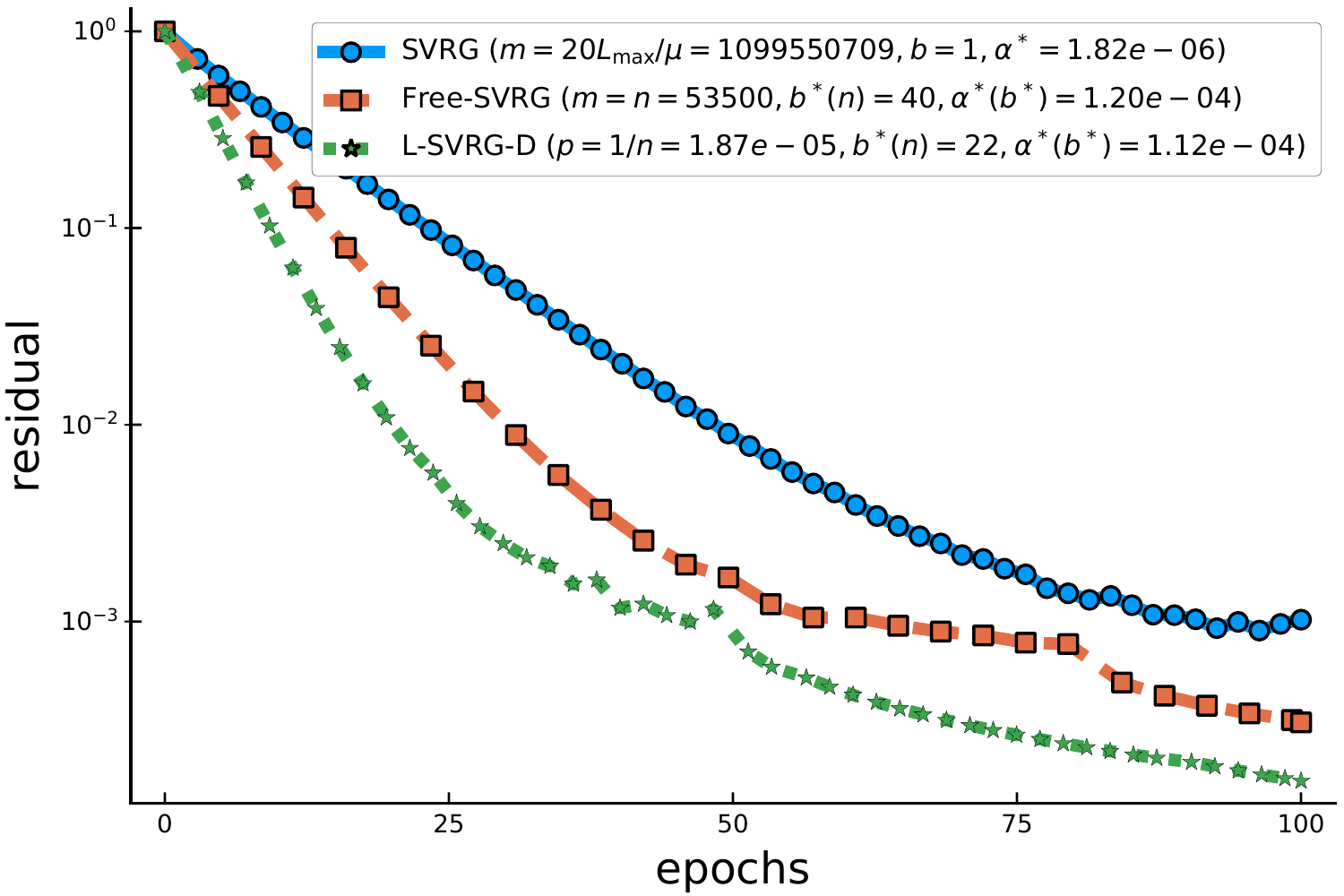}
      \includegraphics[width=0.45\textwidth]{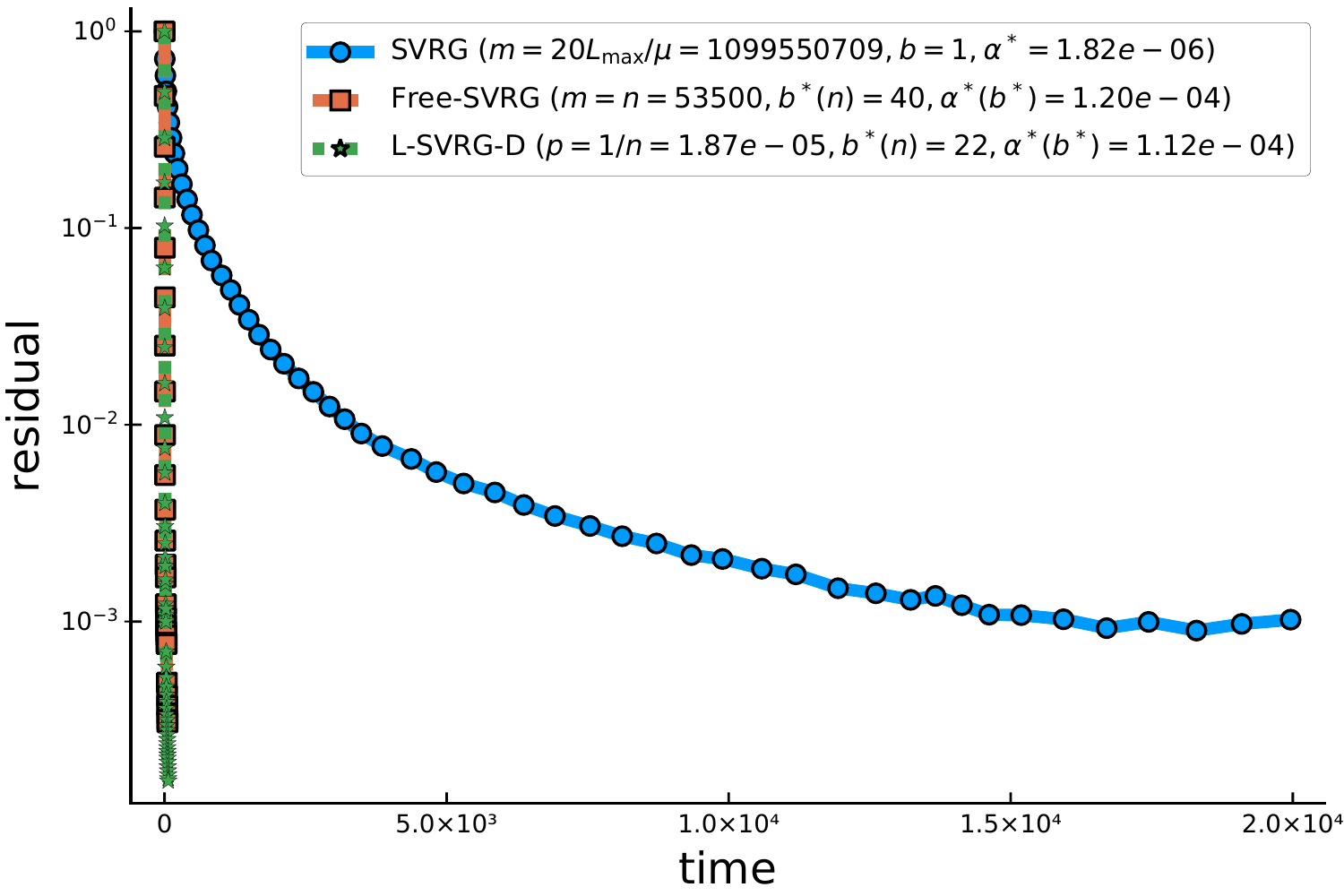}
      \caption{$\lambda = 10^{-3}$}
      \label{fig:exp1B_slice_1e-3}
    \end{subfigure}
  \caption{Comparison of theoretical variants of SVRG with optimal mini-batch size $b^*$ when theoretically available on the \textit{slice} data set.}
  \label{fig:exp1B_slice}
 \end{center}
 \vskip -0.2in
\end{figure}

\begin{figure}[!htb]
  \vskip 0.2in
  \begin{center}
    \begin{subfigure}[b]{\textwidth}
      \centering
      \includegraphics[width=0.45\textwidth]{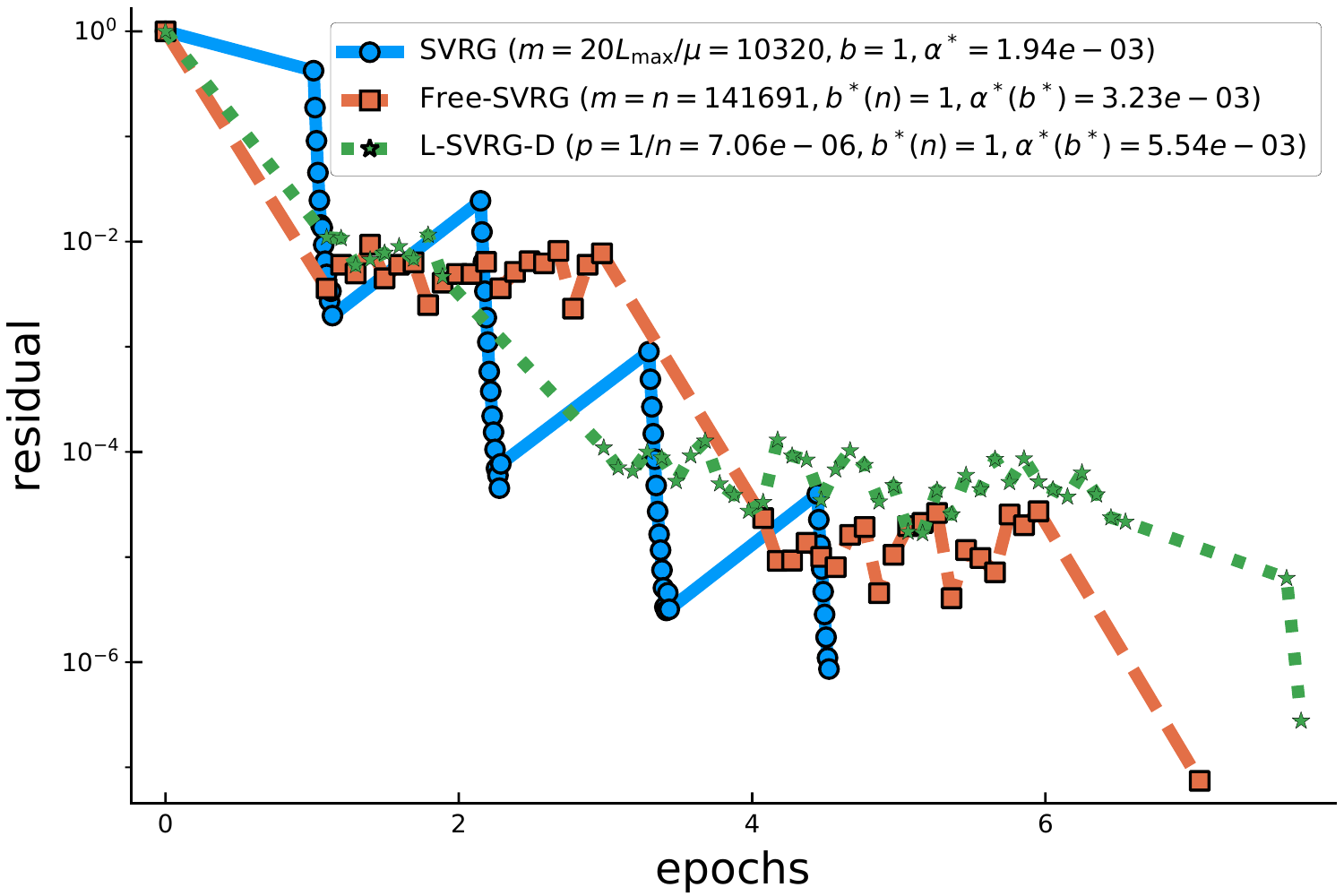}
      \includegraphics[width=0.45\textwidth]{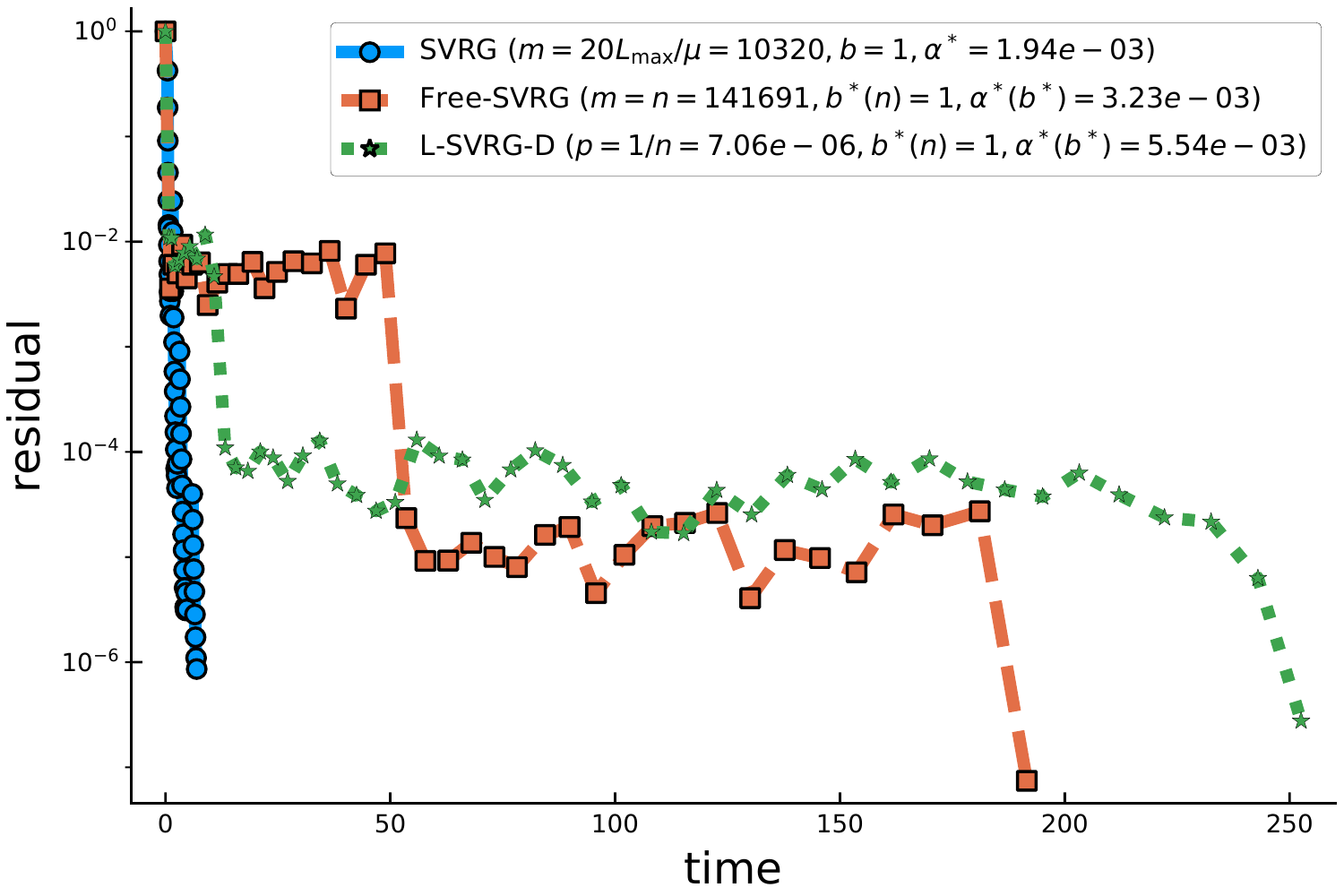}
      \caption{$\lambda = 10^{-1}$}
    \end{subfigure}\\
    \begin{subfigure}[b]{\textwidth}
      \centering
      \includegraphics[width=0.45\textwidth]{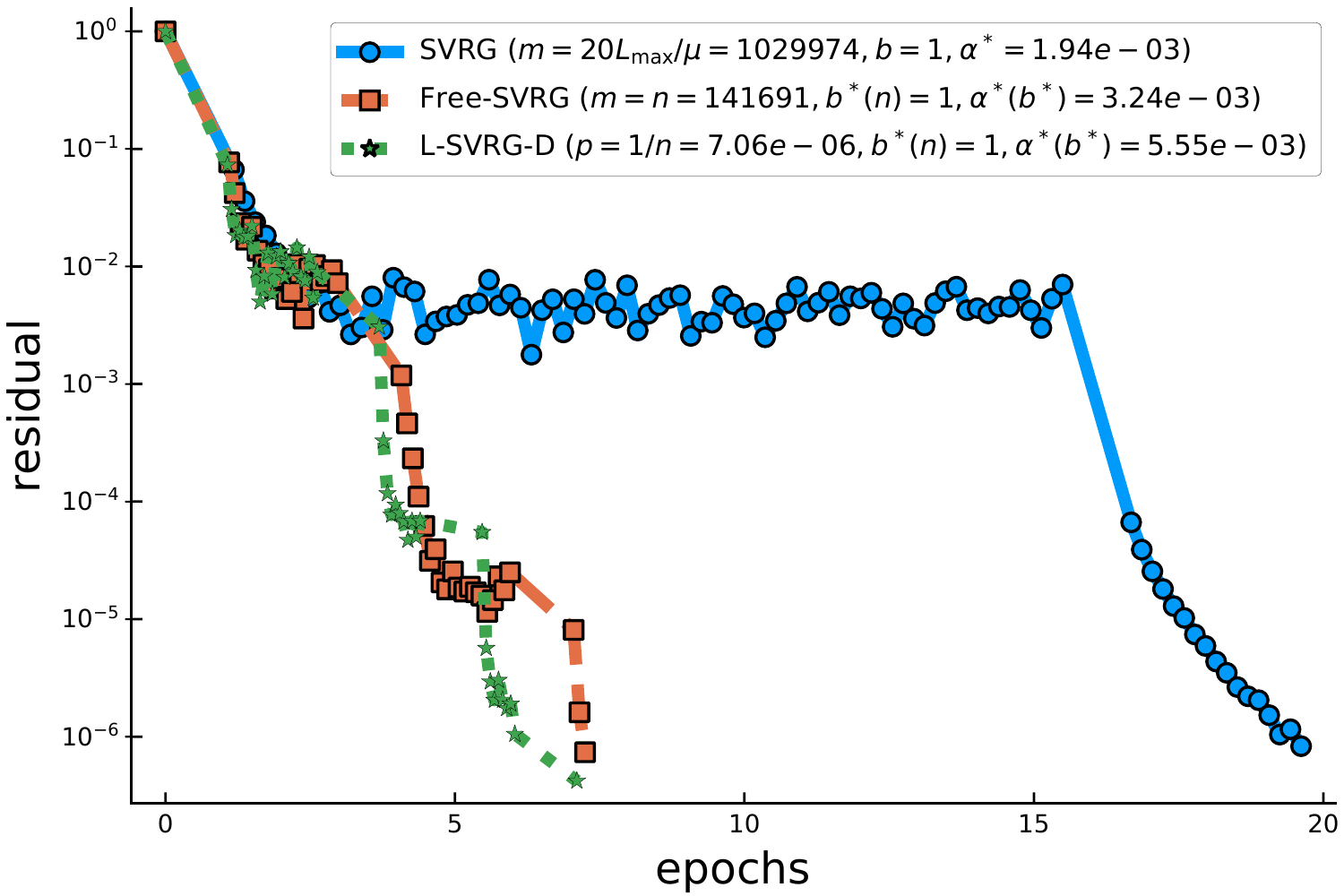}
      \includegraphics[width=0.45\textwidth]{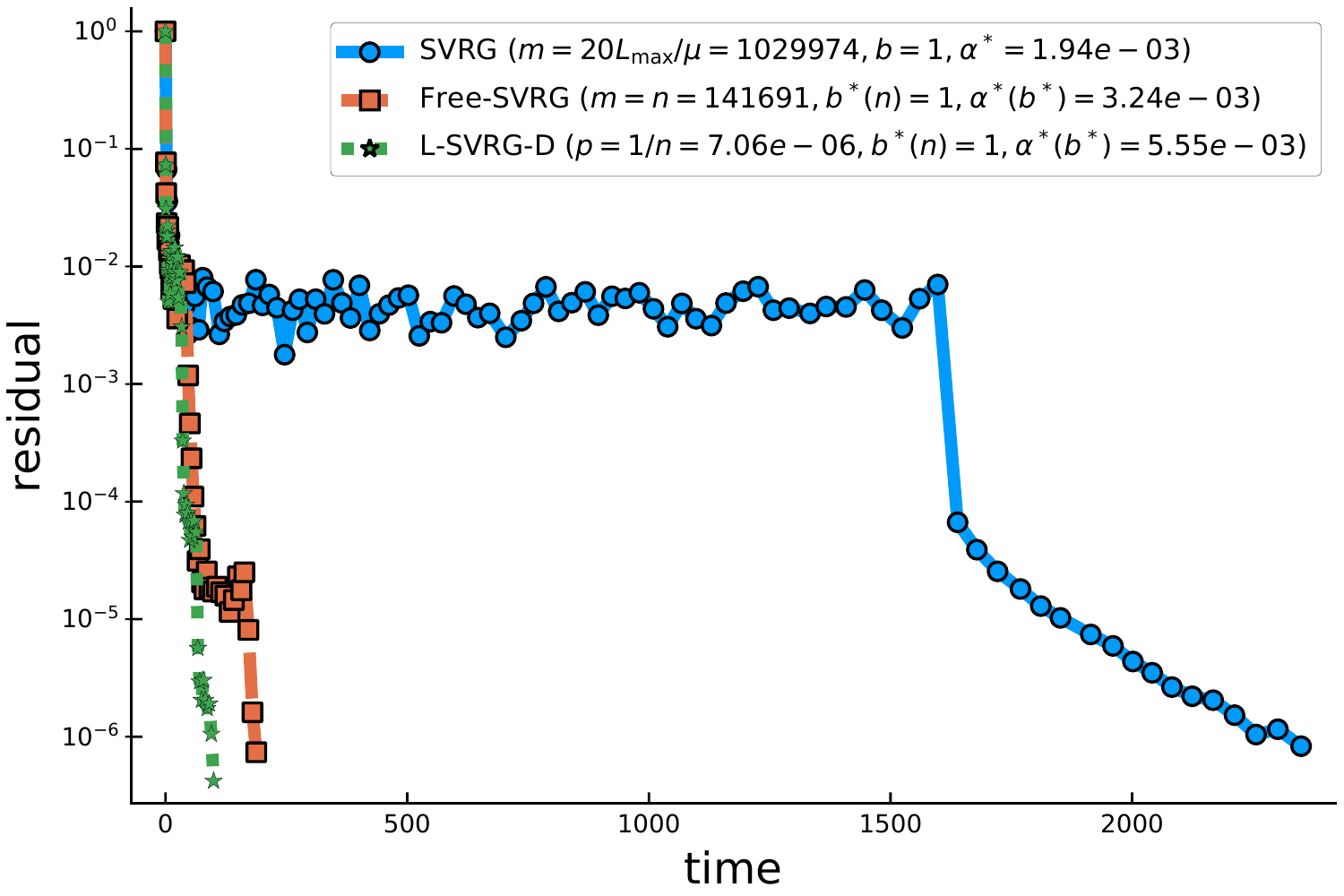}
      \caption{$\lambda = 10^{-3}$}
    \end{subfigure}
  \caption{Comparison of theoretical variants of SVRG with optimal mini-batch size $b^*$ when theoretically available on the \textit{ijcnn1} data set.}
  \label{fig:exp1B_ijcnn1}
  \end{center}
  \vskip -0.2in
\end{figure}

\begin{figure}[!htb]
  \vskip 0.2in
  \begin{center}
    \begin{subfigure}[b]{\textwidth}
      \centering
      \includegraphics[width=0.45\textwidth]{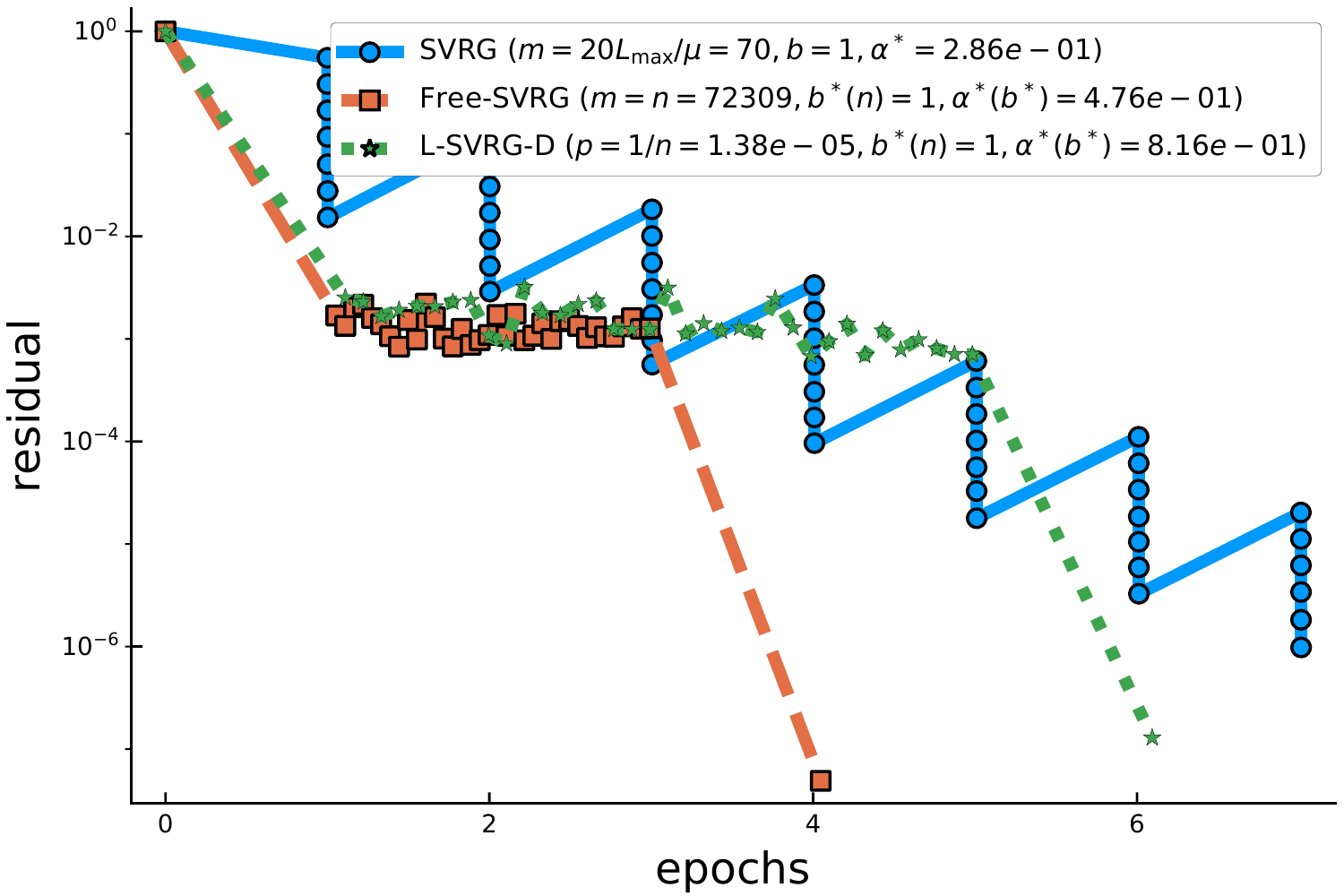}
      \includegraphics[width=0.45\textwidth]{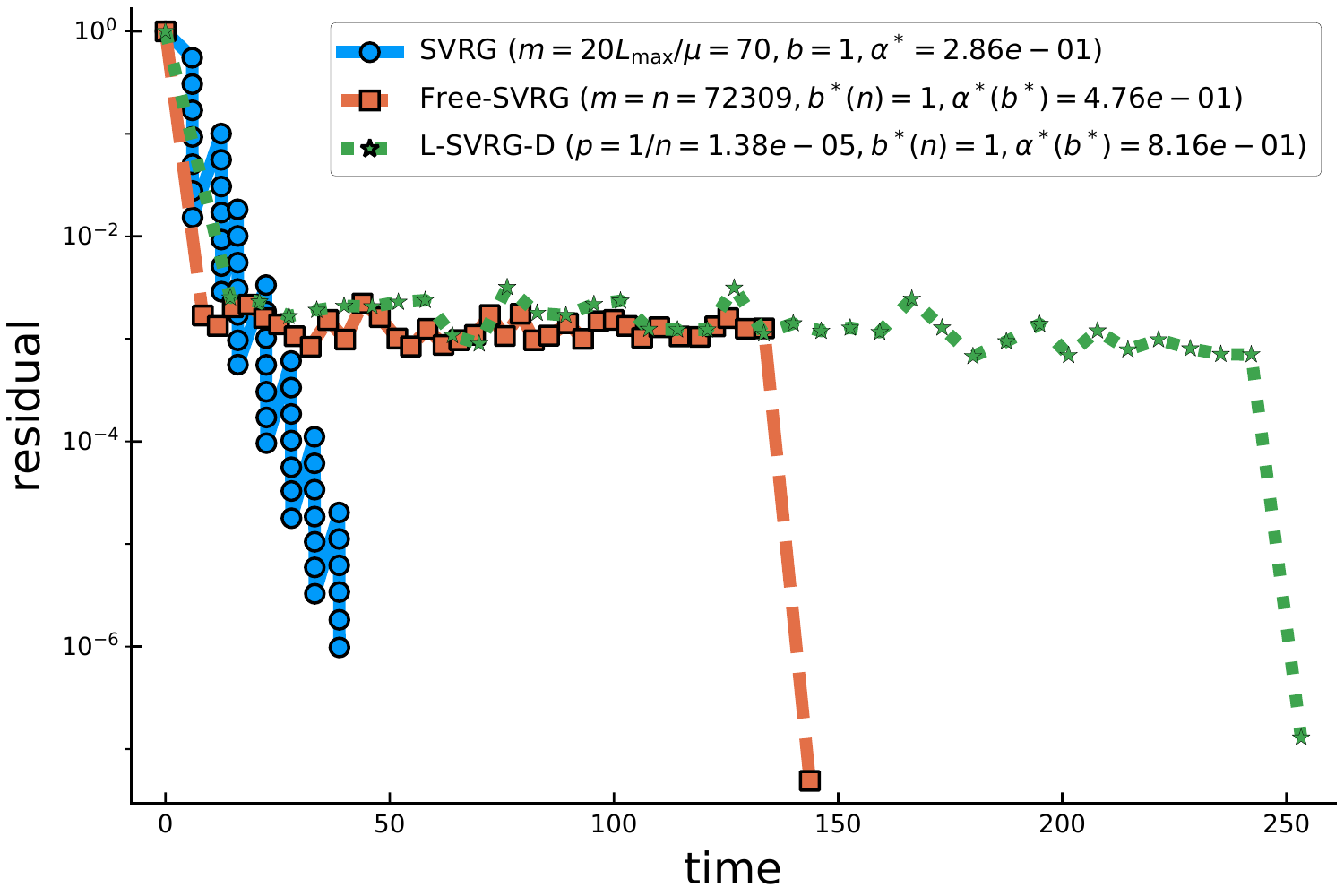}
      \caption{$\lambda = 10^{-1}$}
    \end{subfigure}\\
    \begin{subfigure}[b]{\textwidth}
      \centering
      \includegraphics[width=0.45\textwidth]{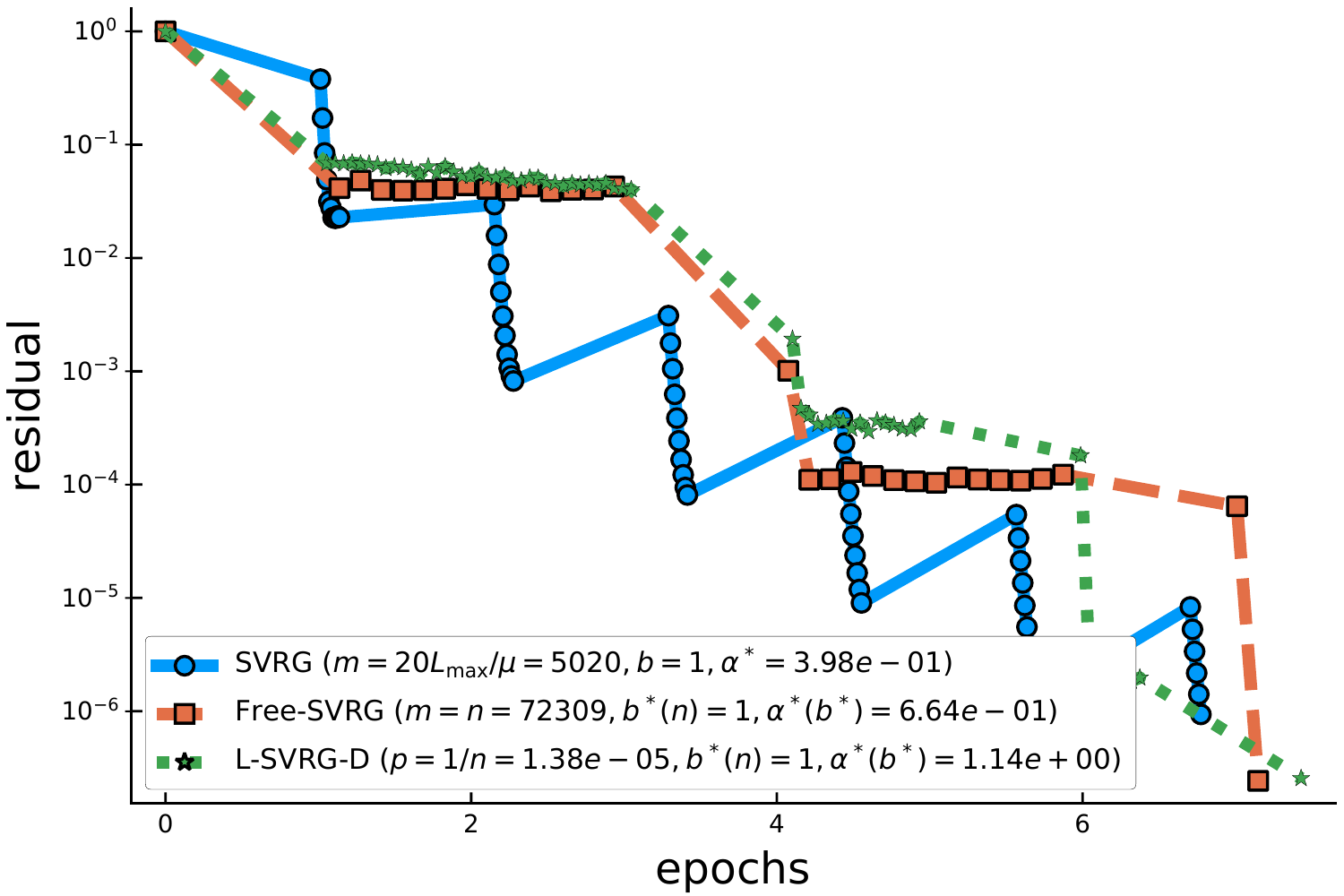}
      \includegraphics[width=0.45\textwidth]{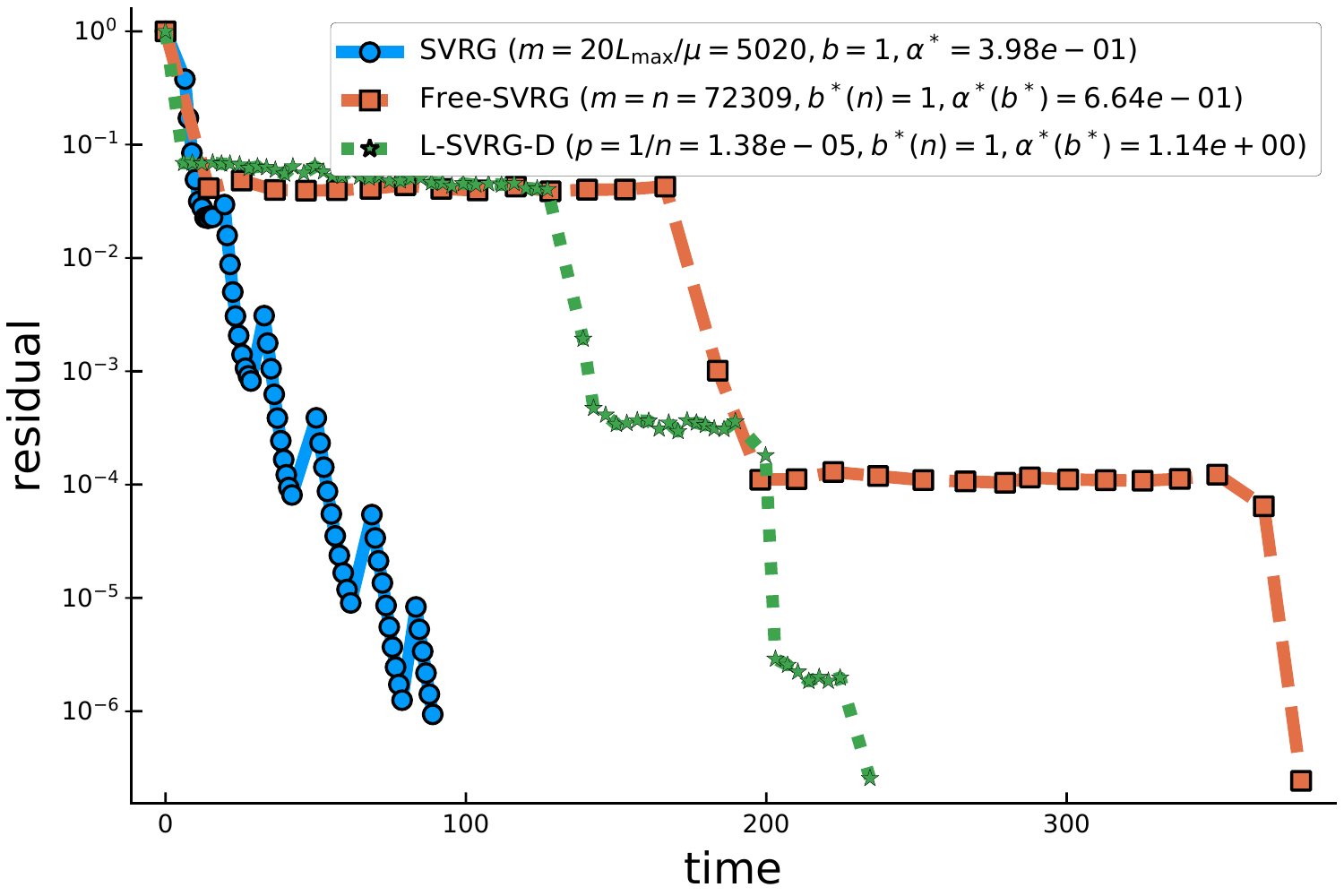}
      \caption{$\lambda = 10^{-3}$}
    \end{subfigure}
  \caption{Comparison of theoretical variants of SVRG with optimal mini-batch size $b^*$ when theoretically available on the \textit{real-sim} data set.}
  \label{fig:exp1B_real-sim}
  \end{center}
  \vskip -0.2in
\end{figure}

\subsubsection{Experiment 1.c: theoretical inner loop size or update probability without mini-batching}
Here, using $b=1$, we set the inner loop size for \textit{Free-SVRG} to its optimal value $m^* = 3L_{\max}/\mu$ that we derived in Proposition~\ref{prop:optimalloopsize}. We set $p = 1/m^*$ for \textit{L-SVRG-D}. The inner loop length is set like in Section~\ref{sec:exp_1a_no_minibatching}.  See Figures~\ref{fig:exp1C_YearPredictionMSD},~\ref{fig:exp1C_slice},~\ref{fig:exp1C_ijcnn1} and~\ref{fig:exp1C_real-sim}. By setting the size of the inner loop to its optimal value $m^*$, the results are similar to the one in experiments 1.a and 1.b. Yet, when comparing Figure~\ref{fig:exp1A_slice} and Figure~\ref{fig:exp1C_slice}, we observe that it leads to a clear speed up of \textit{Free-SVRG} and \textit{L-SVRG-D}.

\begin{figure}[!htb]
  \vskip 0.2in
  \begin{center}
    \begin{subfigure}[b]{0.8\textwidth}
        \includegraphics[width=\textwidth]{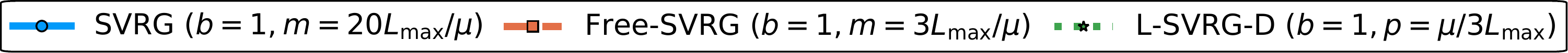}
      \end{subfigure}\\
      \begin{subfigure}[b]{\textwidth}
        \centering
        \includegraphics[width=0.45\textwidth]{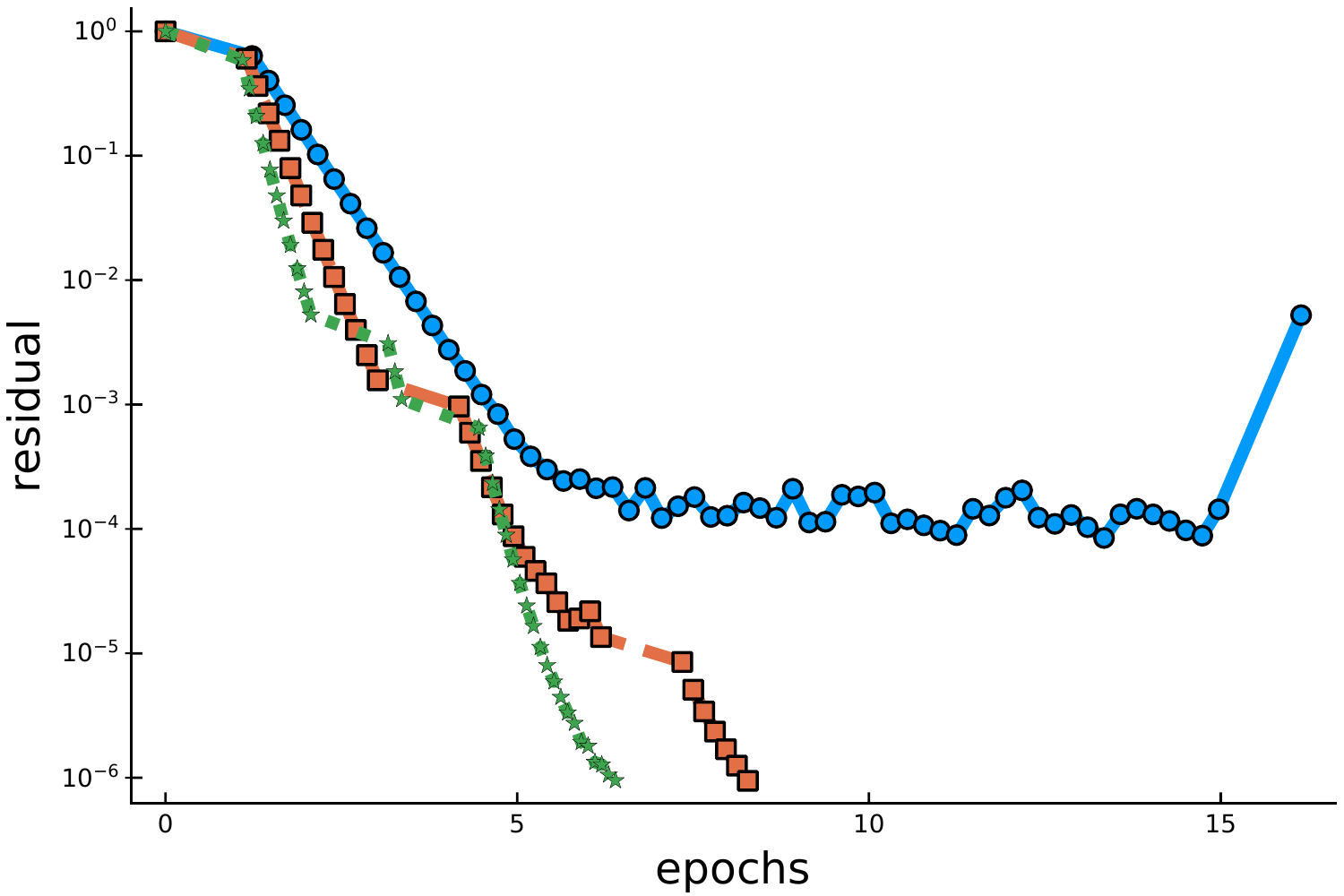}
        \includegraphics[width=0.45\textwidth]{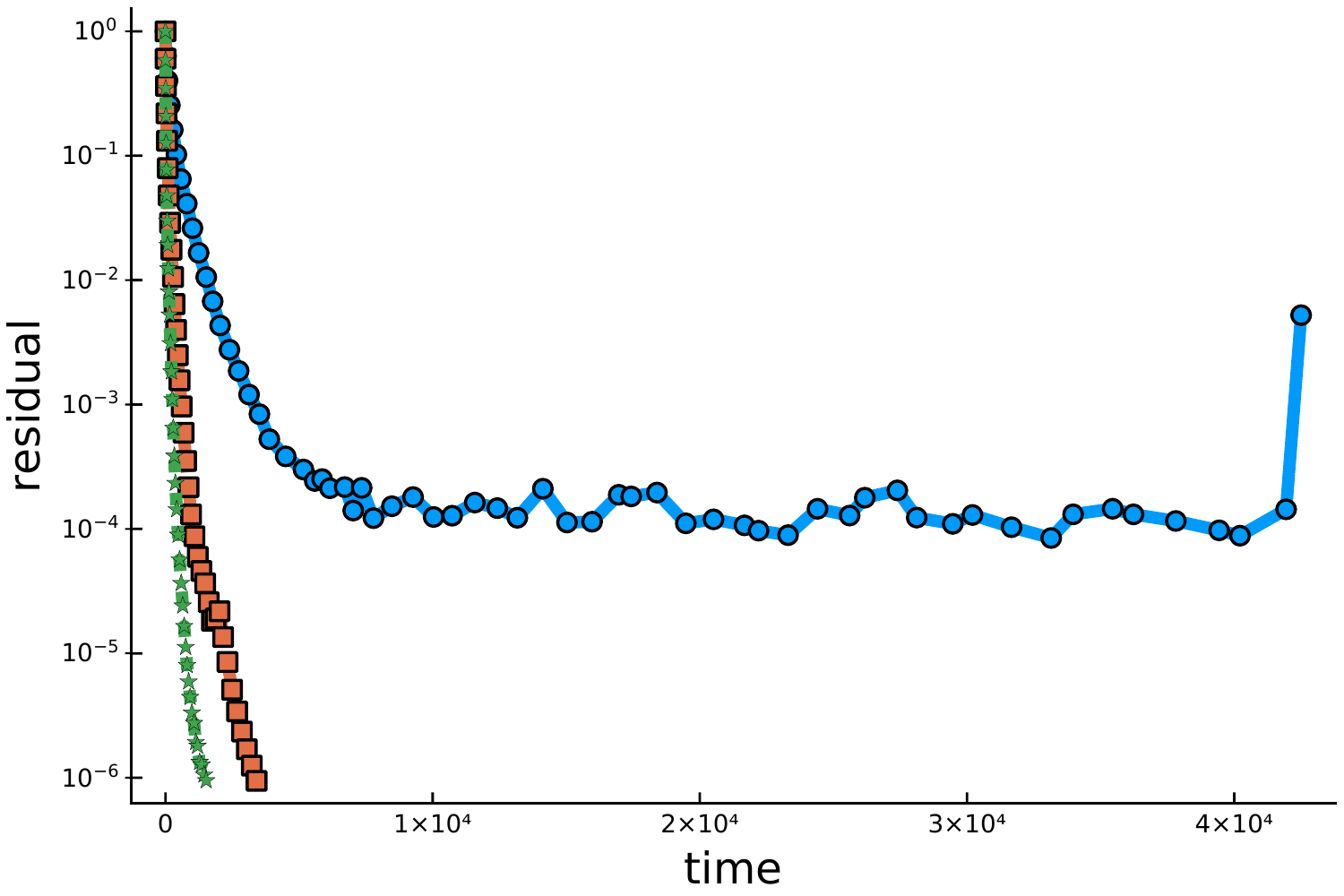}
        \caption{$\lambda = 10^{-1}$}
      \end{subfigure}\\
      \begin{subfigure}[b]{\textwidth}
        \centering
        \includegraphics[width=0.45\textwidth]{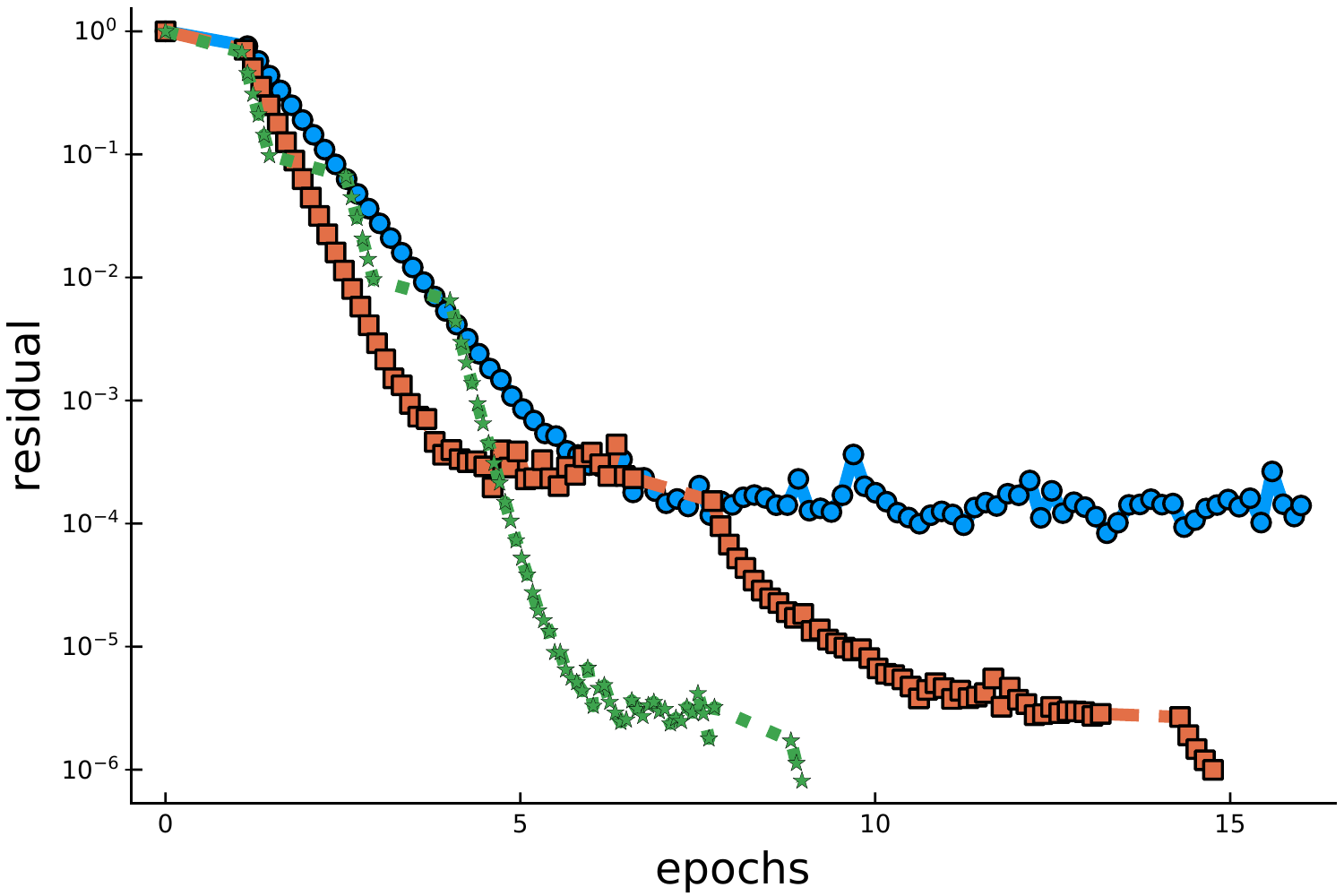}
        \includegraphics[width=0.45\textwidth]{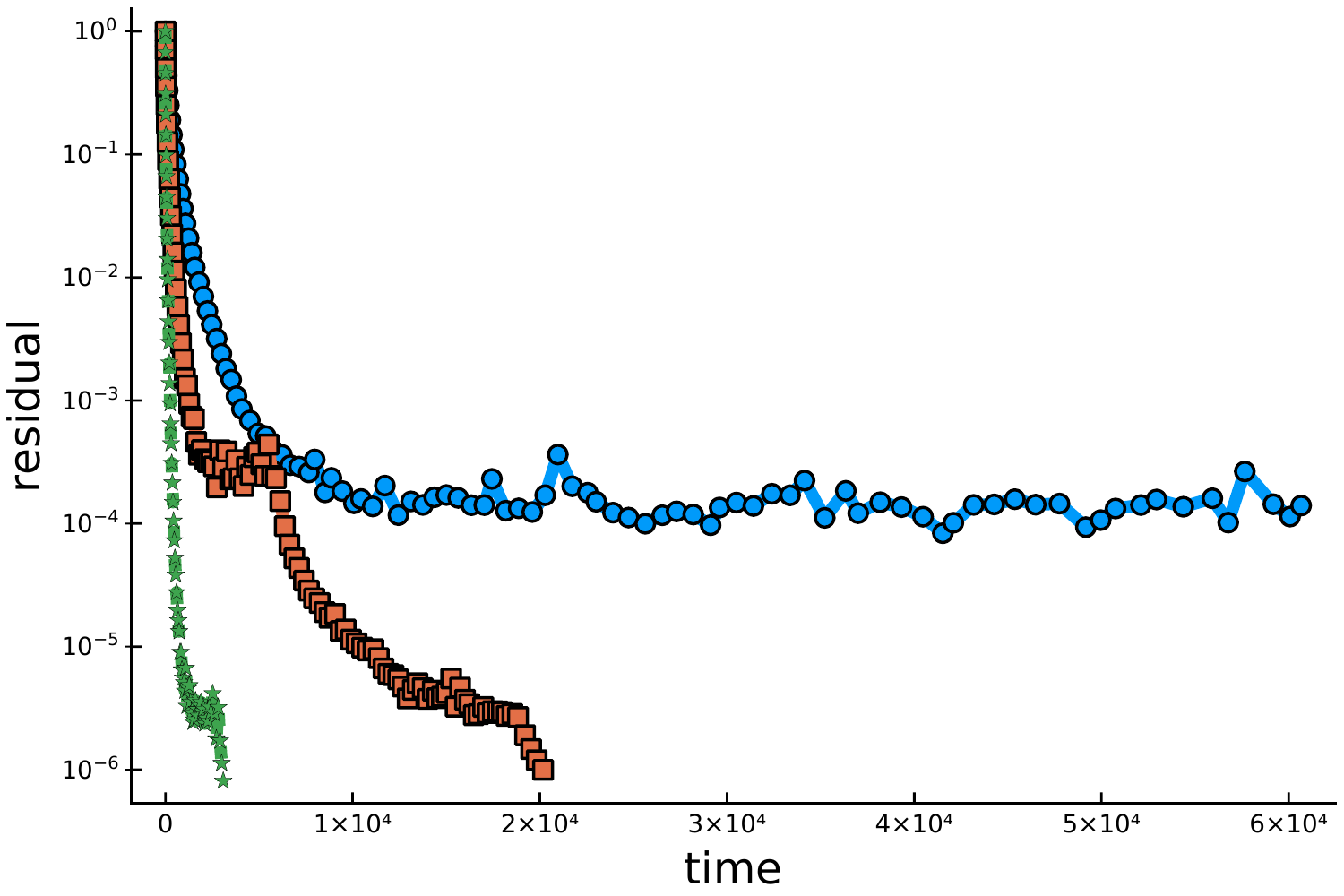}
        \caption{$\lambda = 10^{-3}$}
      \end{subfigure}
  \caption{Comparison of theoretical variants of SVRG with optimal inner loop size $m^*$ when theoretically available ($b=1$) on the \textit{YearPredictionMSD} data set.}
  \label{fig:exp1C_YearPredictionMSD}
  \end{center}
  \vskip -0.2in
\end{figure}

\begin{figure}[!htb]
 \vskip 0.2in
 \begin{center}
   \begin{subfigure}[b]{0.9\textwidth}
       \includegraphics[width=\textwidth]{exp1c/legend_exp1c_horizontal}
     \end{subfigure}\\
     \begin{subfigure}[b]{\textwidth}
        \centering
        \includegraphics[width=0.45\textwidth]{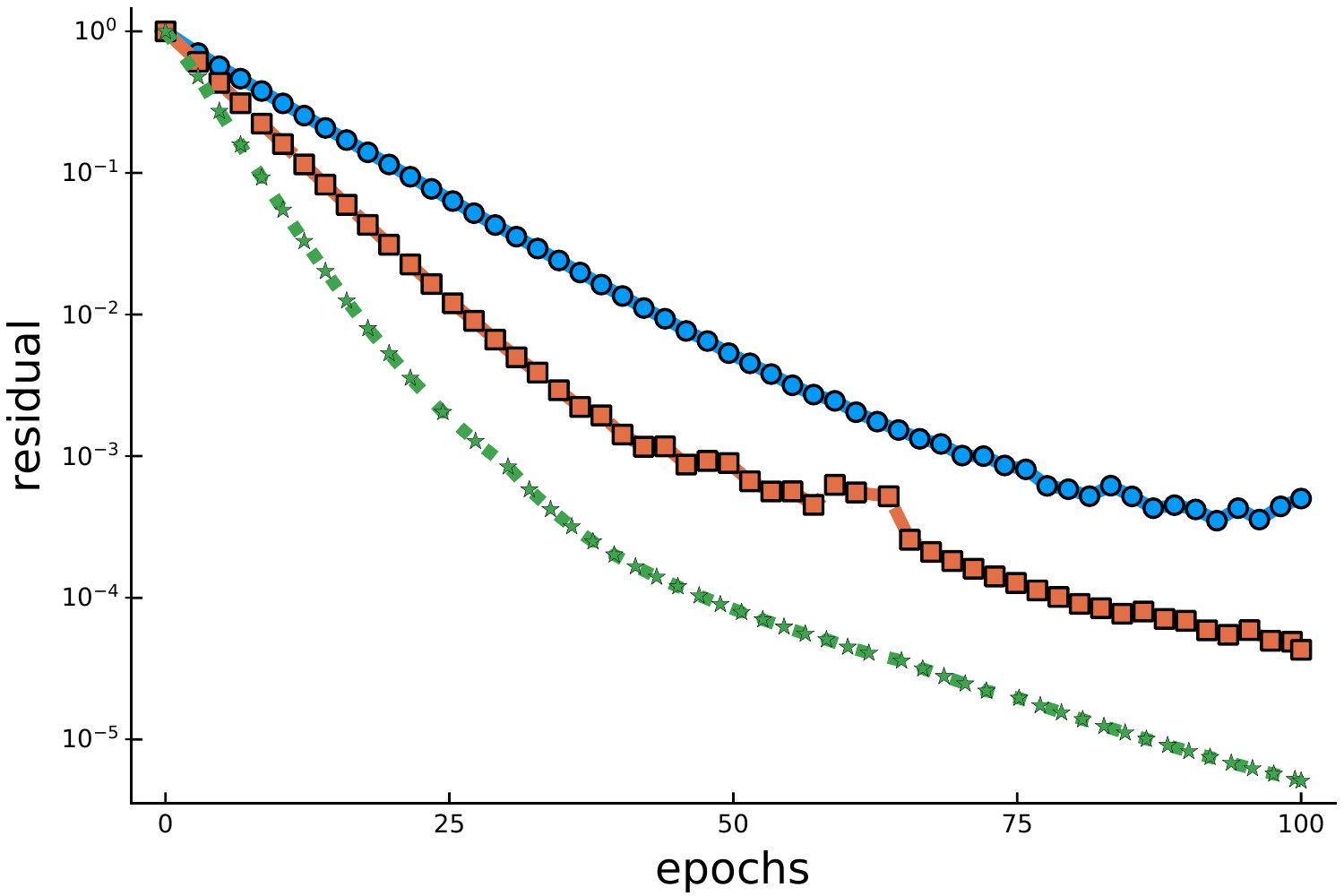}
        \includegraphics[width=0.45\textwidth]{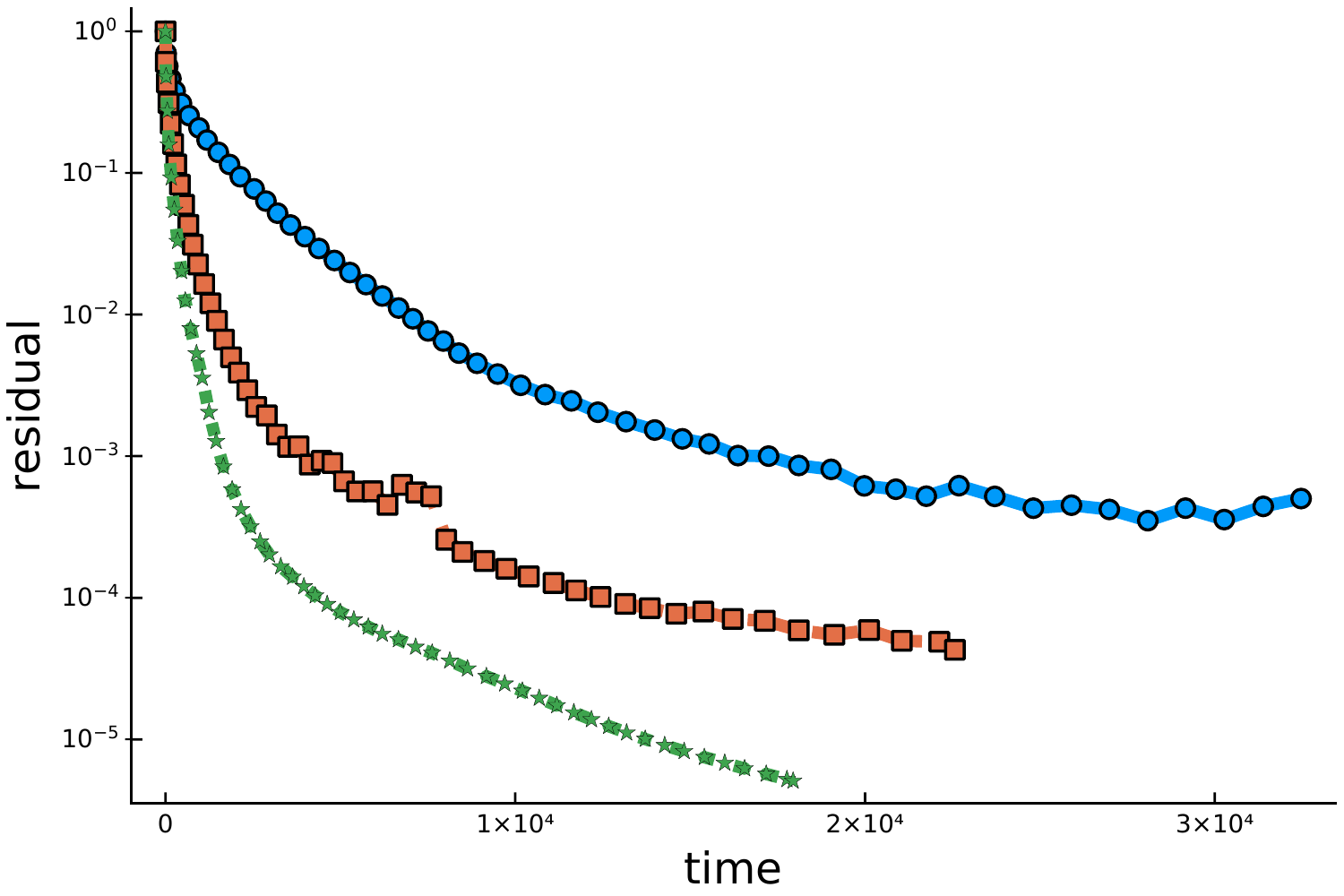}
        \caption{$\lambda = 10^{-1}$}
     \end{subfigure}\\
     \begin{subfigure}[b]{\textwidth}
        \centering
        \includegraphics[width=0.45\textwidth]{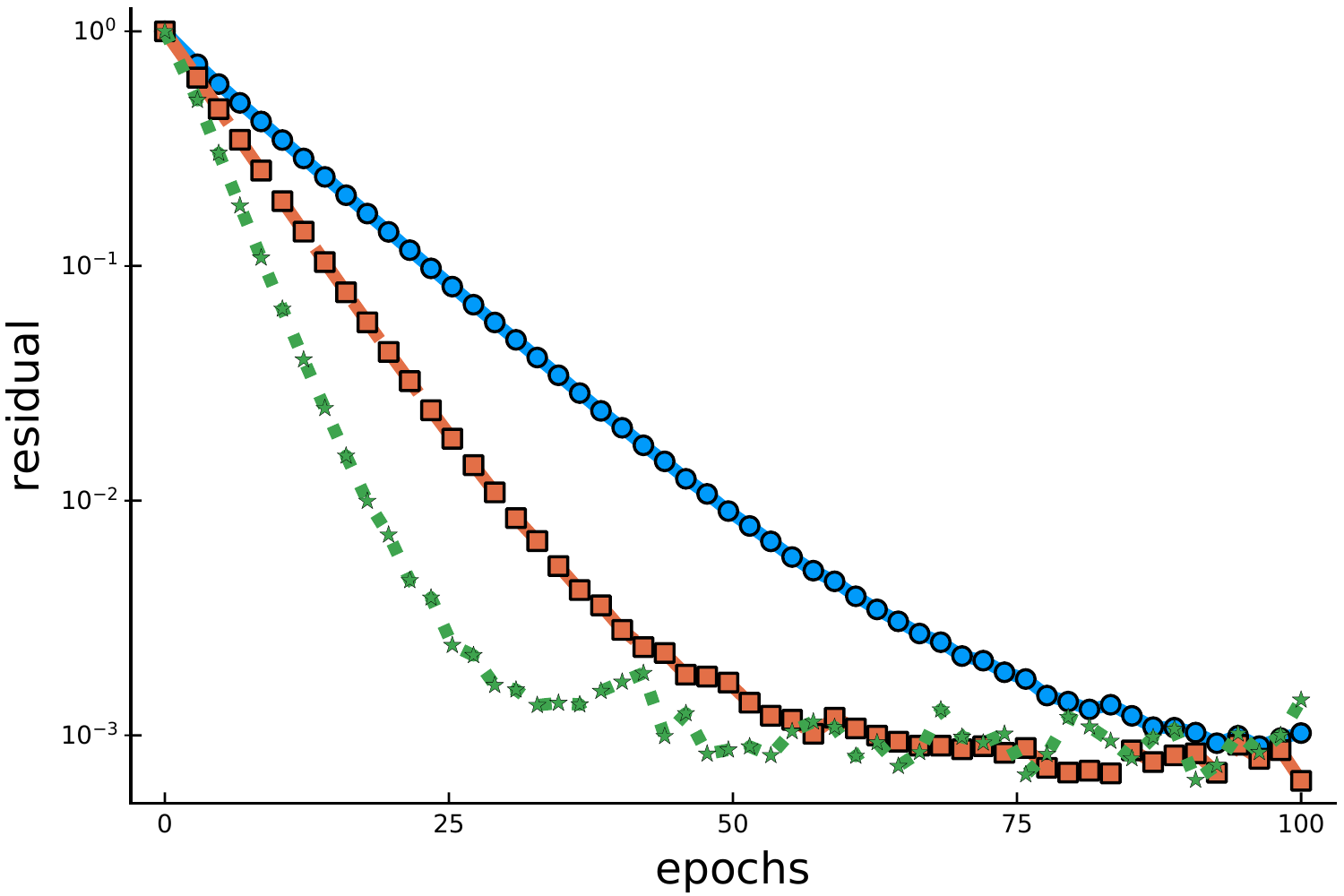}
        \includegraphics[width=0.45\textwidth]{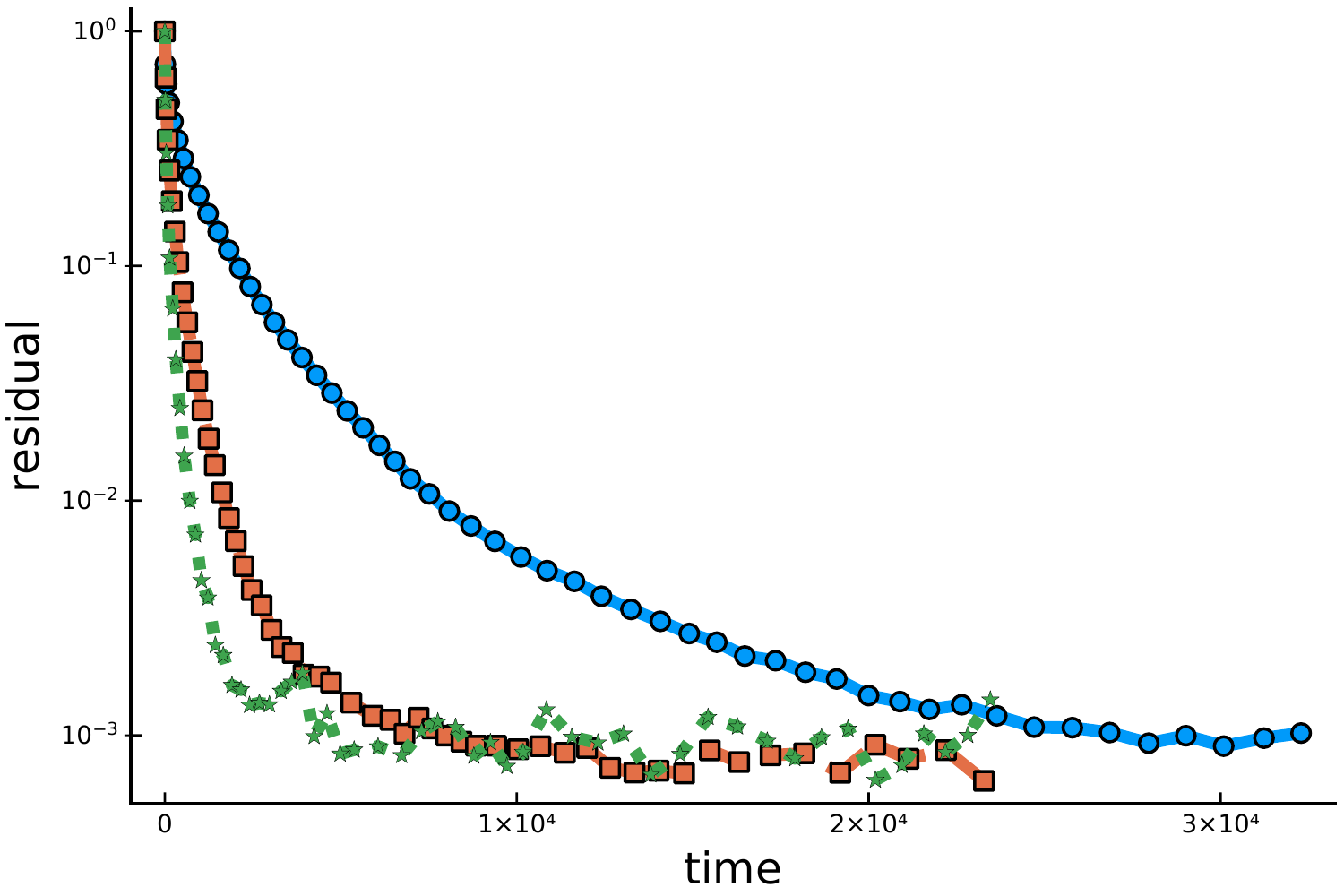}
        \caption{$\lambda = 10^{-3}$}
     \end{subfigure}
 \caption{Comparison of theoretical variants of SVRG with optimal inner loop size $m^*$ when theoretically available ($b=1$) on the \textit{slice} data set.}
 \label{fig:exp1C_slice}
 \end{center}
 \vskip -0.2in
\end{figure}

\begin{figure}[!htb]
  \vskip 0.2in
  \begin{center}
    \begin{subfigure}[b]{0.9\textwidth}
        \includegraphics[width=\textwidth]{exp1c/legend_exp1c_horizontal}
      \end{subfigure}\\
      \begin{subfigure}[b]{\textwidth}
        \centering
        \includegraphics[width=0.45\textwidth]{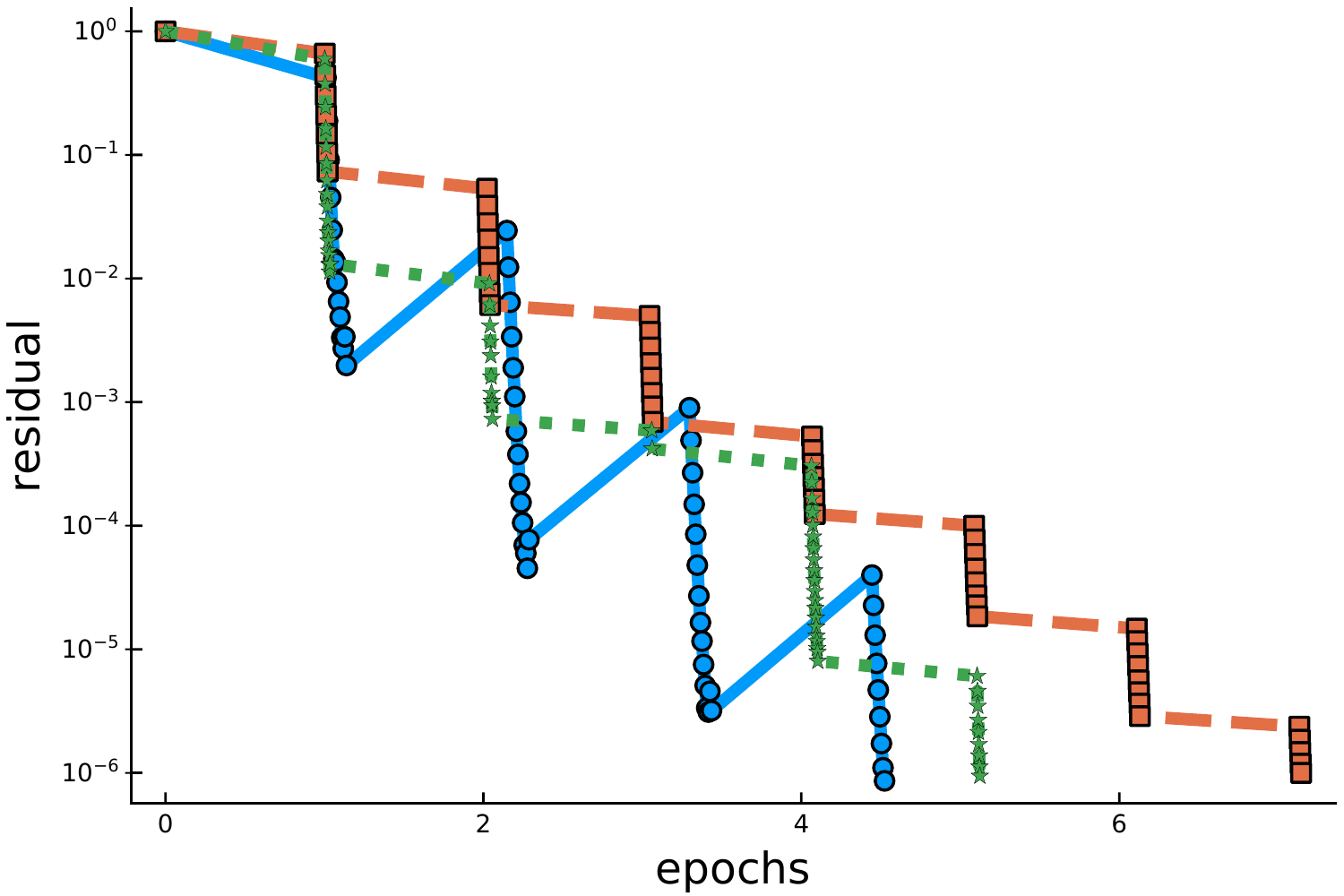}
        \includegraphics[width=0.45\textwidth]{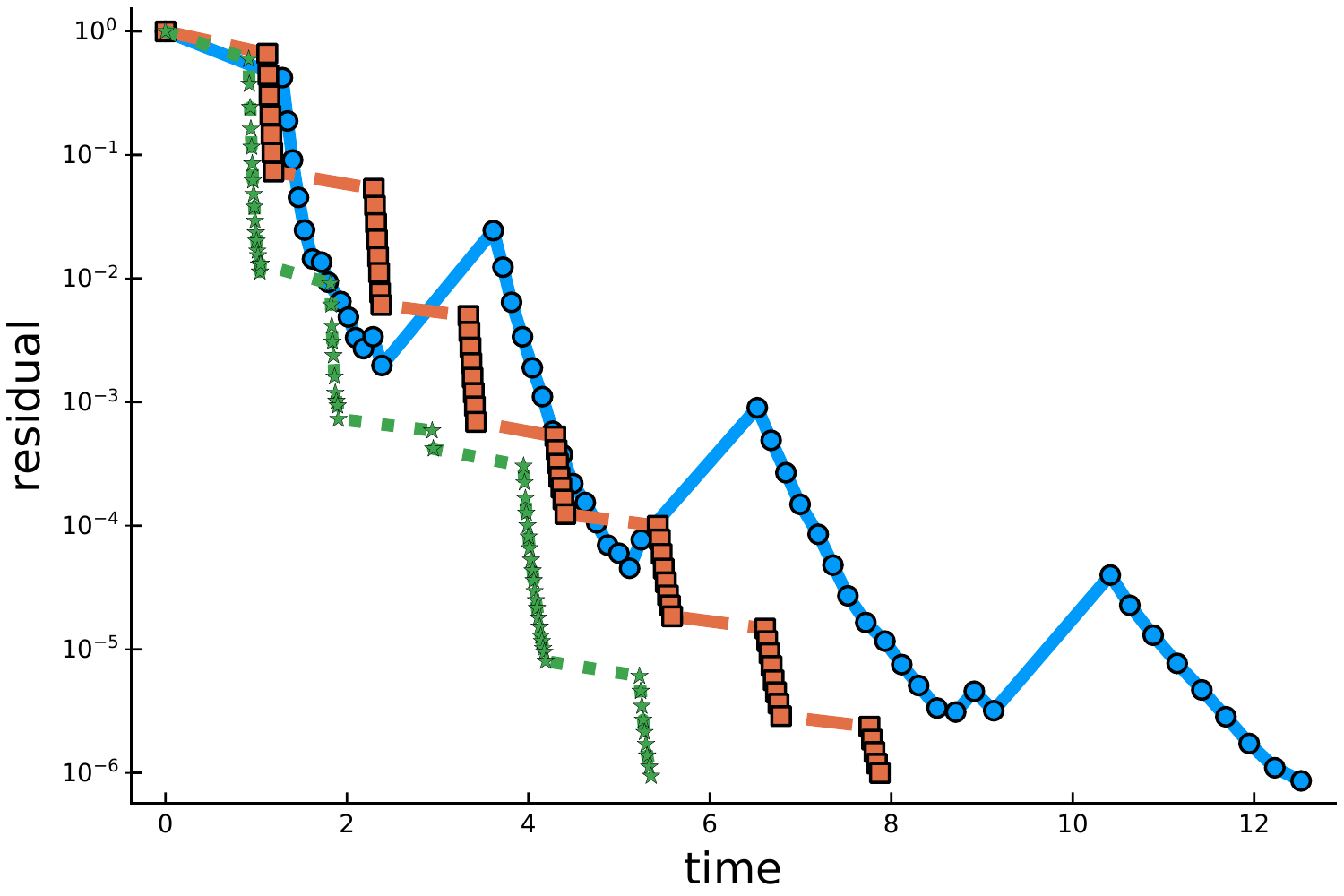}
        \caption{$\lambda = 10^{-1}$}
      \end{subfigure}\\
      \begin{subfigure}[b]{\textwidth}
        \centering
        \includegraphics[width=0.45\textwidth]{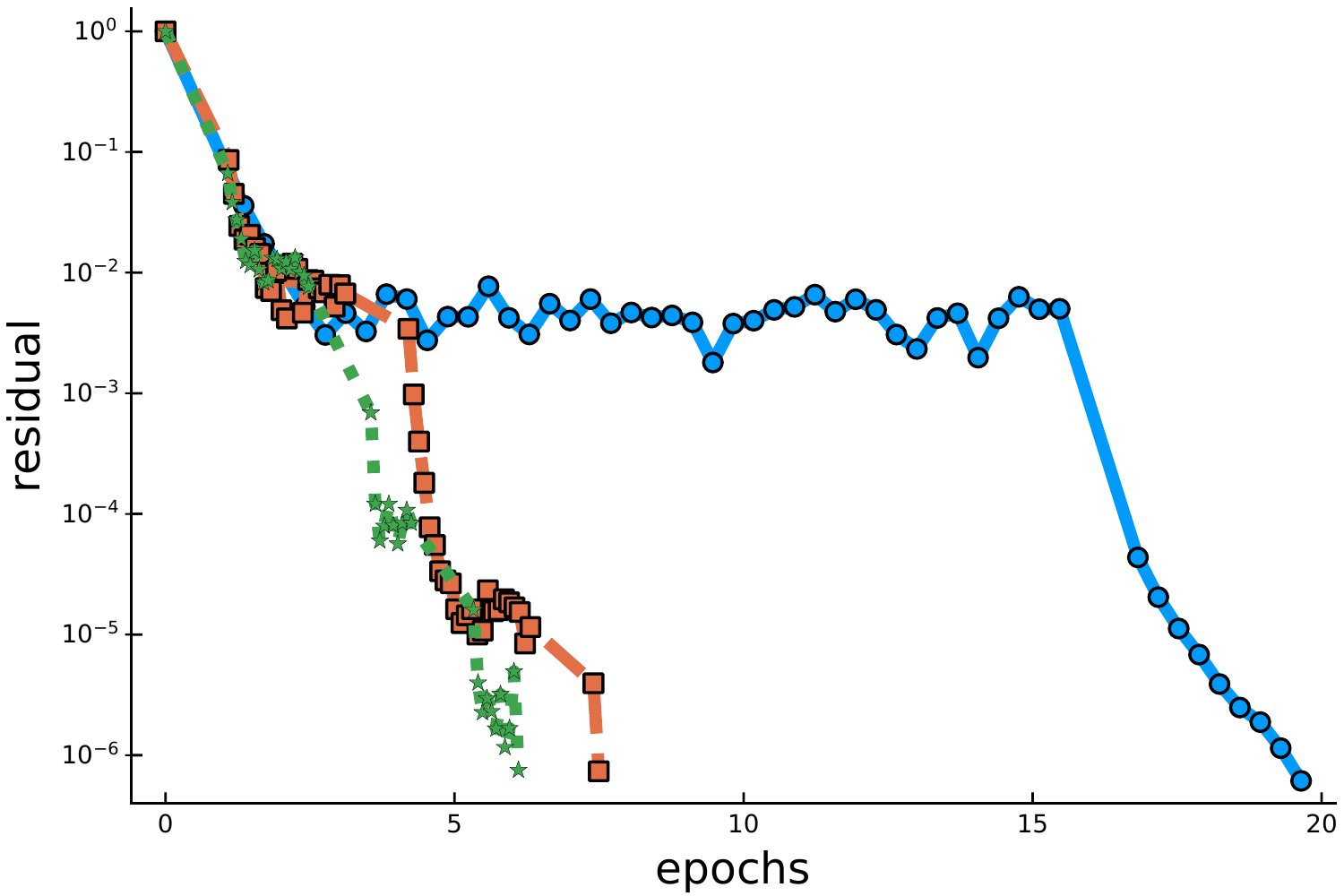}
        \includegraphics[width=0.45\textwidth]{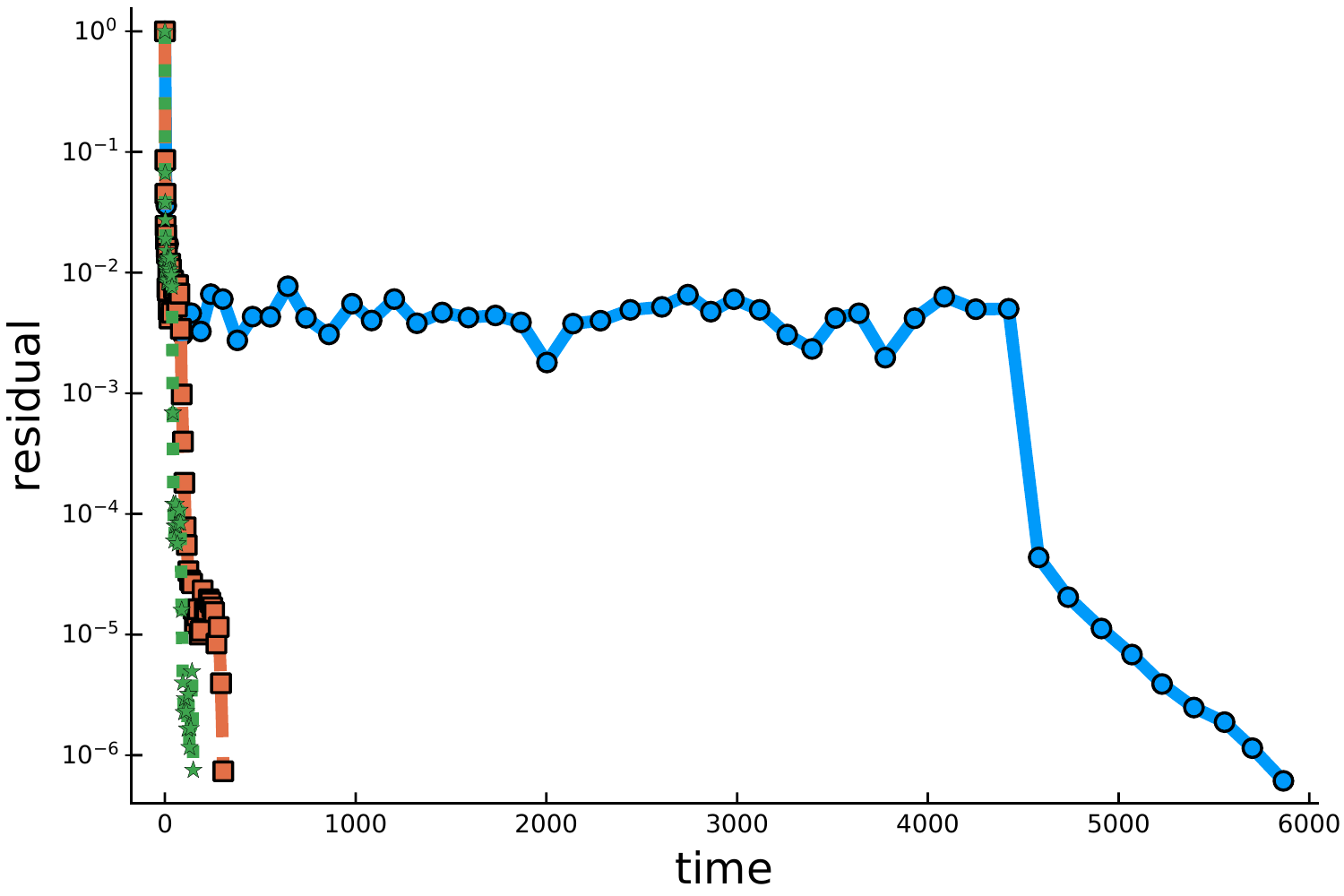}
        \caption{$\lambda = 10^{-3}$}
      \end{subfigure}
  \caption{Comparison of theoretical variants of SVRG with optimal inner loop size $m^*$ when theoretically available ($b=1$) on the \textit{ijcnn1} data set.}
  \label{fig:exp1C_ijcnn1}
  \end{center}
  \vskip -0.2in
\end{figure}

\begin{figure}[!htb]
  \vskip 0.2in
  \begin{center}
    \begin{subfigure}[b]{0.9\textwidth}
        \includegraphics[width=\textwidth]{exp1c/legend_exp1c_horizontal}
      \end{subfigure}\\
      \begin{subfigure}[b]{\textwidth}
        \centering
        \includegraphics[width=0.45\textwidth]{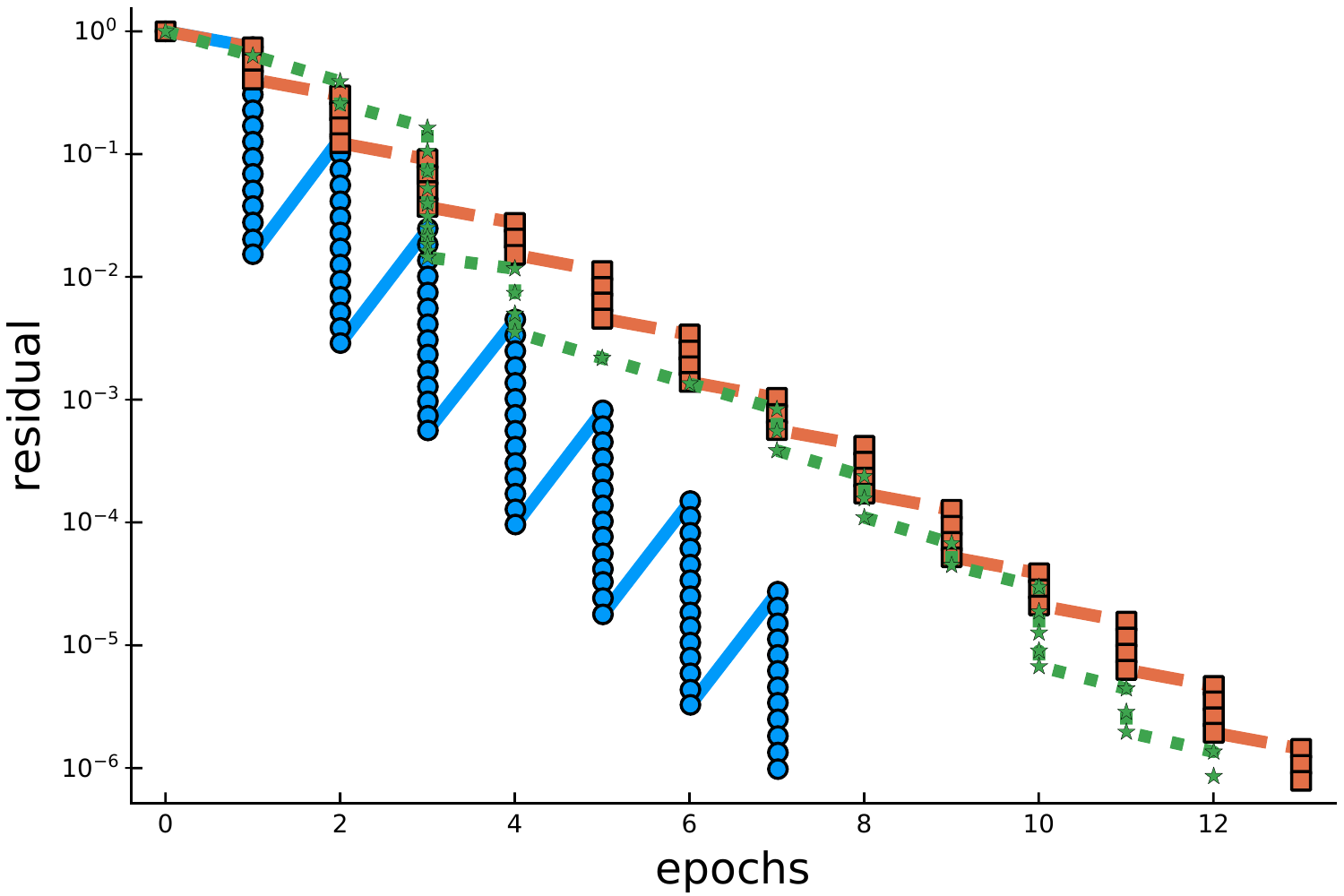}
        \includegraphics[width=0.45\textwidth]{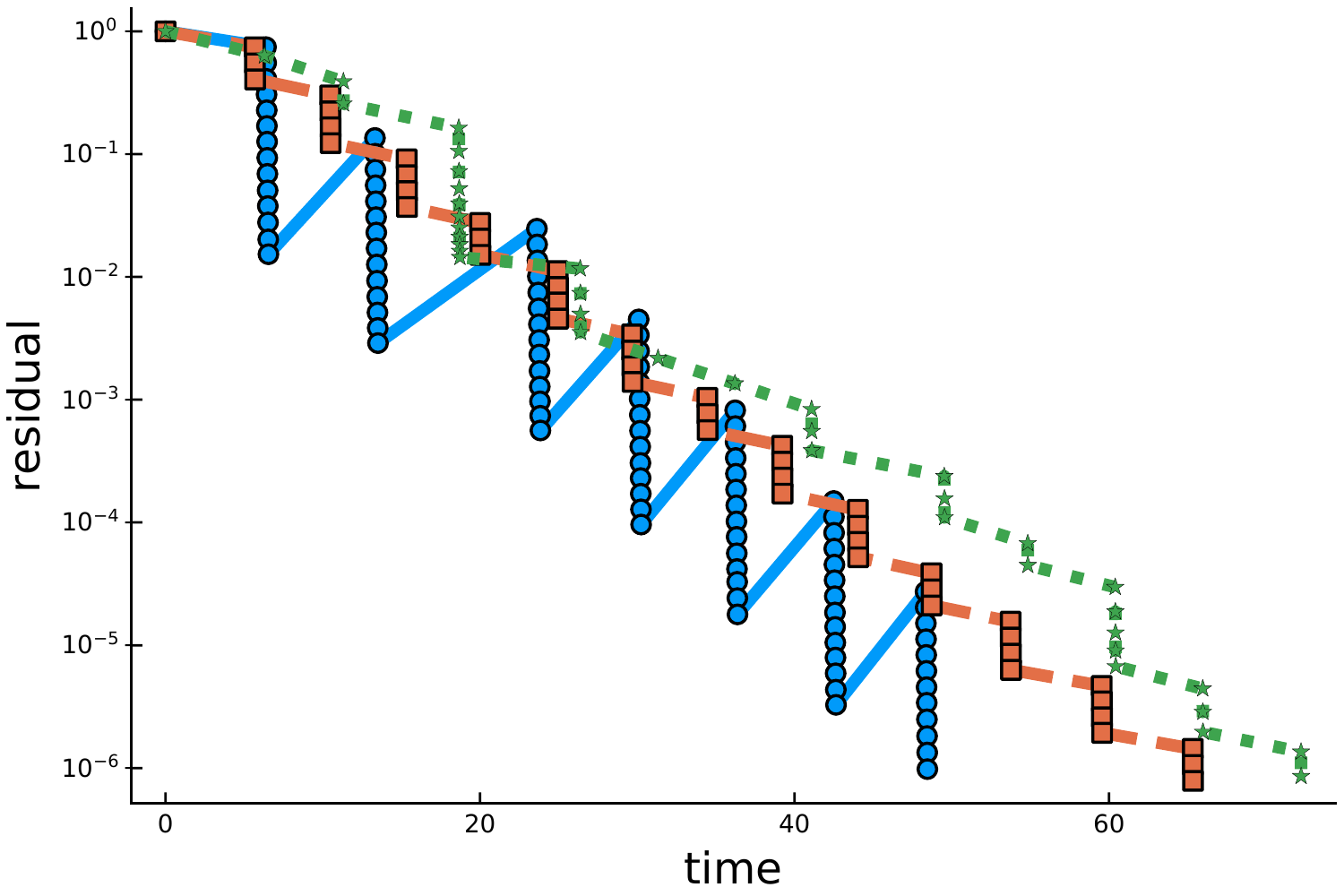}
        \caption{$\lambda = 10^{-1}$}
      \end{subfigure}\\
      \begin{subfigure}[b]{\textwidth}
        \centering
        \includegraphics[width=0.45\textwidth]{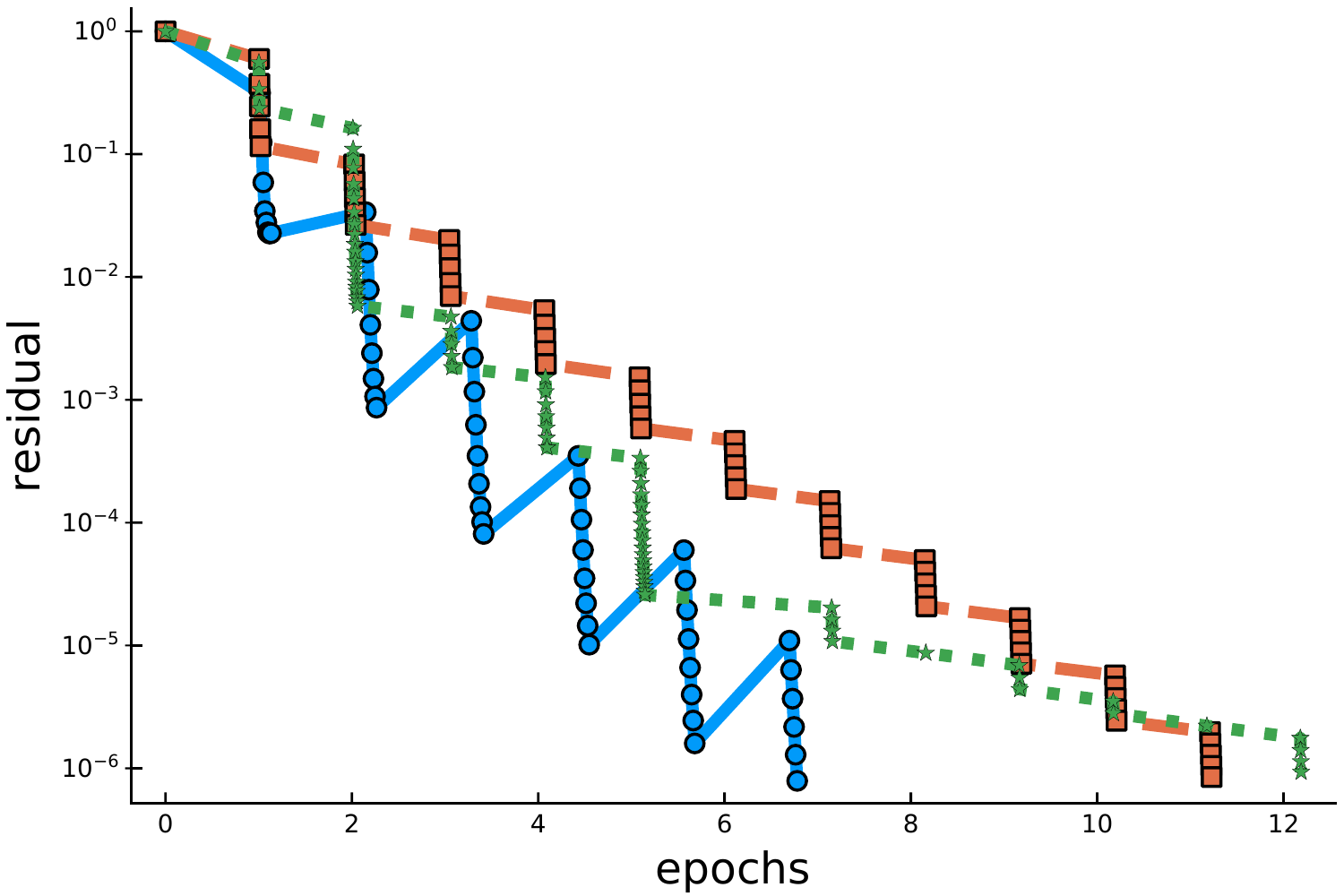}
        \includegraphics[width=0.45\textwidth]{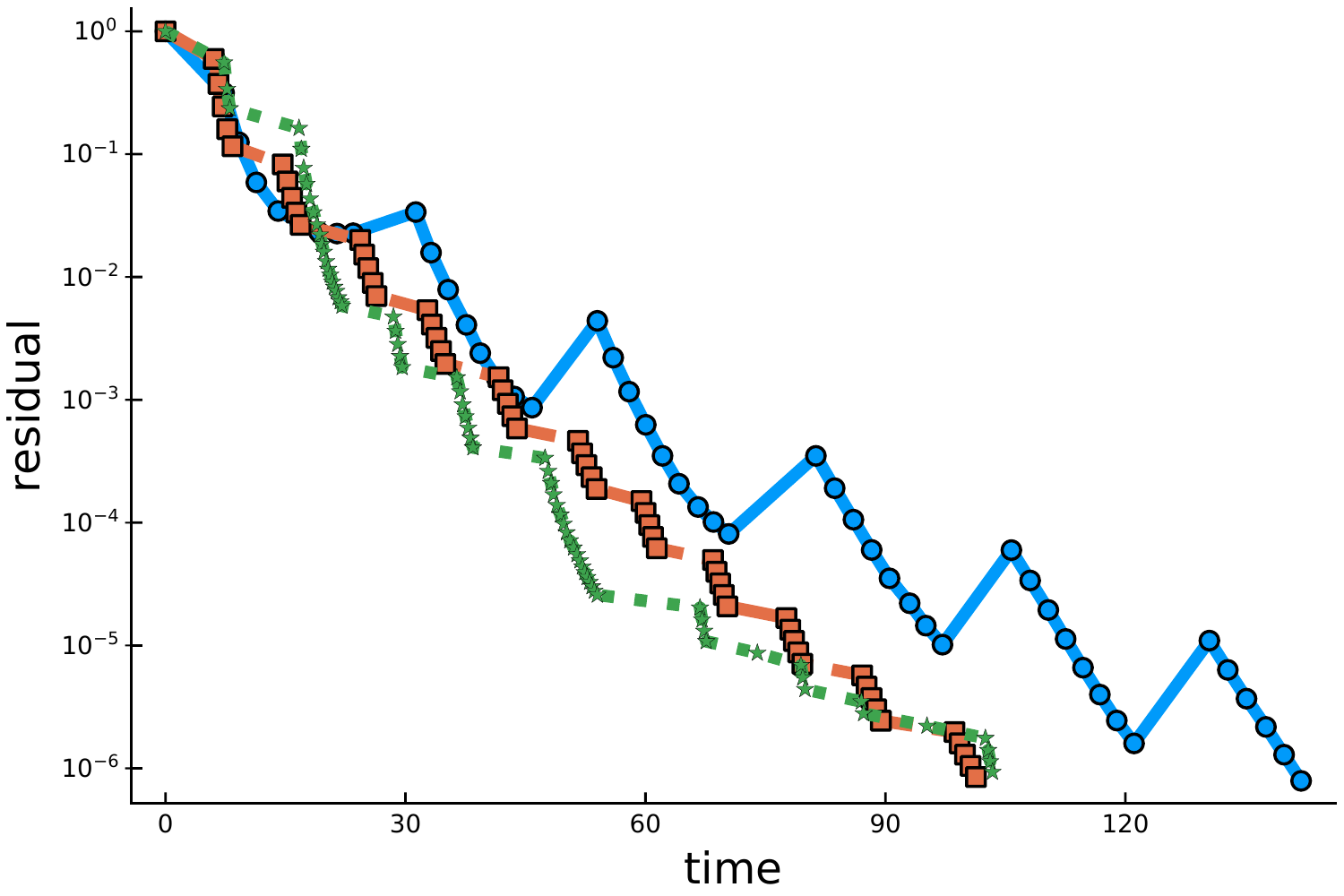}
        \caption{$\lambda = 10^{-3}$}
      \end{subfigure}
  \caption{Comparison of theoretical variants of SVRG with optimal inner loop size $m^*$ when theoretically available ($b=1$) on the \textit{real-sim} data set.}
  \label{fig:exp1C_real-sim}
  \end{center}
  \vskip -0.2in
\end{figure}

\subsection{Optimality of our theoretical parameters}

In this series of experiments, we only consider \textit{Free-SVRG} for which we evaluate the efficiency of our theoretical optimal parameters, namely the mini-batch size $b^*$ and the inner loop length $m^*$.

\subsubsection{Experiment 2.a: comparing different choices for the mini-batch size}
\label{sec:app_optimal_minibatch}
Here we consider \textit{Free-SVRG} and compare its performance for different batch sizes: the optimal one $b^*$, $1$, $100$, $\sqrt{n}$ and $n$. In Figure~\ref{fig:exp2A_YearPredictionMSD},~\ref{fig:exp2A_slice},~\ref{fig:exp2A_ijcnn1} and~\ref{fig:exp2A_real-sim}, we show that the optimal mini-batch size we predict using Table~\ref{tab:optimal_mini-batch} always leads to the fastest convergence in epoch plot (or at least near the fastest in Figure~\ref{fig:exp2A_YearPredictionMSD_1e-3}).

\begin{figure}[!htb]
  \vskip 0.2in
  \begin{center}
      \begin{subfigure}[b]{\textwidth}
        \centering
        \includegraphics[width=0.45\textwidth]{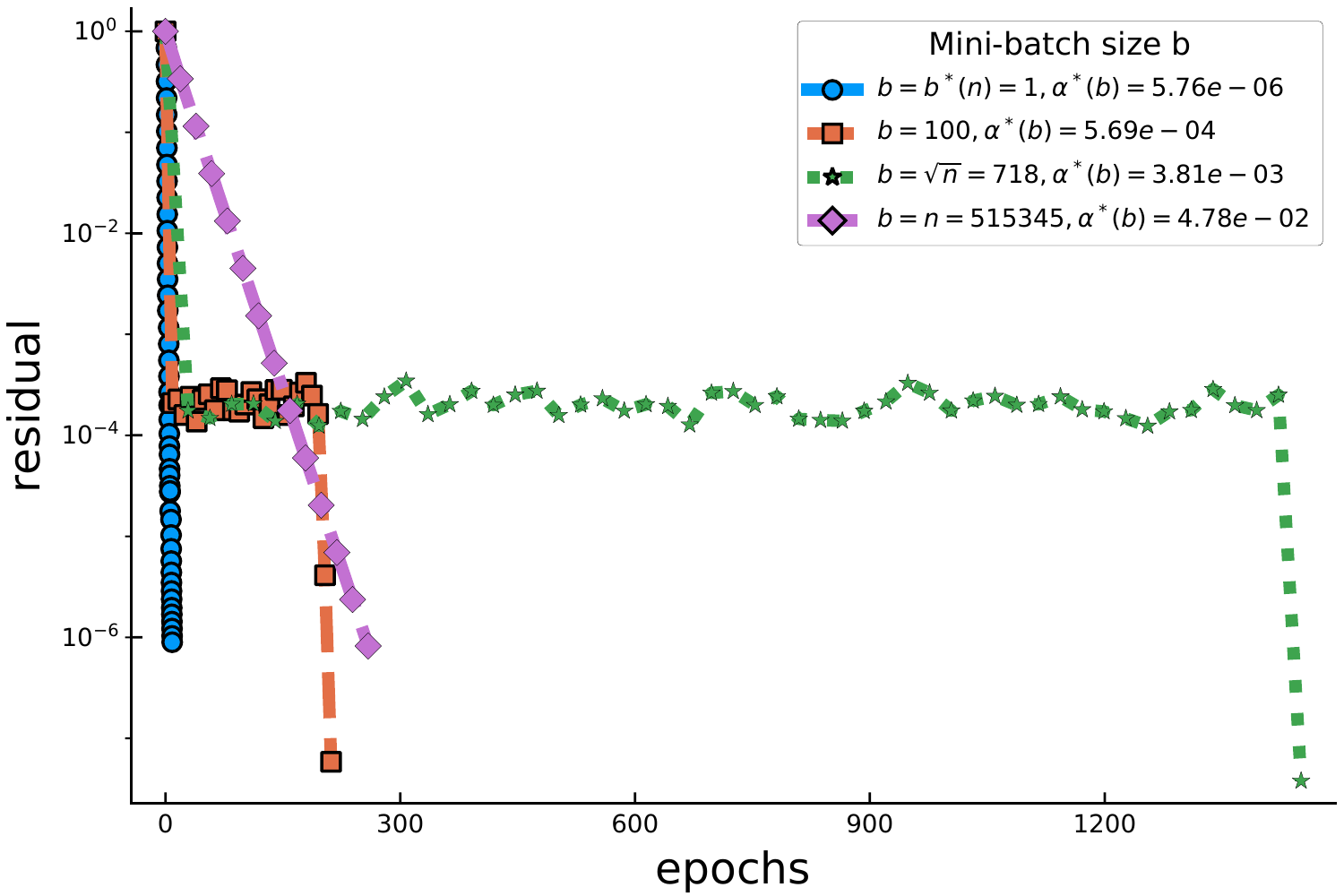}
        \includegraphics[width=0.45\textwidth]{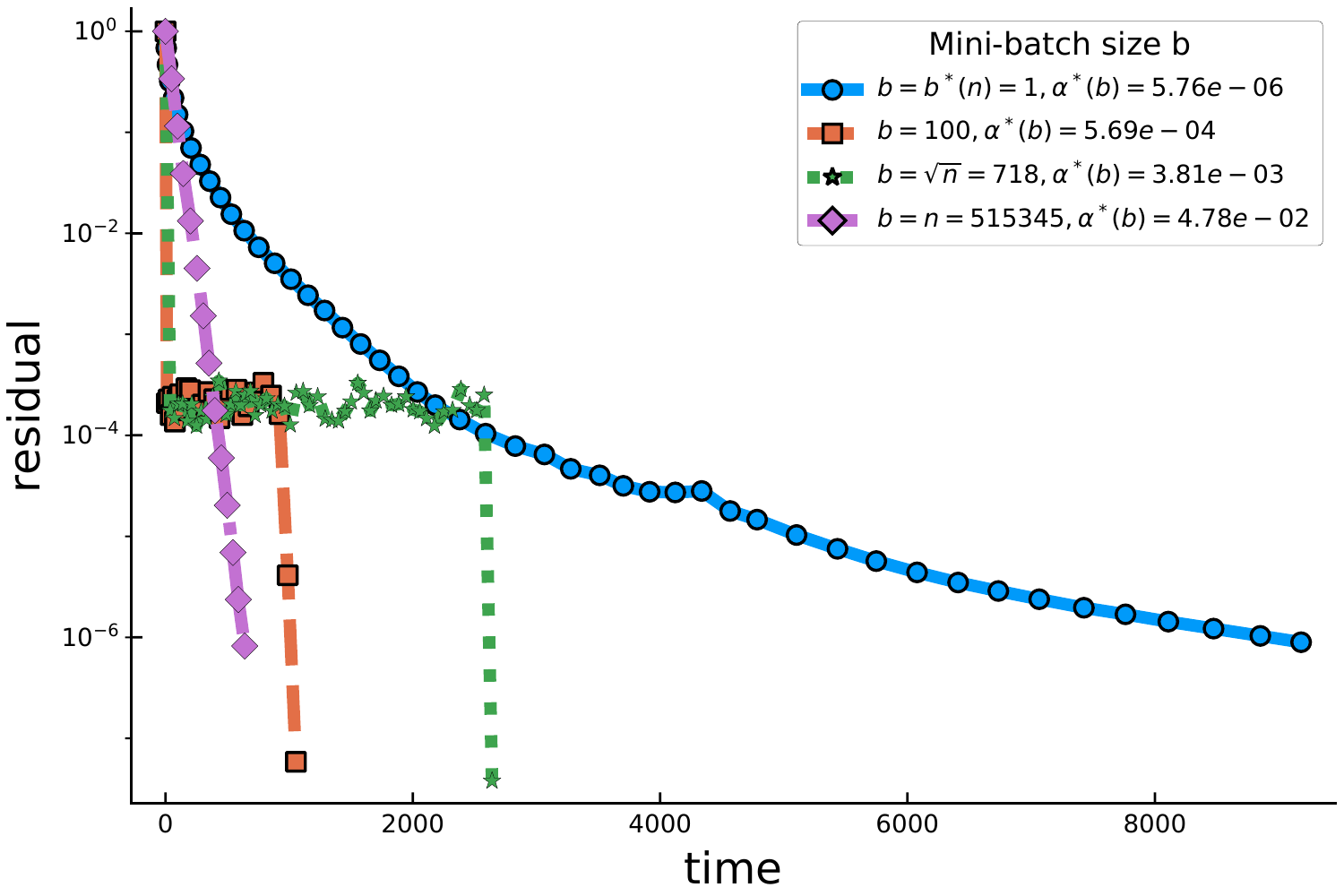}
        \caption{$\lambda = 10^{-1}$}
      \end{subfigure}\\
      \begin{subfigure}[b]{\textwidth}
        \centering
        \includegraphics[width=0.45\textwidth]{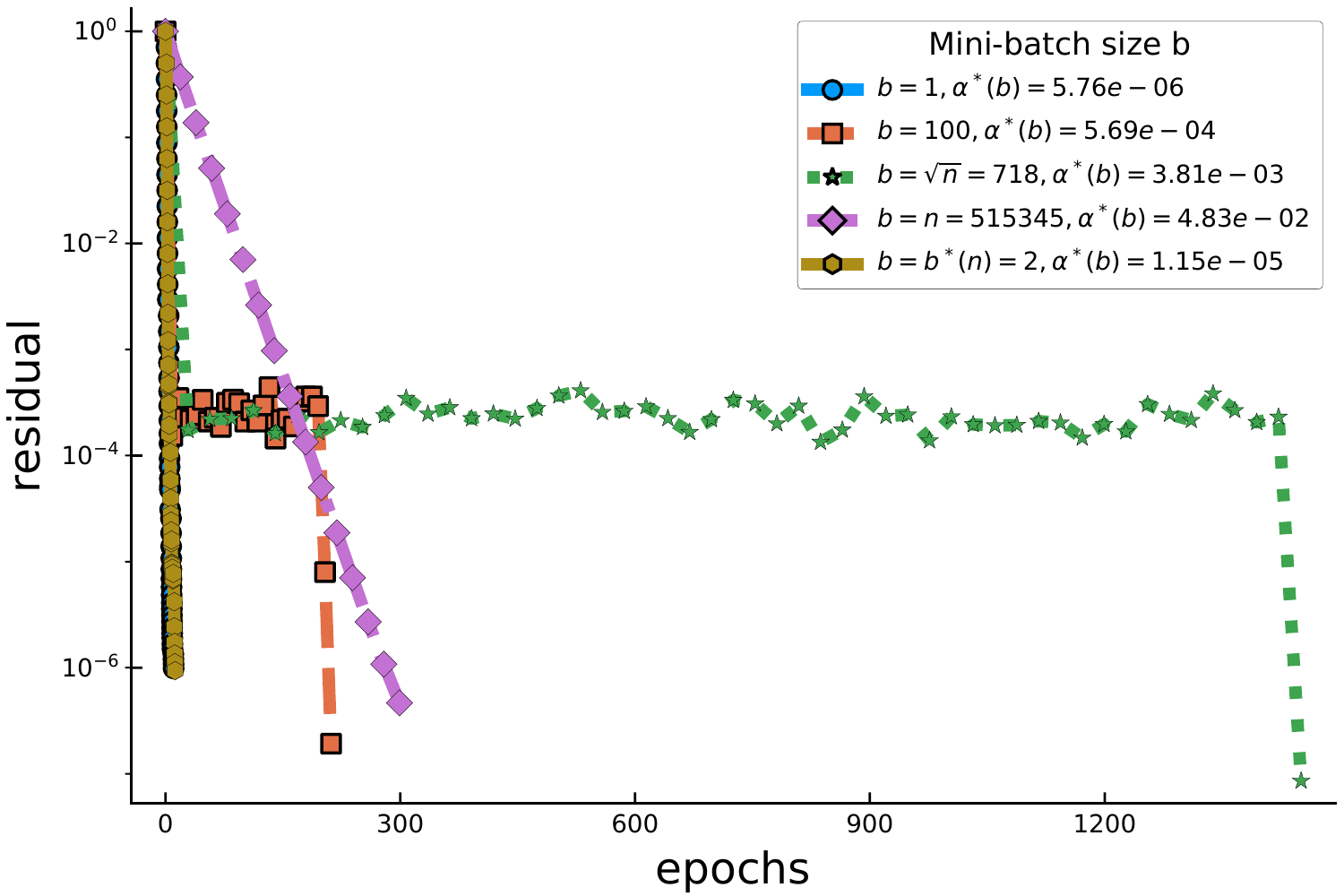}
        \includegraphics[width=0.45\textwidth]{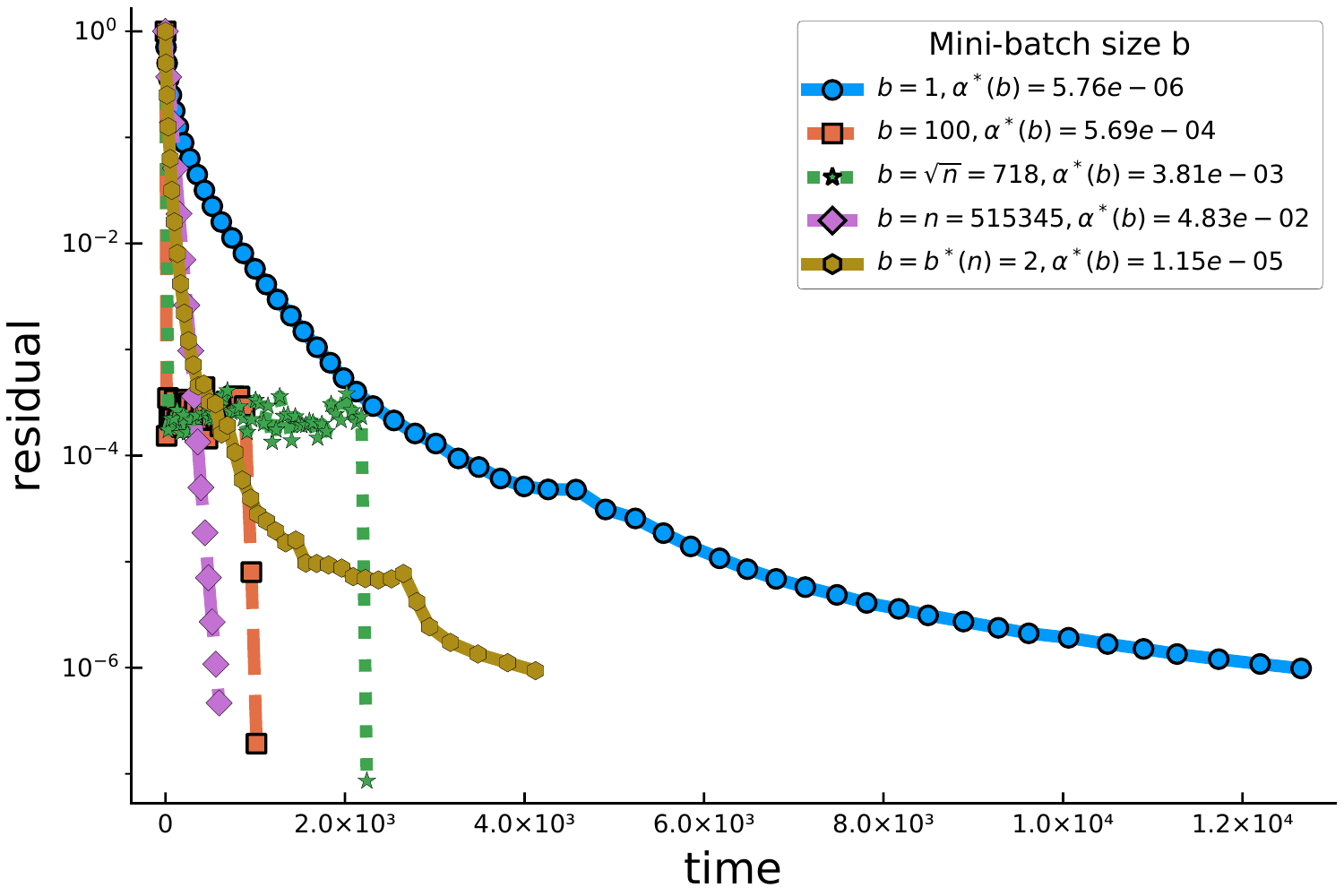}
        \caption{$\lambda = 10^{-3}$}
        \label{fig:exp2A_YearPredictionMSD_1e-3}
      \end{subfigure}
  \caption{Optimality of our mini-batch size $b^*$ given in Table~\ref{tab:optimal_mini-batch} for \textit{Free-SVRG} on the \textit{YearPredictionMSD} data set.}
  \label{fig:exp2A_YearPredictionMSD}
  \end{center}
  \vskip -0.2in
\end{figure}

\begin{figure}[!htb]
 \vskip 0.2in
 \begin{center}
     \begin{subfigure}[b]{\textwidth}
        \centering
        \includegraphics[width=0.45\textwidth]{exp2a/ridge_slice-column-scaling-regularizor-1e-01-exp2a-lame23-final-epoc}
        \includegraphics[width=0.45\textwidth]{exp2a/ridge_slice-column-scaling-regularizor-1e-01-exp2a-lame23-final-time}
        \caption{$\lambda = 10^{-1}$}
     \end{subfigure}\\
     \begin{subfigure}[b]{\textwidth}
        \centering
        \includegraphics[width=0.45\textwidth]{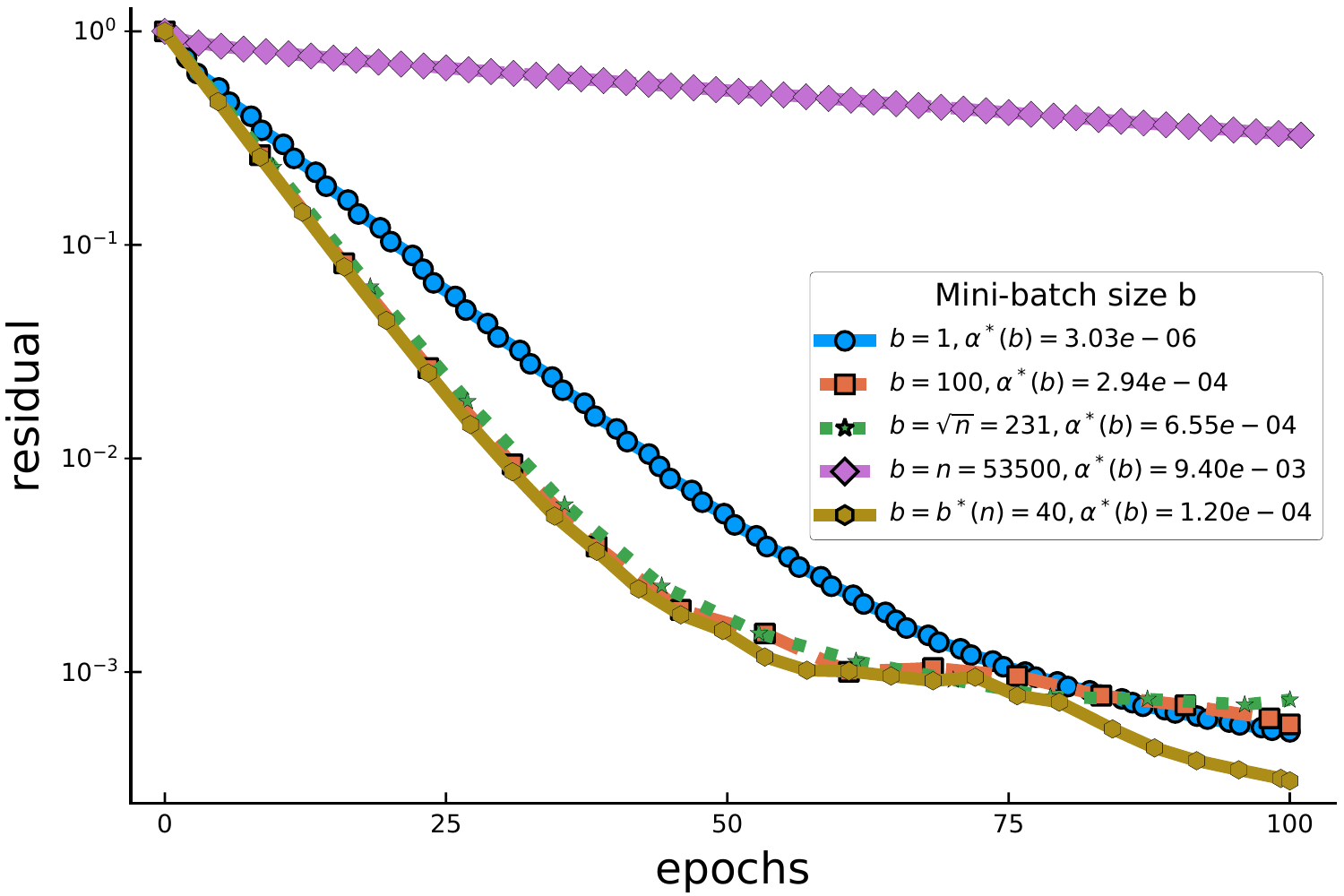}
        \includegraphics[width=0.45\textwidth]{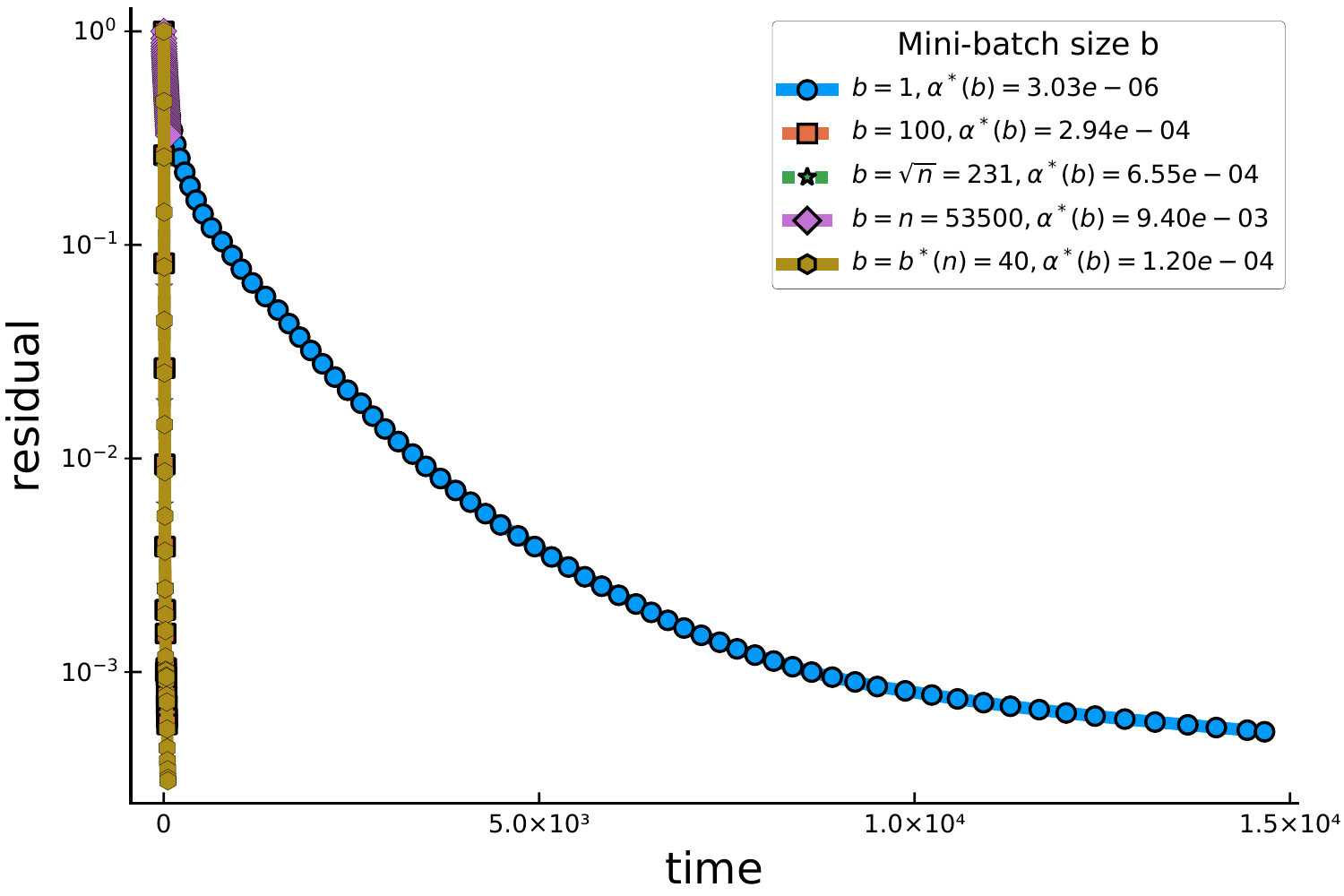}
        \caption{$\lambda = 10^{-3}$}
     \end{subfigure}
 \caption{Optimality of our mini-batch size $b^*$ given in Table~\ref{tab:optimal_mini-batch} for \textit{Free-SVRG} on the \textit{slice} data set.}
 \label{fig:exp2A_slice}
 \end{center}
 \vskip -0.2in
\end{figure}

\begin{figure}[!htb]
  \vskip 0.2in
  \begin{center}
      \begin{subfigure}[b]{\textwidth}
        \centering
        \includegraphics[width=0.45\textwidth]{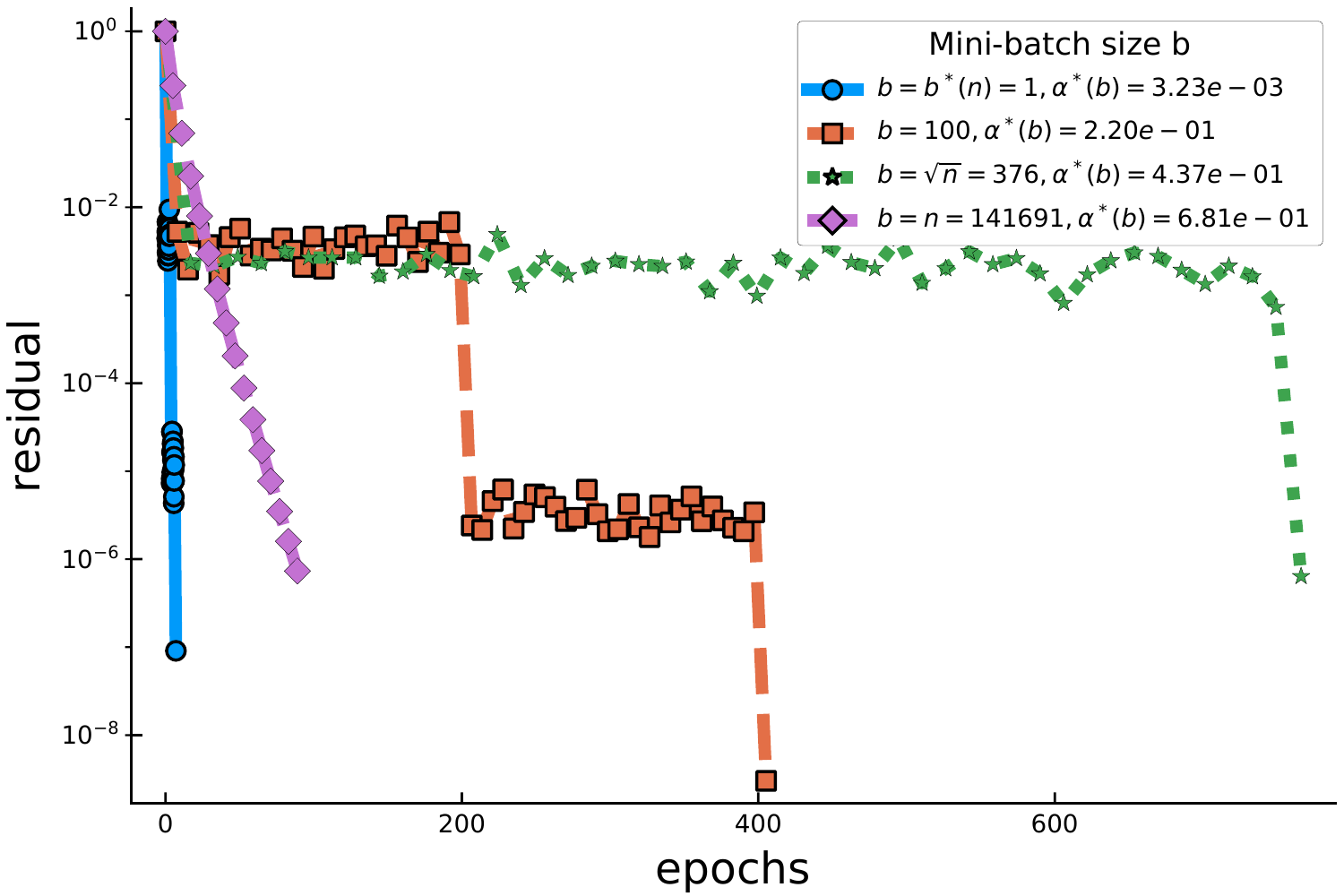}
        \includegraphics[width=0.45\textwidth]{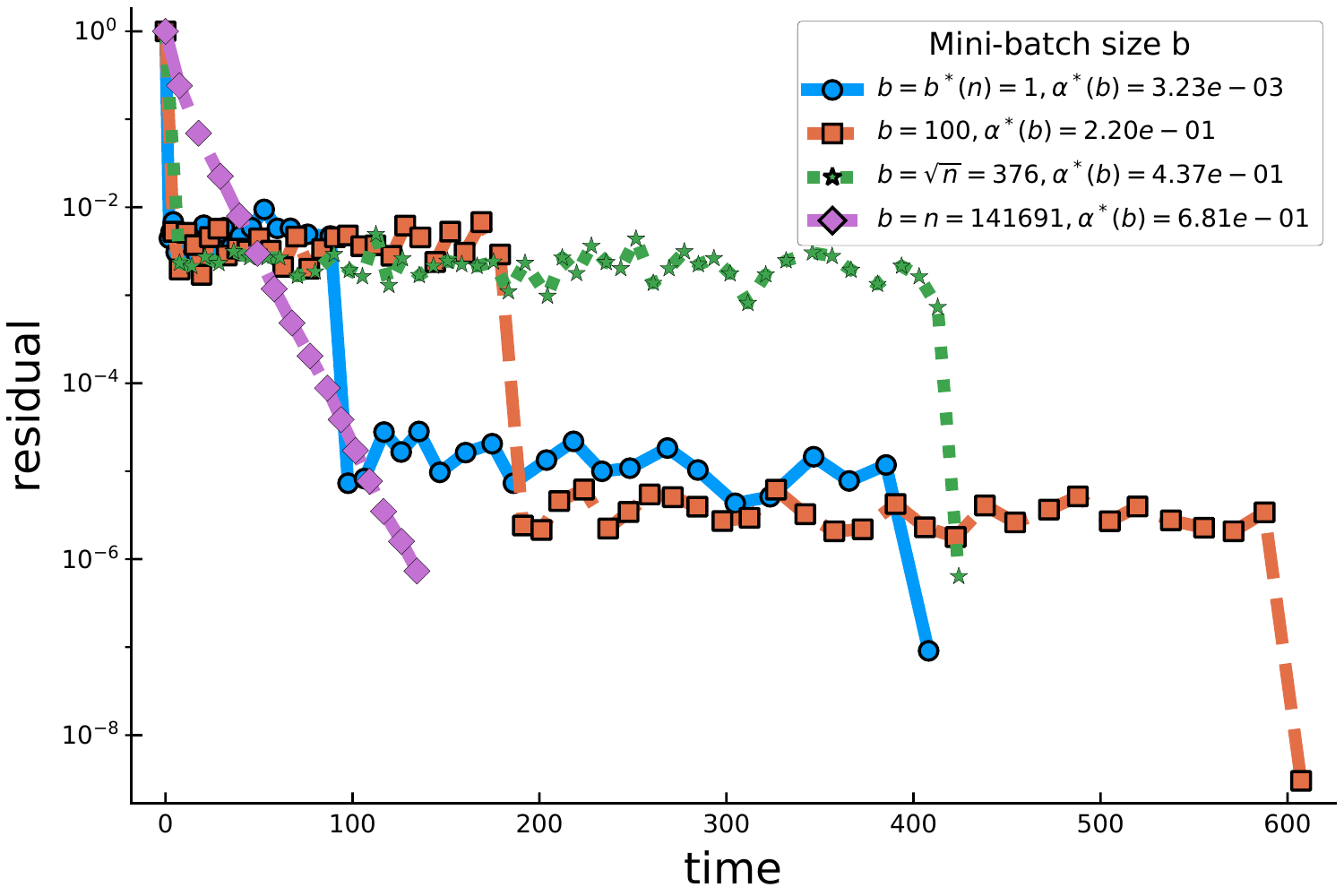}
        \caption{$\lambda = 10^{-1}$}
      \end{subfigure}\\
      \begin{subfigure}[b]{\textwidth}
        \centering
        \includegraphics[width=0.45\textwidth]{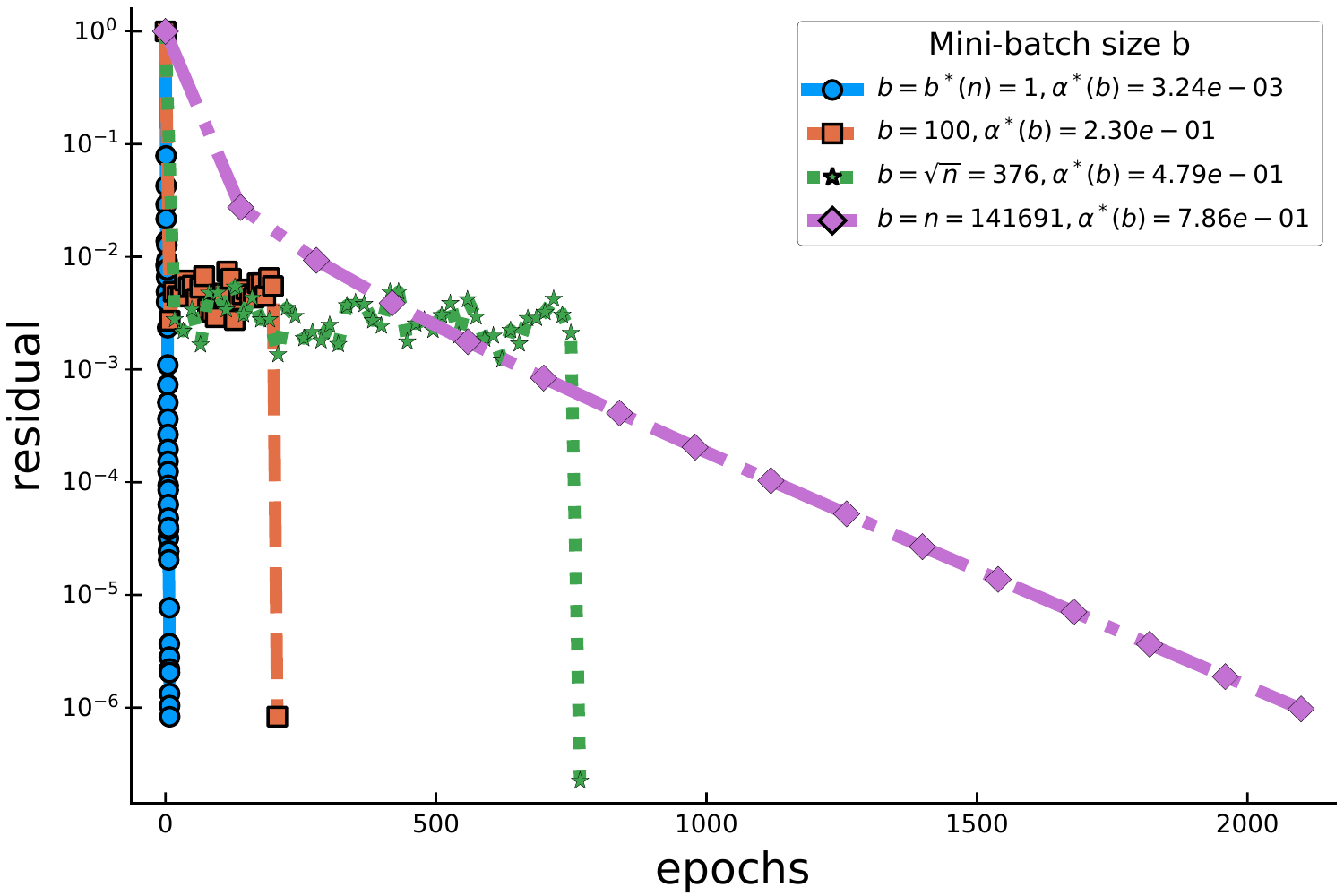}
        \includegraphics[width=0.45\textwidth]{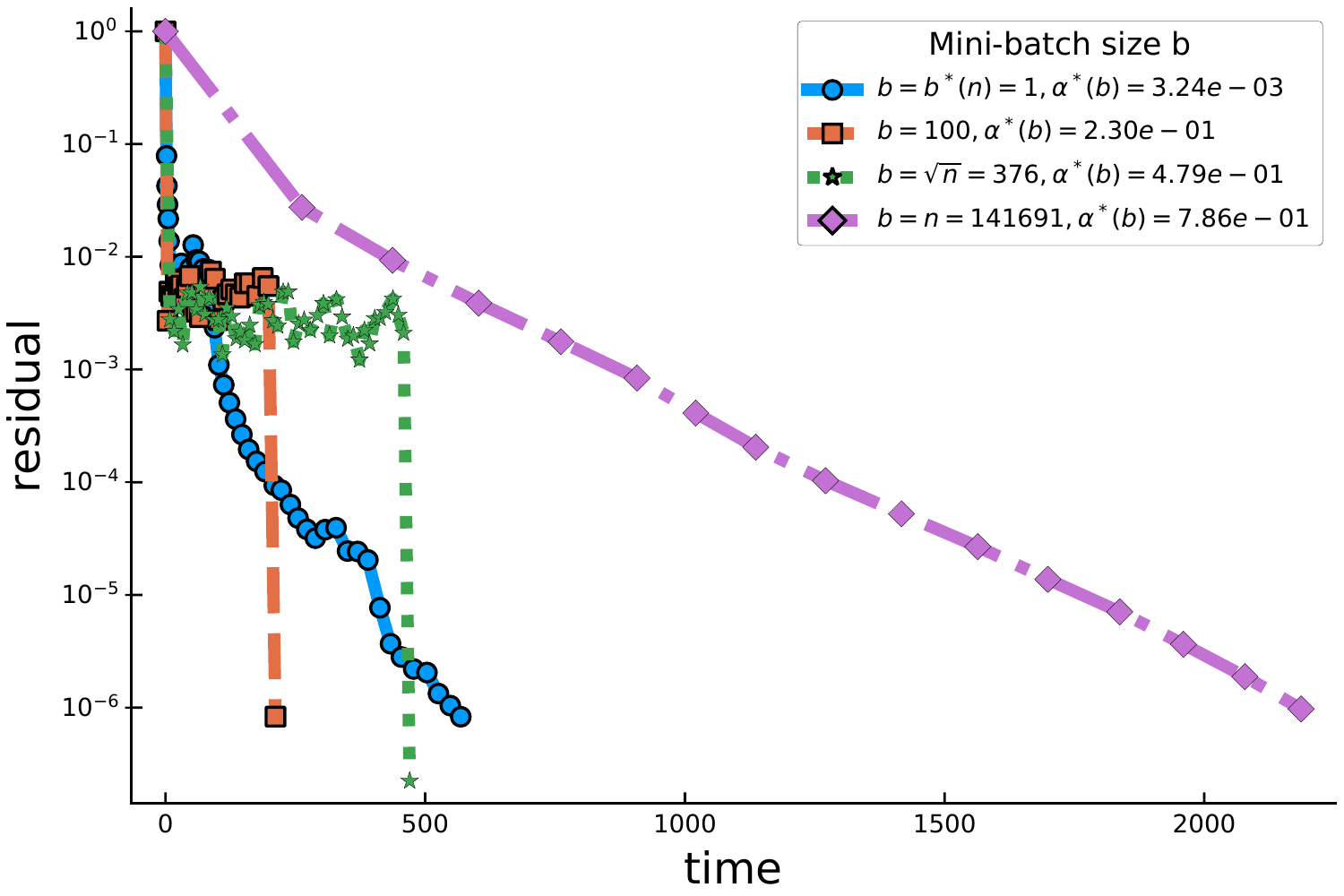}
        \caption{$\lambda = 10^{-3}$}
      \end{subfigure}
  \caption{Optimality of our mini-batch size $b^*$ given in Table~\ref{tab:optimal_mini-batch} for \textit{Free-SVRG} on the \textit{ijcnn1} data set.}
  \label{fig:exp2A_ijcnn1}
  \end{center}
  \vskip -0.2in
\end{figure}

\begin{figure}[!htb]
  \vskip 0.2in
  \begin{center}
      \begin{subfigure}[b]{\textwidth}
        \centering
        \includegraphics[width=0.45\textwidth]{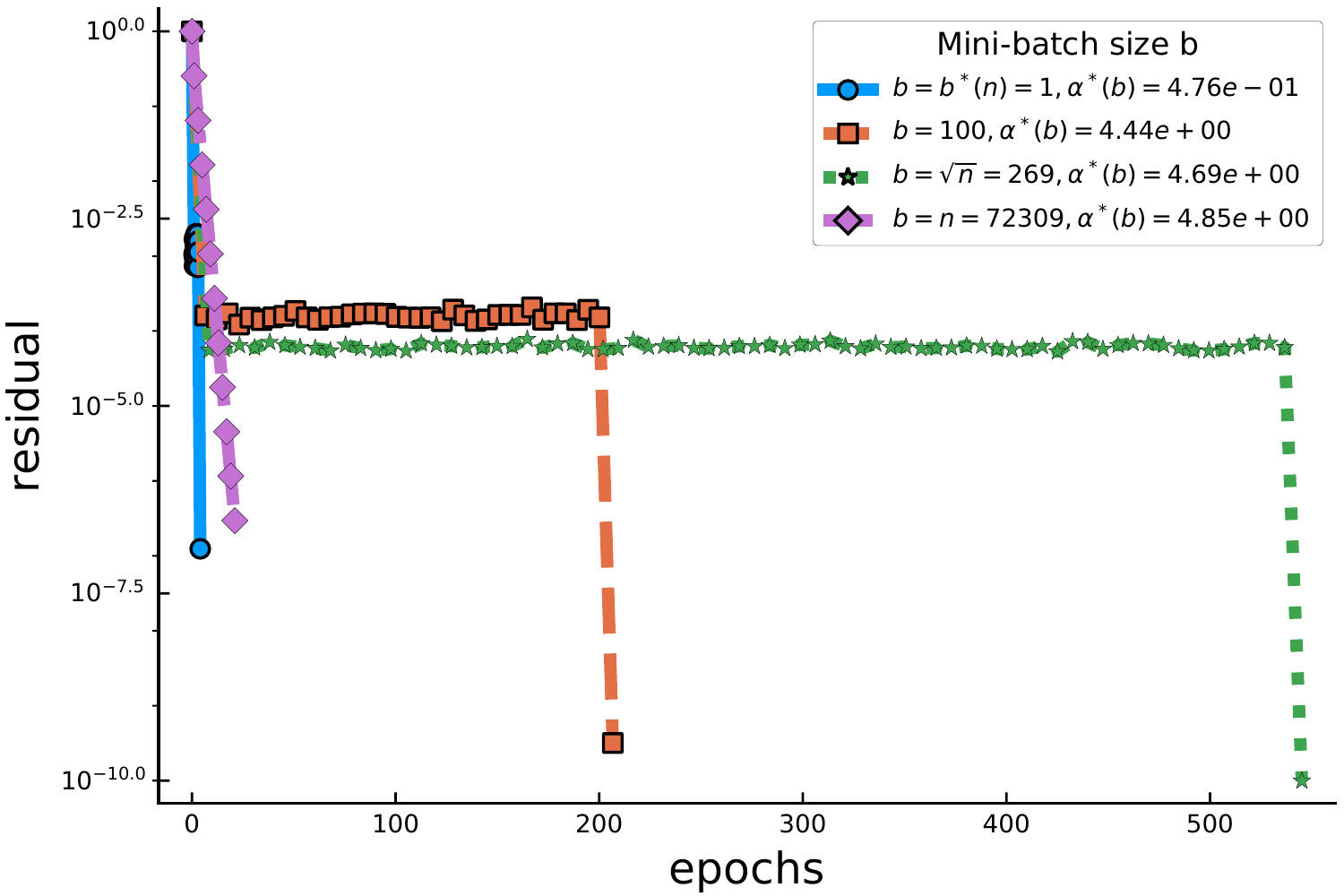}
        \includegraphics[width=0.45\textwidth]{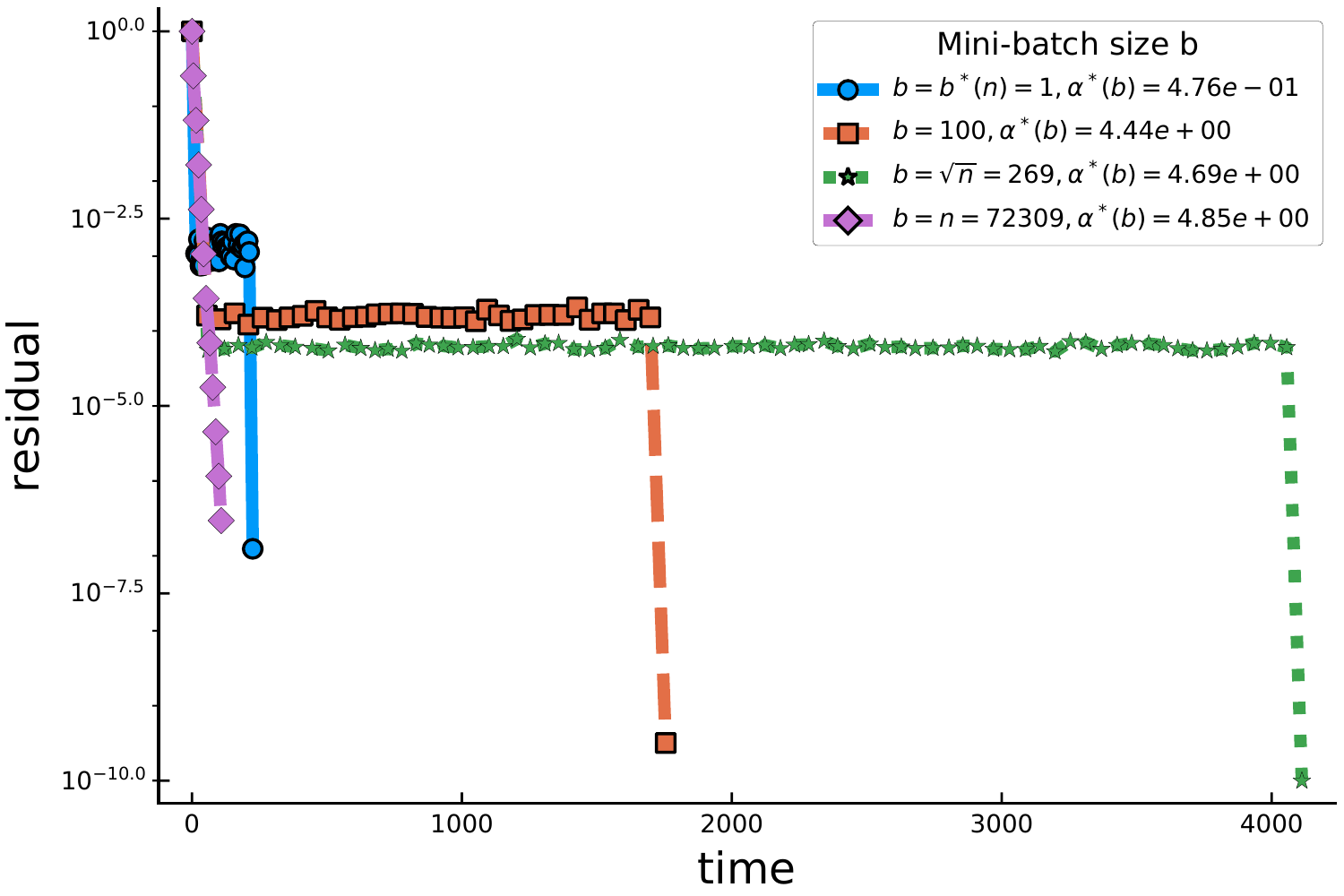}
        \caption{$\lambda = 10^{-1}$}
      \end{subfigure}\\
      \begin{subfigure}[b]{\textwidth}
        \centering
        \includegraphics[width=0.45\textwidth]{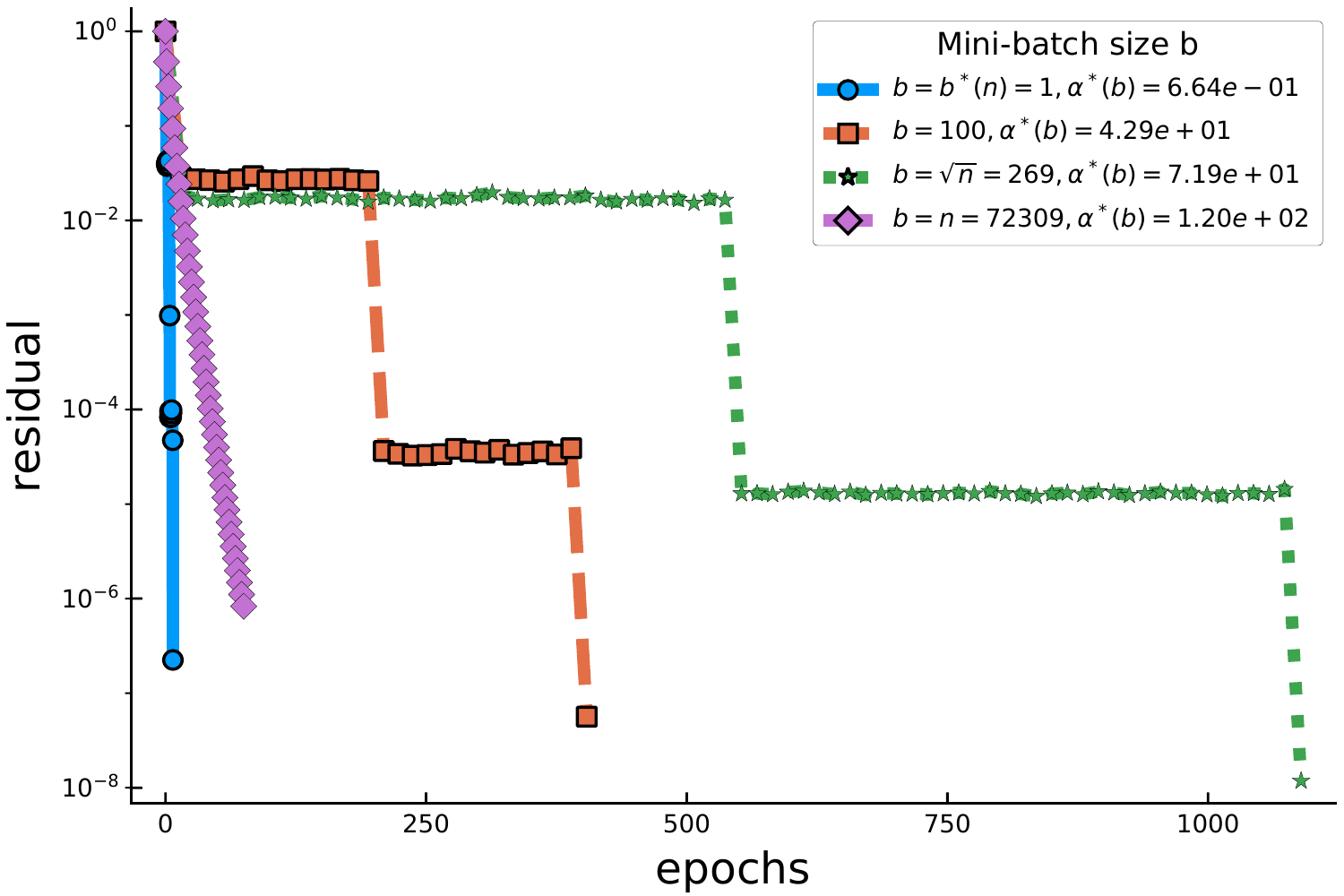}
        \includegraphics[width=0.45\textwidth]{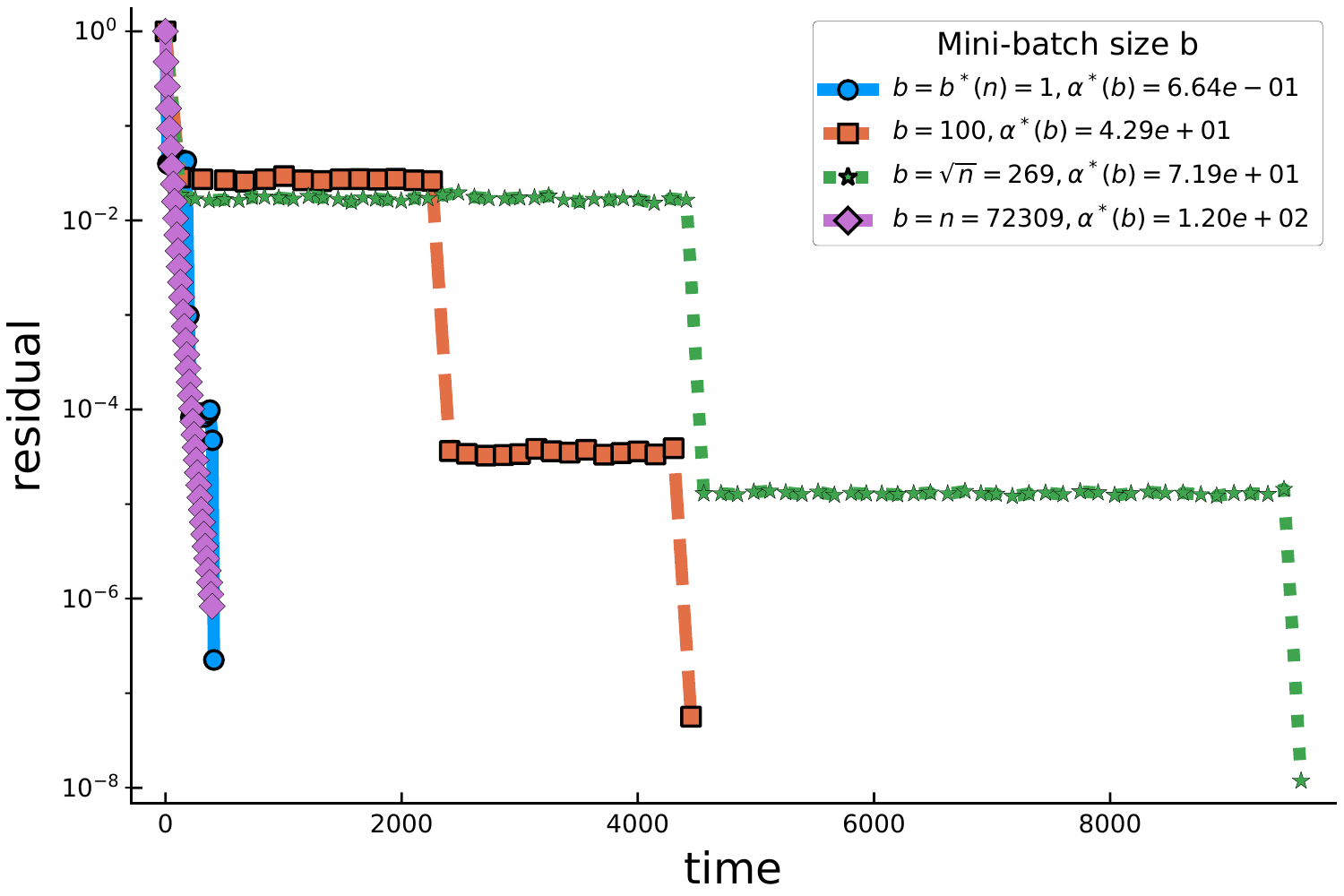}
        \caption{$\lambda = 10^{-3}$}
      \end{subfigure}
  \caption{Optimality of our mini-batch size $b^*$ given in Table~\ref{tab:optimal_mini-batch} for \textit{Free-SVRG} on the \textit{real-sim} data set.}
  \label{fig:exp2A_real-sim}
  \end{center}
  \vskip -0.2in
\end{figure}

\subsubsection{Experiment 2.b: comparing different choices for the inner loop size}
\label{sec:app_exp_inner_loop}
We set $b=1$ and compare different values for the inner loop size: the optimal one $m^*$, $L_{\max}/\mu$, $3L_{\max}/\mu$ and $2n$ in order to validate our theory in Proposition~\ref{prop:optimalloopsize}, that is, that the overall performance of \textit{Free-SVRG} is not sensitive to the range of values of $m$, so long as $m$ is close to $n$, $L_{\max}/\mu$ or anything in between. And indeed, this is what we confirmed in  Figures~\ref{fig:exp2B_YearPredictionMSD},~\ref{fig:exp2B_slice},~\ref{fig:exp2B_ijcnn1} and~\ref{fig:exp2B_real-sim}.
The choice $m=2n$ is the one suggested by \cite{johnson2013accelerating} in their practical SVRG (Option II). We notice that our optimal inner loop size $m^*$ underperforms compared to $n$ or $2n$ only in Figure~\ref{fig:exp2B_real-sim_1e-1}, which is a very rare kind of problem since it is very well conditioned ($L_{\max}/\mu \approx 4$).

\begin{figure}[!htb]
  \vskip 0.2in
  \begin{center}
    \begin{subfigure}[b]{0.48\textwidth}
        \includegraphics[width=\textwidth]{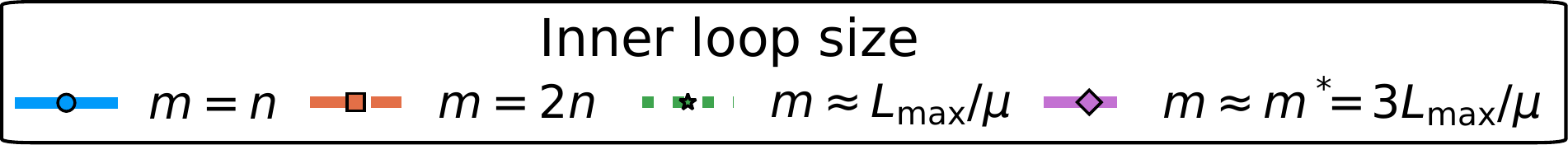}
      \end{subfigure}\\
      \begin{subfigure}[b]{\textwidth}
        \centering
        \includegraphics[width=0.45\textwidth]{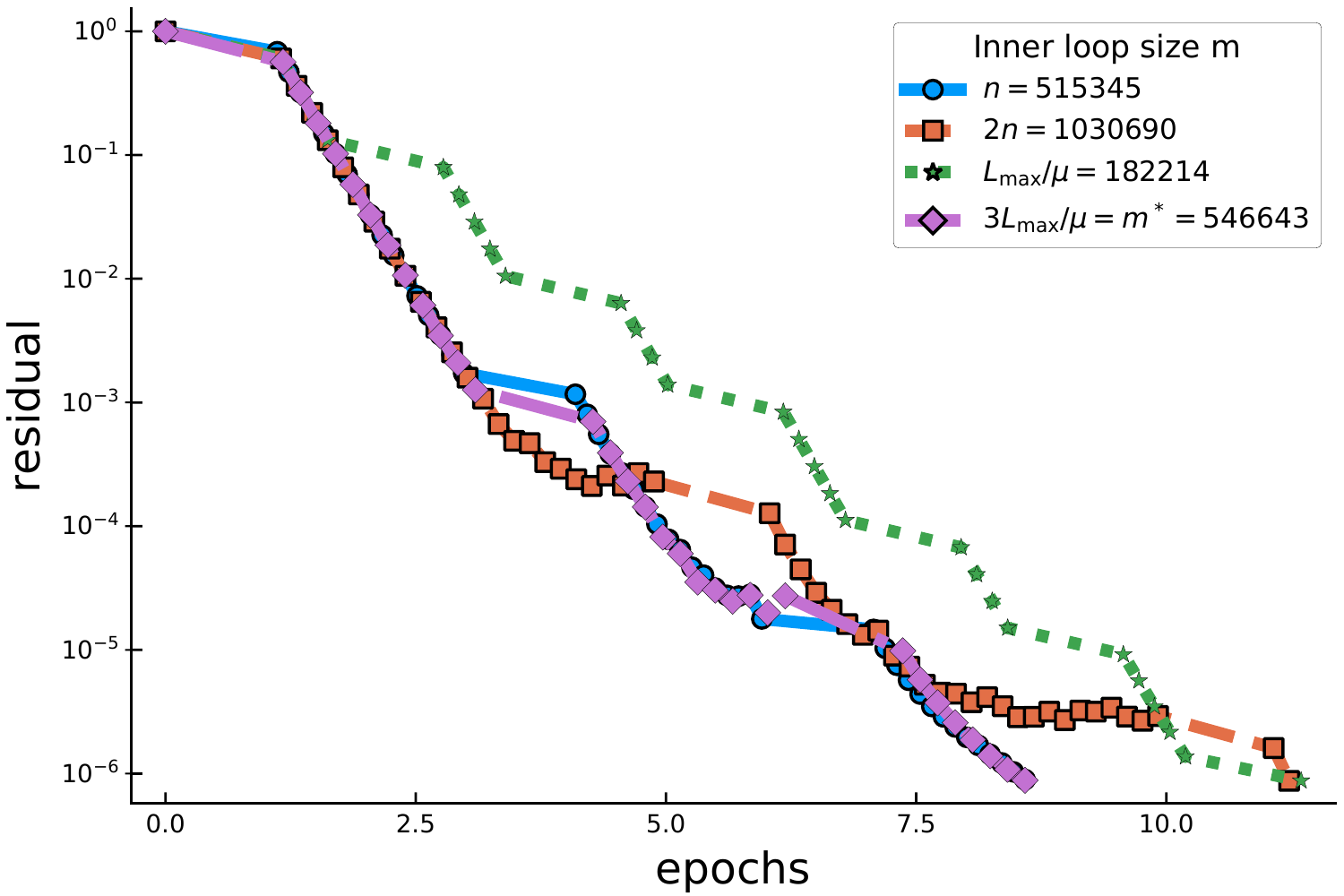}
        \includegraphics[width=0.45\textwidth]{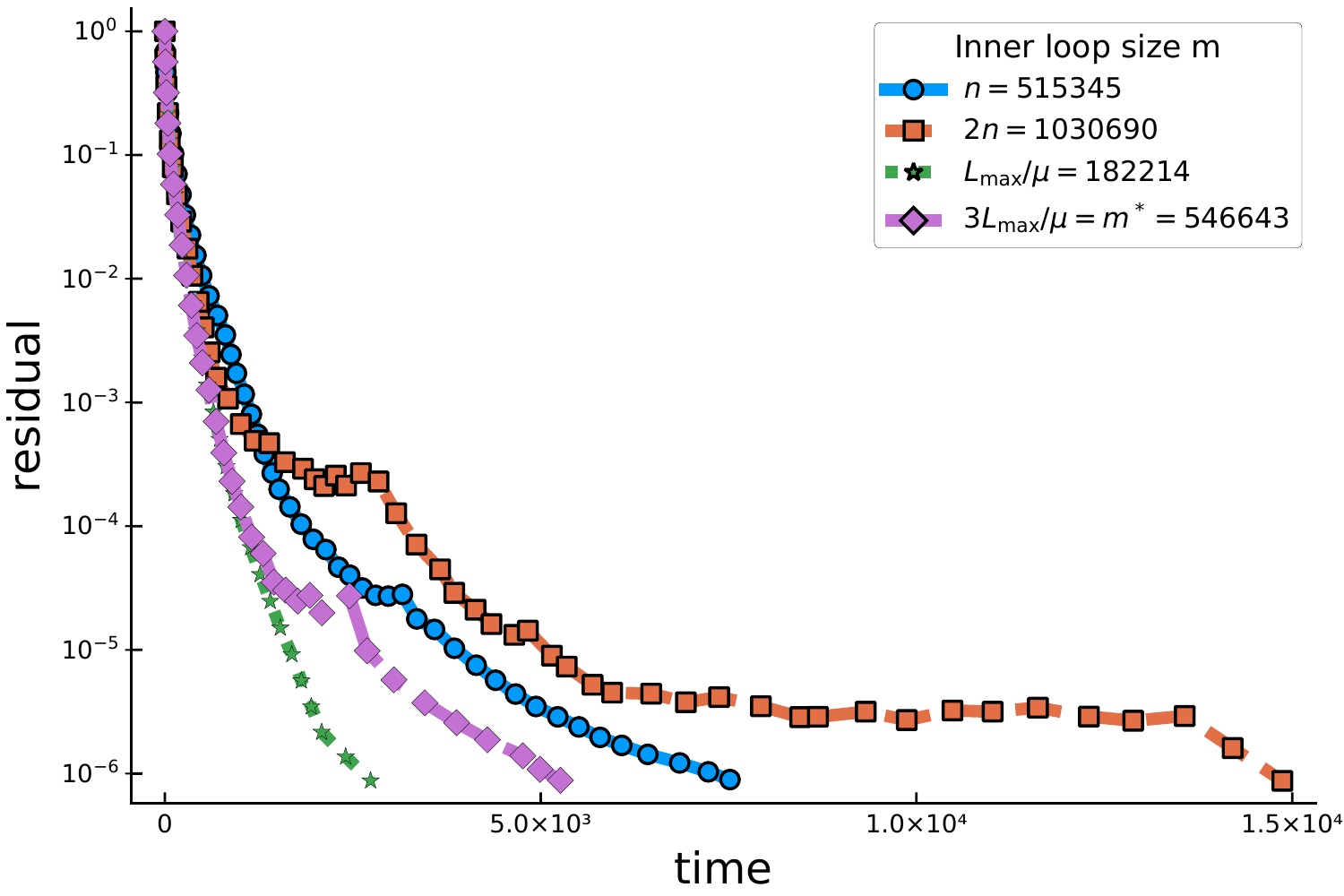}
        \caption{$\lambda = 10^{-1}$}
      \end{subfigure}\\
      \begin{subfigure}[b]{\textwidth}
        \centering
        \includegraphics[width=0.45\textwidth]{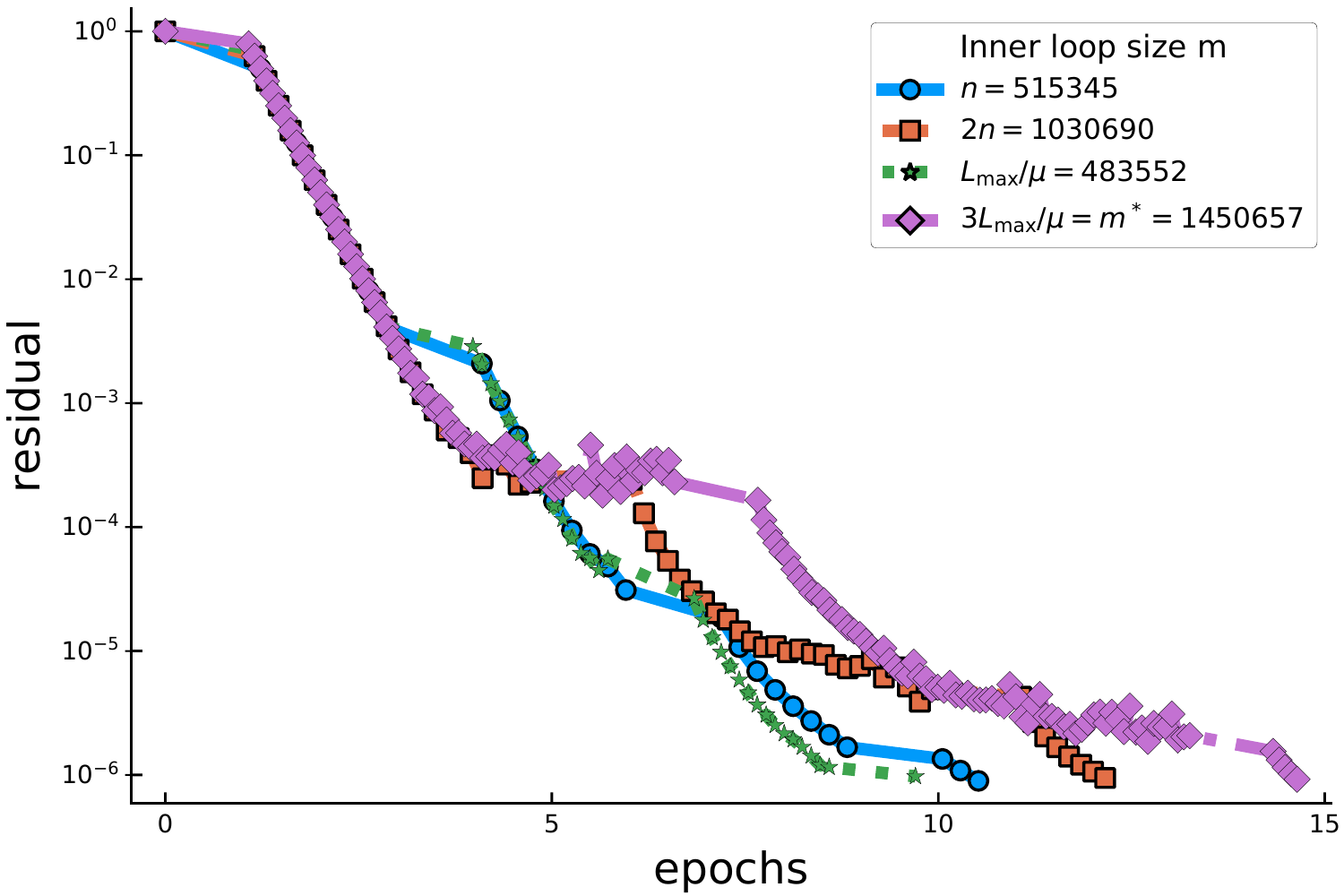}
        \includegraphics[width=0.45\textwidth]{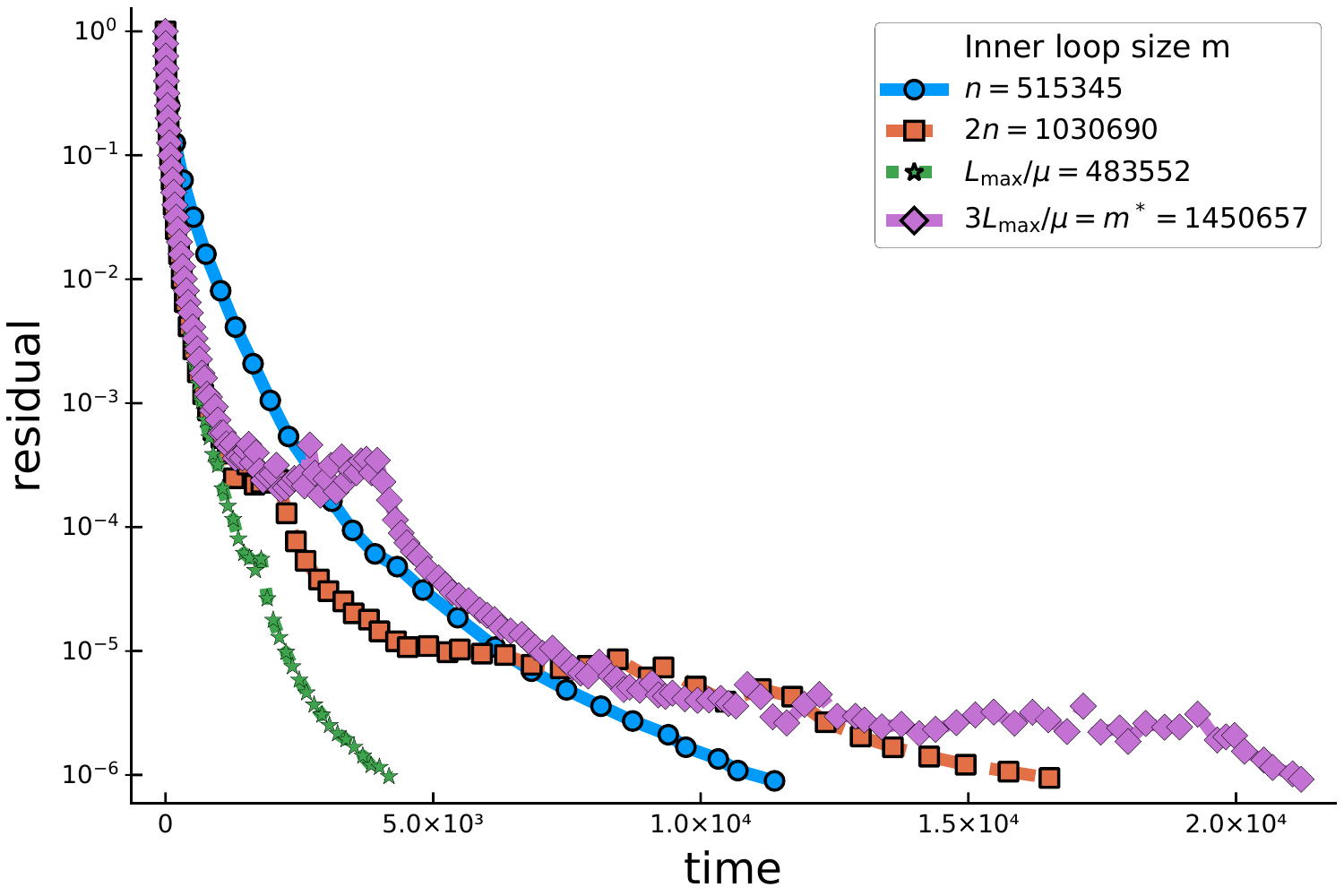}
        \caption{$\lambda = 10^{-3}$}
      \end{subfigure}
  \caption{Optimality of our inner loop size $m^* = 3L_{\max}/\mu$ for \textit{Free-SVRG} on the \textit{YearPredictionMSD} data set.}
  \label{fig:exp2B_YearPredictionMSD}
  \end{center}
  \vskip -0.2in
\end{figure}

\begin{figure}[h]
 \vskip 0.2in
 \begin{center}
   \begin{subfigure}[b]{0.48\textwidth}
       \includegraphics[width=\textwidth]{exp2b/legend_exp2b_horizontal_new}
     \end{subfigure}\\
     \begin{subfigure}[b]{\textwidth}
        \centering
        \includegraphics[width=0.45\textwidth]{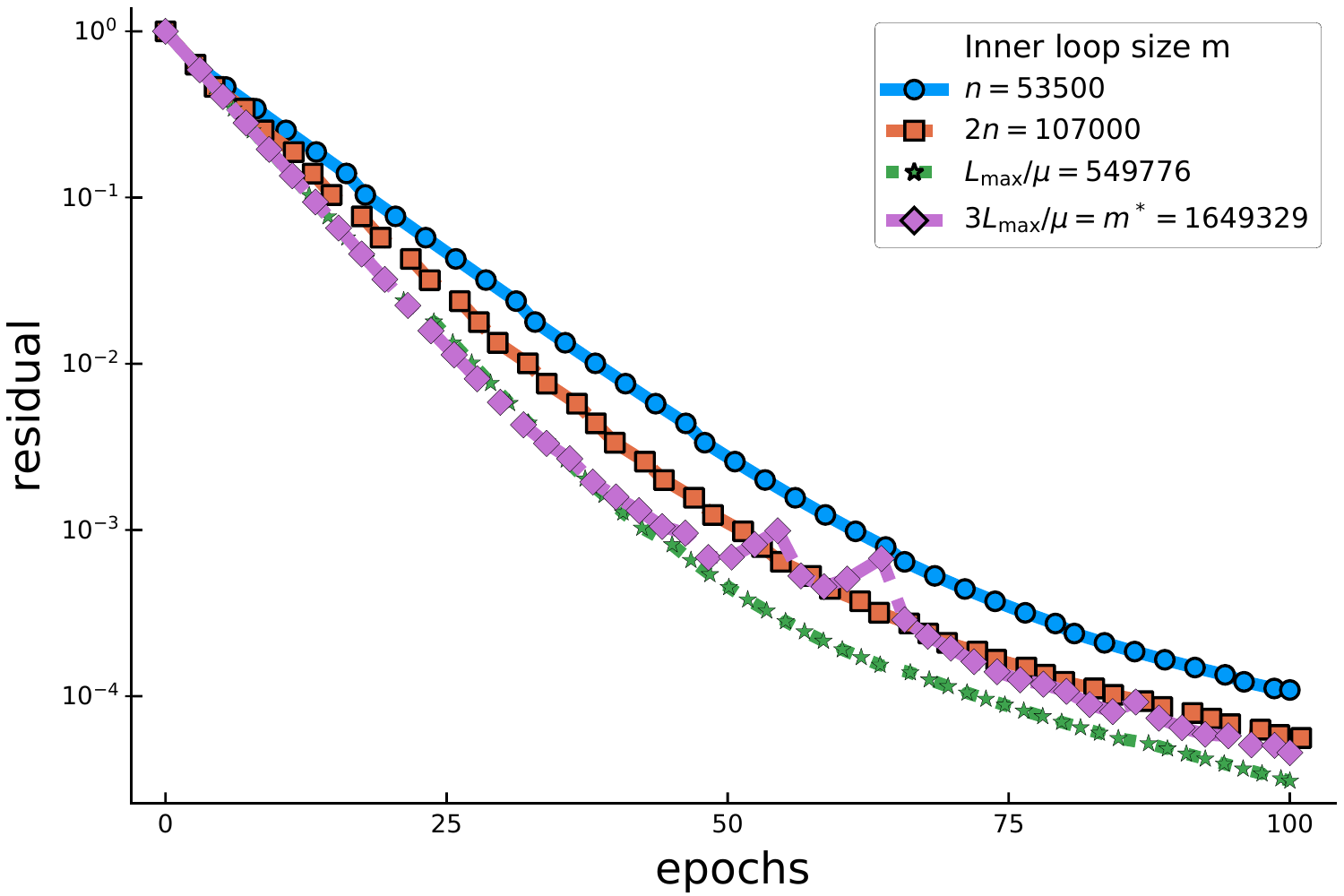}
        \includegraphics[width=0.45\textwidth]{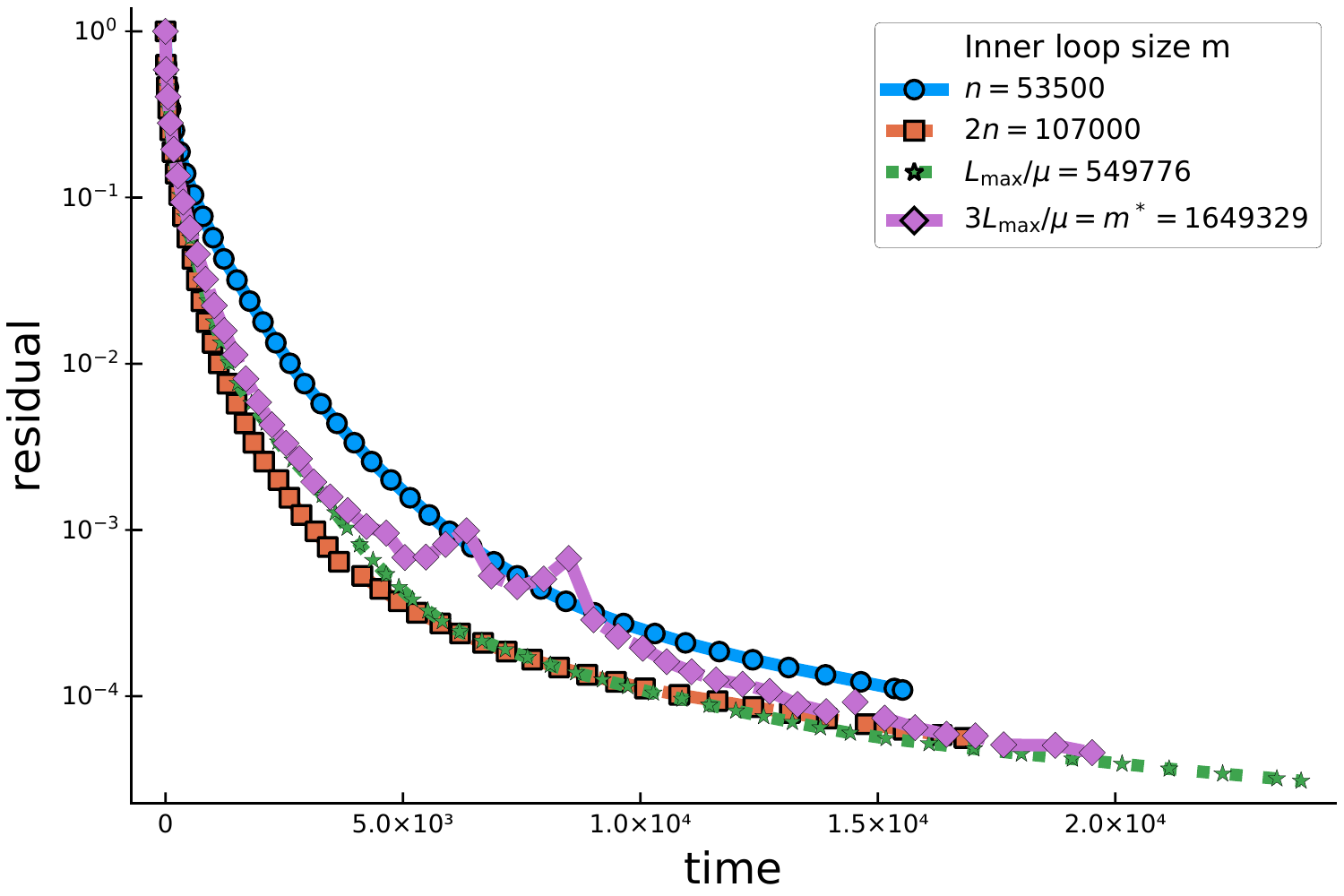}
        \caption{$\lambda = 10^{-1}$}
     \end{subfigure}\\
     \begin{subfigure}[b]{\textwidth}
        \centering
        \includegraphics[width=0.45\textwidth]{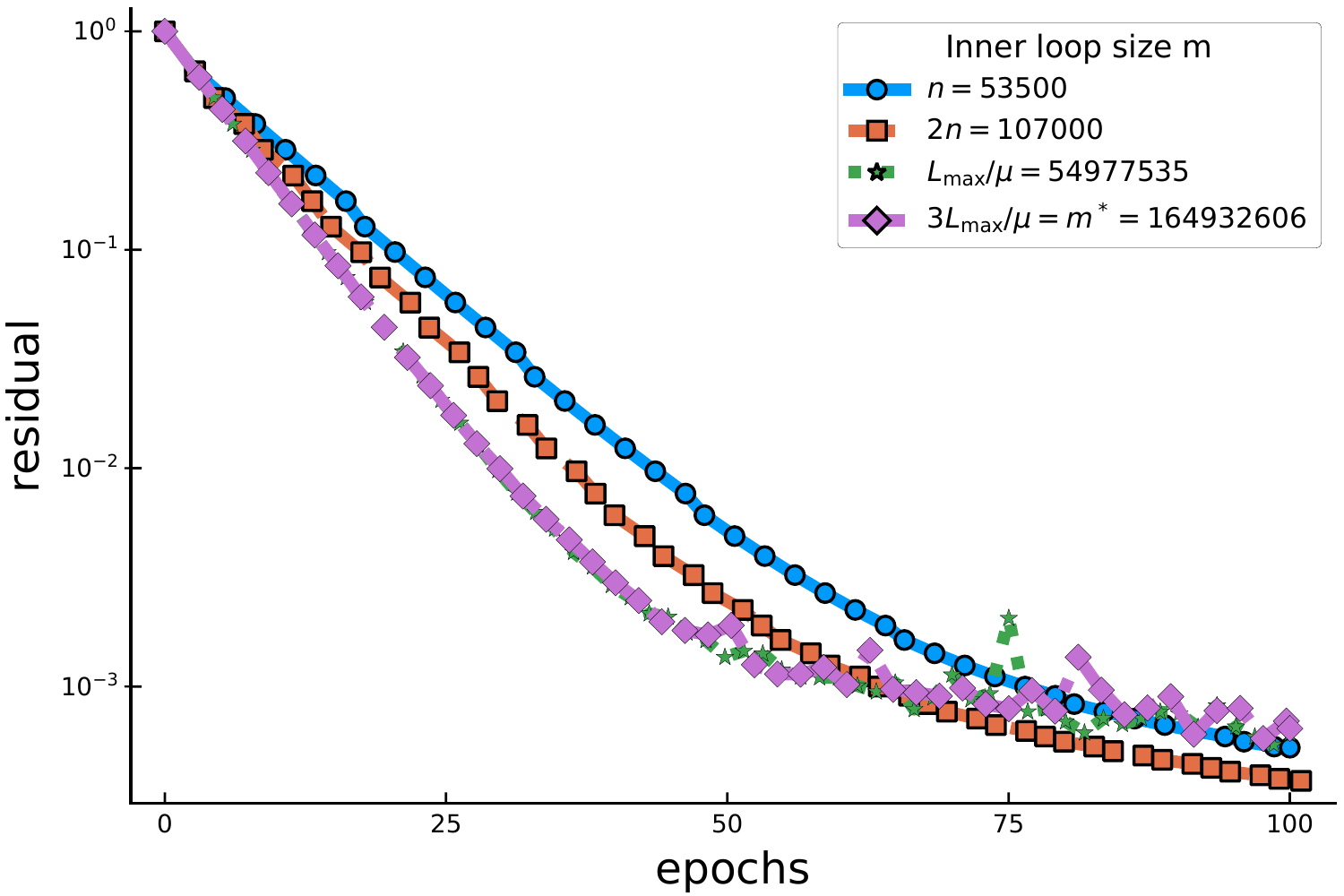}
        \includegraphics[width=0.45\textwidth]{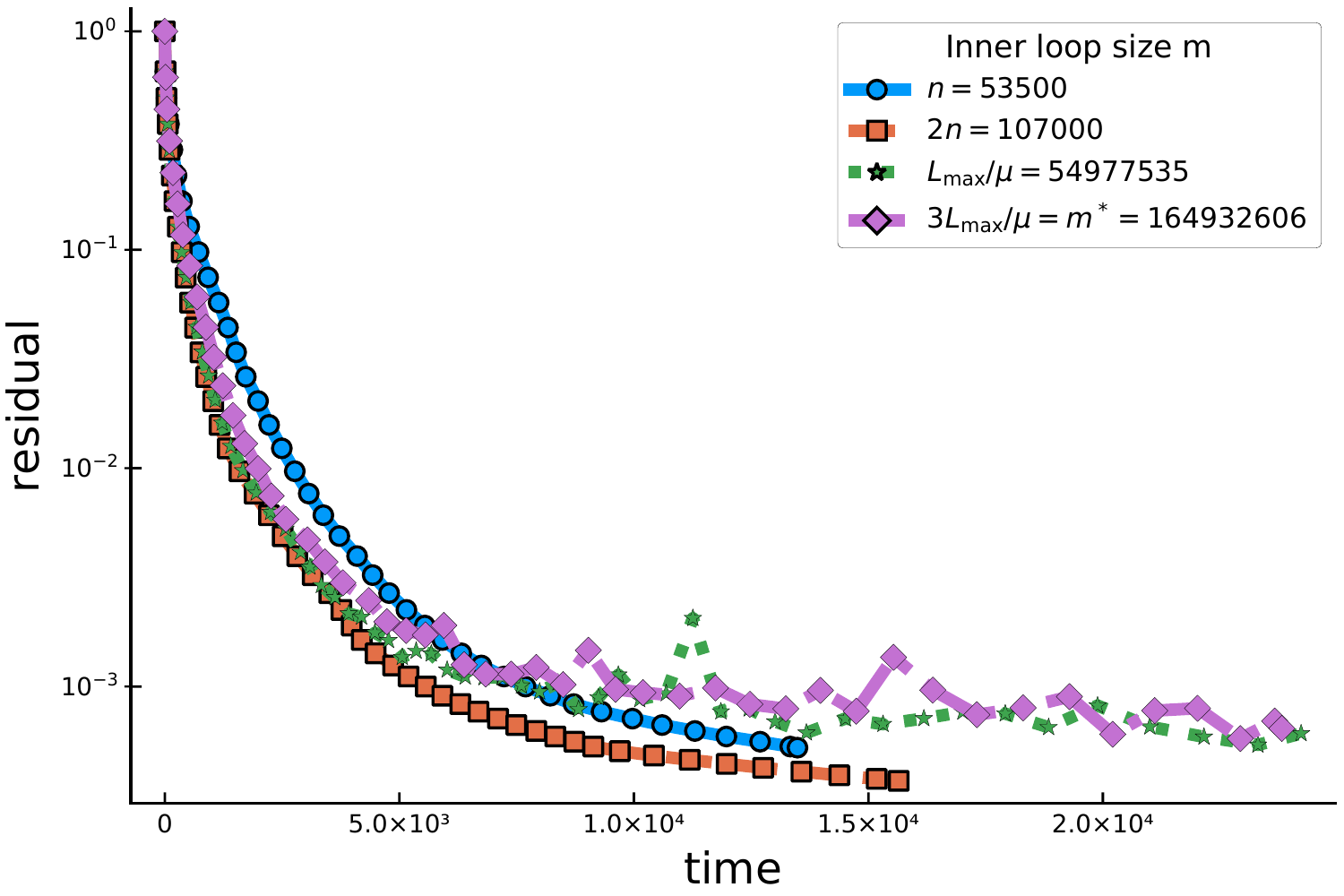}
        \caption{$\lambda = 10^{-3}$}
     \end{subfigure}
 \caption{Optimality of our inner loop size $m^* = 3L_{\max}/\mu$ for \textit{Free-SVRG} on the \textit{slice} data set.}
 \label{fig:exp2B_slice}
 \end{center}
 \vskip -0.2in
\end{figure}

\begin{figure}[!htb]
  \vskip 0.2in
  \begin{center}
    \begin{subfigure}[b]{0.48\textwidth}
        \includegraphics[width=\textwidth]{exp2b/legend_exp2b_horizontal_new}
      \end{subfigure}\\
      \begin{subfigure}[b]{\textwidth}
        \centering
        \includegraphics[width=0.45\textwidth]{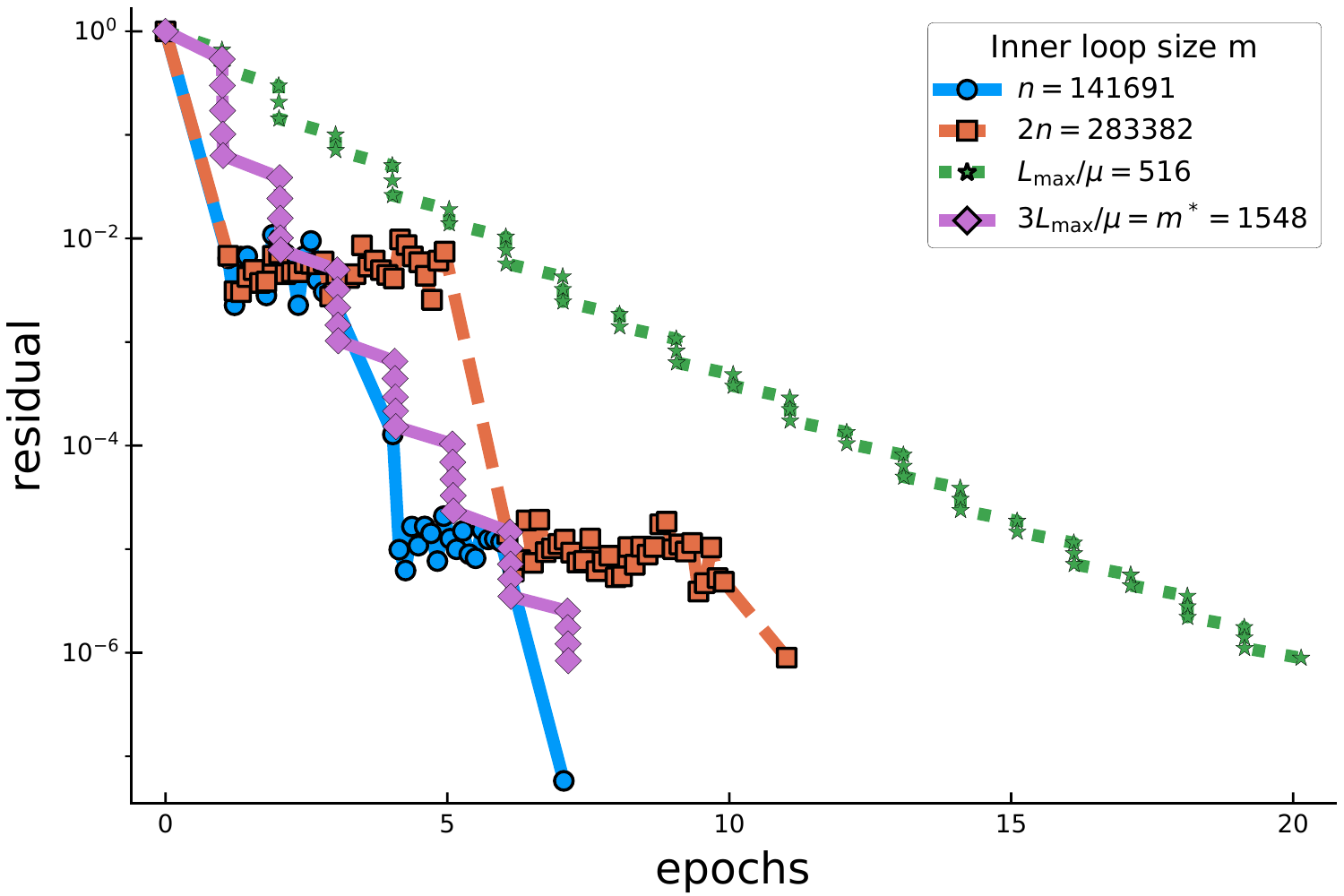}
        \includegraphics[width=0.45\textwidth]{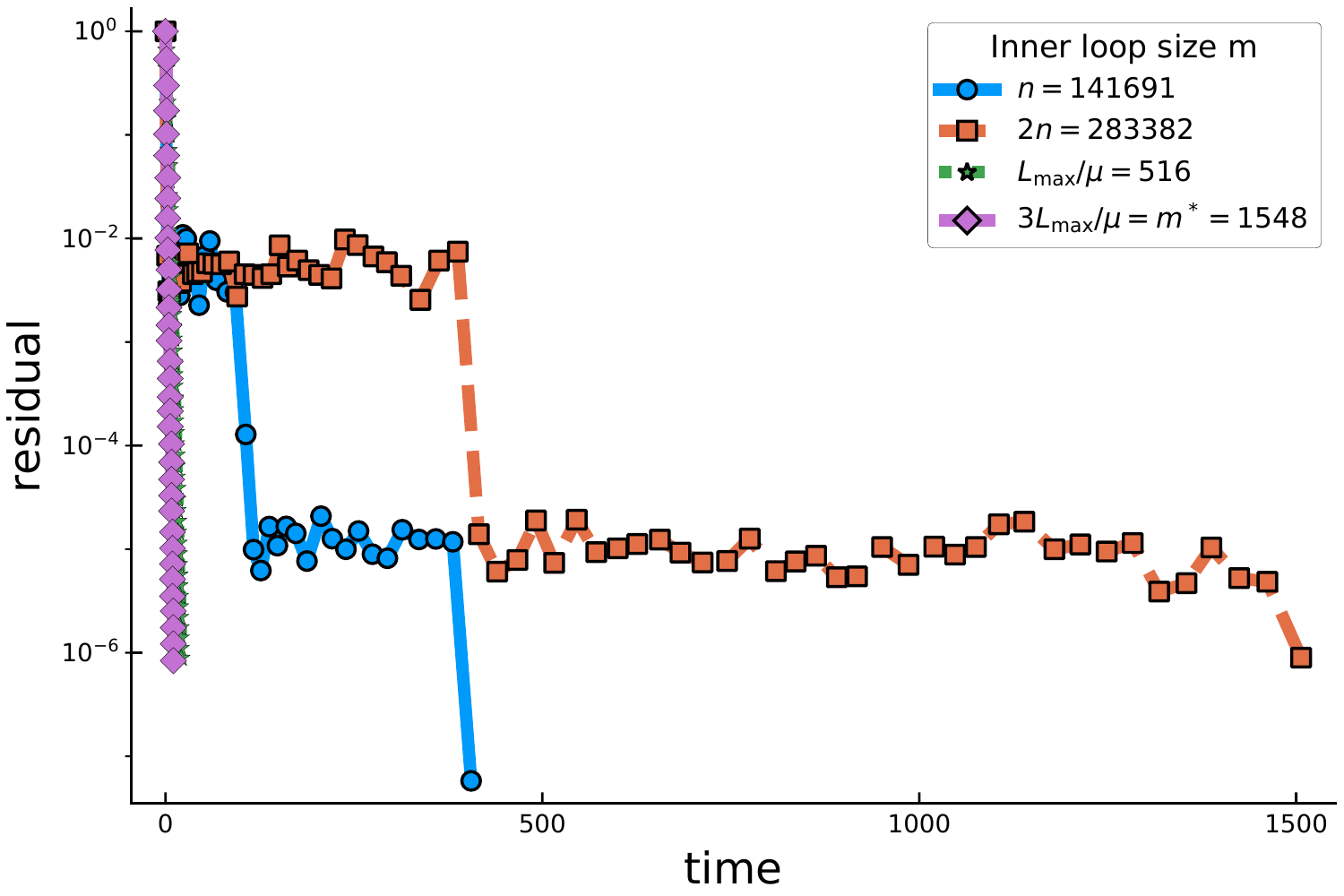}
        \caption{$\lambda = 10^{-1}$}
      \end{subfigure}\\
      \begin{subfigure}[b]{\textwidth}
        \centering
        \includegraphics[width=0.45\textwidth]{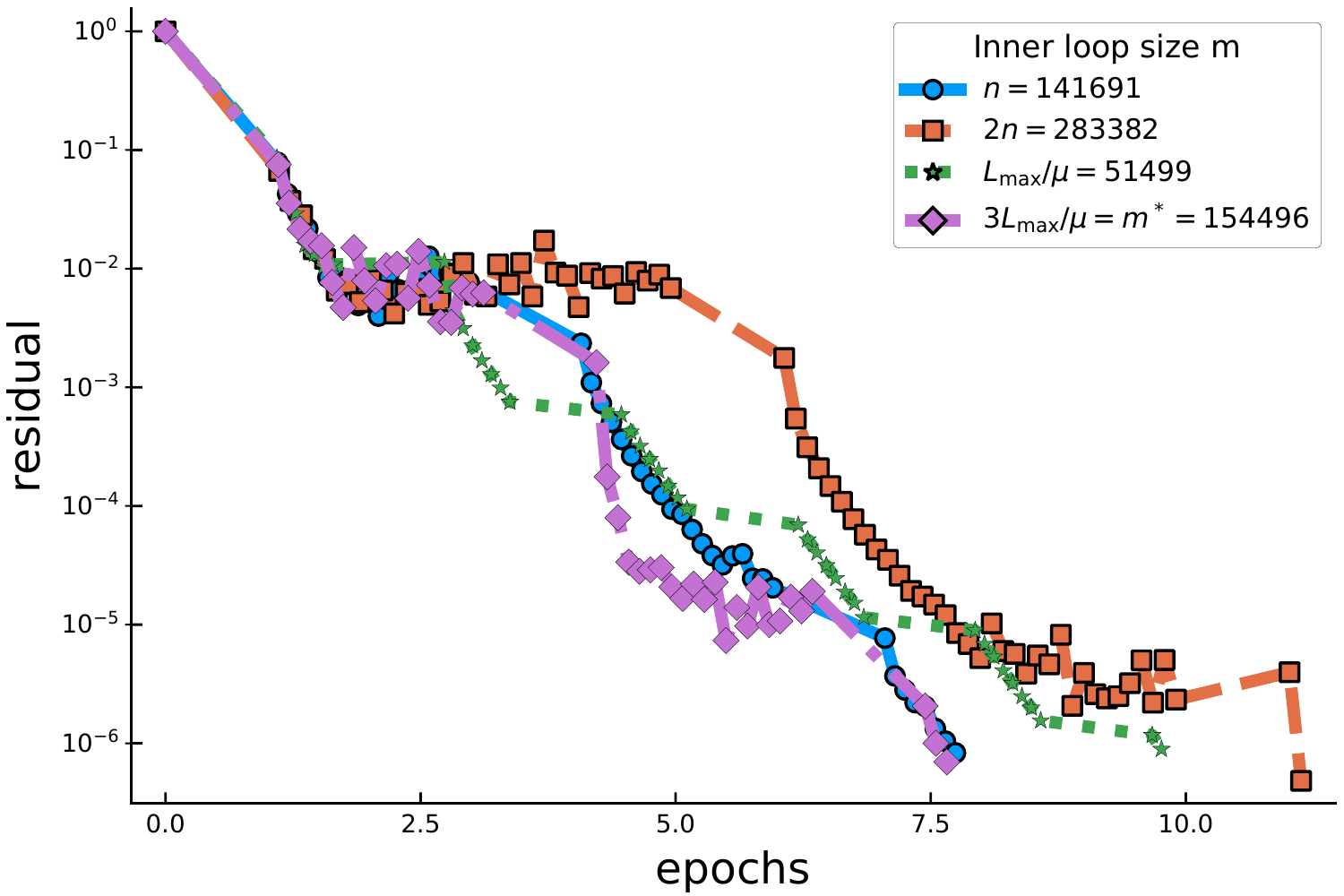}
        \includegraphics[width=0.45\textwidth]{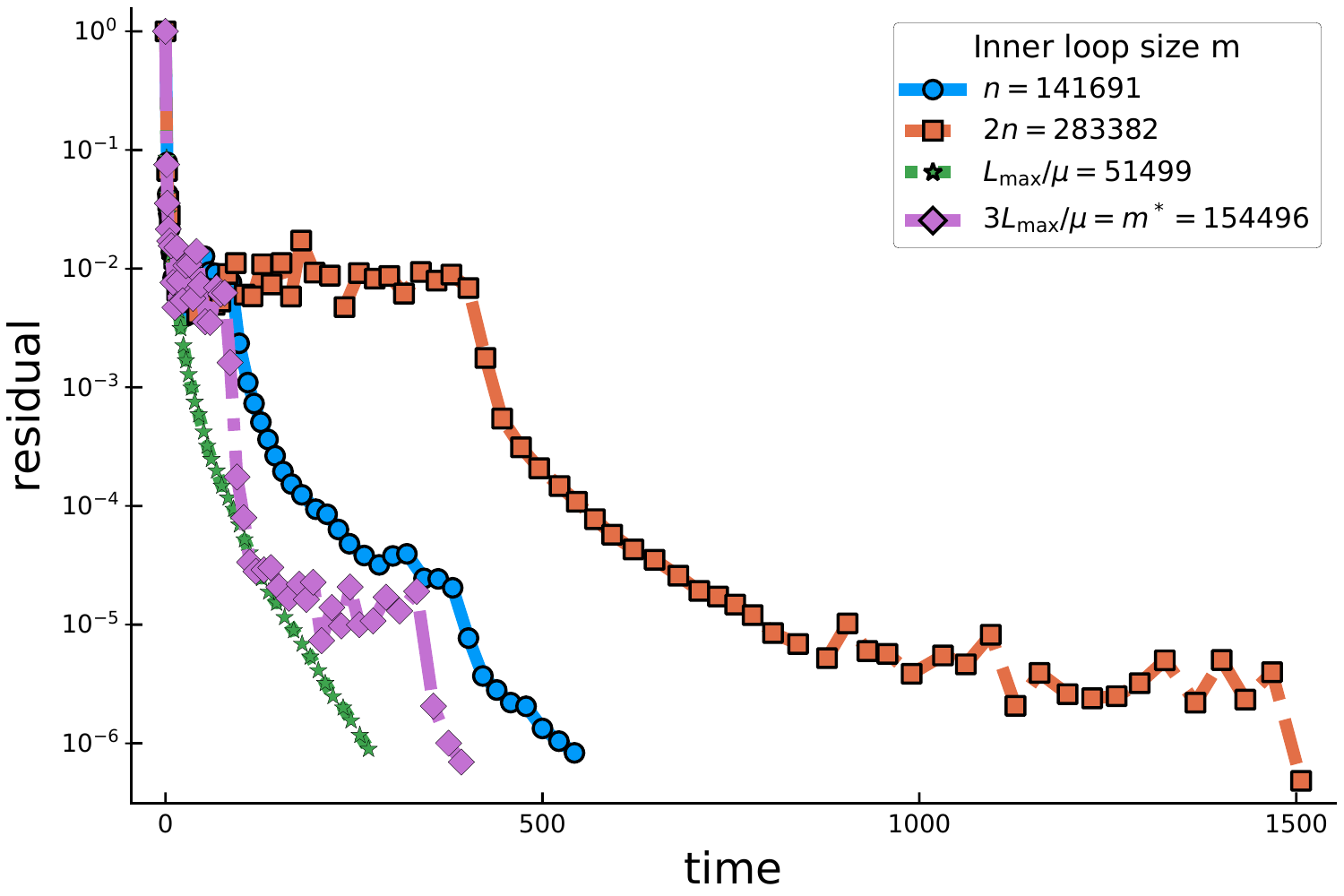}
        \caption{$\lambda = 10^{-3}$}
      \end{subfigure}
  \caption{Optimality of our inner loop size $m^* = 3L_{\max}/\mu$ for \textit{Free-SVRG} on the \textit{ijcnn1} data set.}
  \label{fig:exp2B_ijcnn1}
  \end{center}
  \vskip -0.2in
\end{figure}

\begin{figure}[!htb]
  \vskip 0.2in
  \begin{center}
    \begin{subfigure}[b]{0.48\textwidth}
        \includegraphics[width=\textwidth]{exp2b/legend_exp2b_horizontal_new}
      \end{subfigure}\\
      \begin{subfigure}[b]{\textwidth}
        \centering
        \includegraphics[width=0.45\textwidth]{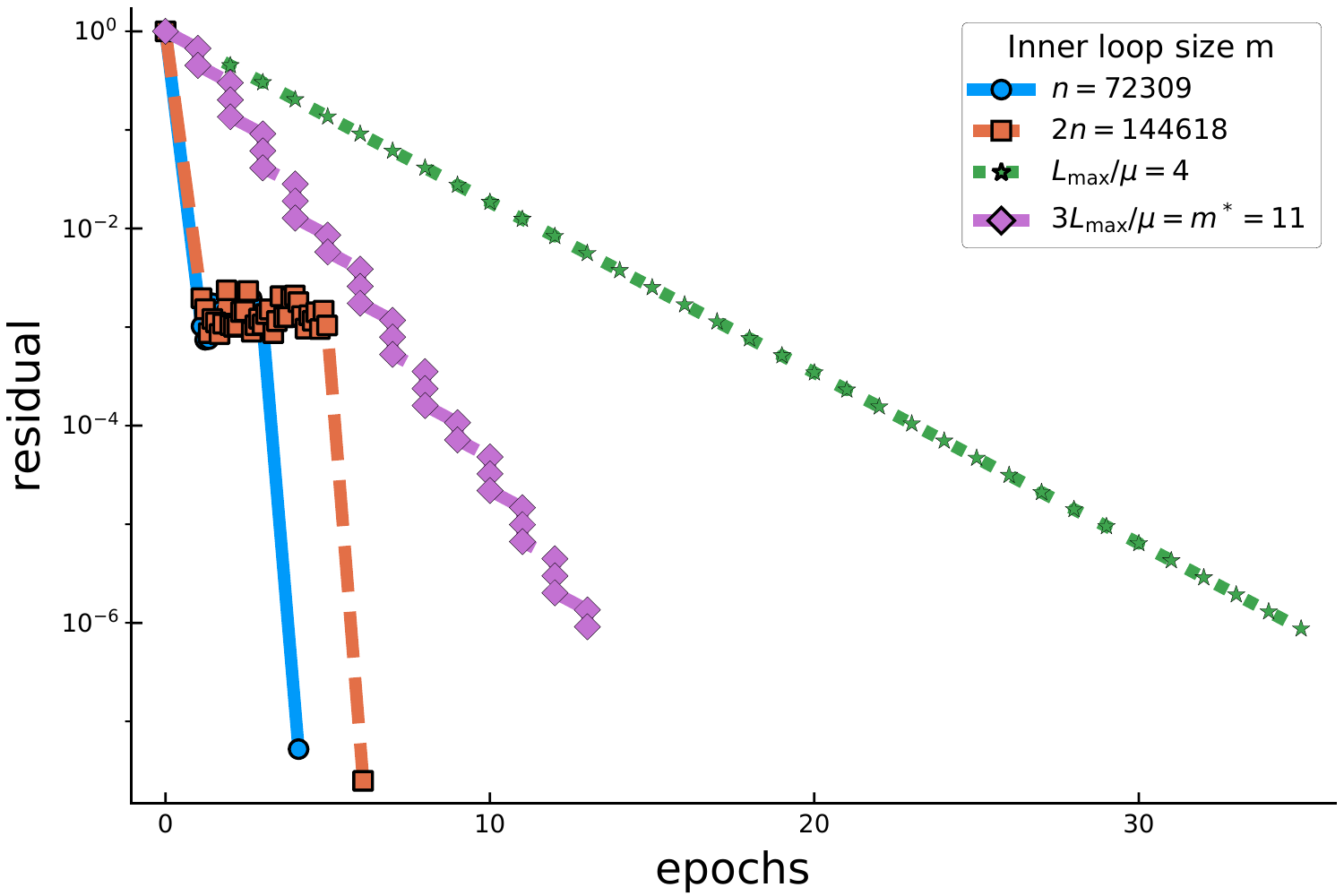}
        \includegraphics[width=0.45\textwidth]{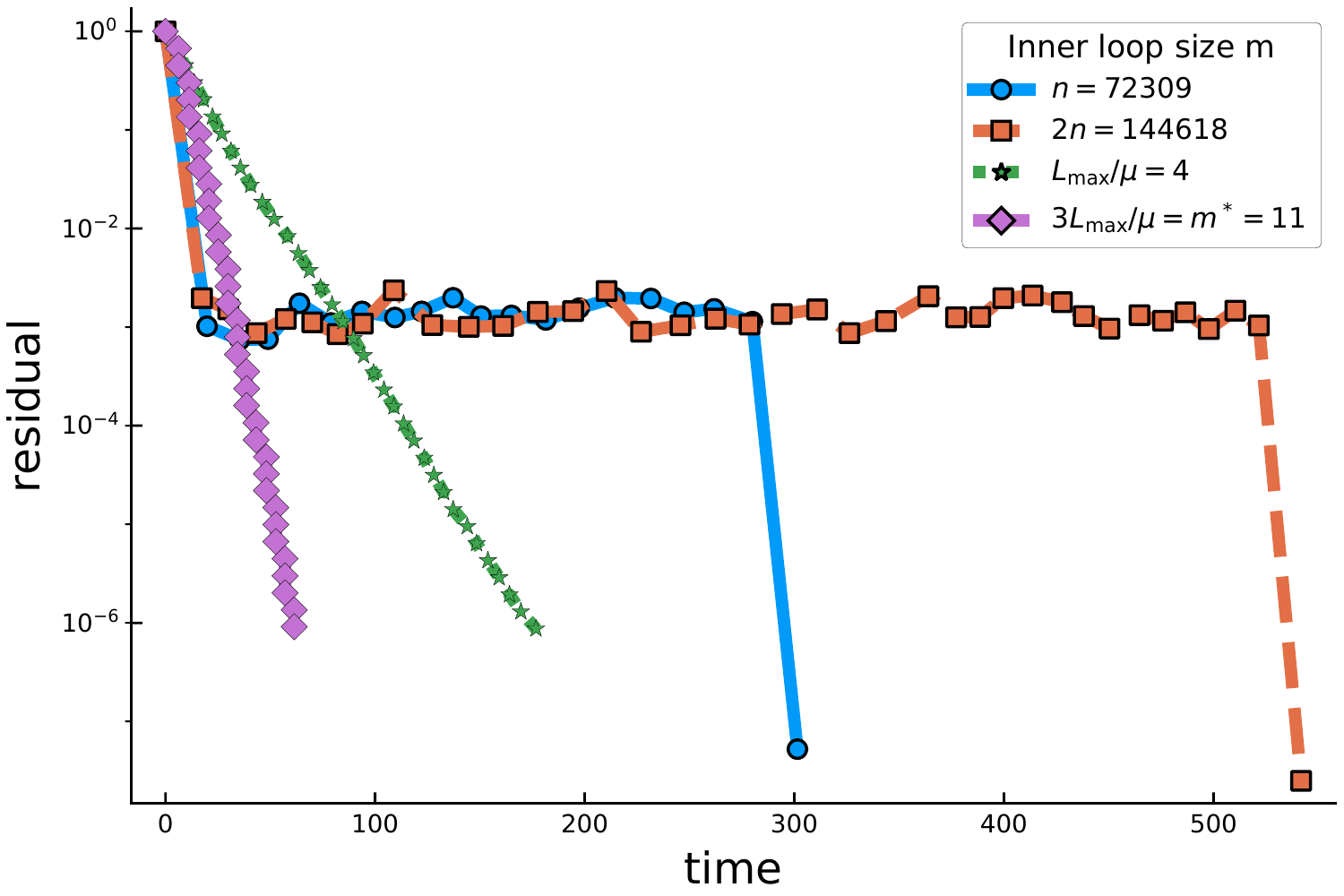}
        \caption{$\lambda = 10^{-1}$}
        \label{fig:exp2B_real-sim_1e-1}
      \end{subfigure}\\
      \begin{subfigure}[b]{\textwidth}
        \centering
        \includegraphics[width=0.45\textwidth]{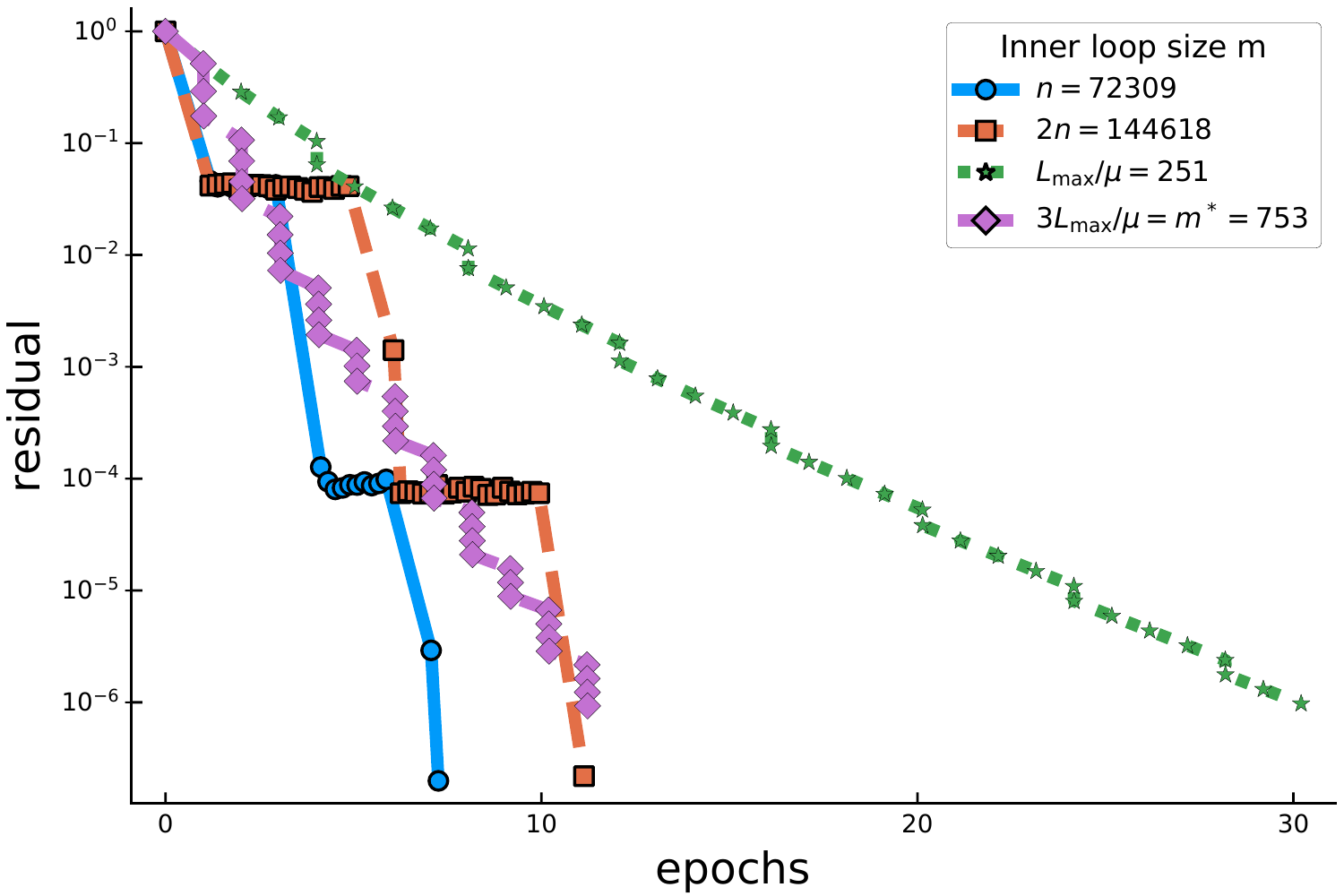}
        \includegraphics[width=0.45\textwidth]{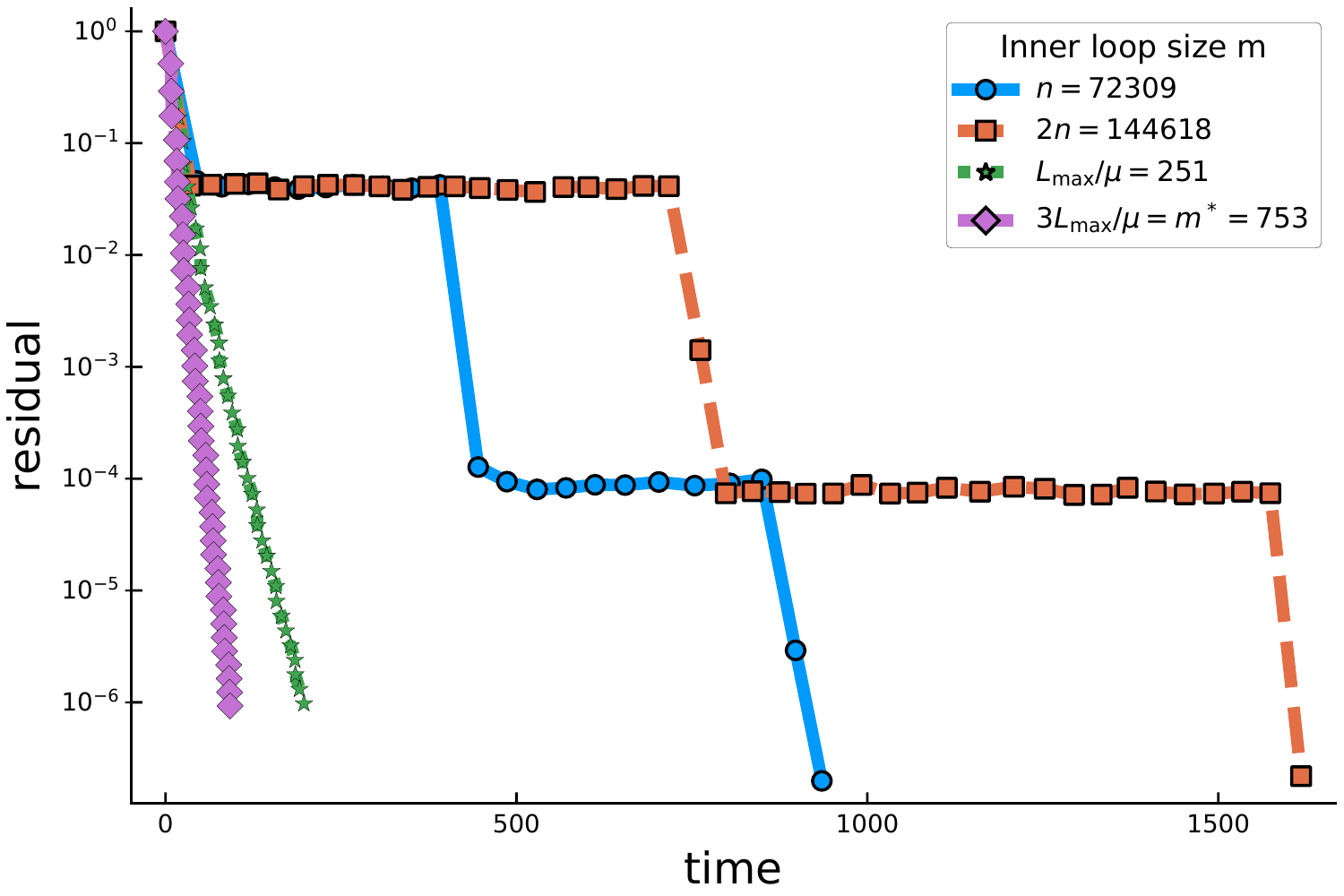}
        \caption{$\lambda = 10^{-3}$}
      \end{subfigure}
  \caption{Optimality of our inner loop size $m^* = 3L_{\max}/\mu$ for \textit{Free-SVRG} on the \textit{real-sim} data set.}
  \label{fig:exp2B_real-sim}
  \end{center}
  \vskip -0.2in
\end{figure}

\end{document}